\documentclass[11pt]{article}
\usepackage [margin=3cm]{geometry}
\usepackage{fancyhdr}
\usepackage{amsfonts,amsmath,amssymb}
\usepackage{todonotes}
\usepackage{latexsym,microtype}
\usepackage{theorem}
\usepackage{xr,xr-hyper}
\usepackage{epsfig,euscript}
\usepackage{nicefrac}
\usepackage{xypic,xspace,tikz}
\usepackage{bbm,ifpdf}
\usepackage{setspace,mathrsfs}
\usepackage [margin=3cm]{geometry}
\ifpdf
  \usepackage{pdfsync,hyperref}
\fi

\newtheorem{theorem}{Theorem}[section]
\newtheorem{lemma}[theorem]{Lemma}
\newtheorem{proposition}[theorem]{Proposition} 
\newtheorem{corollary}[theorem]{Corollary}

\newtheorem{conjecture}[theorem]{Conjecture}

{
\theorembodyfont{\rmfamily}
\theoremstyle{plain}
\newtheorem{definition}[theorem]{Definition}
\newtheorem{example}[theorem]{Example}

\newtheorem{remark}[theorem]{Remark}
}

\newcommand{\qed}{\hfill \mbox{$\Box$}\medskip\newline}
\newenvironment{proof}{\noindent {\bf Proof:}}{\qed}
\newenvironment{proofnoqed}{\noindent {\bf Proof:}}{\par}

\newenvironment{proofIrv}{\noindent {\bf Proof of Proposition \ref{shuf-twist}:}}{\qed \par}

\newenvironment{proofsf}{\noindent {\bf Proof of Theorem \ref{special filtrations}:}}{\qed \par}

\newcommand{\Spec}{\operatorname{Spec}}

\newcommand{\id}{\operatorname{id}}

\newcommand{\codim}{\operatorname{codim}}
\renewcommand{\dim}{\operatorname{dim}}
\newcommand{\Ann}{\operatorname{Ann}}
\newcommand{\Sym}{\operatorname{Sym}}
\newcommand{\rk}{\operatorname{rk}}

\newcommand{\Gr}{\operatorname{Gr}}
\newcommand{\fJ}{\mathfrak{J}}
\newcommand{\fP}{\mathfrak{P}}

\newcommand{\Bv}{\mathbf{v}}
\newcommand{\Bw}{\mathbf{w}}
\newcommand{\coho}{\mathbbm{H}}

\newcommand{\scrM}{\mathscr{M}}

\newcommand{\scrN}{\mathscr{N}}
\newcommand{\HCa}{\mathbf{HC}^{\operatorname{a}}}
\newcommand{\HCg}{\mathbf{HC}^{\operatorname{g}}}
\newcommand{\HCaS}{\HCa_{\!S}}
\newcommand{\HCgS}{\HCg_{\!S}}
\newcommand{\slehat}{\mathfrak{\widehat{sl}}_e}
\newcommand{\HCadS}{\HCa_{\!\partial S}}
\newcommand{\HCgdS}{\HCg_{\!\partial S}}
\newcommand{\cOaS}{\cOa^{S}}
\newcommand{\cOgS}{\cOg^{S}}
\newcommand{\cOadS}{\cOa^{\partial S}}
\newcommand{\cOgdS}{\cOg^{\partial S}}

\newcommand{\LLoc}{\mathbb{L}\!\operatorname{Loc}}
\newcommand{\Rsecs}{\mathbb{R}\Gamma_\bS}

\newcommand{\tCg}{\widetilde{\C^2/\mck }}
\newcommand{\scrD}{\mathscr{D}}
\newcommand{\Hilb}{\mathsf{Hilb}}
\newcommand{\scrB}{\mathscr{B}}

\newcommand{\wt}{\operatorname{wt}}

\newcommand{\Z}{\mathbb{Z}}

\newcommand{\N}{\mathbb{N}}
\newcommand{\R}{\mathbb{R}}
\newcommand{\C}{\mathbb{C}}

\newcommand{\la}{\leftarrow}
\newcommand{\Ra}{\Rightarrow}

\newcommand{\M}{\mathfrak{M}}

\newcommand{\sHom}{{\cal H}\mathit{om}}

\renewcommand{\la}{\lambda}
\newcommand{\cs}{\C^\times}

\renewcommand{\a}{\alpha}
\renewcommand{\b}{\beta}

\newcommand{\Hom}{\operatorname{Hom}}
\newcommand{\A}{ A}

\newcommand{\IH}{I\! H}
\newcommand{\IC}{\operatorname{IC}^{\scriptscriptstyle{\bullet}}}
\newcommand{\hookto}{{\hookrightarrow}}

\newcommand{\quiv}{Q}
\newcommand{\mck}{\Gamma}

\newcommand{\cP}{\mathcal{P}}

\newcommand{\cOg}{\cO_{\!\operatorname{g}}}
\newcommand{\cOa}{\cO_{\!\operatorname{a}}}

\newcommand{\cO}{\mathcal{O}}
\renewcommand{\cL}{\mathcal{L}}

\newcommand{\cT}{\mathcal{T}}

\newcommand{\cM}{\mathcal{M}}
\newcommand{\cC}{\mathcal{C}}

\newcommand{\cI}{\mathcal{I}}
\newcommand{\cJ}{\mathcal{J}}
\newcommand{\cK}{\mathcal{K}}

\newcommand{\dOg}{D^b_{\cOg}(\Dmod)}
\newcommand{\dOa}{D^b_{\cOa}(\Amod)}

\newcommand{\becircled}{\mathaccent "7017}

\newcommand{\Ext}{\operatorname{Ext}}
\newcommand{\fS}{\mathfrak{S}}
\newcommand{\cF}{\mathcal{F}}
\newcommand{\cS}{\mathcal{S}}
\renewcommand{\cR}{\mathcal{R}}
\renewcommand{\cD}{\mathcal{D}}
\renewcommand{\cH}{\mathcal{H}}
\newcommand{\Ctw}{C_{\operatorname{tw}}}
\newcommand{\Csh}{C_{\operatorname{sh}}}
\newcommand{\cHt}{\mathcal{H}_{\operatorname{tw}}}
\newcommand{\cHs}{\mathcal{H}_{\operatorname{sh}}}
\newcommand{\T}{\mathbb{T}}
\newcommand{\gr}{\operatorname{gr}}

\newcommand{\cQ}{\mathcal{Q}}

\newcommand{\Loc}{\operatorname{Loc}}

\newcommand{\cE}{\mathcal{E}}

\newcommand{\End}{\operatorname{End}}

\newcommand{\Mirkovic}{Mirkovi\'c\xspace}

\newcommand{\excise}[1]{}
\newcommand{\bT}{\mathbb{T}}
\newcommand{\zetaw}{\zeta_{w}}
\newcommand{\bt}{\mathbbm{t}}
\newcommand{\bS}{\mathbb{S}}
\newcommand{\mg}{\mathfrak{g}}
\newcommand{\ml}{\mathfrak{l}}
\newcommand{\mm}{\mathfrak{m}}

\newcommand{\mh}{\mathfrak{h}}
\newcommand{\mt}{\mathfrak{t}}
\newcommand{\fp}{\mathfrak{p}}
\newcommand{\cbi}{{_{\la'}\!\cT_{\la}}}
\newcommand{\bi}{{_{\la'}\!T_{\la}}}

\newcommand{\pphpp}{(\!(h)\!)}
\newcommand{\pptpp}{(\!(t)\!)}

\newcommand{\cN}{\mathcal{N}}

\newcommand{\hkr}{\mathcal{H}(k,r)}

\newcommand{\mkr}{\mathcal{M}(k,r)}

\newcommand{\fM}{\mathfrak{M}}
\newcommand{\fN}{\mathfrak{N}}

\newcommand{\fX}{\mathfrak{X}}
\renewcommand{\Loc}{\operatorname{Loc}}
\newcommand{\secs}{\Gamma_\bS}

\newcommand{\Amod}{A\mmod}

\newcommand{\Dmod}{\cD\mmod}

\newcommand{\mmod}{\operatorname{-mod}}
\newcommand{\gmmod}{\operatorname{-gmod}}
\renewcommand{\and}{\qquad\text{and}\qquad}

\newcommand{\mb}{\mathfrak{b}}

\newcommand{\fL}{\mathfrak{L}}
\newcommand{\rWeyl}{%
\mathbb{W}}
\newcommand{\mk}{\mathfrak{k}}

\newcommand{\suppc}{\operatorname{CC}}
\newcommand{\supp}{\operatorname{Supp}}

\newcommand{\fQ}{\mathfrak{Q}}

\newcommand{\gst}{\mho}
\newcommand{\gcst}{\Omega}

\newcommand{\gsi}{\Lambda}

\newcommand{\Lleq}{\overset{L}{\leq}}
\newcommand{\Rleq}{\overset{R}{\leq}}
\newcommand{\Tleq}{\overset{2}{\leq}}
\usetikzlibrary{matrix}
\newcommand{\Htz}{H^2(\fM;\Z)}
\newcommand{\Htr}{H^2(\fM;\R)}
\newcommand{\Ht}{H^2(\fM;\C)}

\newcommand{\aone}{\mathbb{A}^{\! 1}}
\newcommand{\spe}{\mathscr{S}^{\operatorname{sp}}}
\newcommand{\all}{\mathscr{S}}

\ifpdf
  \usepackage{pdfsync,hyperref}
\fi
\newcommand{\hmon}{h^{\nicefrac{-1}{n}}}
\newcommand{\hon}{h^{\nicefrac{1}{n}}}
\newcommand{\Zmod}{Z\!\mmod}

\newcommand{\Aka}{A_{k,\a}}
\newcommand{\Bka}{B_{k,\a}}
\newcommand{\scrU}{\mathscr{U}}

\newcommand{\Lotimes}{\overset{L}\otimes}
\newcommand{\op}{{\operatorname{op}}}
\newcommand{\an}{{\operatorname{an}}}

\newcommand{\fZ}{\mathfrak{Z}}
\newcommand{\CLg}{\mathcal{C}^\fL}
\newcommand{\CMpg}{\mathcal{C}^{\fM^+}}
\newcommand{\CMpa}{C^{\fM^+_0}}
\newcommand{\CLa}{C^{\fL_0}}

\newcommand{\DCLg}{D_{\fL}^b(\Dmod)}
\newcommand{\DCLa}{D_{\fL_0}^b(A\mmod)}

\newcommand{\dimfM}{2d}
\newcommand{\HZZ}{H^{4d}_\fZ(\fM\times \fM; \Z)}
\newcommand{\HZC}{H^{4d}_\fZ(\fM\times \fM; \C)}
\newcommand{\HLZ}{H^{\dimfM}_\fL(\fM; \Z)}

\newcommand{\HLCh}{H^{\dimfM}_\fL\big(\fM; \C\pphpp \big)}

\newcommand{\hcd}{H^{\dimfM}_{\fM_0^+}}
\newcommand{\bm}{H^{\dimfM}_{\fM^+}\big(\fM; \Z\big)}
\newcommand{\bmc}{H^{\dimfM}_{\fM^+}\big(\fM; \C\big)}

\newcommand{\scrv}{\tilde{v}}
\newcommand{\SL}{\operatorname{SL}}
\newcommand{\tC}{\tilde{\cC}}
\newcommand{\LCP}{\operatorname{LCP}}
\newcommand{\LtC}{\LCP(\tC)}
\newcommand{\Ve}[1]{\Delta_{#1}}
\newcommand{\Pro}[1]{P_{#1}}
\newcommand{\Si}[1]{L_{#1}}

\newcommand{\fXeP}{{\fX_{\!P}^e}}
\newcommand{\fXePo}{{\fX_{\!P,0}^e}}
\newcommand{\Htzf}{\Htz_{\operatorname{free}}}
\newcommand{\Ham}{{\operatorname{Ham}}}
\newcommand{\dHCa}{D^b_{\HCa}(A\operatorname{-mod-}A)}
\newcommand{\dHCg}{D^b_{\HCg}(\cD\operatorname{-mod-}\cD)}
\newcommand{\dHCaS}{D^b_{\HCaS}(A\operatorname{-mod-}A)}
\newcommand{\dHCgS}{D^b_{\HCgS}(\cD\operatorname{-mod-}\cD)}
\newcommand{\dHCadS}{D^b_{\HCadS}(A\operatorname{-mod-}A)}
\newcommand{\dHCgdS}{D^b_{\HCgdS}(\cD\operatorname{-mod-}\cD)}
\newcommand{\dOaS}{D^b_{\cOaS}(\Amod)}
\newcommand{\dOadS}{D^b_{\cOadS}(\Amod)}

\begin{document}
\spacing{1.5}

\noindent {\Large \bf 
Quantizations of conical symplectic resolutions 
II:\\ category $\cO$ and symplectic duality
}
\\
\spacing{1.2}
\noindent
{\bf Tom Braden}\footnote{Supported by NSA grants H98230-08-1-0097 and H98230-11-1-0180.}\\
Department of Mathematics and Statistics, University of Massachusetts,
Amherst, MA 01003\smallskip \\
{\bf Anthony Licata}\footnote{Supported by an ARC Discovery Early Career fellowship.}\\
Mathematical Sciences Institute, Australian National University,
Canberra, ACT 0200\smallskip \\
{\bf Nicholas Proudfoot}\footnote{Supported by NSF grant DMS-0950383.}\\
Department of Mathematics, University of Oregon,
Eugene, OR 97403\smallskip \\
{\bf Ben Webster}\footnote{Supported by NSF grant DMS-1151473.}\\
Department of Mathematics, University of Virginia, Charlottesville, VA 22903
\bigskip\\
{\small
\begin{quote}
\noindent {\em Abstract.}
We define and study category $\cO$ for a symplectic resolution, generalizing
the classical BGG category $\cO$, which is associated with the
Springer resolution.  This includes the development of intrinsic
properties paralleling the BGG case, such as a highest weight
structure and analogues of twisting and shuffling functors, along with an extensive discussion of individual examples.  

We observe that category $\cO$ is often Koszul, and its Koszul dual is
often equivalent to category $\cO$ for a different symplectic
resolution.  This leads us to define the notion of a symplectic
duality between symplectic resolutions, which is a collection of
isomorphisms between representation theoretic and geometric
structures, including a Koszul duality between the two
categories.  This duality has various cohomological consequences, including
(conjecturally) an identification of two geometric realizations, due to Nakajima and Ginzburg/\Mirkovic-Vilonen,
of weight spaces of simple representations of simply-laced simple algebraic groups.

An appendix by Ivan Losev establishes a key step in the proof that $\cO$ is highest weight.
\end{quote} } \bigskip

\section{Introduction}
In this paper, we have two main goals:
\begin{itemize}
\item to introduce a version of category $\cO$ attached to a
  symplectic variety with extra structure,
\item to describe a conjectured relationship, which we call {\bf symplectic
duality}, between pairs of symplectic varieties.  The most striking
manifestation of this duality is a Koszul duality between the
associated categories $\cO$.  
\end{itemize}
The motivating example is the classical BGG category $\cO$, and
the remarkable theorem of Beilinson, Ginzburg and Soergel \cite{BGS96} showing that
a regular integral block of category $\cO$ is Koszul self-dual.
In our formulation, this means that the Springer resolution of the nilpotent cone
is self-dual as a symplectic variety.

Our perspective throughout is to study the geometry of symplectic
varieties using deformation quantizations and their representation
theory.   The specific varieties that we want to study are called
{\bf conical symplectic resolutions}. The prequel to this paper \cite{BLPWquant} 
introduced these varieties, their quantizations, and the categories of modules over these quantizations.
Here we will concentrate on a particular subcategory of this module category: category $\cO$.

Versions of category $\cO$ have appeared in many places in the
literature: for representations of $U(\mg)$ in \cite{BGG}, for
rational Cherednik algebras in \cite{GGOR}, for W-algebras in
\cite{BGK,LosO}, and for hypertoric enveloping algebras in
\cite{BLPWtorico}.  Our general definition includes
all of these particular examples as particular cases, and 
we are able to prove many basic facts about these categories in a unified 
way.  We will discuss the details of their structure further below.  

There is a one striking observation about these categories that we
wish to give special prominence: they are often standard Koszul, and
Koszul dual to the category $\cO$ attached to a
 different variety.  This is the heart of our definition of symplectic 
 duality; much of this paper is concerned with fleshing out the structures
 surrounding this observation and explaining how it looks in the 
 various examples where it is known to hold.

 We interpret symplectic duality as evidence of a 
 hidden mirror symmetry-like connection between the two varieties, though 
 at the moment it is difficult to make the nature of this connection mathematically 
 precise. However, the same pairs of examples have arisen in moduli spaces of vacua for
 certain S-dual pairs of field theories in physics, suggesting this is
 not pure coincidence.\bigskip

\noindent\textbf{BGG category $\cO$.}
Let us discuss the content of the paper in more detail. As
mentioned above, our motivating example is the representation theory of
$U(\mg)$, whose geometric avatar is the Springer resolution of the
nilcone by the cotangent bundle $T^*(G/B)$.  
% Before
% discussing how this generalizes to other conical symplectic resolutions, we first recall 
% several facts about the relationship between representation theory and
% the flag variety.  
%In each case, we will later see that this is a special case of a
%theorem applying to other conical symplectic resolutions.
Fix a regular
%\footnote{By this we mean that $\la$ does not lie on the complexification of any of the root hyperplanes.} 
class $\la\in\mh^*\cong H^2(G/B; \C)$ and let $\cOa$ be
the subcategory of of BGG category $\cO$ consisting of modules over $U(\mg)$ with the same generalized central character as
the simple highest weight module with highest weight $\la-\rho$.  
%If $\la\in\Htz$, this is a block in the categorical sense.  More generally, it is called an
%{\bf infinitesimal block}, and it is a union of finitely many indecomposable blocks.  
The subscript stands for {\em algebraic}, since $\cOa$ is defined as a category of modules over an algebra.
 
Let $\cOg$ be the category of finitely generated $(\la-\rho)$-twisted
D-modules\footnote{That is, modules over the sheaf of twisted
differential operators denoted by $\mathscr{D}_\la$ in \cite{BB};
if $\la$ is integral, this is simply the sheaf of differential operators
on the line bundle with Euler class $\la$.} on $G/B$ that are smooth with respect to the
Schubert stratification.  
Here the subscript stands for {\em geometric}, since $\cOg$ is defined as a category of sheaves.
The following list gives some of the known structures and properties of the categories $\cOa$ and $\cOg$.  Our main goal in the paper will be to generalize these statements from $T^*(G/B)$ to arbitrary
conical symplectic resolutions.

\begin{enumerate}
\item There exist localization and section functors relating $\cOa$ and $\cOg$.  These functors are
always inverse derived equivalences, and they are Abelian equivalences if $\la$ is dominant \cite{BB}.
% \footnote{In fact, 
% the two Abelian categories are always equivalent (recall that $\la$ is regular), but the localization and section functors are not
% equivalences when $\la$ is not dominant.  This fact will not generalize to other symplectic resolutions.}
\item The two categories are both highest weight \cite{CPS88} and have graded lifts which are Koszul \cite{BGS96}.
\item If $\la$ is integral, then the center of the Yoneda algebra of $\cOg$ is canonically isomorphic to $H^*(G/B; \C)$ \cite{Soe90}.
\item The Grothendieck group $K(\cOg)$ is isomorphic, via the characteristic cycle map, to the top Borel-Moore
homology group of the union of the conormal varieties to the Schubert strata on $G/B$.  This isomorphism
intertwines the Euler form with a geometrically-defined intersection form.
\item The group $K(\cOg)$ decomposes as a direct sum over all nilpotent orbits by looking at microlocal supports
of D-modules.  The top Borel-Moore homology group of the union of the conormal varieties to the Schubert strata on $G/B$ 
decomposes as a direct sum over all nilpotent orbits via the Beilinson-Bernstein-Deligne (BBD) decomposition \cite{BBD,CG97}.
If $\la$ is integral and  $G\cong\SL_r$, then these two decompositions agree.
\item There are two collections of derived auto-equivalences of $\cOa$, and of its graded lift,
given by twisting and shuffling functors \cite{AS, Irvshuf}.
These functors define two commuting actions of the Artin braid group of $\mg$ \cite{AS, MOS, BBM04},
and they categorify the left and right actions of the Weyl group on its group algebra.
\item The category $\cOa$ is Koszul self-dual \cite{BGS96}.  The induced derived auto-equivalence of the graded lift of $\cOa$ 
exchanges twisting and shuffling functors \cite[6.5]{MOS}.
\end{enumerate}\bigskip

\noindent\textbf{Category $\cO$ in general}.
We now explain how these results generalize.  Let $\fM_0$ be a Poisson cone, and let $\fM\to\fM_0$ be a symplectic resolution
of $\fM_0$, equivariant with respect to the conical scaling action of $\bS:=\cs$.  Let $\cD$ be an $\bS$-equivariant quantization
of $\fM$, 
% in the sense of Section \ref{quant-quantizations}, 
and let $A$ be the ring of $\bS$-invariant global sections of $\cD$.
Many rings of independent interest arise this way, such as spherical rational Cherednik algebras \cite{EtGi}, central quotients of finite
W-algebras \cite{Sk}, central quotients of hypertoric enveloping algebras \cite{BeKu}, and (conjecturally) quotients
of shifted Yangians \cite{KWWY} (see Section \ref{quant-quantizations} for more details).

Let $\bT:=\cs$ act on $\fM$ by Hamiltonian
symplectomorphisms that commute with $\bS$, and assume that the fixed point set $\fM^\bT$ is finite.
The action of $\bT$ on $\fM$ lifts to $\cD$ and induces a $\Z$-grading on $A$.  Let $A^+\subset A$ be the non-negatively
graded part.  We define $\cOa$ to be the category of finitely generated $A$-modules that are locally finite with respect to $A^+$.
Versions of this category have already been studied for Cherednik algebras \cite{GGOR,RouqSchur,GoLo},
for finite W-algebras \cite{BGK,LosO,WebWO,MR3016303}, and for hypertoric enveloping algebras \cite{GDKD,BLPWtorico}.
The classical case is where $\fM = T^*(G/B)$ and $A$ is a central quotient of the universal enveloping algebra of $\mg$;
if the period of the quantization is a regular element of $\mh^*\cong H^2(\fM; \C)$, 
then $\cOa$ is equivalent to the BGG category $\cOa$ (Remark \ref{BGG-Oa}).\footnote{This statement really requires
regularity of the period, otherwise it fails.}  
%There has been a rich development in recent years in the study of these various categories, and one of our aims is to unify the subject. 

Let $$\fM^+ := \{p\in\fM\mid\lim_{\bT\ni t\to 0}t\cdot p\;\text{exists}\}.$$
We define $\cOg$ to be the category of $\cD$-modules that are set-theoretically supported on $\fM^+$ 
and admit a particularly nice lattice for a certain subalgebra $\cD(0)\subset\cD$;
see Sections \ref{gco-sec:loc} and \ref{cOg-definition} for a precise definition.
If $\fM = T^*(G/B)$ and $\bT$ is a generic cocharacter of $G$, then $\fM^+$ is equal to the union of the conormal varieties
to the Schubert strata, and $\cOg$ is equivalent to the category $\cOg$ above.
The aforementioned results generalize as follows.

\begin{enumerate}
\item There exist localization and section functors relating $\cOa$ and $\cOg$ (Corollary \ref{OQ}).
These functors are inverse derived equivalences for most quantizations (Theorem \ref{derived localization}), and they are Abelian equivalences if $\la$ is sufficiently positive (Theorem \ref{localization}).
\item The category $\cOa$ is highest weight for most quantizations (Theorem \ref{highest weight}\footnote{The proof of this theorem relies
heavily on an appendix by Ivan Losev.}), and
$\cOg$ is always highest weight (Proposition \ref{ghw}).
We conjecture that both categories are Koszul (Conjectures \ref{Koszul}
and \ref{geom-Koszul}).  
We can verify this conjecture in many examples, including 
cotangent bundles of partial flag varieties, 
% (Section \ref{sec:flag-varieties}), 
S3-varieties,
% \footnote{This is our name for a resolution of a normal slice to one nilpotent orbit inside of the closure
% of another.  The term S3 stands for Slodowy, Spaltenstein, and Springer.}, 
% (Section \ref{sec:S3}), 
hypertoric varieties,
% (Section \ref{sec:hypertoric-varieties}), 
Hilbert schemes on ALE spaces,
% (Section \ref{sec:hilbert-schemes-ale}), 
and some quiver varieties 
% (Section \ref{sec:quiver-varieties}).
(Section \ref{sec:examples}).
\item There is a natural graded ring homomorphism from $H^*(\fM; \C)$ to the Yoneda algebra of $\cOg$.
We conjecture that, whenever $\cOg$ is indecomposable (this will depend on the choice of quantization),
this homomorphism will be an isomorphism (Conjecture \ref{HH}).  We can prove this conjecture for cotangent bundles of partial flag varieties, 
% (Section \ref{sec:flag-varieties}), 
S3-varieties in type A, 
% (Proposition \ref{typeA-HH}), 
and hypertoric varieties 
% (Section \ref{sec:hypertoric-varieties}).
(Section \ref{sec:examples}).
We also formulate a stronger version of Conjecture \ref{HH}, relating the equivariant cohomology of $\fM$
to the center of the universal deformation of the Yoneda algebra (Conjecture \ref{HH-eq}), which we prove in the latter two cases.
\item The Grothendieck group $K(\cOg)$ is isomorphic, via the characteristic cycle map, to the top Borel-Moore
homology group of $\fM^+$.  This isomorphism
intertwines the Euler form on the Grothendieck group with the equivariant intersection form 
defined using the localization formula (Theorem \ref{support isomorphism}).
\item The group $K(\cOg)$ decomposes as a direct sum over all symplectic leaves of $\fM_0$ by looking at supports
of sheaves (Equation \eqref{O-decomposition}).  The top Borel-Moore homology group of $\fM^+$ 
decomposes via as a direct sum over all symplectic leaves of $\fM_0$ 
via the BBD decomposition (Equation \eqref{BBD-decomposition}).
Under special assumptions that are satisfied by hypertoric varieties and S3-varieties in type A,
these two decompositions agree (Corollary \ref{the special case}).  A weakening of this relationship holds
more generally (Theorem \ref{special filtrations}).
\end{enumerate}\bigskip

\noindent\textbf{Twisting and shuffling.} To state the appropriate generalization of item 6, we need some more definitions.
Let $W$ be the Namikawa Weyl group of $\fM_0$; this is a finite group that acts faithfully on $\Htr$.  Namikawa
shows that there is a hyperplane arrangement $\cHt$ in $\Htr$ whose chambers are equal to
the $W$-translates of the ample cones of the collection of symplectic resolutions of $\fM_0$ (Remark \ref{dream}).
Let $E_{\operatorname{tw}}\subset\Ht$ be the complement of the complexification of $\cHt$;
this space may also be interpreted as the locus of points over which the universal deformation of $\fM$ is affine.
In the special case where $\fM$ is the cotangent bundle of $G/B$, $W$ is the Weyl group of $G$,
and $\cHt$ is the Coxeter arrangement.

Next, let $\rWeyl$ be the Weyl group of the group of Hamiltonian symplectomorphisms of $\fM$ that commute with $\bS$, and let 
$T$ be a maximal torus.   
Let $\cHs$ be the arrangement in $\mt_\R$ whose hyperplanes describe the  cocharacters of $T$ with infinite fixed-point sets, 
and let $E_{\operatorname{sh}}\subset\mt$
be the complement of the complexification of $\cHs$.  If $\fM$ is the cotangent bundle of $G/B$, then the group of Hamiltonian
symplectomorphisms commuting with $\bS$ is $G$, and everything is the same as in the previous paragraph.  This example,
however, is misleading; in general, $W$ and $\rWeyl$ are unrelated, as are $\cHt$ and $\cHs$.
For example, if $\fM$ is a crepant resolution of $\C^2/\mck$,
then $W$ is isomorphic to the Weyl group corresponding to $\mck$ under the McKay correspondence, 
but $\rWeyl$ is trivial unless $\mck = \Z/2\Z$.

\begin{enumerate}
\item[6.]  We construct two commuting collections of derived endomorphisms of $\cOa$, 
called twisting and shuffling functors.
We construct an action of $\pi_1(E_{\operatorname{tw}}/W)$ on $D^b(\cOa)$
via twisting functors (Theorem \ref{gco-twisting braid}) and an action of 
$\pi_1(E_{\operatorname{sh}}/\rWeyl)$ on $D^b(\cOa)$ via shuffling functors
(Theorem \ref{gco-shuffling braid}).
\end{enumerate}\bigskip

\noindent\textbf{Symplectic duality.} Item 7 cannot generalize verbatim because, as mentioned above, the groups that act by twisting and shuffling functors
are in general unrelated.  The correct generalization involves two different symplectic resolutions, $\fM\to\fM_0$ and
$\fM^!\to\fM_0^!$.

\begin{enumerate}
\item[7.]  We define a symplectic duality between $\fM$ and $\fM^!$ to be a pair of isomorphisms
$$E_{\operatorname{tw}}/W \cong E^!_{\operatorname{sh}}/\rWeyl^!
\and
E_{\operatorname{sh}}/\rWeyl \cong E^!_{\operatorname{tw}}/W^!$$
and a Koszul duality between $\cOa$ and $\cOa^!$ that exchanges twisting and shuffling functors
(see Definition \ref{def:duality} for a more precise formulation).
We have already seen that $T^*(G/B)$ is self-dual (or, more naturally, dual to its Langlands dual).
Furthermore, we show that every type A S3-variety is dual to a different type A S3-variety (Theorem \ref{S3-duality}),
every hypertoric variety is dual to a different hypertoric variety (Theorem \ref{hypertoric-duality}),
and every affine type A quiver variety is dual to a different affine type A quiver variety (Theorem \ref{hkr-mkr}
and Corollary \ref{the real hkr-mkr}).
We conjecture the existence of dualities between quiver varieties and slices in the affine Grassmannian
(Remark \ref{quiver-Gr}) and between pairs of moduli spaces of instantons on ALE spaces (Remark \ref{instanton}).
\end{enumerate}

%\begin{remark}
The simplest examples of symplectic duality are between $T^*\mathbb{P}^{\ell-1}$ and a crepant resolution
of $\C^2\big{/}(\Z/\ell\Z)$.  These are special cases of every class of examples mentioned above.  Part of the interest in twisting and shuffling functors is that they can be used to construct braid group actions and homological invariants of knots.  The exchange of twisting and shuffling functors under symplectic duality then provides an explanation for different geometric constructions of the same knot homology (see Section \ref{knot}).\\
%\end{remark}

%\begin{remark}
Symplectic duality appears to be closely related to a mirror duality in physics.  
Seiberg and Intrilligator \cite{IS3D} propose a notion of {\bf mirror duality}\footnote{This duality should not be confused with the homological mirror symmetry of Calabi-Yau manifolds, which is perhaps better known to algebraic and symplectic geometers.} between three dimensional gauge
theories which carry N=4 supersymmetry .  Such a gauge theory has a moduli space attached to it with a number
of different components, including two distinguished components called the {\bf Higgs branch} and the {\bf Coulomb branch}.
Mirror duality exchanges these two components; that is, the Higgs branch of one theory is isomorphic to the Coulomb branch
of the dual theory.

It was pointed out to us by Gukov and Witten that our list of known and conjectural examples of
symplectic duality coincides almost perfectly with the known list of Higgs branches of mirror dual gauge theories
(or, equivalently, with the known list of Higgs/Coulomb pairs for a single gauge theory).
For example:
\begin{itemize}
\item Type A S3-varieties are mirror to other type A S3-varieties \cite[\S 3.3]{dBHOOY2}.
\item Hypertoric varieties are mirror to other hypertoric varieties \cite[\S 4]{dBHOOY2}.
\item Affine type A quiver varieties are mirror to other affine type A quiver varieties \cite[\S 3.3]{dBHOOY2}.
\item  An ALE space is mirror to the instanton moduli space for the corresponding simply-laced Lie group on $\mathbb{R}^2$ \cite{IS3D}.
\end{itemize}

These examples strongly suggest that symplectic duality and mirror duality are two perspectives on the same phenomenon.
Unfortunately, mirror duality and Coulomb branches do not yet have precise
mathematical definitions, so there is not yet a rigorous mathematical
statement for us to propose in an attempt to relate symplectic duality
to mirror duality in full generality.
After the appearance of the first version of this paper, the authors
became aware of work in progress of Nakajima, Braverman, and
Finkelberg \cite{Nak-Coulomb}, as well as simultaneous work by 
Bullimore, Dimofte, and Gaiotto \cite{BDG-Coulomb},
which proposes a construction
of the Coulomb branch of the gauge theory associated to a symplectic representation of a compact Lie group
(for which the Higgs branch would be the hyperk\"ahler quotient).  One may therefore regard this construction
as a conjectural construction of the symplectic dual of any conical symplectic resolution that arises
via a hyperk\"ahler quotient construction. Preliminary
calculations suggest that their approach agrees with ours in the
special cases which we understand well, but a
precise comparison of these two theories will have to be left for future work.
\bigskip
% However, it is natural to speculate that if there is a symplectic duality between $\fM$ and $\fM^!$, then $\fM$ and $\fM^!$ will arise as the Higgs and Coulomb branches of a single three dimensional supersymmetric gauge theory.
%\end{remark}

\noindent\textbf{Cohomology.} A symplectic duality between $\fM$ and $\fM^!$ has two interesting cohomological implications.
First, consider the decomposition of $K(\cOg)_\C$ from item 5 into direct summands indexed by symplectic leaves of $\fM_0$.
A consequence of symplectic duality is that the summand indexed by a leaf in $\fM_0$ is canonically
dual to the summand indexed by a corresponding leaf in $\fM_0^!$ (Proposition \ref{cohsd}).
In the case of type A S3-varieties, this duality of vector spaces is known as skew Howe duality.
In the case of affine type A quiver varieties, it is rank-level duality.  For hypertoric varieties, it is a reflection
of the behavior of the Tutte polynomial under Gale duality (Example \ref{dualities}).
When $\fM$ is a finite type ADE quiver variety and $\fM^!$ is a transverse slice in the affine Grassmannian,
this duality relates Nakajima's geometric construction of weight spaces of simple representations
to Ginzburg and Mirkovic-Vilonen's geometric construction of the same weight spaces (Example \ref{awesome example}).

The second cohomological implication comes from the last sentence of
item 3, in which we conjecture that $H^*_T(\fM; \C)$ is isomorphic to
the center of the universal deformation of the Yoneda algebra of
$\cOg$.  If this conjecture holds, then symplectic duality implies a
relationship between the equivariant cohomology rings of $\fM$ and
$\fM^!$ that was previously studied in several examples by Goresky and
MacPherson \cite{GM} (Theorem \ref{gm-thm}).  Thus, symplectic duality
may be regarded as a categorification of many different previously
studied dualities.  Interestingly, neither of these two cohomological
phenomena seems to have been familiar to physicists who study mirror
duality.\bigskip

\noindent\textbf{Summary.} The paper is structured as follows.  Section \ref{sec:quant} is a review of all of the relevant background on conical symplectic
resolutions that do not involve choosing a Hamiltonian action of
$\bT$.  Most of this material is taken from \cite{BLPWquant}.  Section
\ref{sec:gco} is devoted to the definitions and basic properties of
$\cOa$ and $\cOg$, including the localization and section functors
that relate them.  Section \ref{structure} is a review of the
background material on Koszul, highest weight, and standard Koszul
categories, which we apply to $\cOa$ and $\cOg$ in Section \ref{Oa and
  Og structure}.  Sections
\ref{sec:Grothendieck}-\ref{sec:twist-shuffl-funct} deal with items
4-6 on our list.  Section \ref{sec:examples} consists of analyses of
all of the structures that we have defined in the special cases of
cotangent bundles of partial flag varieties, S3-varieties, hypertoric
varieties, Hilbert schemes on ALE spaces, quiver varieties, and slices
in the affine Grassmannian.  Finally, Section \ref{duality} is devoted
to the definition, examples, and consequences of symplectic duality.

\vspace{\baselineskip}
\noindent
{\em Acknowledgments:}
The authors would like to thank Roman Bezrukavnikov, Justin Hilburn, Dmitry Kaledin, and Ivan Losev
for useful conversations.  In addition, the authors are grateful to the
Mathematisches Forschungsinstitut Oberwolfach for its hospitality
and excellent working conditions during the initial stages of work on this paper.

\section{Quantizations of conical symplectic resolutions}\label{sec:quant}
In this section we review the necessary background on conical symplectic resolutions.
Roughly, the section is a summary of all of the definitions and constructions in this paper
that do not involve choosing a Hamiltonian action of $\bT$.\footnote{This is not quite accurate,
as twisting functors, which are not introduced until Section \ref{twisting-gco}, also do not involve
the torus $\bT$.  We wait until Section \ref{sec:twist-shuffl-funct} to introduce twisting functors
in order to emphasize the similarities between twisting functors and shuffling functors,
which {\em do} involve the choice of $\bT$.}
Most of the material that appears here is taken from \cite{BLPWquant};
the main exception is Section \ref{sec:integrality}, which is new.

\subsection{Conical symplectic resolutions}\label{gco-sec:resolutions}
Let $\fM$ be a smooth, complex algebraic variety with an algebraic symplectic form
$\omega$.  Suppose that $\fM$ is equipped with an action of the
multiplicative group $\bS\cong\cs$ such that $s^*\omega = s^n\omega$
for some integer $n\geq 1$.
% Furthermore, we assume that $\bS$ acts on the whole of $H^2(X;\C)$ with weight $n$.
We will assume that $\bS$ acts on the coordinate ring $\C[\fM]$ with
only non-negative weights and that the trivial weight space
$\C[\fM]^\bS$ is 1-dimensional, consisting only of the constant
functions.  Geometrically, this means that the affinization $\fM_0 :=
\Spec \C[\fM]$ is contracted by the $\bS$-action to a single cone point $o\in \fM_0$.  
We will assume that the minimal symplectic leaf of $\fM_0$ consists only of the point $o$,
thus eliminating the possibility that $\fM_0$ contains a factor of a symplectic vector space.\footnote{We did
not include this condition as part of the definition of a conical symplectic resolution in \cite{BLPWquant},
but it will be useful in the current work.}
Finally, we assume that the canonical
map from $\fM$ to $\fM_0$ is a projective resolution of singularities
(that is, it must be an isomorphism over the smooth locus of $\fM_0$).  
We will refer to this collection of data as a {\bf conical symplectic resolution}.

Examples of conical symplectic resolutions include the following:
\begin{itemize}
\item $\fM$ is a crepant resolution of $\fM_0 = \C^2/\mck$, where $\mck$
is a nontrivial finite subgroup of $\SL_2$.
The action of $\bS$ is induced by the inverse of the diagonal action on $\C^2$, and $n=2$.
\item $\fM$ is the Hilbert scheme of a fixed number of points on the crepant resolution of $\C^2/\mck$,
and $\fM_0$ is the symmetric variety of unordered collections of points on the singular space.
Once again, $\bS$ acts by the inverse diagonal action on $\C^2$, and $n=2$.
\item $\fM = T^*(G/P)$ for a reductive algebraic group $G$ and a parabolic subgroup $P$,
and $\fM_0$ is the affinization of this variety (when $G=\SL_r$, this
always be the closure of a nilpotent orbit in the Lie algebra of $G$).
The action of $\bS$ is the inverse scaling action on the cotangent fibers, and $n=1$.
\item $\fM$ is a hypertoric variety associated to a simple, unimodular,
hyperplane arrangement in a rational vector space \cite{BD,Pr07}, and $\fM_0$ is the hypertoric
variety associated to the centralization of this arrangement.
These varieties admit an action of $\bS$ with $n=1$ if and only if the arrangement has a bounded chamber;
they always admit an action of $\bS$ with $n=2$.
\item $\fM$ and $\fM_0$ are Nakajima quiver varieties \cite{Nak94,Nak98}.  
These varieties admit an action of $\bS$ with $n=1$ if and only if the quiver has no loops;
they always admit an action of $\bS$ with $n=2$.
\item $\fM_0$ is a transverse slice to $\operatorname{Gr}^\mu$ inside of $\operatorname{Gr}^\la$,
where $\operatorname{Gr}^\mu$ and $\operatorname{Gr}^\la$ are Schubert varieties inside of
the affine Grassmannian for a reductive group $G$.
When $\la$ is a sum of minuscule
coweights for $G$, $\fM_0$ has a natural conical symplectic resolution
constructed from a convolution variety; in most other cases, it
seems to possess no such resolution.  This example is discussed in
greater generality in \cite{KWWY}.
\end{itemize}

\begin{remark}
The fifth class of examples overlaps significantly with each of the
others.  The first two examples are special
cases of quiver varieties, where the underlying graph of the quiver is the extended Dynkin diagram corresponding
to $\quiv$.  The third and sixth examples can be realized as quiver varieties if the group $G$ is of type A.  
Finally, a hypertoric variety is a quiver variety if and only if the associated hyperplane arrangement is cographical.
\end{remark}

\begin{remark}
  All of these examples admit complete hyperk\"ahler metrics, and in fact we know of no examples
  that do not admit complete hyperk\"ahler metrics.  (Such spaces do exist if we drop the hypothesis
  that $\fM$ is projective over $\fM_0$; some examples will appear in subsequent work by 
  Arbo and the third author.)
  The unit circle in $\bS$ acts by hyperk\"ahler isometries, but is
  Hamiltonian only with respect to the real symplectic form.  Our
  assumptions about the $\bS$-weights of $\C[\fM]$ translate to the
  statement that the real moment map for the circle action is proper and bounded below.
\end{remark}

\subsection{Deformation theory and birational geometry}\label{sec:nam}
Let $\fM$ be a conical symplectic resolution.  The following result is stated in \cite[2.7]{BLPWquant}; 
it is due in this form to Namikawa \cite{Namiflop}, and is closely related to earlier results of Kaledin and Verbitsky \cite{KV02}.

\begin{theorem}\label{universal family}
The variety $\fM$ has a universal Poisson deformation $\pi\colon\scrM\to H^2(\fM;\C)$ which
is flat.  The variety $\scrM$ admits an action of $\bS$ extending the action on $\fM\cong \pi^{-1}(0)$,
and $\pi$ is $\bS$-equivariant with respect to the weight $-n$ action on $H^2(\fM;\C)$.
This family is trivial in the category of smooth manifolds with circle actions. 
\end{theorem}

For any $\eta\in\Ht$, we will also be interested in the {\bf twistor deformation} $$\scrM_\eta := \scrM\times_{\Ht}\aone.$$
Let $\scrM_\eta(\infty):= (\scrM_\eta\smallsetminus\fM)/\bS$ be the generic fiber of $\scrM_\eta$.  A fundamental result of Kaledin
\cite[2.5]{KalDEQ} says that, if $\eta$ is the Euler class of an ample line bundle on $\fM$, then $\scrM_\eta(\infty)$ is affine.
More generally, Namikawa \cite{NamikawaNote} shows that there is a finite set $\cHt$ of hyperplanes in $\Htr$ such that the union
$\bigcup_{H\in\cHt} H_\C\subset\Ht$ is equal to the locus over which the fibers of $\pi$ fail to be affine.\footnote{The subscript tw 
stands for ``twisting", and is explained by Theorem \ref{gco-twisting braid}.
There will also be a ``shuffling" arrangement $\cHs$, and an analogous Theorem \ref{gco-shuffling braid}.}

Namikawa constructs a universal Poisson deformation of $\fM_0$ over the base $H\! P^2(\fM_0)$ \cite{Namiflop}.
Since $\Spec\C[\scrM]$ is itself a Poisson deformation of $\fM_0$, we obtain a map from $\Ht$ to $H\! P^2(\fM_0)$.
Namikawa shows that this map is a quotient by a finite subgroup $W$ of the general linear group of $\Ht$ \cite[1.1]{NamiaffII}.
In the case of the Springer resolution, $\Ht$ is isomorphic to the Cartan subalgebra and $W$ is isomorphic to the Weyl group.
For this reason, we refer to $W$ more generally as the {\bf Namikawa Weyl group}.

\begin{remark}\label{dream}
The Namikawa Weyl group in fact acts on $\Htr$, with a fundamental domain equal to the closure of the
movable cone of $\fM$ \cite[2.17]{BLPWquant}.  This movable cone can be further divided into chambers
given by ample cones of various conical symplectic resolutions of $\fM_0$.  (For any conical symplectic resolution
$\fM'$, its second cohomology group and its movable cone are canonically identified with those of $\fM$.)
Namikawa \cite{NamikawaNote} proves that $\fM$ is a relative Mori dream space over $\fM_0$ in the sense of \cite[2.4]{AW},
and that the chambers of $\cHt$ are exactly equal to the $W$-translates of the ample cones of the various
resolutions of $\fM_0$.  If $\fM$ is obtained as a symplectic quotient of a vector space by the action of a group $G$
and the Kirwan map from $\chi(G)_\R$ to $\Htr$ is an isomorphism, these chambers coincide with the maximal cones
of the GIT fan.
\end{remark}

\subsection{Quantizations}\label{quant-quantizations}
Let $\fM$ be a conical symplectic resolution.
A {\bf quantization} of $\fM$ is defined to be 
\begin{itemize}
\item an $\bS$-equivariant sheaf $\cQ$ of flat $\C[[h]]$-algebras on $\fM$, complete in the $h$-adic topology,
where $\bS$ acts on $h$ with weight $n$ (see \cite[\S 3.2]{BLPWquant} for a precise definition of $\bS$-equivariance)
\item an $\bS$-equivariant isomorphism from $\cQ/h\cQ$ to the structure sheaf $\fS_\fM$ of $\fM$
\end{itemize}
satisfying the condition that, if $f$ and $g$ are functions over some
open set and $\tilde f$ and $\tilde g$ are lifts to $\cQ$, the image
in $\fS_\fX \cong \cQ/h\cQ\cong h\cQ/h^2\cQ$ of the element $[\tilde
f, \tilde g]\in h\cQ$ is equal to the Poisson bracket $\{f,g\}$.

Using the work of Bezrukavnikov and Kaledin \cite{BK04a}, who classify quantizations in a (much more general) non-equivariant setting, Losev \cite[2.3.3]{Losq} proves the following classification result (see also \cite[3.5]{BLPWquant}).
 
\begin{theorem}\label{periods}
Quantizations of a conical symplectic resolution $\fM$ are in bijection with $\Ht$ via
the period map of \cite{BK04a}.
\end{theorem}

Fix a quantization $\cQ$ of $\fM$.  Let $\cD(0) := \cQ[h^{\nicefrac{1}{n}}]$, and let
$\cD(m) := h^{\nicefrac{-m}{n}}\cD(0)$ for all $m\in\Z$.  Let 
$$\cD := \cQ[h^{\nicefrac{-1}{n}}] = \bigcup_{m=0}^\infty \cD(m);$$
we will often abuse notation by referring to $\cD$ as a quantization of $\fM$.
Let $$A := \Gamma_\bS(\cD)$$
be the ring of $\bS$-invariant sections of $\cD$.  This ring
inherits an $\mathbb{N}$-filtration $$A(0)\subset A(1)\subset \ldots\subset A$$
given by putting 
$$A(m) := \Gamma_\bS(\cD(m)).$$
The associated graded of $A$ may be canonically identified with $\C[\fM]$ as an 
$\N$-graded ring.
Many of our examples of conical symplectic resolutions in the previous section admit
quantizations for which the ring $A$ is of independent interest.

\begin{itemize}
\item If $\fM$ is the Hilbert scheme of $k$ points on a crepant resolution of
  $\C^2/\mck$, then is $A$ is isomorphic to
  a spherical symplectic reflection algebra for the wreath product
  $S_k\wr\mck$ \cite[1.4.4]{EGGO}.
\item If $\fM = T^*(G/B)$ for a reductive algebraic group $G$ and a Borel subgroup $B\subset G$,
then $A$ is a central quotient of the
  universal enveloping algebra $U(\mg)$ \cite[Lemma 3]{BB}.
\item If $\fM$ is the resolution of a Slodowy slice to a nilpotent orbit in $\mg$,
then $A$ is a central quotient of a finite W-algebra \cite[6.4]{Sk}.
\item If $\fM$ is a hypertoric variety, then $A$ is a central quotient
of a hypertoric enveloping algebra \cite[\S 5]{BeKu}, \cite[5.9]{BLPWtorico}.
\item If $\fM_0$ is a slice to one affine Schubert variety inside another, then $A$ is conjecturally
isomorphic to a quotient of a shifted Yangian \cite{KWWY}.
\end{itemize}

Note that $\cD$ and $A$ also carry a grading by the group $\Z/n\Z$, where $\cQ\subset\cD$
lies in degree $\bar 0$ and $h^{\nicefrac{1}{n}}$ has degree $\bar 1$.  
% When we
% refer to an element of $A$ as ``homogeneous,'' we mean with respect to
% this grading.   
The grading on $A$ is compatible with the filtration and thus descends to a grading on $\gr A$,
which is equal to the grading induced by the natural semigroup homomorphism from $\N$ to $\Z/n\Z$.

\subsection{Integrality}
\label{sec:integrality}

We would like to have some notion of what it means for a quantization
to be {\bf integral}.
Let $\Htzf$ be the quotient of $\Htz$ by its torsion
subgroup.\footnote{In the situation of greatest interest to us, when
  there is a Hamiltonian $\C^*$-action commuting with $\bS$ that has
  isolated fixed points, there is no torsion in this group, as we show
  in Proposition \ref{no-torsion}.}
The naive definition would be that $\cQ$ or $\cD$ is integral if its period lies in the lattice
$\Htzf\subset\Ht$, but this is not suitable for our purposes.
% \footnote{Throughout the paper, we will abusively
% write $\Htz$ to denote the quotient of the second integer cohomology group by the torsion subgroup,
% so that $\Htz$ naturally sits inside of $\Ht$.}  
For example, if $\fM = T^*X$ for a projective
variety $X$ and $\cD$ is the quantization with period $\la\in\Ht \cong H^2(X; \C)$, 
then $A$ is isomorphic to ring of differential operators on $X$, twisted by $\la + \frac 1 2 \varpi_X$,
where $\varpi_X$ is the Euler class of the canonical bundle of $X$ \cite[4.4]{BLPWquant}.
In this case, we would like to say that $\cD$ is integral if and only if $\la + \frac 1 2 \varpi_X\in\Htzf$.
More generally, the set of integral periods should be a coset
$\Lambda$ of $\Htzf\subset\Ht$ that
satisfies the following properties.
\begin{itemize}
\item We have $\la\in\Lambda$ if and only if $-\la\in\Lambda$.  Equivalently, $2\Lambda$ is contained in $\Htzf$.
We include this condition because the quantization with period $-\la$ is the opposite ring
of the quantization with period $\la$ \cite[3.2]{BLPWquant}, and the opposite of an integral quantization should be integral.
\item If $X\subset\fM$ is a smooth Lagrangian subvariety, then the restriction of $\Lambda$ to $X$
is equal to $\frac 1 2 \varpi_X + H^2(X; \Z)_{\operatorname{free}}$.  In particular, this uniquely determines
$\Lambda$ if $\fM$ is a cotangent bundle.
\item Suppose that $G$ is a reductive group acting on a symplectic
  vector space $V$, and $\fM$ is a smooth symplectic
quotient of $V$ by $G$.  (For example, all quiver varieties and smooth hypertoric varieties are of this form.)
Given a Lagrangian $G$-subspace $L\subset V$, we may identify the Weyl algebra of $V$
with the ring of differential operators on $L$.  Consider the quantized moment map $\mu_L:U(\mg)\to\operatorname{Diff}(L)$
that takes an element of $\mg$ to the induced vector field on $L$, and consider the induced quantization $\cD_L$ of $\fM$,
as in \cite[2.8(i)]{KR}.  The period of this  quantization can be
calculated from \cite[3.16]{BLPWquant}.  We should choose $\Lambda$ to
be the coset of this period.  Note that if $L$ and $L'$ are two different Lagrangian
$G$-subspaces, then $\cD_L$ and $\cD_{L'}$ need not be equal, but
\cite[3.16]{BLPWquant} shows that their periods will always differ by
an element of $\Htzf$, corresponding to the determinant character of $G$
acting on $L'/(L'\cap L)$.
\end{itemize}

\begin{remark}
By the first property above, the coset $\Lambda$ is uniquely determined by the image $c_\Lambda$ of $2\Lambda$ in 
$\Htzf/H^2(\fM; 2\Z)_{\operatorname{free}} \subset H^2(\fM; \Z/2\Z)$.
% \footnote{Note that this containment is an isomorphism if and only if $\Htz$ has no 2-torsion.}  
The second property above is equivalent to the statement 
that the restriction of $c_\Lambda$ to any smooth
Lagrangian subvariety should equal the second Stiefel-Whitney class of that subvariety.
Unfortunately, this condition may not uniquely determine $c_\Lambda$, as it is possible that $\fM$ has no smooth Lagrangian
subvarieties at all.
\end{remark}

Very little of what we do in this paper depends on the notion of integrality.  In Sections \ref{sec:quant}-\ref{sec:examples},
we will only refer to integral quantizations in the context of cotangent bundles, hypertoric varieties, and quiver varieties,
in which case the meaning is completely determined by the second and third conditions above.  In Section \ref{duality},
the notion of integrality will become important; in that section, we simply assume that every conical symplectic resolution
comes with a choice of $\Lambda$ that is consistent with our three conditions.

\subsection{Sheaves of modules}\label{gco-sec:loc}
Let $\fM$, $\cQ$, and $\cD$ be as in Section \ref{quant-quantizations}.
A $\cD(0)$-module $\cN(0)$ is called {\bf coherent} if it is a quotient of a sheaf which is locally free of
finite rank.  Setting $\cN(m) := h^{\nicefrac{-m}{n}}\cN(0)$, Nakayama's lemma tells us that the
following three conditions are equivalent:
\begin{itemize}
\item $\cN(0)$ is coherent 
\item $\cN(0)/\cN(-1)$ is a coherent sheaf of modules over $\cD(0)/\cD(-1)\cong \fS_{\fM}$
\item $\cN(0)/\cN(-n) = \cN(0)/h\cN(0)$ is a coherent sheaf of modules over $\cQ/h\cQ\cong \fS_{\fM}$.
\end{itemize}
An $\bS$-equivariant $\cD$-module $\cN$ is called {\bf good} if it 
admits a coherent $\bS$-equivariant $\cD(0)$-lattice $\cN(0)$.
We call  a good $\cD$-module {\bf holonomic} if it has Lagrangian support.

Given a choice of lattice $\cN(0)$, we will refer to the coherent sheaf $\cN(0)/\cN(-n)$ as the {\bf
big classical limit} of $\cN$, and to $\cN(0)/\cN(-1)$ as the {\bf small classical limit} of $\cN$.
Note that the big classical limit is an $n$-fold extension of the small classical limit, and this extension need not split.

\subsection{Localization}\label{sec:localization}
Let $\fM$, $\cQ$, $\cD$, and $A$ be as in Section \ref{quant-quantizations}.
Let $\Dmod$ denote the category of good $\bS$-equivariant $\cD$-modules.
Note that the choice of lattice is not part of the data of an object of $\cD\mmod$.

Let $\Amod$ be the category of finitely generated $A$-algebras.  A {\bf good filtration}
of an $A$-module $N$ is defined to be a filtration such that $\gr N$ is finitely generated over $\gr A$.
For any $N$, we can choose a good filtration by picking a finite generating set $Q\subset N$
and putting $N(m) := A(m)\cdot Q$.

We have a functor $$\secs:\Dmod\to\Amod$$ given by taking $\bS$-invariant global sections.
The left adjoint functor $$\Loc:\Amod\to\Dmod$$ is defined by putting
$\Loc(N) := \cD\otimes_AN.$
To see that $\Loc(N)$ is indeed an object of $\Dmod$, choose a good filtration of $N$.  
We define the {\bf Rees algebra} $R(A)$ to be the $h$-adic completion of 
$$A(0)[[\hon]] + \hon A(1)[[\hon]] + h^{\nicefrac{2}{n}}A(2)[[\hon]] + \ldots\subset A[[\hon]]$$
and the {\bf Rees module} $R(N)$ to be the $h$-adic completion of
$$N(0)[[\hon]] +\hon N(1)[[\hon]] + h^{\nicefrac{2}{n}}N(2)[[\hon]] + \ldots\subset N[[\hon]].$$
Note that $R(N)$ is a module over $R(A)\cong \Gamma(\cD(0))$,
and $\cD(0)\otimes_{R(A)}R(N)$ is a coherent lattice in $\Loc(N)$.

\begin{remark}\label{lattice-filtration}
If $N$ is an object of $\Amod$, we have shown that $\Loc(N)$ always admits a coherent lattice,
but the construction of that lattice depends on a choice of filtration of $N$.  Conversely,
any coherent lattice $\cN(0)$ for an object $\cN$ of $\Dmod$ induces a filtration of
$N:= \secs(\cN)$ by putting $N(m) := \secs\big(\cN(m)\big)$.
\end{remark}

If $\secs$ and $\Loc$ are quasi-inverse equivalences of categories, we
will say that {\bf localization holds for \boldmath$\cQ$} 
or {\bf localization holds for \boldmath$\cD$} or {\bf localization holds at \boldmath$\la$},
where $\la$ is the period of $\cQ$.  If their derived functors induce quasi-inverse equivalences
of derived categories, we say that {\bf derived localization holds}.
Localization and/or derived localization is known to hold in many special cases, including quantizations of the Hilbert
scheme of points in the plane \cite[4.9]{KR}, the cotangent bundle of $G/P$ \cite{BB},
resolved Slodowy slices  \cite[3.3.6]{Gin08} \& \cite[7.4]{DK}, and hypertoric varieties \cite[5.8]{BeKu}.
In \cite[A \& B.1]{BLPWquant}, we have shown that localization and derived localization hold for ``many" quantizations.

\begin{theorem}\label{localization} 
If $\eta$ is the Euler class of an ample line bundle on $\fM$, then for any $\la$,
localization holds at $\la + k\eta$ for sufficiently large integers $k$.
\end{theorem}

\begin{theorem}\label{derived localization}
If $\eta$ is the Euler class of an ample line bundle on $\fM$, then for any $\la$,
derived localization holds at $\la + k\eta$ for all but finitely many complex numbers $k$.
\end{theorem}
Forthcoming work of McGerty and Nevins \cite{MN2} gives a considerable strengthening of Theorem \ref{localization}, 
showing that the locus where localization fails is
contained in countably many translates of hyperplanes from the
discriminant locus.
In earlier work \cite{MN}, they also gave a cohomological
criterion for when derived localization holds: when the section
algebra has finite global dimension.

\subsection{Modules with supports}\label{quant-lag}
Let $\fM$, $\cQ$, $\cD$, and $A$ be as in Section
\ref{quant-quantizations}.  
Let $\fL_0\subset\fM_0$ be the subscheme defined by a graded ideal $J\subset \C[\fM_0]$,
and let $\fL\subset\fM$ be the subscheme defined by a graded ideal sheaf $\cJ\subset\fS_\fM$.
We will often assume that $\fL$ is the scheme-theoretic preimage of $\fL_0$,
which is equivalent to saying that $\cJ = \fS_{\fM}\otimes_{\C[\fM_0]} J$
(see Propositions \ref{C-to-C} and \ref{derived-C-to-C}).
We denote by $J+h\cdot R(A)\subset R(A)$ the
preimage of $J\otimes \C[\hon]/\langle h\rangle$ under the natural map $$R(A)\to
\C[\fM_0]\otimes \C[\hon]/\langle h\rangle .$$
The following definitions appeared in \cite[\S 6.1]{BLPWquant}.

\begin{definition}\label{def:CLa}
  Let $\CLa$ be the full subcategory of $\Amod$ consisting of all
  modules $N$ that admit good filtrations with either of the following two
  equivalent properties:
  \begin{itemize}
\item Let $a\in A(k)$ be homogeneous of degree $\bar k$ for the $\Z/n\Z$ grading, and suppose that 
its symbol $\bar a\in A(k)/A(k-1) \cong \C[\fM_0]_k$ lies in $J$.
Then $a\cdot N(m)\subset N(k+m-n)$.
\item For any $a\in J+h\cdot R(A)$, we have $a\cdot R(N)\subset h\cdot R(N)$.
\end{itemize}
  Let $\DCLa$ be the full subcategory of $D^b(\Amod)$ consisting of
  objects with cohomology in $\CLa$.  
\end{definition}

\begin{remark}
Note that if $N$ is an object of $\CLa$, then the associated graded $\gr N$ will be killed by the ideal
$J$, but the converse is not true unless $n=1$.  
% A counterexample to the converse can be found by applying $\secs$ to the module
% in Remark \ref{classical-limit}, where $\cL_0\subset\C^2$ is one of the axes with its reduced scheme structure.
% Let $a\in A(1)$ be any homogeneous lift of a generator of $J$.  Then $\gr N$ is killed by $J$ if and only if $a\cdot N(k)\subset N(k)$, whereas $N$ lies in $\CLa$ if and only if $a\cdot N(k)\subset N(k-1)$.
% In this case the first statement holds, but not the second.
\end{remark}

\begin{definition}\label{def:CLg}
Let $\CLg$ be the full subcategory of $\Dmod$ consisting of modules
with big classical limits that are scheme-theoretically supported on $\fL$.
More precisely, a $\cD$-module $\cN$ is in $\CLg$ if it admits a lattice $\cN(0)$ that is preserved by
$h^{-1}\tilde f$ for any section $\tilde f$ of $\cQ$ whose image in
$\cQ/h\cQ\cong\fS_\fM$ lies in $\cJ$.
Let $\DCLg$ be the full subcategory of $D^b(\Dmod)$ consisting of objects with cohomology in $\CLg$.
\end{definition}

\begin{proposition}\label{C-to-C}
If $\fL$ is the scheme-theoretic preimage of $\fL_0$, then
  $\Loc$ takes $\CLa$ to $\CLg$ and $\secs$ takes $\CLg$ to $\CLa$.  
\end{proposition}

\begin{proof}
Let $N$ be an object of $\CLa$.  Choose a filtration of $N$ as in Definition \ref{def:CLa},
and let $\cN(0):=\cD(0)\otimes_{R(A)}R(N)$ be the induced lattice in $\cN:=\Loc(N)$.  Let $\tilde f$
be a global section of $\cQ$ whose image $f\in\cQ/h\cQ\cong\fS_\fM$ lies in $\cJ$.
After decomposing $\tilde f$ into eigenvectors for the $\bS$ action, we may assume that there exists
an integer $k$ such that $h^{\nicefrac{-k}{n}}\tilde f$ is $\bS$-invariant.
Thus $h^{\nicefrac{-k}{n}}\tilde f\in\secs(\cD)=A(k)$ is homogeneous of degree $\bar k$
for the $\Z/n\Z$ grading, so 
$h^{\nicefrac{-k}{n}}\tilde f\cdot N(m) \subset N(k+m-n)$. 

On any sufficiently small open subset $U$, we have $\cN(0)(U)\cong \cD(0)(U)
\otimes_{R(A)}R(N)$; 
moreover,
\[\cN(0)(U)=\sum_m
\cD(0)(U)\otimes h^{\nicefrac{m}{n}} N(m),\]
where we write $\cD(0)(U)\otimes h^{\nicefrac{m}{n}} N(m)$ to denote the image of
the tensor product over $\C$ inside of $\cN(0)(U)$.  
Thus
\begin{align*}
\tilde f\cdot \cN(0)(U)& \subset \sum_m \left(
  h^{\nicefrac{m}{n}}[\tilde f,\cD(0)(U)]\otimes N(m)+\cD(0)(U)\otimes
  h^{\nicefrac{m+k}{n}}\cdot N(m+k-n)\right) \\
  & \subset h\cdot \cN(0)(U).
\end{align*}
Since the ideal sheaf $\cJ$ is generated by global sections, this
suffices to show that $\Loc(\cN)\in \CLg$.

For the opposite direction, let $\cN$ be an object of $\CLg$, and let $\cN(0)$ be a lattice preserved by
$h^{-1}\tilde f$ for every section $\tilde f$ of $\cQ$ whose image $f\in \cQ/h\cQ\cong\fS_\fM$ lies in $\cJ$.
Let $N := \secs(\cN)$, and let $N(m) := \secs\big(\cN(m)\big)$ be the induced filtration.
Let $a\in A(k)$ be homogeneous of degree $\bar k$ with symbol in $J$.
Then $h^{\nicefrac{k}{n}}a$ is a section of $\cQ$ whose image lies in $\cJ$, so
$h^{\nicefrac{k}{n}-1}a\cdot \cN(m) \subset \cN(m)$,
and therefore $a\cdot \cN(m)\subset h^{1-\nicefrac{k}{n}}\cN(m) = \cN(k+m-n)$.
Applying $\secs$, we see that $a\cdot N(m)\subset N(k+m-n)$, so $N$ is an object of $\CLa$.
\end{proof}

\begin{proposition}
  \label{derived-C-to-C}
 If $\fL$ is the scheme-theoretic preimage of $\fL_0$ and derived
 localization holds at $\la$, then $\LLoc$ takes $\DCLa$ to $\DCLg$ and $\Rsecs$ takes $\DCLg$ to $\DCLa$.  
\end{proposition}
\begin{proof}
  Let $\cN$ be an object in $\CLg$, and let $\cN(0)\subset\cN$ be a
  lattice satisfying the required condition.  There is a spectral
  sequence (see \cite[\S6.1]{BLPWquant}, particularly the proof of Theorem 6.5)
  \[ H^p(\fM;\cN(0)/\cN(-n))\Ra
  R\left(\coho^p({\Rsecs(\cN)})\right)/h
  R\left(\coho^p({\Rsecs(\cN)})\right).\]
  Since the left-hand side is killed by the ideal $J$, the same is
  true of the right hand side,
   which implies that $\coho^p({\Rsecs(\cN)})$ is in $\CLa$, and is
   only non-zero in finitely many degrees since the map $\pi$ is projective.

Now let $N$ be an object of $\CLa$ and put $\cN:=\LLoc(N)$.  This only
has cohomology in finitely many degrees since $A_\la$ has finite
global dimension, by a result of McGerty and Nevins \cite[\S7.5]{MN}.
A filtration of $N$ induces a lattice in $\coho^p(\cN)$.
For any $a\in J+h\cdot R(A)$, we have that
  $a\cdot
  R(N)\subset h\cdot R( N);$  thus, on any projective resolution, the
  map induced by  $a$ is
  null-homotopic mod $h$; this
  implies that our lattice in $\coho^p(\cN)$ has the required property.
\end{proof}

\begin{remark}\label{why-bounded} 
If derived localization does not hold, then the functor $\LLoc$ is not bounded.
If we were to replace $D^b$ by the bounded-above category $D^-$, then Proposition \ref{derived-C-to-C}
would hold for arbitrary quantizations.  This is discussed in greater detail
in \cite[\S 4.3]{BLPWquant}.
\end{remark}

\subsection{Harish-Chandra bimodules and characteristic cycles}\label{quant-HC}
We continue with the notation $\fM$, $\cQ$, and $\cD$ from Section \ref{quant-quantizations}.  The product $\fM \times \fM \to \fM_0 \times \fM_0$ is a conical symplectic resolution with quantization $\cD \boxtimes \cD^{\op}$ (the tensor product is taken over $\C(\!(h)\!)$) and section ring $A \otimes A^{\op}$.  Thus we can apply the previous definitions and results to $A$-bimodules and $\cD$-bimodules. 

Consider the diagonal
$\fZ_0 \subset \fM_0\times\fM_0$ (with its reduced scheme structure),
and its preimage $\fZ := \fM\times_{\fM_0}\fM$, the {\bf Steinberg}
scheme (which may not be reduced).

\begin{definition}\label{def:HCa}\label{def:HCg}
A finitely generated $A$-bimodule (resp.  $\cD$-bimodule) is called
{\bf Harish-Chandra} if it lies in $C^{\fZ_0}$ (resp. $\mathcal{C}^{\fZ}$).
We will use the notation
$$\HCa := C^{\fZ_0}\and \HCg := \mathcal{C}^{\fZ}$$
for the abelian categories of algebraic and geometric Harish-Chandra bimodules,
along with $\dHCa$ and $\dHCg$
for the subcategories of the bounded derived categories of all bimodules consisting of objects with Harish-Chandra cohomology.
\end{definition}

The following results appear in \cite[\S 6.1]{BLPWquant}.

\begin{proposition}\label{tensor action}
The category $\HCa$ is a monoidal category under the operation of tensor product, and the category $\CLa$
is a module category over $\HCa$; similarly, when $A$ has finite
global dimension, $\dHCa$ has a monoidal
structure induced by derived tensor product, and an action on $\DCLa$.

 There is a geometric version of this derived tensor product, induced
 by convolution on  $\dHCg$, 
and the category $\DCLg$ is naturally a module category over $\dHCg$.  
These structures are compatible
with the derived $\bS$-invariant section functors.  
\end{proposition}

Let $\cH$ be an object of $\HCg$ and let $\cN$ be an object of $\CLg$.
Let $d = \frac 1 2 \dim\fM$.
In \cite[\S 6.2]{BLPWquant} we constructed maps\footnote{We will review the definition of these maps in Section \ref{gco-grothendieck}.}
$$\suppc: K(\HCg)\to \HZZ
\and
\suppc:  K(\CLg)\to \HLZ,$$
and we proved the following result \cite[6.15 \& 6.16]{BLPWquant}.
  
\begin{proposition}\label{categorification}
The map $\suppc$ intertwines the monoidal structure on $\dHCg$ with the convolution product
on $\HZZ$, and it also intertwines the action of $\dHCg$ on $\DCLg$ with the convolution action of $\HZZ$ on $\HLZ$.
\end{proposition}

\begin{remark}
The statements in \cite[\S 6.2]{BLPWquant} are somewhat more technical than what we have stated above,
because there we consider all quantizations at once.  That is, for any pair $\la,\la'\in\Ht$, we define what it means
for an $(A_\la, A_{\la'})$-bimodule or a $(\cD_\la, \cD_{\la'})$-bimodule to be Harish-Chandra, and so on.
We will in fact need this stronger version when we discuss twisting functors in Section \ref{twisting-gco}, but for the purposes
of this summary we have elected to keep things clearer by fixing a particular quantization.
\end{remark}

\section{The categories \texorpdfstring{$\cOa$}{Oa} and \texorpdfstring{$\cOg$}{Og}}\label{sec:gco}
To define the categories $\cOa$ and $\cOg$ we need one more piece of geometric structure,
namely a Hamiltonian action of the multiplicative group $\bT\cong\cs$, commuting with the action of $\bS$, 
such that $\fM^\bT$ is finite.
First, let us make some observations about the integral cohomology of such
a symplectic resolution.
\begin{proposition}\label{no-torsion}
  If $\fM$ is a conical symplectic resolution that admits a
  $\bT$-action as above, then $H^*(\fM;\Z)$ is
  torsion-free and concentrated in even degrees.
\end{proposition}
\begin{proof}
By Poincar\'e duality, we can instead consider the Borel-Moore
homology of the same variety.
By \cite[7.1.5]{Na}, it suffices to show that the same is true of the smooth
projective variety
$\fM^{\bS}$.  The action of $\bT$ preserves $ \fM^{\bS}$; thus, this
projective variety has a torus action with isolated fixed points.  The
Bia{\l}ynicki-Birula decomposition of $\fM^{\bS}$ shows that it has
even torsion-free cohomology.
\end{proof}

\vspace{-\baselineskip}
\begin{remark}
The analogous result is shown for quiver varieties even when they
don't have a $\bT$-action with finite fixed-point set in \cite[7.3.5]{Na}.
\end{remark}

The action of $\bT$ lifts canonically to an action on $\cQ$, where $\bT$ fixes $h$.
By \cite[3.11]{BLPWquant}, there exists an element $\xi$ of $A(n)\subset A$, unique up to translation by $A(0)\cong \C$, 
such that the endomorphism of $\cD$ induced by the generator
of the Lie algebra $\bt := \operatorname{Lie}(\bT)$ is given by conjugation with $\xi$.

\begin{remark}
% The second condition implies that the action of $\bT$ preserves the symplectic form $\omega$ 
% (unlike the action of $\bS$).\footnote{Though $\bS$
% and $\bT$ are isomorphic as groups, they play very different roles in our setup.}
For any choice of $\xi$, the image $\bar\xi$ of $\xi$ in $A(n)/A(n-1)\subset\gr A\cong \C[\fM]$ is 
the unique $\bS$-equivariant moment map for the action of $\bT$ on $\fM$.  
Another way to say this is to note that $\xi$ induces a homomorphism
from $U(\bt)$ to $A$, and the associated graded of this homomorphism is the co-moment map.
\end{remark}

\subsection{The relative core}\label{sec:rel-core}
Choose an indexing set $\cI$ for the $\bT$-fixed points of $\fM$, so that
$\fM^\bT = \{p_\a\mid \a\in\cI\}$.
For each $\a\in\cI$, let $X_\a\subset \fM$ be the closure of the set 
$$X_\a^\circ := \big\{p\in \fM\mid \lim_{\bT\ni t\to 0}t\cdot p = p_\a\big\},$$ and let $\fM^+ := \bigcup X_\a$.
The set $\fM^+$ is called the {\bf relative core} of $\fM$.

The fact that the action of $\bT$ preserves the symplectic form 
implies that each $X_\a$ is Lagrangian (though possibly singular), and the open subvariety $X^\circ_\a$ is isomorphic to $d$-dimensional affine space.

In the affine variety $\fM_0$, let
 \[\fM^+_0 := \big\{p\in \fM_0\mid \lim_{\bT\ni t\to 0}t\cdot p = o\big\}\] be the locus of points that limit to the unique $\bT$-fixed point $o\in\fM_0$.  Since $\fM$ is projective over $\fM_0$, a point in $\fM$ has a limit if and only if its image in $\fM_0$ does; it follows that $\fM^+ = \bigcup X^\circ_\a$
is the preimage of $\fM^+_0$.

Let $X_{\a,0}$ be the image of $X_\a$ in $\fM_0$.

\begin{example}
If $\fM$ is a crepant resolution of $\C^2/\mck$ with $\mck\cong \Z/k\Z$, 
the relative core components $\{X_\a\mid \a\in\cI\}$ 
will consist of a chain of $k-1$ projective lines, along with a copy of $\C$ at one end of the chain.  
If $\fM$ is $T^*(G/P)$, they will
be the conormal varieties to the Schubert strata of $G/P$.  If $\fM$ is a hypertoric variety, 
they will all be toric varieties.  The cotangent bundle of $\mathbb{P}^1$ is a special case of
all three of these examples; in this case we have two subvarieties: the zero section and one of the fibers.
\end{example}

\begin{remark}\label{core}
The preimage of $o\in \fM_0$ in $\fM$ is called the {\bf core}, and is a subset of $\fM^+$ consisting of
the union of all of the projective components.  Our requirement that $\{o\}$ is a symplectic leaf of $\fM_0$ 
guarantees that the core is a Lagrangian subvariety of $\fM$.
Note that the core is independent of the choice of $\bT$-action,
while the relative core depends on this choice.
\end{remark}

Let $J\subset\C[\fM]$ be the ideal in the coordinate ring of $\fM$ generated by functions of non-negative $\T$-weight
and $\bS$-weight greater than or equal to $n$.

\begin{lemma}\label{vanishing}
The relative core $\fM^+\subset\fM$ is the vanishing locus of $J$.
\end{lemma}

\begin{proof}
Let $f\in\C[\fM]$ be a function of non-negative $\T$-weight and
$\bS$-weight greater than or equal to $n$.
Then $f$ vanishes on $\bS$-fixed points, and the core (being projective) contains at least one such
point.  Thus $f$ vanishes on the entire core.  For any $p\in\fM^+$, 
$$f(p) = \lim_{t\to 0}(t\cdot f)(t\cdot p) = 0,$$
since $t\cdot f$ is approaching either $f$ or $0$, and $t\cdot p$ is approaching an element
of the core.  Thus $f$ vanishes on all of $\fM^+$, so
$\fM^+$ is contained in the vanishing locus of $J$.
 
Now suppose that $p\in\fM\smallsetminus\fM^+$;
we must produce an element of $J$ that does not vanish at $p$.
Let $p_0 \in \fM_0\smallsetminus\fM^+_0$ be the image of $p$.
Since the limit as $t$ goes to zero of $t\cdot p_0$ does not exist, there must exist
a function $f\in\C[\fM_0]\cong\C[\fM]$ such that
$$\lim_{\bT\ni t\to 0}f(t\cdot p_0) = \lim_{\bT\ni t\to 0}(t^{-1}\cdot f)(p_0)$$ does not exist;
if we require $f$ to be a $\bS\times\bT$-weight function, this means that $f$ has positive $\bT$-weight and
does not vanish at $p_0$.  Since it has positive $\bT$-weight, it is non-constant, 
and therefore has positive $\bS$-weight; taking a power, we
may assume its $\bS$-weight is at least $n$.
\end{proof}

\begin{remark}\label{scheme structure}
Until now we have only defined $\fM^+\subset\fM$ and
$\fM^+_0\subset\fM_0$ as subsets; we will now endow them with
subscheme structures given by the ideal $J$, as suggested by
Lemma \ref{vanishing}.
\end{remark}

Recall that $$\bar\xi\in A(n)/A(n-1)\subset \gr A\cong \C[\fM]$$ is defined as the symbol of $\xi\in A(n)$.
We define a {\bf $\bar\xi$-equivariant coherent sheaf} on $\fM$ to be a coherent sheaf $F$ along 
with an endomorphism  $d:F\to F$ such that, for all locally defined sections $v$ and
functions $f$, we have $d(fv)=\{\bar\xi,f\}v+fd(v)$, where $\{,\!\}$ is the Poisson bracket on $\C[\fM]$.
This definition is motivated by the following lemma.

\begin{lemma}\label{classical limit}
Let $\cN$ be a good $\cD$-module with a coherent lattice $\cN(0)\subset\cN$ that is preserved by $\xi$.\footnote{Note
that $\xi\in A(n)$, so {\em a priori} we only know that $\xi\cdot\cN(0)\subset\cN(n)$; here we are assuming
that $\xi\cdot\cN(0)\subset\cN(0)$.}
Let $\bar \cN := \cN(0)/\cN(-1)$ be the small classical limit.  
Then the action of $\xi$ defines a $\bar\xi$-equivariant structure
on $\bar\cN$.  Furthermore, the $\C[\fM]$-module $\Gamma(\bar\cN)$ is isomorphic to the associated graded of 
the filtered $A$-module $\secs(\cN)$ as a module-with-endomorphism.
\end{lemma}

\begin{proof}
The action of $\xi$ clearly descends to an endomorphism $d:\bar\cN\to\bar\cN$.
Let $f$ be a function on $\fM$, and lift it to a section $\tilde f$ of $\cD(0)$.
Let $v$ be a section of $\bar\cN$, and lift it to a section $\tilde v$ of $\cN(0)$.
Then $d(fv)$ is the image in $\bar\cN$ of the section
$\xi\tilde f\tilde v = [\xi,\tilde f]\tilde v + \tilde f \xi \tilde v$ of $\cN(0)$.
{\em A priori}, $[\xi,\tilde f]$ is a section of $\cD(n) = h^{-1}\cD(0)$, but since it is a commutator,
it in fact lies in $\cD(0)$, and it descends to the function $\{\bar\xi,f\}$; this proves that $\bar\cN$ is $\bar\xi$-equivariant.
The statement about the associated graded follows formally. 
\end{proof}

\begin{lemma}\label{coherent-fd}
For any $\bar\xi$-equivariant coherent sheaf $F$ on $\fM$ which is
set-theoretically supported on $\fM^+$, the generalized $\bar\xi$-eigenspaces 
of $\Gamma(\fM;F)$ are finite dimensional and the real parts of the eigenvalues are bounded above.
\end{lemma}

\begin{proof}
The condition holds for a sheaf if it holds for the
successive quotients of a filtration of the sheaf, thus we may
assume that $F$ is scheme-theoretically supported on a single relative core component $X_\a$.
Since $X_{\a,0}$ is affine, the pushforward
of $F$ to $X_{\a,0}$ is a quotient of the tensor product of the structure sheaf
of $X_{\a,0}$ with some finite-dimensional $\bar\xi$-module $W$.  
Thus we only need
to prove that $\bar\xi$ acts on the coordinate ring of $X_{\a,0}$ with finite dimensional generalized
eigenspaces, and that the eigenvalues that appear (all of which are integers, since the $\bar\xi$-action 
comes from the $\bT$-action) are bounded above.
This follows from the fact that every element of $X_{\a,0}$ limits to $o$ under the $\bT$-action.
\end{proof}

\subsection{The category \texorpdfstring{$\cOa$}{Oa}}\label{sec:Oa}
The action of $\bT$ on $\cQ$ induces an integer grading of $A$, where the $k^\text{th}$ graded piece is 
$A^k := \{a\in A\mid [\xi,a] = ka\}$.
Let $$A^+ := \bigoplus_{k\geq 0}A^k.$$
\begin{definition}
  We define {\bf algebraic category $\cO$} to be the full subcategory
  $\cOa$ of finitely generated $A$-modules for which the subalgebra
  $A^+$ acts locally finitely.  We define $\dOa$ to be the full subcategory
  of objects of $D^b(\Amod)$ with cohomology in $\cOa$.
\end{definition}

\begin{remark}\label{BGG-Oa}
Let $\mg$ be a simple Lie algebra with Borel subalgebra $\mb$ and Cartan subalgebra $\mh$.
An infinitesimal block of the classical BGG category $\cO$ is by definition the full subcategory of finitely generated $U(\mg)$
modules for which $U(\mb)$ acts locally finitely, $U(\mh)$ acts semisimply, and the center
of $U(\mg)$ acts with a fixed generalized character.  It is a theorem
of Soergel \cite[Theorem 1]{Soe86} that, in the case
of a regular character, this is equivalent to the category obtained by dropping the condition that $U(\mh)$
acts semisimply but adding the condition that the center of $U(\mg)$ acts with a fixed honest character.

In our setup, $A$ is the analogue of a central quotient of $U(\mg)$, 
$A^+$ is the analogue of $U(\mb)$, and we have no analogue of $U(\mh)$.
When $\fM = T^*(G/B)$ and the period of the quantization is regular (that is, it has trivial stabilizer in the Namikawa Weyl group,
which in this case is the same as the usual Weyl group), 
our category $\cOa$ is equivalent to the corresponding infinitesimal block of BGG category $\cO$ by Soergel's theorem.
If the period is not regular, then our category will be genuinely different from the corresponding
infinitesimal block of BGG category $\cO$.
\end{remark}

% We conclude this section with a pair of lemmas that we will need in Section \ref{sec:loc-O}.

\begin{lemma}\label{fg}
For all integers $k$, $A^k$ is finitely generated over $A^0$.
\end{lemma}

\begin{proof}
This follows from the corresponding statement for $\gr A \cong \C[\fM]$, which is a consequence
of the fact that $\C[\fM]$ is finitely generated as a commutative algebra, and $\C[\fM]^\bT$ is finitely generated since $\bT$ is reductive.
\end{proof}

We call an $A$-module $N$ a {\bf weight module} if it decomposes into generalized weight spaces
for the action of $\xi\in A$.  More precisely, for any $\ell\in\C$, let 
$$N^\ell := \{x\in N\mid\,\,\text{there exists $q\in\Z$ such that $(\xi-\ell)^q\cdot x = 0$}\}.$$
Then $N$ is a weight module if and only if $N = \displaystyle\bigoplus_{\ell\in\C} N^\ell$.  Note that for all $k\in\Z$ and $\ell\in\C$, $A^k\cdot N^\ell \subset N^{k+\ell}$.

\begin{lemma}\label{fdba}
A finitely generated $A$-module $N$ lies in $\cOa$ if and only if $N$ is a weight module,
$N^\ell$ is finite dimensional for
all $\ell$, and $N^\ell=0$ for all $\ell$ with sufficiently large real part.
\end{lemma}

\begin{proof}
First suppose that the three conditions are satisfied.  For any $x\in N^\ell$, 
$$A^+\cdot x\subset\bigoplus_{k\geq 0} N^{k+\ell},$$ which is finite dimensional.
Thus $A^+$ acts locally finitely, and $N\in\cOa$.

Conversely, suppose that $N\in\cOa$.  The fact that $N$ is a weight module follows
from the fact that $\xi\in A^+$ acts locally finitely.
The fact that $N^\ell = 0$ for all $\ell$ with sufficiently large real part follows from the fact
that $N$ is finitely generated, thus the vector space obtained by applying $A^+$
to a generating set is finite dimensional.
Finally, the fact that each generalized weight space of $N$ is finite dimensional
follows from Lemma \ref{fg} and the fact that $N$ is locally finite for the action of $A^0\subset A^+$.
\end{proof}

\begin{proposition}\label{Oa vs CMpa}
The category $\cOa$ is equal to the category $\CMpa$ of Definition \ref{def:CLa},
where we use the scheme structure on $\fM^+$ coming from Lemma \ref{vanishing}.
\end{proposition}

\begin{proof}
First suppose that $N\in\cOa$.  
To show that $N\in\CMpa$, we must find a good filtration of $N$
such for $k\geq n$, we have $A^+(k)\cdot N(m) \subset N(k+m-n)$.
Choose a finite dimensional subspace
$S\subset N$ which is closed under $A^+$ and generates $N$, and 
define a filtration on $N$ by putting $N(m) := A(m)\cdot S$.
If $k\geq n$, then we have 
\begin{equation*}
  A^+(k)\cdot N(m)\subset A(m)A^+(k)\cdot
  S+[A(m),A^+(k)]\cdot S
 \subset A(m)\cdot S+A(k+m-n)\cdot S = N(k+m-n).
\end{equation*}

Next, suppose that $N\in\CMpa$.  By Definition \ref{def:CLa}, this means that we may choose a filtration of $N$
such such that, for all $k\geq n$, $A^+_{\bar k} (k) \cdot N(m) \subset N(k+m-n)$.
In particular, $\xi \cdot N(m)\subset N(m)$, and $\gr N$ is set-theoretically supported
on $\fM^+_0$.
Let $F$ be the pullback of $\gr N$ from $\fM_0$ to $\fM$.  By Lemma \ref{coherent-fd} applied to $F$,
$\gr N$ has finite-dimensional $\bar\xi$ weight spaces, and the real parts of the eigenvalues are bounded above.
Thus, the same holds for $N$, and Lemma \ref{fdba} tells us that $N$ is in $\cOa$.
\end{proof}

\subsection{The category \texorpdfstring{$\cOg$}{Og}}
\begin{definition}\label{cOg-definition}
  We define {\bf geometric category $\cO$} to be the full subcategory
  $\cOg$ of $\Dmod$ consisting of modules $\cN$ such that
  \begin{itemize}
  \item  the sheaf $\cN$ is set-theoretically supported on the relative core
    $\fM^+$
  \item there exists a $\cD(0)$-lattice $\cN(0)\subset \cN$ such that
    $\xi\cdot \cN(0)\subset \cN(0)$.
\end{itemize}
We define $\dOg$ to be the full subcategory of objects of $D^b(\Dmod)$ with cohomology in $\cOg$.
\end{definition}

\begin{remark}\label{BGG-Og}
Suppose that $\fM = T^*(G/P)$.  By \cite[4.5]{BLPWquant},
$\Dmod$ is equivalent to the category of finitely generated twisted D-modules on $G/P$, where the twist
is determined by the period of the quantization.  Then $\cOg$ consists of regular twisted D-modules
with microlocal supports in $\fM^+$.
\end{remark}

Our first result is that, unlike $\cOa$, the category $\cOg$ depends only on the image of the period in the quotient
$\Ht/\Htz$.

\begin{lemma}\label{geometric twist}
Let $\la,\la'\in\Ht$ be two classes that differ by an element of $\Htz$.
Let $\cD$ and $\cD'$ be the quantizations with periods $\la$ and $\la'$, and 
let $\cOg$ and $\cOg'$ be the associated categories.
Then $\cOg$ and $\cOg'$ are canonically equivalent.
\end{lemma}

\begin{proof}
Let $\cL$ be a line bundle with first Chern class $\la-\la'\in\Htz$.
In \cite[\S 5.1]{BLPWquant} we construct a $\cQ-\cQ'$ bimodule ${}_{\la'}\cT_{\la}$ which as, a left $\cQ$-module, 
is a quantization of $\cL$.  Consider the functor ${}_{\la'}\cT_{\la}[\hmon]\otimes_{\cD}-$ from $\Dmod$ to $\cD'\mmod$.
Since ${}_{\la'}\cT_{\la}[\hmon]$ is a quantization of a line bundle, this functor does not affect the support
of an object.  Furthermore, if $\cN(0)\subset\cN$ is a lattice satisfying the second condition of the definition of $\cOg$,
then ${}_{\la'}\cT_{\la}(0)\otimes_{\cD(0)}\cN(0)$ will be such a lattice, as well.  Thus our functor takes $\cOg$ to $\cOg'$.
To show that it is a canonical equivalence, it is sufficient to show that 
${}_{\la}\cT_{\la'}[\hmon]\otimes_{\cD'}{}_{\la'}\cT_{\la}[\hmon]\cong\cD$, where ${}_{\la'}\cT_{\la}$ is a quantization of $\cL^{-1}$.
This follows from uniqueness of quantizations of line bundles \cite[5.2]{BLPWquant}.
\end{proof}

\begin{proposition}\label{Og vs CMpg}
The category $\cOg$ is equal to the category $\CMpg$ of Definition \ref{def:CLg}.
\end{proposition}

\begin{proof}
By Lemma \ref{vanishing}, a $\cD$-module $\cN$ is in $\CMpg$ if and only if it admits a lattice $\cN(0)$ that is preserved by
$h^{-1}\tilde f$ for any section $\tilde f$ of $\cQ$ whose image $f\in \cQ/h\cQ\cong\fS_\fM$
has non-negative $\bT$-weight and $\bS$-weight at least $n$.
In particular, $\xi\in \Gamma(h^{-1}\cQ)$ is of this form, and thus preserves
this lattice.  In addition, the big classical limit $\cN(0)/\cN(-n)$ is
killed by an ideal whose vanishing set is $\fM^+$, and thus is
set-theoretically supported on this locus.  The same follows for
$\cN$, so we can conclude that $\cN\in \cOg$.  Thus $\CMpg\subset\cOg$, and it remains to show the reverse inclusion.

Using Lemma \ref{geometric twist} and the analogous statement for $\CMpg$ (which can be proved in the same way),
we may add a large multiple of an ample class to the period of our quantization.
Then by Theorem \ref{localization}, we may assume that localization holds.
Let $\cN$ be an object of $\cOg$, and let $N := \secs(\cN)$. 
By Lemmas \ref{classical limit}, \ref{coherent-fd}, and \ref{fdba},
$N$ is an object of $\cOa$, which is equal to $\CMpa$ by Proposition \ref{Oa vs CMpa}.
Then by Proposition \ref{C-to-C}, $\cN = \Loc(N)$ is in $\CMpg$, and we are done.
\end{proof}

% Recall the functors $$\secs:\Amod\to\Dmod\and\Loc:\Dmod\to\Amod$$ defined
% in Section \ref{sec:localization}.  Proposition \ref{C-to-C} thus
% implies that:

% Finally, let $\cOZ$ be the category of modules $\{N_k\}$ over the
% $\Z$-algebra attached to the ample line bundle $\cL_\eta$ such that
% $N_k\in\cOa$ for $k\gg 0$, modulo objects of finite total dimension.

The following corollary follows directly from Propositions \ref{C-to-C}, \ref{derived-C-to-C}, \ref{Oa vs CMpa}, and \ref{Og vs CMpg}.

\begin{corollary}\label{alg-geom}\label{OQ}
$\Loc$ takes $\cOa$ to $\cOg$ and $\secs$ takes $\cOg$ to $\cOa$.  
In particular, if localization holds for $\cD$, then the geometric category $\cOg$ is equivalent to the algebraic category $\cOa$. Similarly, if
derived localization holds, then $\LLoc$ and $\Rsecs$ induce
an equivalence between $\dOa$ and $\dOg$.
\end{corollary}

\begin{remark}
As discussed in Remark \ref{why-bounded}, a version of Corollary \ref{alg-geom}
holds in bounded above derived categories even if derived location fails.
\end{remark}

\begin{example}
Combining Remarks \ref{BGG-Oa} and \ref{BGG-Og} with Corollary \ref{OQ},
we obtain Beilinson and Bernstein's equivalence between an infinitesimal block of BGG category $\cO$
and the category of finitely generated twisted D-modules on $G/B$, smooth with respect to the Schubert
stratification.
\end{example}

\begin{remark}
It seems slightly dissatisfying to use the unreduced scheme
structure of Remark \ref{scheme structure} on $\fM^+$.  One could
also consider the category
$\mathcal{C}^{\fM^+_{\operatorname{red}}}$ attached to the reduced
scheme structure on the subset $\fM^+$; that is to say, the category
of regular $\cD$-modules that are set-theoretically supported on
$\fM^+$.  It is clear that
$\mathcal{C}^{\fM^+_{\operatorname{red}}}\subset\CMpg = \cOg$, but
it is not clear whether or not this containment is an equality.
Since $\fM^+_{\operatorname{red}}$ is not the scheme-theoretic preimage of any subvariety of $\fM_0$,
there is no obvious choice of a corresponding subcategory of $\cOa$.
\end{remark}

%\end{proposition}
The following lemma provides an alternative formulation of the second condition of Definition \ref{cOg-definition};
it will be used to prove Proposition \ref{Serre subcategory}.

\begin{lemma}\label{stable-submodule}
  A good $\cD$-module $\cN$ possesses a $\xi$-stable coherent lattice if and only if, for
  every finitely generated $\cD(0)$-submodule $\cP$, the
  sum \[\cP_j := \cP+\xi\cdot \cP+\cdots +\xi^j\cdot \cP\]
  stabilizes for $j\gg 0$.  
\end{lemma}
\begin{proof}
First suppose that the sum stabilizes for every finitely generated $\cP$.
If we take $\cP$ to be a coherent lattice and take $j$ in the stable range, then $\cP_j$
is a $\xi$-stable lattice.

Now assume that $\cN$ admits a $\xi$-stable coherent lattice $\cN(0)$, and let $\cP\subset\cN$
be any finitely generated $\cD(0)$-submodule.
Then for some $m$, $\cP\subset \cN(m)$ and therefore $\cP_j\subset \cN(m)$ for all $j\geq 0$.  
The stabilization of $\cP_j$ then follows from the fact that $\cN(m)$ is a finitely generated module over a Noetherian ring.
\end{proof}

\begin{proposition}\label{Serre subcategory}
  The category $\cOg$ is an Abelian Serre subcategory of $\Dmod$.
\end{proposition}
\begin{proof}
In order to check that a full subcategory of an Abelian category is
Serre and Abelian, we need only check that it is closed under
quotients, submodules, and extensions.  Obviously, all these are
compatible with the support condition, so we need only consider the
existence of a $\xi$-invariant lattice.  

Let $\cN\subset\cM$ be objects of $\Dmod$.
The image of a $\xi$-invariant lattice in $\cM$ is such a lattice in
$\cM/\cN$, so $\cM\in \cOg\Rightarrow \cM/\cN\in \cOg$.  Since any
finitely generated $\cD(0)$-module $\cN$ is also a finitely
generated submodule of $\cM$, Lemma
\ref{stable-submodule} shows that $\cM\in \cOg\Rightarrow \cN\in \cOg$.

Finally, assume that $\cN\in \cOg$ and $\cM/\cN\in \cOg$.   Let $\cP$ be any finitely generated
$\cD(0)$-submodule of $\cM$; by Lemma \ref{stable-submodule}, it is sufficient to show that the sequence
$\{\cP_j\}$ of $\cD(0)$-submodules stabilizes.  Replacing $\cP$ by some $\cP_j$, we
may assume that the image of $\cP$ in $\cM/\cN$ is $\xi$-stable.  Choose
a finite generating set $\{m_i\in \Gamma(U_i;\cP)\}$ for $\cP$, 
along with $\{a^j_i\in \Gamma(U_i;\cD(0))\}$ such that for all $i$,
$\chi_i :=\xi m_i-\sum_j a^j_im_j$ lies in $\Gamma(U_i;\cN)$.
Let $\cR\subset\cN$ be the $\cD(0)$-submodule generated by $\{\chi_i\}$.

For any section $m$ of $\cP$ on an open subset $U$, we may choose
sections $b_i$ of $\cD(0)$ such that $m=\sum b_im_i$ (perhaps after
shrinking $U$),
and therefore
\[\xi  m=\sum \xi b_im_i=\sum_i [\xi,b_i]m_i+\sum_i b_i\xi m_i=\sum_i [\xi,b_i]m_i+\sum_i b_i\chi_i+\sum_{i,j}
b_ia^j_im_j.\]
Thus, $\xi\cdot \cP\subset \cP+\cR$; by induction, this implies that $\cP_j\subset \cP+\cR_{j-1}$ for all $j$.
% $\cP_j\subset \cP+\cR_{j-1}$ since \[\cP_j=\cP_{j-1}+\xi\cdot \cP_{j-1}\subset
% (\cP+\cR_{j-2})+\xi \cdot (\cP+\cR_{j-2})=\cP+\cR_{j-1}.\]  
Since $\cR$ is a submodule of $\cN$, the submodules $\{\cR_j\}$ stabilize, and thus so do the
submodules $\{\cP_j\}$.
\end{proof}

\section{Categorical preliminaries}\label{structure}
In this section we will collect various definitions and basic results about Koszul, highest weight,
and standard Koszul categories, which we will apply to $\cOa$ and $\cOg$ in the next section.

\subsection{Koszul categories}\label{hwksk}
Much of the
material in this section comes from the seminal work \cite{BGS96},
though our presentation follows more closely that of \cite{MOS}, to
which we refer the reader for further details. 

Let $\tilde {\cal C}$ be a $\C$-linear Abelian category, which we will
assume throughout is Noetherian and Artinian with enough projectives
and finite projective dimension.  Consider a choice of weight
$\wt(L)\in \Z$ for each simple in $\tilde {\cal C}$, and assume
further that there are finitely many simples of any given weight.  The category
$\tilde {\cal C}$ is said to be {\bf mixed} if, whenever
$\Ext^1(L,L')\neq 0$, we have $\wt(L')< \wt(L)$.  A {\bf Tate
  twist} on a mixed category is an autoequivalence $M\mapsto M(1)$
such that $\wt(L(1))=\wt(L)-1$.

Let $\tilde {\cal C}$ be a mixed category, and let $\tilde {\cal C}/\Z$ be the category 
whose objects are the same as those of $\tilde {\cal C}$,
but whose morphism spaces are the graded vector spaces $$\Hom_{\tilde {\cal
    C}/\Z}(M,M'):=\bigoplus_{d\in\Z} \Hom_{\tilde {\cal C}}(M,M'(-d)).$$ Note that every object in $\tilde {\cal
  C}/\Z$ is isomorphic to all of its Tate twists.  

If $P$ is an indecomposable projective in $\tilde{\cal C}$,
or more generally a projective whose head is concentrated in a single
weight, then $\End_{\tilde {\cal
    C}/\Z}(P)$ is a mixed algebra in the sense of \cite[4.1.5]{BGS96},
and thus positively graded.  In fact,
any mixed category with Tate twist is equivalent to the category of
finite dimensional graded modules over a projective limit of finite
dimensional positively graded algebras
with semi-simple degree zero part \cite[4.1.6]{BGS96}; if we assume in addition that the are
finitely many simples of weight 0, then the condition of having enough projectives
guarantees that this algebra can be taken to be finite dimensional.  

Let $\operatorname{Vect}_\C$ be the category of finite-dimensional complex vector spaces.
Define the {\bf degrading} $\cal C$ of $\tilde{\cal C}$ to 
be the category of additive functors
$(\tilde {\cal C}/\Z)^{\operatorname{op}}\to \operatorname{Vect}_\C$ for which the 
composition with the natural functor $\tilde{\cal C}^\op\to (\tilde{\cal C}/\Z)^\op$ is  
left-exact\footnote{This is a degrading in
    the sense of \cite[\S 4.3]{BGS96}, but it is somewhat stronger, since the
    condition $(*)^i_{M,N}$ is automatic from the Yoneda lemma.}.  This
  is the same as the category of additive functors from the opposite
  category of projectives in $\tilde{\cal C}$ to $\operatorname{Vect}_\C$, since every object in $\tilde{\cal C}$ 
  can be presented as the cokernel of a map between projectives.  In
  more concrete language, $\cal C$ is the category of finite dimensional right
  modules over the endomorphism ring in $\tilde{\cal C}/\Z$ of the sum of
  the projective covers of all simples of weight 0.  

We
  say that $\tilde{\cal C}$ is a {\bf graded lift} of $\cal C$.
  Similarly, we can define the degrading of a functor between mixed
  categories, and speak of graded lifts of functors.
 We call an object of $\cal C$ {\bf gradeable} if it is representable, that is, if it is of
  the form $\Hom(M,-)$ for $M$ in $\tilde {\cal
    C}/\Z$.  If $\tilde{\cal C}$ is the category of finite dimensional graded modules over a
  positively graded finite dimensional algebra, then $\cal C$ is the
  category of finite dimensional ungraded modules.
  
We now define the category of {\bf linear complexes} of projectives, which plays a central role in our discussion of Koszul duality below.

\begin{definition}
  Let $\LtC$ denote the category whose objects are complexes
  $X_\bullet$ of projective objects in $\tC$ such that all summands of
  the head of $X_j$ have weight $j$.  This is ``linear'' in the sense
  that if we shifted every term to have head which is weight 0, then
  every differential would have ``degree 1.''
  The morphisms in $\LtC$ are chain maps.  
\end{definition}

Remarkably, $\LtC$ is an Abelian subcategory of the category
of complexes in $\tilde{\cal C}$.  The simple modules of $\LtC$
are the complexes given by a single indecomposable projective in a
single degree $j$; we can weight the category $\LtC$ by
endowing a complex concentrated in degree $j$ with weight $j$. The
Tate twist on this category is given by $[-1](-1)$.   An indecomposable injective
in this category arises as a quotient of a minimal projective
resolution of a simple by the subcomplex consisting of objects
with head in weight less than $j$ in the $j^\text{th}$ term.

\begin{remark}
The category $\LtC$ is canonically equivalent to the quadratic dual of the category $\tilde{\cal C}$ 
\cite[Theorem 12]{MOS}, even if $\tilde{\cal C}$ is not itself quadratic.
\end{remark}

\begin{definition}\label{def:Koszul}
% The category $\tilde{\cal C}$ is said to be {\bf Koszul} if the minimal projective resolution of every simple module is linear.
The category $\cal C$ is said to be {\bf Koszul} if it admits a graded lift $\tC$
with the property that the minimal projective resolution of every simple object in $\tC$ is linear.
If $\cal C$ is Koszul, then any two Koszul graded lifts are canonically equivalent as mixed categories \cite[2.5.2]{BGS96}.
\end{definition}

Mazorchuk, Ovsienko, and Stroppel \cite[\S 5.1]{MOS} define a pair of
adjoint functors\footnote{In \cite{MOS}, a 
%slightly 
different finiteness condition is used on the derived 
  category, but the Artinian and finite global dimension hypotheses
  guarantee that the functors are bounded.}
$$
	K_{\tilde{\cal C}} : D^b(\tilde{\cal C}) \longrightarrow D^b(\LtC)
\and	K'_{\tilde{\cal C}} :  D^b(\LtC) \longrightarrow D^b(\tilde{\cal C}).
$$
We refer there for the complete definition; 
the facts we will need about these functors are summarized in the following result \cite[Theorem 30]{MOS}.

\begin{theorem} \label{prop:koszulformula}
The following are equivalent for a non-negatively graded category $\tilde{\cal C}$:
\begin{enumerate}
\item The degrading $\cC$ is Koszul.
\item The functors $K_{\tilde{\cal C}}$ and $K'_{\tilde{\cal C}}$ above are mutually inverse equivalences of categories.
\item The functor $K_{\tilde{\cal C}}$ takes each indecomposable projective to the corresponding simple.
\item The functor $K'_{\tilde{\cal C}}$ takes each indecomposable injective to the corresponding simple.
\end{enumerate}
\end{theorem}

\begin{proposition}
  If $\cC$ is Koszul, then so is $\LtC$, and there is a canonical
  equivalence of categories between $\tilde{\cal C}$ and the double
  dual $\LCP\!\big(\LtC\big)$.
\end{proposition}
\begin{proof}
By \cite[Theorem 12]{MOS}, Koszulity of $\LtC$ is equivalent to the quadratic dual of $\cC$ being Koszul, which follows
from \cite[2.9.1]{BGS96}.

Consider the composite equivalence
\[K_{\LtC }\circ K_{\tilde{\cal C} }\colon D^b(\tilde{\cal C})\to
D^b(\LCP\!\big(\LtC\big) );\] this is an equivalence of derived
categories sending projectives to injectives.  Composing
with the inverse of the derived Nakayama functor,  we
obtain an equivalence of derived categories sending projectives to
projectives, and thus inducing an equivalence of Abelian categories
$\tilde{\cal C}\cong \LCP\!\big(\LtC\big) $.  
\end{proof}

\begin{remark}\label{rem:Yoneda}
Though we have emphasized categories rather than algebras in the above
definition of Koszul duality, it is sometimes convenient to reconsider
the Koszul duality statements above from the point of view of
algebras.   Let $P$ be the sum of the indecomposable
projectives in $\cC$; then $\cC$ is equivalent to the category of finite-dimensional $\End_\cC(P)^{\op}$-modules.  
Let $L$ be the sum of the simples in $\cC$, and choose a projective resolution $\Pi_\bullet$
of $L$.  The algebra \[\cE := \bigoplus_{m\in\Z}\Hom(\Pi_\bullet,\Pi_\bullet[m])\] is naturally a dg-algebra, which we call
the {\bf dg-Yoneda algebra} of $\cC$; this algebra depends on the choice of $\Pi_\bullet$ only up to quasi-isomorphism.  
The cohomology ring 
$$E := H^\bullet(\cE) \cong \Ext_\cC(L,L)$$ is the ordinary {\bf Yoneda algebra} of $\cC$.  
The algebra $E$ carries a natural (cohomological) grading, and thus has a category of graded modules $E \gmmod$.

Now assume that $\cC$ is Koszul, and
let $\tilde\Pi_\bullet$ be a linear projective resolution of the weight 0 graded lift $\tilde L$ of $L$.
Then $\tilde\Pi_\bullet$ is an injective generator in $\LtC$, and
the algebra of endomorphisms (of arbitrary degree) of $\tilde\Pi_\bullet$ in $\LtC/\Z$ is a
quotient dg-algebra of $\cE$ with trivial differential, killing all elements of $
\Hom(\tilde\Pi_\bullet,\tilde\Pi_\bullet[m](m))$ of positive degree.  The quotient 
map is a quasi-isomorphism, so $$E\cong \End_{\LtC/\Z} (\tilde\Pi_\bullet,\tilde\Pi_\bullet)\cong \bigoplus_m\End_{\LtC} \left(\tilde\Pi_\bullet,\tilde\Pi_\bullet[m](m)\right).$$
Thus $E\gmmod\cong \LtC$, and 
the algebra $\End_\cC(P) = \End_{\tilde{\cal C}/\Z}(\tilde P)$ is isomorphic as a graded algebra to the Yoneda algebra of $E\gmmod$.
This demonstrates explicitly that, if $\cC$ is Koszul, it has a unique graded lift $\tC$ (see Definition \ref{def:Koszul}).
\end{remark}

\begin{remark}
The Yoneda algebra of the category $\cOg$ will be studied in Section \ref{sec:hochsch-cohom-cent}.
\end{remark}

For the purposes of Section \ref{duality}, it will be convenient to introduce the following definition.

\begin{definition}\label{def:dual}
Let $\cC$ and $\cC^!$ be two Koszul categories, and let $\tC$ and $\tC^!$ be their graded lifts.
A {\bf Koszul duality} from $\cC$ to $\cC^!$ is an equivalence
of mixed categories $$\psi:\LtC\to\tC^!.$$  
Taking derived functors and precomposing with $K_{\tC}$, we also obtain
  an equivalence of  triangulated categories
$$\Psi: D^b(\tC) \to D^b({\tC^!})$$
sending projective objects of $\tC$ to simple objects of $\tC^!$ and sending the Tate twist $(1)$ to the functor $(-1)[-1]$.
Conversely, any such equivalence of triangulated categories must
induce an equivalence $\LtC\cong\tC^!$, as these are the hearts of $t$-structures for
which the functor is exact.  Thus, we will also refer to $\Psi$ as a Koszul duality from $\cC$ to $\cC^!$.
We say that $\cC^!$ is {\bf Koszul dual} to $\cC$ if there is a Koszul duality between them.
\end{definition}

While we have not defined Koszul duality in a way which is obviously
symmetric, the following observation demonstrates that it is a true duality.

\begin{proposition}\label{Koszul-symmetric}
If $\psi:\LtC\to\tC^!$ is a Koszul duality from $\cC$ to $\cC^!$, then the inverse of
$$\LCP(\psi):\tC\cong\LCP\!\big(\LtC\big)\to\LCP(\tC^!)$$ is a Koszul duality from $\cC^!$ to $\cC$.
\end{proposition}

\begin{remark}\label{Koszul-symmetric-derived}
We could also phrase Proposition \ref{Koszul-symmetric} in terms
of the derived equivalence $\Psi: D^b(\tC) \to D^b({\tC^!})$.
The induced Koszul duality in the reverse direction is the precomposition of 
$\Psi^{-1}$ with the right derived Nakayama functor. 
\end{remark}

\subsection{Highest weight and standard Koszul categories}
Let $\mathcal C$ be a $\C$-linear Abelian, Noetherian category with simple objects $\{L_\a\mid\a\in\cI\}$,
projective covers $\{P_\a\mid\a \in \cI\}$, 
and injective hulls $\{I_\a\mid\a\in\cI\}$.
Let $\leq$ be a partial order on $\cI$.

\begin{definition}\label{high-weight-def}
  We call $\mathcal C$ {\bf highest weight} with respect to this partial order if there is a collection of objects
  $\{\Ve\a\mid\a \in \cI\}$ and epimorphisms 
  $\Pro \a\overset{\Pi_\a}\to\Ve \a\overset{\pi_\a}\to\Si \a$
  such that  for each $\a\in \cI$, the following conditions hold:
  \begin{enumerate}
  \item The object $\ker\pi_\a$ has a filtration such that each sub-quotient is isomorphic to $\Si \b$ for some $\b< \a$.
  \item The object $\ker \Pi_\a$ has a filtration such that each sub-quotient is isomorphic to $\Ve \gamma$ for some $\gamma> \a$.
  \end{enumerate}
The objects $\Ve\a$ are called {\bf standard objects}.
Classic examples of highest weight categories in representation theory include integral
blocks of parabolic BGG category $\cO$ \cite[5.1]{FM}.
\end{definition}

In any highest weight category, we also have a notion of {\bf costandard} objects.

\begin{definition}
  Let $\nabla_\a$ be the largest subobject of $I_\a$ whose composition
  factors are all isomorphic to $L_\b$ with $\b\leq \a$.  By
  \cite[Theorem 1]{DlabRingel}, the category  $\mathcal C$ is highest
  weight if and only if $I_\a$ has a filtration by costandards $\nabla_\gamma$
  analogous to the standard filtration on projectives.
\end{definition}

If $\cC$ admits a graded lift $\tC$, then every standard object of $\cC$ is gradeable.
More precisely, if $\tilde P_\a$ is a graded lift of $P_\a$, then we may define $\tilde\Delta_\a$
to be the largest quotient of $\tilde P_\a$ with no composition factors of the form $\tilde L_\b$ for $\b>\a$,
and $\tilde \Delta_\a$ will be a graded lift of $\Delta_\a$.  We refer to the graded lifts of standard objects of $\cC$
as standard objects of $\tC$.

\begin{definition}
A highest-weight category $\cC$ is {\bf standard Koszul} if it admits a graded lift $\tC$
with the property that the minimal projective resolution of every standard object in $\tC$
is linear.  (Compare this definition to Definition \ref{def:Koszul}, in which ordinary Koszulity is defined.)
\end{definition}

The following result is the main theorem of \cite{ADL03}.

\begin{theorem}\label{standard Koszul implies Koszul}
A highest-weight category $\cC$ is standard Koszul with respect to a given partial order if and only if it is Koszul and
its Koszul dual $\cC^!$ is highest weight with respect to the opposite partial order.
\end{theorem}

\section{The structure of \texorpdfstring{$\cOa$}{Oa} and \texorpdfstring{$\cOg$}{Og}}\label{Oa and Og structure}
In this section we fix a conical symplectic resolution and a Hamiltonian
$\bT$-action with isolated fixed points and we investigate the structure of the categories 
$\cOa$ and $\cOg$ for various different quantizations.
We prove that $\cOa$ is highest weight for most quantizations
(Proposition \ref{B} and Theorem \ref{highest weight}), and that $\cOg$ is highest weight
for all quantizations (Proposition \ref{ghw}).  Furthermore, we conjecture that 
both categories are standard Koszul (Conjectures \ref{Koszul} and \ref{geom-Koszul}), based on the evidence from the 
theory of hypertoric varieties (Example \ref{htkoszul}) and from classical Lie theory (Example \ref{spkoszul}).

We also include a brief discussion of the Yoneda algebra of $\cOg$ (Section \ref{sec:hochsch-cohom-cent}).
In particular, we define a natural map from $H^*(\fM; \C)$ to the center of the Yoneda algebra, and conjecture
that this map is often an isomorphism (Conjecture \ref{HH}).

\subsection{The B algebra}\label{sec:B}
For any $\Z$-graded ring $A$, let \[B(A):= A^0\Big{/}\sum_{k>0}A^{-k}A^{k}.\] 

\begin{proposition}
If $A$ is the $\bS$-invariant
section ring of a quantized conical symplectic resolution $\fM$, with grading induced by the action of $\bT$,
then $B := B(A)$ is finite dimensional as a vector space.
\end{proposition}

\begin{proof}
Consider the associated graded ring $\gr B$; 
this ring admits a surjection from
$B(\C[\fM])$, where $\C[\fM]$ is also graded by the action of $\bT$.
% (It is not clear whether this surjection is an isomorphism, since, {\em a priori},
% $\C[\fM]^{-k}\cdot \C[\fM]^k$ might be strictly smaller than the associated graded of $A^{-k}\cdot A^k$.)
It therefore suffices to show that $B(\C[\fM])$ is finite dimensional.

Let $p_0$ be an element of $\fM_0$.
In the proof of Lemma \ref{vanishing}, we showed that if $p_0 \notin \fM_0^+$, then
there exists a function of positive $\bT$-weight that does not vanish at $p_0$.
Let $\fM^-$ be the relative core for the opposite $\bT$-action, and let $\fM^-_0$ be its image in $\fM_0$. By the same reasoning,
if $p_0 \notin \fM_0^-$, then there exists a function of negative $\bT$-weight that does not vanish at $p_0$.
It follows that the set-theoretic vanishing locus in $\fM_0$ 
of the ideal $I$ generated by all homogeneous elements of nonzero $\bT$-weight
is equal to $\fM_0^+\cap \fM_0^- = \{o\}$, and therefore that this ideal has finite codimension
in $\C[\fM_0] = \C[\fM]$.  Thus $B(\C[\fM]) = (\C[\fM]/I)_0$ is finite dimensional.
\end{proof}

For each $\a\in\cI$, let $U_\a$ be a formal neighborhood of $p_\a$.  
Although the groups $\bS$ and $\bT$ do not act on $U_\a$ or on $\cD|_{U_\a}$, their Lie algebras do,
so we can make sense of the ring $\tilde A_\a := \secs(\cD|_{U_\a})$.
Let $A_\a$ be the subring of $\tilde A_\a$ that is additively spanned by $\bT$-weight vectors (this means passing from
power series to polynomials), and let
$B_\a := B(A_\a)$.  Then $A_\a$ and $B_\a$
admit natural maps from $A$ and $B$, respectively.

\begin{lemma}\label{Aa-weyl}
Let $d = \frac{1}{2}\dim\fM$.
The algebra $A_\a$ is isomorphic to the ring of global differential operators on $\C^d$, 
and $B_\a$ is isomorphic to $\C$.
\end{lemma}

\begin{proof}
There is only one quantization
of the formal polydisk \cite[1.5]{BK04a}, thus the ring of sections $\Gamma(\cD|_{U_\a})$
must be isomorphic to the Weyl algebra
$$\C[[x_1, y_1, \ldots, x_d, y_d, \hon]](\hmon)\Big{/}\Big\langle [h,x_i], [h,y_i],[x_i,x_j], [y_i,y_j], [x_i,y_j]- h\delta_{ij}\Big\rangle.$$ 
% (recall that $n$ is the $\bS$ weight of the symplectic form).
We may choose $x_1, y_1, \ldots, x_d, y_d, h$ to be simultaneous
weight vectors for $\bS$ and $\bT$, with each $x_iy_i$ and $h$ having $\bS$ weight $n$ and $\bT$ weight 0.
If $\chi_i$ is the $\bS$ weight of $x_i$, let 
$$z_i := h^{\nicefrac{-\chi_i}{n}}x_i\and w_i:=  h^{\nicefrac{\chi_i}{n}-1}y_i.$$
Then $A_\a$ is generated by $$\{z_1,w_1,\ldots,z_d,w_d\},$$
subject to the relations $$[z_i, z_j] = 0, \qquad [w_i,w_j] =0, \and [z_i,w_j] = \delta_{ij}.$$
The $\C$-vector space spanned by the $z_i$ and $w_i$ is isomorphic as a $\bT$-space to the tangent space $T_{p_\a}\M$.  Since $p_\a$ is an isolated fixed point, none of the $z_i$ or $w_i$ can have zero weight.  Without loss of generality suppose that the $\bT$-weight of $z_i$ is negative.  The ring $A_\a$ has a PBW basis given by monomials of the form $z^aw^b$ for $a, b \in \N^d$.  All such monomials with $a \ne 0$ are clearly in $A_\a^{-k}A_\a^k$ for some $k > 0$, so $B_\a$ is at most one-dimensional.  On the other hand, the action of $A_\a^0$ on the $\bT$-invariant part of $A_\a/A_\a\langle w_1,\dots, w_d\rangle$ descends to a nontrivial action of $B_\a$, so $B_\a$ cannot be zero. 
\end{proof}

In what follows, fix an $\bS$-equivariant line
bundle $\cL$, very ample over $\fM_0$, and let $\eta\in\Ht$ be its
Euler class.  Fix another class $\la\in\Ht$, and for all $k\in\C$, let
$\cQ_k$ be the quantization of $\fM$ with period $\la + k\eta$.
Let $$A_k := \secs(\cD_k) \and B_k := B(A_k).$$ For each $\a\in\cI$,
define $\Aka$ and $\Bka$ as above.

\begin{proposition}\label{B}
The natural map $\varphi_k :B_k\to\displaystyle\bigoplus_{\a\in\cI} \Bka$
is an isomorphism for all but finitely many values of $k\in \C$.
\end{proposition}

\begin{proof}
Our plan is to construct a family of maps, parametrized by the affine line, such that the fiber over $k$
is the map $\varphi_k$, and to show that the generic map $\varphi_\infty$ is an isomorphism.

To accomplish this, we work with the twistor deformation $\scrM_\eta$
of $\fM$ over $\aone$, introduced in \ref{sec:nam}.  Let $\Delta :=
\Spec\C[[h]]$ be the formal disk, and let
$\sigma_k:\Delta\to\Delta\times\aone$ be the map that is the identity
on the first coordinate and pulls back the coordinate on $\aone$ to
$kh$.  Following the argument in \cite[4.17]{BLPWquant}, there exists
a quantization $\scrD$ of $\scrM_\eta$ such that $\cD_k$ is isomorphic
to the pull-back of $\scrD$ along the map $\sigma_k$ via the pull-back
construction described in \cite[\S 3.1]{BLPWquant}.  The action of
$\bT$ extends to this situation by \cite[1.5]{KalPois}.

Let $\scrU_\a\subset\scrM_\eta$ be a formal neighborhood of the
component of $\scrM_\eta^\bT$ corresponding to $\a$, so that
$\scrU_\a$ is a deformation of $U_\a$ over $\aone$.  Let $\pi$ be the
projection from $\scrM_\eta$ to $\aone$.  Let $$\scrB := B(\pi_*\scrD)
\and \scrB_\a := B(\pi_*\scrD|_{\scrU_\a});$$ both are sheaves of
algebras over $\aone$, and we have a natural map
$\varphi:\scrB\to\scrB_\a$ whose fiber over $k$ is $\varphi_k$.

By a result of Kaledin \cite[2.5]{KalDEQ}, the generic fiber $\scrM_{\eta}(\infty)$ is affine.
This tells us that the attracting sets to the fixed points are all closed affine spaces, so 
the associated graded algebras of $B(\infty)$ and $\oplus B_{\a}(\infty)$ are both isomorphic to the coordinate ring
of $\scrM_{\eta}(\infty)^\bT$.  It follows that the map $\gr\varphi(\infty)$ is an isomorphism, and thus so is $\varphi(\infty)$.
\end{proof}

\vspace{-\baselineskip}
\begin{remark}\label{nonisolated B}
  If we use an action of $\bT$ which does not have isolated fixed
  points, then these results proceed through in almost the same way,
  but with one important change: the algebras $B_\a$ should now be
  indexed by components of $\fM^{\bT}$, and each one will be given by global
  sections of an induced quantization on the corresponding component. 
\end{remark}

\subsection{The category  \texorpdfstring{$\cOa$}{Oa} is highest weight (for most quantizations)}\label{objects}
Throughout this section we will assume that the map $\varphi:B\to\oplus B_\a\cong\C^\cI$ is an isomorphism.
By Proposition \ref{B}, this is the case for ``most" quantizations.

For each $\a\in\cI$, let $$\Delta_\a := A \otimes_{A^+} B_\a
\and \nabla_{\!\a} := \Big(B_\a^*\otimes_{A^+} A\Big)^\star,$$
where $B_\a$ is regarded as a quotient of $B$ (and therefore also of $A^+$).
Here $*$ denotes ordinary vector space duality and $\star$ denotes restricted duality:
if $N$ is a finitely generated right weight module (as defined in Section \ref{sec:Oa}), 
then $N^\star := \oplus_{\ell\in\C} (N^\ell)^*$.
We will refer to $\Delta_\a$ and $\nabla_{\!\a}$ as the {\bf
standard} and {\bf costandard} modules indexed by $\a$.

\begin{lemma}\label{standards}
The modules $\Delta_\a$ and $\nabla_{\!\a}$ lie in $\cOa$.
\end{lemma}

\begin{proof}
The fact that $\Delta_\a$ lies in $\cOa$ follows from Lemmas \ref{fg} and \ref{fdba}, and the proof
that $\nabla_{\!\a}$ lies in $\cOa$ is identical.
\end{proof}

\vspace{-\baselineskip}
\begin{lemma}\label{simples}
Each standard object $\Delta_\a$ has a unique simple quotient $L_\a$.
Furthermore, every simple object of $\cOa$ is isomorphic to a unique
element of the set $\{L_\a\mid \a\in\cI\}$.
\end{lemma}

\begin{proof}
Let $\ell_\a$ be the highest weight (measured by its real part) that appears in $\Delta_\a$.
Then $\Delta_\a^{\ell_\a}$ is annihilated by $A^k$ for all positive $k$, and is therefore naturally a $B$-module;
it is isomorphic as a $B$-module to $B_\a$.
Let $N_\a$ be the sum of all submodules of $\Delta_\a$ that do not contain $\Delta_\a^{\ell_\a}$.  Then 
$L_\a:=\Delta_\a/N_\a$ is evidently nonzero and simple.  Furthermore, it is the only simple quotient of $\Delta_\a$,
since $\Delta_\a$ is generated by its highest weight space.

If $\a\neq\a'$, then the highest weight spaces of $L_\a$ and $L_{\a'}$
are not isomorphic as $B$-modules, therefore $L_\a$ and $L_{\a'}$ cannot be isomorphic as $A$-modules.
Now suppose that $L$ is an arbitrary simple object of $\cOa$.  The highest weight space of $L$
must be isomorphic as a $B$-module to $B_\a$ for some $\a\in\cI$.  We get a natural $A^+$-module homomorphism $B_\a \to L$, which induces an $A$-module homomorphism $\Delta_\a \to L$, which is a surjection since $L$ is simple.  Thus $L$ is a quotient of $\Delta_\a$, so it is isomorphic to $L_\a$.
\end{proof}

\vspace{-\baselineskip}
\begin{lemma}\label{finite length}
All objects of $\cOa$ have finite length.
\end{lemma}

\begin{proof}
Lemma \ref{simples} tells us that there are finitely many simple objects,
so it is enough to prove that each simple object appears finitely many times
in the composition series of any object of $\cOa$.
This follows from Lemma \ref{fdba}, which says that each generalized weight
space of an object of $\cOa$ is finite dimensional.
\end{proof}

\vspace{-\baselineskip}
\begin{lemma}\label{basic}
For all $\a\in\cI$, $\End_{\cOa}(L_\a) = \C$.
\end{lemma}

\begin{proof}
The natural maps $\C = \End_B(B_\a)\to\End_{\cOa}(\Delta_\a)\to\End_{\cOa}(L_\a)$ are isomorphisms.
\end{proof}

For any subset $\cK\subset\cI$, let $\cOa(\cK)$ be the full subcategory of $\cOa$ consisting
of objects whose simple subquotients all lie in the set $\{L_\a\mid \a\in\cK\}$.
Consider the partial order on $\cI$ generated by putting $\a\leq \a'$ if $L_\a$ is isomorphic to a subquotient of $\Delta_{\a'}$ or of $\nabla_{\a'}$. 

\begin{remark}  We will show in Corollary \ref{costandard equals standard in K} below that $\Delta_{\a}$ and $\nabla_{\a}$ have the same composition series multiplicities for most quantizations, so in fact this partial order can also be defined using only one of these classes of objects.
\end{remark}

\begin{lemma}\label{cover and hull}
Let $\cK\subset\cI$ be closed in the order topology (that is, $\a\leq\a'\in\cK \Rightarrow \a\in\cK$)
and let $\a\in\cK$ be a maximal element.
Then the natural surjection $\Delta_\a\to L_\a$ is a projective cover in $\cOa(\cK)$ and 
the natural injection $L_\a\to\nabla_{\!\a}$ is an injective hull in $\cOa(\cK)$.
\end{lemma}

\begin{proof} Consider the functor from $\cOa(\cK)$
to the category of vector spaces taking $N$ to $\Hom(\Delta_\a, N)\cong\Hom_{A^+}(B_\a, N)$.
We wish to show that this functor is exact; it is obviously
left-exact, so we need only show that it induces a surjection when
applied to a surjection.

Assume not, and let $\phi\colon \Delta_\a \to N/N'$ be a homomorphism in $\cOa(\cK)$ which cannot be lifted to a map $\Delta_\a \to N$.  Without loss of generality we can assume that $\phi$ is surjective, and that $N$ is generated as an $A$-module by a vector $v\in N^{\ell_\a}$ which lifts a nonzero vector in $\phi(B_\a)$.  
We can further assume that $N/N'$ is isomorphic to $L_\a$.

Let $\ell_\b$ be the highest weight appearing in $N$.  By adjunction we get a homomorphism $\psi\colon N \to \nabla_\b$ which is an isomorphism on the $\ell_\b$-weight space.  Since $N$ is generated by $v$, it follows that $\psi(v)\ne 0$, and so $L_\a$ appears in a composition series of $\nabla_\b$.  Thus $\a < \b$, and $L_\b$ appears in a composition series of $N$, contradicting the fact that $N$ lies in $\cOa(\cK)$.

Thus $\Delta_\a$ is projective.  That it is the projective cover
of $L_\a$ follows from the fact that $\Delta_\a^{\ell_\a}$ is 1-dimensional.
The second statement follows similarly from the fact that $\nabla_{\!\a}$ corepresents the vector space dual
of the same functor.
\end{proof}

For all $\a\in\cI$, let $\cK_{\a} := \{\a'\in \cI\mid \a'<\a\} = \overline{\{\a\}}\smallsetminus\{\a\}$.

\begin{lemma}\label{ker and coker}
For any $\a\in\cI$, the kernel of $\Delta_\a\to L_\a$ and the cokernel of $L_\a\to\nabla_{\!\a}$
both lie in the subcategory $\cOa(\cK_\a)$.
\end{lemma}

\begin{proof}
It suffices to show that $L_\a$ appears in the composition series of both $\Delta_\a$
and $\nabla_{\!\a}$ with multiplicity exactly 1.  This follows from the fact
that $\dim\Delta_\a^{\ell_\a} = \dim\nabla_{\!\a}^{\ell_\a} = 1$.
\end{proof}

Consider the set  $\mathfrak{U}\subset\Ht$ consisting of periods of
quantizations such that for all $\a, \a'\in\cI$, $\Ext^k_{A}\!\big(\Delta_\a,
\nabla_{\!\a'}\big) = 0$ for $k>0$.  Note that by \cite[3.2.3]{BGS96}, this
  further implies that for all $\a, \a'\in\cI$, $\Ext^2_{\cOa}\!\big(\Delta_\a,
  \nabla_{\!\a'}\big) = 0$.
By Losev's Theorem \ref{Thm:ext_coinc1}, $\mathfrak{U}$ contains a non-empty Zariski open subset.
and for every $\eta\in\Ht$ with $\scrM_\eta$ affine,
we have $\kappa\eta+\la\in\mathfrak{U}$ for all but finitely many $\kappa\in \C$.
% \begin{lemma}\label{exts}
% The set $\mathfrak{U}$ contains a non-empty Zariski open subset, and for every $\eta\in\Ht$ with $\scrM_\eta$ 
% affine, we have $\kappa\eta+\la\in\mathfrak{U}$ for all but finitely many $\kappa\in \C$.
% \end{lemma}
The following theorem can be deduced from Theorem \ref{Thm:ext_coinc1} and Lemmas \ref{standards}-\ref{ker and coker}
via \cite[3.2.1]{BGS96}.

\begin{theorem}\label{highest weight}
Assuming that the quantization $\cD$ is chosen such that $\varphi$ is an isomorphism and 
the period of $\cD$ lies in $\mathfrak{U}$ (both generic conditions),
the category $\cOa$ has enough projectives and is highest weight with respect to our partial order.
\end{theorem}

The following corollary follows by an argument identical to that in \cite[3.3.2]{BGS96}.

\begin{corollary}\label{FF}
  For any $\la \in \mathfrak{U}$, the inclusion $D^b(\cOa)\to
  \dOa$ is an equivalence of categories.
\end{corollary}

\begin{proof}
These categories have a common generator as triangulated
categories, given by a projective generator $P$ or injective generator
$I$.  Thus,
  it suffices to check that the map $\Ext^i_{\cOa}(P,I)\to
  \Ext^i_{A_\la}(P,I)$ is an isomorphism for all $i$, and all projectives
  $P$ and injectives $I$ in $\cOa$.  For $i=0$, this
  is just the fact that $\cOa$ is a full subcategory.  For $i>0$, both
  sides are 0, since $P$ is standard filtered and $I$ costandard filtered.
\end{proof}

\vspace{-\baselineskip}
\begin{conjecture}\label{Koszul}
Whenever $\cOa$ is highest weight, it is also standard Koszul.
\end{conjecture}

\begin{example}\label{htkoszul}
If $\fM$ is a hypertoric variety and $\cD$ is chosen correctly, 
then $\cOa$ is standard Koszul by \cite[4.10]{BLPWtorico}.  The rings $R$ and $E$ are
isomorphic to the rings $A$ and $B$ introduced in \cite{GDKD}.
\end{example}

\begin{example}\label{spkoszul}
If $\cOa$ is a regular infinitesimal block of BGG category $\cO$ (see Remark \ref{BGG-Oa}),
then it is known to be standard Koszul by \cite{BGS96, RC, ADL03}.  (See \cite[9.2]{kosdef}
for more details.)
\end{example}

\subsection{The category \texorpdfstring{$\cOg$}{Og} is highest weight}
We begin by using Theorem \ref{highest weight} to prove that $\cOg$ is always highest weight.

\begin{proposition}\label{ghw}
For any choice of quantization, the category $\cOg$ is highest weight
and the inclusion $D^b(\cOg)\to \dOg$ is an equivalence of categories.
\end{proposition}

\begin{proof}
Let $\la\in\Ht$ be the period of the quantization, and let $\eta\in\Htz$ be an ample class.
By Lemma \ref{geometric twist}, we may replace $\la$ by $\la + k\eta$ for any $k\in\Z$.
By Proposition \ref{localization}, we may choose $k$ large enough so that localization holds.
By Theorem \ref{highest weight}, we may also choose $k$
large enough so that $\cOa$ is highest weight.
By Corollary \ref{OQ}, this implies that $\cOg$ is highest weight, and
by Corollary \ref{FF}, the full faithfulness follows as well.
\end{proof}

By the same argument, the following conjecture would be implied by Conjecture \ref{Koszul}.

\begin{conjecture}\label{geom-Koszul}
For any choice of quantization, the category $\cOg$ is standard Koszul.
\end{conjecture}

\begin{remark}
Conjectures \ref{Koszul} and \ref{geom-Koszul} will not come up again until Section \ref{duality},
where they will play a central role in the definition of symplectic duality of conical symplectic resolutions.
\end{remark}

In the remainder of this section we give an explicit construction for the costandard modules in $\cOg$, which will be useful
for our study of the Grothendieck group of this category in Section \ref{gco-intform}.
Let $\Theta_\a := A_\a\otimes_{A_\a^+}B_\a$, regarded as a module over $A$.

\begin{proposition}\label{Theta-character}
Let $d = \frac{1}{2}\dim\fM$.
The $\bT$-character of $\Theta_\a$ is $e^{w_\a}\prod_{i=1}^{d} (1-e^{-\chi_i})^{-1}$, where $\chi_1,\ldots,\chi_d$ 
are the positive weights (with multiplicity) for the action of $\bT$ on $T_{p_\a}\fM$ and $w_\a$ is the $\bT$-weight of $B_\a$.
\end{proposition}

\begin{proof}
Lemma \ref{Aa-weyl} tells us that $A_\a$ is isomorphic to the Weyl algebra for $\C^d$
with generators $z_1,w_1,\ldots,z_d,w_d$, where the $z_i$ have positive $\bT$-weight.  
Then, as a $\bT$-vector space, $\Theta_\a$ is isomorphic to $\C[w_1,\ldots,w_d]\otimes B_\a$.
The result follows.
\end{proof}

Note this shows that $\Theta_{\a}$ lies in the subcategory $\cOa(\leq \a)$ of modules whose composition factors are either $L_{\a}$ or have maximal $\xi$-eigenvalue with  real part strictly less than that of $L_\a$. Since ${\nabla}_{\!\a}$ is injective in this subcategory, we have an induced map of $A$-modules $\Theta_\a\to {\nabla}_{\!\a}$.  
The module $\Theta_{\a}$ also has a universal property; by the adjunction of $\Hom$ and $\otimes$, we have for all $A$-modules $M$ that:
\[\Hom_{A}(M,\Theta_{\a})=\Hom_{A_{\a}}(A_{\a}\otimes_{A} M, \Theta_{\a}).  \]
%This property is more useful when we dualize. Let $\Theta'_{\a}=(B_\a\otimes _{A^{+}_{\a}} A_{\a})^{\star}$ be the dual version of this module.  In this case, we have \[\Hom_{A}(\Theta_{\a}',M)\cong \Hom_{A_{\a}}(\Theta'_{\a},(M^\star\otimes_{A}A_\a)^\star)\]
%Of course, the costandard $\nabla_{\!\a}$ has a similar property:
%\[\Hom_{A}(M,\nabla_{\!\a})=\Hom_{A_{+}}(M,B_{\alpha}).\]

%This result also shows that:
%\begin{lemma} 
%  The $A$-module $\Theta_{\a}$ is finitely generated for all $\a$.  
%\end{lemma}
%\begin{proof}
%  Since $B(A)$ is finite-dimensional, there is a polynomial $p\in
%  \C[t]$ such that $p(\xi)\in A_0$ has trivial image in $B(A)$; that
%  is, we have $p(\xi)\in \sum_{k\in \Z_{>0}}A_{-k}A_k$.  Thus, if $p(\ell)\neq
%  0$, then $\Theta_{\a}^{\ell}=\sum_{k\in \Z_{>0}} A_{-k}\Theta_{\a}^{\ell +k}$.  Let
%  $m=\min\{\operatorname{Re}(z) | p(z)=0\}$.   This
%  shows that $N=\sum_{\operatorname{Re}(\ell)\geq m}\Theta_{\a}^{\ell}$ is a
%  generating subpace of $\Theta_{\a}$, and is finite dimensional by
%  Proposition \ref{Theta-character}.
%\end{proof}
%
The modules $\Theta_\alpha$ and $\Delta_\alpha$ 
have analogues which are families over the twistor
deformation $\mathscr{M}_\eta$.  Let  $\tilde{\Delta}_{\eta,\a}:= \mathscr{A}_\eta
\otimes_{\mathscr{A}_\eta^+}\mathscr{B}_{\eta,\a}$ be the deformed
standard module attached to $\eta$ and $\a$, and let $\tilde\Theta_\a$ be the restriction of
$\mathscr{A}_{\eta,\a}\otimes_{\mathscr{A}_{\eta,\a}^+}\mathscr{B}_{\eta,\a}$ to
$\mathscr{A}$.  While $\tilde{\Delta}_{\eta,\a}$ is a more natural algebraic
object, the family $\tilde\Theta_{\eta,\a}$ has a more regular structure, and
in particular is flat over $\aone$ by Proposition \ref{Theta-character}.  On the other hand, it is not clear (in fact, seems to not be true) that $\tilde\Theta_{\eta,\a}$ is finitely generated.  We can fix this defect by considering its restriction $\hat\Theta_{\eta,\a}$ to the formal neighborhood $\widehat{\aone}$ of the origin in $\aone$, which is finitely generated by Nakayama's lemma applied to the individual $\xi$-weight spaces.

As usual, we use the subscript $k$ below to denote the period $\la + k\eta$.
The following lemma says that the two modules are isomorphic if the period is sufficiently large.

\begin{lemma}\label{Delta-Theta}
  The natural homomorphism of $A_k$-modules $\Theta_{k,\a}\to \nabla_{\!k,\a}$ is an isomorphism for $k$ sufficiently large. 
\end{lemma}

\begin{proof}
First, note that  $A_{\a}\otimes_{A}\nabla_{\!k,\a}$ only depends on the restriction of $\Loc(\nabla_{\!k,\a})$ to a formal neighborhood of the fixed point $\a$. In particular, if localization holds, then the functor $A_{\a}\otimes_{A}-$ is isomorphic to the completion of localization at the fixed point $\a$ and thus is exact for all $\a$.  

	Recall that the deformed bimodule ${_k\cT_m}$ has an inner action of $\xi$, which induces a $\bT$-equivariant structure on the line bundle $\cL$. For each fixed point $\a$, we have a weight $n_\a$; as we change $k$, the $\xi$-weight $\xi_{k,\a}$ of $B_{k,\a}$ varies as $\xi_{k,\a}=kn_{\a}+\xi_{0,\a}$.  In particular, if $n_{\a}>n_{\b}$, then $ \xi_{k,\a}>\xi_{k,\b}$ for $k\gg 0$, and similarly with inequalities reversed. We can choose $k$ sufficiently large, so this holds for each pair of fixed points, and localization holds as well.  
	
	Fix  $M\in \cOa$ and choose  a fixed point $\b$ in the support of $\Loc(M)$ has $n_{\b}$ maximal.  The paragraph above implies that $\supp\Loc(M)$ contains no $X_{\gamma}$ with $\b\in X_{\gamma}$ and $\b\neq \gamma$.  Since localization holds, the sheaf $\Loc(M)$ has a section which does not vanish in the classical limit at the fixed point $\b$.  This section must have non-zero component in the weight $\xi_{k,\b}$ eigenspace of the operator $\xi$.  
	Thus, if $M$ is an $A_k$-module in category $\cO$ such that all $\xi$-weights in $M$ are  $ <\xi_{k,\b}$, then $\b$ does not lie in the support of $\Loc(M)$.  
	
	Consequently, if $M$ is a module in $\cOa(< \a)$, then $\a$ does not lie in the support of $\Loc(M)$, and by adjunction, we have that $\Hom(M,\Theta_{k,\a})=0$.  If $M\in \cOa(\leq \a)$, then as argued above, the module $A_{\a}\otimes_{A}M$ is supported on the completion of $T_{\a}X_{\a}\subset T_{\a}\fM$; in particular, its $\xi$-weights are bounded above.  Any such module over the Weyl algebra is the sum of some number of copies of $\Theta_{\a}$.  Thus, on the subcategory $\cOa(\leq \a)$, the functor $\Hom_{A}(M,\Theta_{k,\a})$ is exact, implying that $\Theta_{k,\a}$ is an injective object in this category.  Since we have already shown that $\Hom(L_\b,\Theta_{k,\a})=0$ for $\b< \a$, it follows that $\Theta_{k,\a}$ is the sum of some number of copies of ${\nabla}_{\!\a}$, and since $\Theta$ has a 1-dimensional space of vectors that transform as $B_{\a}$, we have the desired isomorphism.  
	\end{proof}

In order to describe the sheaves in $\cOg$ which give standard objects we recall a construction from \cite{BLPWquant}. 
Proposition 5.2 from that paper shows that for every pair of integers $k$ and $m$ there is a good, $\bS$-equivariant $(\cQ_k, \cQ_m)$-bimodule ${_k\cT_m}$ (unique up to canonical equivalence) with ${_k\cT_m}/h\,{_k\cT_m}\cong\cL^{k-m}$.
Let $${_kZ_m} := \secs(_k\cT_m[\hon])\and
Z:= \bigoplus_{k\geq m\geq 0}{_kZ_m}.$$
This is a $\Z$-algebra in the sense of Gordon and Stafford \cite[\S 5]{GS}, with multiplication given by tensor products of sections.  
We have a localization functor $\Loc^\Z:\Zmod\to\Dmod$ given 
by $\Loc^\Z(N) := \left(\bigoplus_{k\geq 0}{}_0\cT_k[\hmon]\right)\otimes_Z N$; it
becomes an isomorphism after modding out by $Z$-modules which are bounded above.

Let ${}_k(Z_\a)_m$ be the space of $\bS$-invariant and $\bT$-finite vectors in the
completion of ${}_k\cT_m[\hmon]$ at the point $p_\a$, and let 
%$$Z_\a := \bigoplus_{k\geq m\geq 0}{}_k(Z_\a)_m.$$  Then $Z_\a$ is a $\Z$-algebra;
%it is a local version of the $\Z$-algebra $Z$ discussed in Section \ref{gco-sec:Z-algebra}.
%Let 
$${}_k(\Theta_\a^\Z)_{m}:= {}_k(Z_\a)_m\otimes_{A_{m,\a}^+}B_{m,\a}
\and
\Theta_\a^\Z := \bigoplus_{k\geq 0}{}_k(\Theta_\a^\Z)_{0};$$
then $\Theta_\a^\Z$ is a module over $Z$.  Consider the localization $\gcst_\a=\Loc^\Z(\Theta_\a^\Z)$

\begin{proposition}\label{geometric standard}
The sheaf $\gcst_\a$ is the costandard object of $\cOg$ corresponding to $\a$.  In particular, if localization
holds at $\la$, then $\gcst_\a= \Loc(\nabla_{\!\a})$.
\end{proposition}

Note that this result, together with Corollary \ref{alg-geom} shows that
$\gcst_\a$ has a unique simple submodule, which we denote by 
$\gsi_\a$.\medskip

\begin{proof}
By Lemma \ref{Delta-Theta}, we have $\gcst_\a = \Loc(\nabla_{\!\a})$ for the quantization with period $\la + k\eta$
when $k$ is sufficiently large.  If localization holds at $\la$, then consider the following commuting square of
equivalences.
\[\tikz[->,very thick]{
\matrix[row sep=10mm,column sep=10mm,ampersand replacement=\&]{
\node (a) {$\cOg^\la$}; \& \node (c) {$\cOg^{\la+k\eta}$};\\
\node (b) {$\cOa^\la$}; \& \node (d) {$\cOa^{\la+k\eta}$};\\
};
\draw (b) -- (a);
\draw (a) -- (c);
\draw (b) --(d);
\draw (d)--(c);
}\]
The vertical arrows are given by localization, the top horizontal arrow is given by Lemma \ref{geometric twist},
and the bottom horizontal arrow is given by tensor product with ${}_{\la+k\eta}T_\la$. 
We know that the vertical arrow on the right takes $\nabla_{\!\a}$ to $\gcst_\a$, and it is easy to check
that the two horizontal arrows take $\gcst_\a$ to $\gcst_\a$ and $\nabla_{\!\a}$ to $\nabla_{\!\a}$.  The proposition follows.
\end{proof}

We can also construct deformed versions $\tilde Z_{\eta,\a}$ and $\tilde\Theta_{\eta,\a}^\Z$; since we will want to have finitely generated modules, we consider the restrictions $\hat Z_{\eta,\a}$ and $\hat\Theta_{\eta,\a}^\Z$ of these to $\widehat{\aone}$. 
%, which we will need for the proof of Theorem \ref{support isomorphism} below.  
This allows us
to construct a deformed costandard object
$\hat\gcst_{\eta,\a} := \Loc^\Z(\hat\Theta_{\eta,\a}^\Z)$ on $\scrM_\eta\times_{\aone}\widehat{\aone}$, which is flat over $\widehat{\aone}$ (by Proposition \ref{Theta-character}) so that $\hat\gcst_{\eta,\a}|_\fM \cong \gcst_\a$.   We also have standard objects $\gst_{\a}$ and their deformations $\hat \gst_{\a}$, dual to $\gcst_{\a}$ and $\hat\gcst_{\a}$.
This construction will be used in the proof of Theorem \ref{support
  isomorphism} below.  

\subsection{The center of the Yoneda algebra of \texorpdfstring{$\cOg$}{Og}}
\label{sec:hochsch-cohom-cent}

Consider the Hochschild cohomology ring
\[H\! H^*(\cD) := \Ext^\bullet_{\cD\otimes\cD^{\op}}(\cD,\cD).\]
Here the Ext algebra is computed in the bounded below derived category of
sheaves of modules over the sheaf $ \cD\otimes\cD^{\op}$; that is, by taking an injective resolution of the
left-hand term.  (The existence of such an injective resolution follows by
mimicking the argument of \cite[2.2]{Hartshorne} with $\cD(0)_x$ in place of $\fS_{\fM,x}$.)
By the usual formalism, there is a spectral sequence 
\[
H^{i}(\fM;\mathscr{E}\!\mathit{xt}^j_{\cD\otimes
  \cD^{\op}}(\cD,\cD))\Rightarrow H\! H^{i+j}(\cD).
\]
By \cite[3.1]{WX}, the Hochschild cohomology of the Weyl algebra
vanishes in all higher degrees, so the spectral sequence collapses at the $E_2$ page 
and we have $$H\! H^*(\cD)\cong H^*(\fM;Z(\cD)) \cong H^*(\fM;\C\pphpp).$$
Thus, for any object in $D^+(\cD\mmod)$, we obtain a map from
$H^*(\fM;\C)=H^*(\fM;\C\pphpp)^\bS$ to the center of its Ext-algebra.  In particular, $H^*(\fM;\C)$ maps to the Yoneda algebra $E$ of $\cOg$,
as defined in Section \ref{hwksk}.

This map need not be an isomorphism; for example, if the period of $\cD$ is as generic as possible,
then $\cOg$ will be semisimple, and its Yoneda algebra will be concentrated in degree zero.
However, we make the following conjecture, which essentially says that this is the only thing that can go wrong.

\begin{conjecture}\label{HH}
  If the category $\cOg$ is indecomposable (that is, if it has no proper block
  decomposition), then the map $H^*(\fM;\C)\to Z(E)$ is an isomorphism.
\end{conjecture}

\begin{remark}
Conjecture \ref{HH} holds for cotangent bundles of partial flag varieties (Proposition \ref{prop:hochG/P}),
quiver varieties in finite type ADE and affine type A \cite[3.5]{S3}, and
for hypertoric varieties (part (vi) of Section \ref{sec:hypertoric-varieties}).
\end{remark}

\begin{remark} In Section \ref{gm}
we formulate a stronger version of Conjecture \ref{HH} 
which relates the equivariant cohomology
of $\fM$ to the center of the universal deformation of $E$.
\end{remark}

\section{The Grothendieck group of \texorpdfstring{$\cOg$}{Og}}\label{sec:Grothendieck}
We continue to let $d = \frac 1 2 \dim\fM$. In this section, we show that the Grothendieck group $K(\cOg)$ is canonically
isomorphic by the characteristic cycle map to the cohomology $\bm$ with support in $\fM^+$ as lattices with inner products (Theorem \ref{support isomorphism}).
This is accomplished by studying the characteristic cycles of standard objects, 
but we also give some partial results about the images of simple objects under this isomorphism (Section \ref{sec:simples}).

\subsection{Characteristic cycles revisited}\label{gco-grothendieck}
In Section \ref{quant-HC} we alluded to a characteristic cycle map $\suppc:K(\CLg)\to\HLZ$
that was studied in \cite[\S 6.2]{BLPWquant}, following ideas of Kashiwara and Schapira \cite{KSdq}.  In this section we review this construction and study
it in greater detail.

Let $\cN$ be an object of $D^b(\Dmod)$.  We have isomorphisms
 $$ \sHom^\bullet_{\cD}(\cN,\cN)\cong
  \sHom^\bullet_{\cD}(\cN,\cD)\Lotimes_{\cD}\cN \cong
  \cD_{\Delta}\Lotimes_{\cD\boxtimes\cD^{\op}}\big(\cN\boxtimes
  \sHom^\bullet_{\cD}(\cN,\cD)\big),$$
  and evaluation defines a canonical map to the Hochschild homology 
  $$\mathcal{HH}(\cD) := \cD_{\Delta}\Lotimes_{\cD\boxtimes\cD^{\op}}\cD_{\Delta}.$$
All this is completely general, and holds in both the Zariski and the
classical topology.  In the classical topology, we also have an isomorphism  $\mathcal{HH}(\cD^{\an}) \cong
\C_{\fM_\Delta}\pphpp[\dimfM]$ by \cite[6.3.1]{KSdq}.
 
We define the {\bf characteristic cycle} 
$$\suppc(\cN)\in H^0(\mathcal{HH}(\cD^{\an}))\cong H^{\dimfM}\big(\fM; \C\pphpp \big)$$
to be the image of $\id\in H^0(\sHom^\bullet_{\cD}(\cN^{\an},\cN^{\an}))$ along this map.
More generally, if $\cN$ is supported on a subvariety $j\colon\fL\hookrightarrow\fM$,
then we may consider the identity map of $\cN^{\an}$ to be a section of 
$j^!\sHom^\bullet_{\cD}(\cN^{\an},\cN^{\an})$.
Applying our map then gives us a class
$$\suppc(\cN)\in H^0(j^!\mathcal{HH}(\cD^{\an}))\cong \HLCh.$$
(Our abuse of the notation $\suppc(\cN)$ is justified by the fact that this class is functorial
for inclusions of subvarieties.)  By Poincar\'e-Verdier duality, this can also be considered as a Borel-Moore homology class on $\fL$.

If $\fL$ is Lagrangian, then
Kashiwara and Schapira \cite[7.3.5]{KSdq} show that $\suppc(\cN)$ actually lies in $\HLZ$;
more precisely, if $\cN(0)$ is a good lattice, then
\[\suppc(\cN)=\sum_{i=1}^r\rk_{\fL_i}\!\!\big(\cN(0)/\cN(-1)\big)\cdot
[\fL_i]\in \HLZ\subset  \HLCh,\]
where $\{\fL_1,\ldots,\fL_r\}$ are the components of $\fL$ and $\rk_{\fL_i}$ is the rank
at the generic point of $\fL_i$.

We can also take characteristic cycles in families for modules on twistor
deformations $\mathscr{M}_\eta\to \aone$ after base change to $\widehat{\aone}$.  Let $\scrN$ be such a module, and consider the
image of the identity via the natural morphisms
\begin{multline}\label{relative-HH}
  \sHom^\bullet_{\scrD}(\scrN,\scrN)\cong
  \sHom^\bullet_{\scrD}(\scrN,\scrD)\Lotimes_{\scrD}\scrN \cong
  \scrD_{\Delta}\Lotimes_{\scrD\boxtimes_{\aone}\scrD^{\op}}\big(\scrN\boxtimes_{\aone}
  \sHom^\bullet_{\scrD}(\scrN,\scrD)\big)\\ \to
  \scrD_{\Delta}^{\an}\otimes_{\scrD ^{\an}\boxtimes_{\aone}\scrD^{\an,\op}}\scrD_{\Delta}^{\an}\cong
  \pi^{-1}\fS_{\aone}[\dimfM]\pphpp.
\end{multline}
This defines a class in relative 
cohomology $\suppc(\scrN)\in H^{\dimfM}_{\mathscr{L}}(\mathscr{M}_\eta/\aone;\C\pphpp)$ for any 
Lagrangian $\mathscr{L}\supset
\supp(\scrN)$.  
If we let $\fL := \fM\cap \mathscr{L}$, then we have a
natural restriction map $$H^{\dimfM}_{\mathscr{L}}(\mathscr{M}_\eta/\aone;\C\pphpp)\to
\HLCh$$ given by dividing by the coordinate $t$ on $\aone$. We
also have a natural functor of restriction from
$\scrD\mmod\to \cD\mmod$ given by $\scrN|_{\fM} :=
\scrN\overset{L}\otimes_{\C[t]}\C$. 
The following lemma says that these operations are compatible.

\begin{lemma}\label{restriction-commutes}
  If $\scrN$ is a good $\scrD$-module,
  then $\suppc(\scrN|_{\fM})=\suppc(\scrN)|_{\fM}$. 
\end{lemma}
\begin{proof}
  Consider the complex \eqref{relative-HH} of $
  \pi^{-1}\fS_{\aone}$ modules, and
take the derived tensor product with $\C$ over $\C[t]$.
We claim that we obtain a corresponding complex for  $\scrN|_{\fM}$.
That is, we obtain \begin{multline}\label{absolute-HH}
  \sHom^\bullet_{\scrD}(\scrN|_{\fM},\scrN|_{\fM})\cong
  \sHom^\bullet_{\scrD}(\scrN|_{\fM},\cD)\Lotimes_{\cD}\scrN|_{\fM} \cong
  \cD_{\Delta}\Lotimes_{\cD\boxtimes\cD^{\op}}\big(\scrN|_{\fM}\boxtimes
  \sHom^\bullet_{\cD}(\scrN|_{\fM},\cD)\big)\\ \to
  \cD_{\Delta}^{\an}\otimes_{\cD ^{\an}\boxtimes\cD^{\an,\op}}\cD_{\Delta}^{\an}\cong
  \C_\fM[\dimfM]\pphpp.
\end{multline} It
suffices to prove this for $\scrN$ locally free.  In this case,
$\sHom^\bullet (\scrN,\scrD)$ is concentrated in degree 0 and is itself
locally free, so the statement is clear.

Thus $\suppc(\scrN)|_{\fM}$ can be obtained as the image of the identity
under the map \eqref{absolute-HH}.
By definition $\suppc(\scrN|_{\fM})$ is the image of the identity under \eqref{absolute-HH}, so we are done.
\end{proof}

\subsection{Intersection forms for category $\cO$}\label{gco-intform}
We now turn our attention to the subcategory $\cOg = \CMpg \subset \Dmod$, so that
the characteristic cycle map goes from $K(\cOg)$ to $\bm$.  We first need to reinterpret the group
$\bm = H^{\dimfM}(\fM, \fM \setminus \fM^+; \Z)$ equivariantly, so that we can apply localization.  

\begin{lemma} The forgetful homomorphism
\[H^{\dimfM}_\T(\fM, \fM\setminus \fM^+; \Z) \to H^{\dimfM}(\fM, \fM\setminus\fM^+; \Z)\]
is an isomorphism.  The localization map
\[H^{\dimfM}_\T(\fM, \fM\setminus \fM^+; \Z) \to H^{\dimfM}_\bT(\fM^\bT; \Z)\] is an injection; a class lies in the image if and only if its restriction to $p_\a$ is a $\Z$-multiple of the equivariant Euler class $e_{\bT}(T_{p_\a}\fM)$, which is the 
product of the negative weights of the action of $\bT$ on $T_{p_\a}\fM$.
\end{lemma}
\begin{proof}
Choose an ordering $\a_1,\dots,\a_r$ of the index set $\cI$ refining the closure order $\leftharpoonup$, so that $\fM^+_k := \bigcup_{i=1}^k X^\circ_i$ is closed for all $k$.  Let $U_k = \fM \setminus \fM^+_k$ and $U_0 = \fM$.  Then for $1 \le k \le r$, the cohomology
$H^*(U_{k-1},U_k;\Z)$ is isomorphic to the Borel-Moore homology $H^{BM}_{4d-*}(X_k^\circ)$, so it is isomorphic to $\Z$ in degree $2d$ and $0$ in all other degrees.  It follows that $H^*_\T(U_{k-1},U_k; \Z)$ is a free $H^*_{\T}(pt)$-module generated by 
$H^{\dimfM}_\T(U_{k-1},U_k; \Z) \cong H^{\dimfM}(U_{k-1},U_k; \Z)$.  In addition, the restriction of a generator of $H^{\dimfM}_\T(U_{k-1},U_k; \Z)$ to $p_k = p_{\a_k}$ is the equivariant Euler class of $T_{p_k}X^\circ_{p_k}$.

The result now follows by an easy induction using the exact sequence 
\[H^*_\T(U_0,U_{k-1};\Z)\to H^*_\T(U_0,U_{k};\Z) \to H^*_\T(U_{k-1},U_{k};\Z)\]
which is short exact since the left and right terms vanish in odd degrees.
\end{proof}

From the first part of the lemma, we have a canonical map
$$\bm = H^{\dimfM}(\fM^+; j^!\Z_\fM) \to H^{\dimfM}_\bT(\fM^\bT; \Z).$$
For all $\gamma\in\bm$ and $\a\in\cI$, we will write $\gamma|_\a$ to denote the image of $\gamma$ in $H^{\dimfM}_\bT(p_\a; \Z)$.
The second part of the lemma implies that 
The lattice $\bm$ is freely generated by the classes $\{v_\a\mid\a\in\cI\}$,
where $v_\a|_{\a}$ is the product of the negative weights of the action of $\bT$ on $T_{p_\a}\fM$ and 
$v_\a|_{\b} = 0$ for $\b\neq\a$.
 
The classes $v_\a$ form an orthonormal basis
for the {\bf equivariant intersection form}
$$\langle \beta, \gamma\rangle := (-1)^{d}\sum_{\a\in\cI}\frac{\beta|_\a\cdot\gamma|_\a}{e(\a)},$$
where $e(\a)\in H_\bT^{4d}(p_\a; \Z)$ is the product of all of the weights of the action of $\bT$ on $T_{p_\a}\fM$.

On $K(\cOg)$ we have the {\bf Euler form} given by the formula
$$\Big\langle [\cM], [\cN]\Big\rangle := \sum_{i=0}^\infty
(-1)^i\dim\Ext_{\Dmod}(\cM,\cN).$$

\begin{proposition}
  The classes $\{[\gst_\a]\mid \a\in\cI\}$ form an orthonormal basis
  for $K(\cOg)$.  In particular, the Euler form on $K(\cOg)$ is symmetric.
\end{proposition}
\begin{proof}
 This follows from the universal coefficient theorem applied to
 $\Ext^*(\hat{\gcst}_{\a,\eta},\hat{\gcst}_{\b,\eta})$. Generically on $\widehat{\aone}$, the supports of $
\hat{\gcst}_{\a,\eta}$ and $\hat{\gcst}_{\b,\eta}$ are distinct if
$\a\neq \b$, so
$\Ext^*(\hat{\gcst}_{\a,\eta},\hat{\gcst}_{\b,\eta})$ is supported
on $\{0\}\subset \widehat{\aone}$.  The universal coefficient theorem shows
that \[ \Ext^i({\gst}_{\a},{\gst}_{\b})\cong
(\Ext^{i+1}(\hat{\gcst}_{\a,\eta},\hat{\gcst}_{\b,\eta})\oplus
\Ext^i(\hat{\gcst}_{\a,\eta},\hat{\gcst}_{\b,\eta}))\otimes_{\C[[t]]}\C.\]  Thus,
obviously, the Euler characteristic of this complex is 0, and the
classes  $[\gcst_\a]$ and $[\gcst_\b]$ are orthogonal.

On the other hand, we know from the costandard property that 
\[\Ext^i({\gcst}_{\a},{\gcst}_{\a})=
\begin{cases}
  \C & i=0\\
0 & i\neq 0\\
\end{cases}\]
so this establishes orthonormality.
\end{proof}

\begin{corollary} \label{costandard equals standard in K} If the period of a quantization is chosen so that localization holds and the hypotheses of Theorem \ref{highest weight} are satisfied, then for each $\a$ we have $[\nabla_\a] = [\Delta_\a]$ in $K(\cO_a)$; in particular, the multiplicities of any simple in $\nabla_\a$ and $\Delta_\a$ are the same.
\end{corollary}
\begin{proof}
Since $\nabla_{\!\a}$ is sent to $\gcst_\a$ under localization, the proposition implies that 
the costandards $\nabla_{\!\a}$ give an orthonormal basis of $K(\cO_\a)$ under the Euler form.  But in any highest weight category the classes of standards are (left) orthogonal to the classes of costandards, so we must have $[\Delta_\a] = [\nabla_\a]$.
\end{proof}

\begin{theorem}\label{support isomorphism}
The map $\suppc:K(\cOg)\to \bm$ is an isomorphism 
that intertwines the Euler form with the equivariant intersection form.
\end{theorem}

\begin{proof}
Since the costandard modules $\{\gcst_\a\mid \a\in\cI\}$ form an orthonormal basis for $K(\cOg)$, it suffices
to show that $\suppc(\gcst_\a)=v_\a$ for all $\a\in\cI$.
Consider the sheaf $\tilde{\gcst}_{\eta,\a} = \Loc^\Z(\tilde{\Theta}_\a^\Z)$ on $\mathscr{M}_\eta$, which we introduced
at the end of the previous section, along with its Euler class 
$\suppc(\tilde{\gcst}_{\eta,\a})\in H^{\dimfM}_{\mathscr{M}_\eta^+}(\mathscr{M}_\eta/\aone; \C\pphpp)$.
% By Lemma \ref{restriction-commutes}, we have 
% $$\suppc(\tilde{\gst}_{\eta,\a})|_\fM = \suppc(\tilde{\gst}_{\eta,\a}|_\fM) = \suppc \gst_\a.$$
Since $\mathscr{M}_\eta^{\bT}$ is isomorphic
to a disjoint union of $|\cI|$ copies of $\aone$ and the space
$\mathscr{M}_\eta^+$ is an $\mathbb{A}^d$-bundle over this space, the group 
$H^{\dimfM}_{\mathscr{M}_\eta^+}(\mathscr{M}_\eta/\aone; \C\pphpp) $
is a $\C\pphpp$-vector space of dimension $|\cI|$.  Let 
$$\{\scrv_\a\mid\a\in\cI\}\subset H^{\dimfM}_{\mathscr{M}_\eta^+}(\mathscr{M}_\eta/\aone; \C\pphpp)$$ be the $\C\pphpp$-basis
indexed by the components of the fixed point set, so that $\scrv_\a|_\fM = v_\a$.

Over a generic element of $\aone$, the restriction of $\tilde{\gcst}_{\eta,\a}$ is simply the structure sheaf
of the locus of points whose $\bT$-limit is equal to the fixed point labeled by $\a$.  This implies
that $\suppc(\tilde{\gcst}_{\eta,\a}) = \scrv_\a$, and therefore that
$$\suppc(\gcst_\a) = \suppc(\tilde{\gcst}_{\eta,\a}|_\fM) =
\suppc(\tilde{\gcst}_{\eta,\a})|_\fM = \scrv_\a|_\fM = v_\a.$$
This completes the proof.
\end{proof}

\begin{remark}
  While the hypothesis that $\bT$ has isolated
  fixed points was used in an essential way 
  here, it should be possible to generalize this result to more general
  $\bT$-actions, at the cost of downgrading from an isomorphism to an
  injection.  A forthcoming result of Baranovsky and Ginzburg \cite{BaGi} shows
  that the map $\suppc$ is injective in the case where $\bT$ is
  trivial.  In this case, $\suppc$ takes values in the top degree homology group of the core
  (the preimage of $o\in\fM_0$), and it intertwines the Euler form with the ordinary intersection form on $\fM$
  by \cite[6.5.4]{KSdq}.  This map can be extremely far from surjective,
  though; for generic periods, category $\cOg$ for a trivial action
  has no non-zero objects.  A recent preprint of Bezrukavnikov and Losev
  shows how complicated this dependence can be in the case of certain
  quiver varieties \cite{BLet}.  We expect that this result should extend to
  arbitrary $\bT$ as a mix of these two situations.  See Remark \ref{nonisolated B} for a related discussion.
\end{remark}

We conclude this section by noting that we can specify a geometrically-defined partial order with
respect to which the category $\cOg$ is highest weight.  (We already know that it is highest weight
by Proposition \ref{ghw}, but the relation between the partial order $\le$ we used there and the geometry of $\fM$ is not clear.)
Define a partial order $\leftharpoonup $ on $\cI$
by putting $\a\leftharpoonup \b$ if $p_{\a}\in X_{\b}$ (or equivalently $\overline{X^\circ_\a} \cap X_\b^\circ \ne\emptyset$) and then taking the transitive closure.

\begin{proposition}\label{geometric order}
  The support of $\gst_\a$ is contained in $\bigsqcup_{\b\leftharpoonup\a}
  X_{\b}$.  In particular, $\cOg$ is highest-weight with respect to
  the partial order $\leftharpoonup $.
\end{proposition}
\begin{proof}
This follows immediately from the structure on the fixed point
classes, since the change of basis matrix between the bases $\{v_\a\mid \a\in\cI\}$ and
$\{[X_\a]\mid\a\in\cI\}$ is triangular with ones on the diagonal with respect to this
partial order.  Since the simple $\gsi_\a$ defined after Proposition
\ref{geometric standard} has non-trivial support on
$X_\a$, if it occurs in $\gst_\b$, the standard $\gst_\b$ must have
$X_\a$ in its support, so $\a\leftharpoonup \b$.
\end{proof}
It is worth noting that $\leftharpoonup$ is not necessarily the weakest partial order with
respect to which $\cOg$ is highest weight.  For example, $\cOg$ may be
semi-simple, and thus highest weight for the trivial partial order.

\subsection{Supports of simples}\label{sec:simples}
The key to the previous section was the computation of the
characteristic cycles of standard and costandard objects.  It is also interesting to
consider the characteristic cycles of simple objects, though they are
much more difficult to understand.  In this section we will obtain
some partial results about their set-theoretic supports that will be useful in later sections.

We call an $A$-module $N$ {\bf  holonomic} if its derived localization
$\LLoc(N)$ on any resolution has Lagrangian support.
Note that this is independent of the choice of resolution, since the
functors $\LLoc$ for different resolutions are related by convolution
with a Harish-Chandra bimodule, which preserves holonomicity.

For any simple $A$-module $L$, let $\fM_{L} \subset\fM_0$ be
the subscheme defined by the ideal $\gr\Ann(L)\subset \gr A \cong
\C[\fM_0]$.  This subscheme is always the closure of a symplectic leaf
\cite{Giprim}; in particular, it is a subvariety.  A leaf that arises
in this way will be called {\bf special}, in analogy with the existing
terminology for nilpotent orbits.  
% (Note that the subvariety
% $\fM_{L}$, and therefore the set of special leaves, depends
% nontrivially on the choice of quantization.)  
We let $\all$ denote the
set of all symplectic leaves of $\fM_0$, and $\spe$ denote the subset  of
leaves which are special for a fixed quantization $\cD$.

\begin{theorem}\label{unique-minimal}
If $L$ is holonomic, the support of the sheaf $\gr L$ on $\fM_0$ is contained in $\fM_L$, and it intersects
the dense leaf of $\fM_L$ nontrivially; equivalently, a symplectic leaf
closure contains $\supp(\gr L)$ if and only it contains $\fM_L$.
\end{theorem}

Before proving Theorem \ref{unique-minimal} we establish a pair of lemmas.
Consider the Rees algebra $R(A) \cong \Gamma(\cD(0))$.  Following Losev, we wish to consider
the completion of this algebra at a maximal ideal in $\C[\fM_0]\cong
R(A)/h\cdot R(A)$.  For $s\in \fM_0$, we let $R(A)^\wedge_s$ be the
completion of $R(A)$ in the topology induced by the maximal ideal
$\mathfrak{m}_s+h\cdot R(A)$.  Let $S$ be the symplectic leaf
containing $s$ and, following Kaledin \cite[2.3]{KalPPV}, let
$\mathcal{Y}_s$ denote the formal slice to $S$ inside of $\fM_0$.

\begin{lemma}\label{algebra-tensor}
  The completion $R(A)^\wedge_s$ is isomorphic to the tensor product
  $W^\wedge_0 \,\hat{\otimes}_{\C[[h]]}\, C$ where $W$ is the Weyl algebra on the
  symplectic vector space $T_s^*S$ and
  $C$ is a quantization of $\mathcal{Y}_s$.
\end{lemma}

\begin{proof}
The algebra $R(A)^\wedge_s$ is a quantization of a formal neighborhood of the
 fiber over $s$ in $\fM$; this formal scheme is isomorphic as a
 Poisson scheme to the product
 of the completion of $S$ at $s$ with a symplectic resolution
 of $\mathcal{Y}_s$.  By the classification of quantizations in \cite{BK04a}, any
 quantization of the latter will have sections of
 the form $W^\wedge_0 \,\hat{\otimes}\, C$, so we are done.
\end{proof}

Consider a holonomic $A$-module $N$ with a fixed good filtration, and
choose a point $s$ which is a smooth point of $\supp (\gr L)$ and which
is in a symplectic leaf $S$ of maximal dimension amongst those
intersecting $\supp (\gr L)$.  Now, we may form the
completion $R(N)^\wedge_s$, which is a module over $R(A)^\wedge_s$.
\begin{lemma}
  The tangent space $T_s\supp (\gr N)$ is Lagrangian in the symplectic
  space $T_sS$.
\end{lemma}

\begin{proof}
  The component $Y$ of $\supp (\gr N)$ which contains $s$ must be the
  image of a component $Y'$ of $\supp\LLoc(N)$, which is Lagrangian by
  the assumption of holonomicity.  Let $S'$ be the preimage of $S$ in $\fM$;
  this is a coisotropic subvariety of $\fM$ whose closure contains $Y'$.
  Since $Y'\cap S'$ is Lagrangian, it must be a union of the leaves of the
  null-foliation of $S'$, which is given by the fibers of the projection to $S$.
  Thus, $Y\cap S$, which is the image of $Y'\cap S'$, is Lagrangian in $S$.  The result follows.
\end{proof}

\begin{lemma}\label{module-tensor}
The quotient of $R(A)^\wedge_s$ by the annihilator $I^\wedge_s$ of the module $R(N)^\wedge_s$ is a finite extension of $W^\wedge_0$, i.e. finitely generated as a $W^\wedge_0$-module.  
\end{lemma}
\begin{proof}
Note that completing the coherent sheaf $R(N)/h\cdot R(N)$ at $s$ produces
a finite rank locally free sheaf on the completion of $\supp(\gr N)$, since $s$ is a generic point of this support.  
This implies that $R(N)^\wedge_s$ is finitely generated
over $W_0^\wedge$.  The result follows.
\end{proof}
This shows that the GK dimension over $\C((h))$ of $R(A)^\wedge_s/I^\wedge_s[h^{-1}]$ is the same as that of $W^\wedge_0[h^{-1}]$, that is, $\dim S$.

{\bf \noindent Proof of Theorem~\ref{unique-minimal}:}
The theorem can be reformulated as saying that $\fM_L$ is the union of the closures of those
symplectic leaves that intersect the support of $\gr(\Ann L)$.
Since $\gr(\Ann L)$ kills $\gr L$, we have $\bar S\subset\fM_{L}$ for any leaf $S$ intersecting the support of $L$;
this makes one of the two containments clear.  For the reverse inclusion, it suffices to show that, for $s$ and $S$ as above,
we have $\fM_{L} = \bar{S}$.

Let
$I:=\Ann(R(L))\subset R(A)$; then $\fM_{L}$ is defined by the ideal $I/hI = \gr\Ann(L)$.
We have an injective map $R(L)\to R(L)_s^\wedge$ (by the
simplicity of $L$), and thus an injective map 
$$R(A)/I\to R(A)_s^{\wedge}/I^\wedge_s.$$  
After adjoining $h^{-1}$, the latter algebra has GK dimension over $\C((h))$ given by $\dim S$.  Thus, the GK dimension of
$R(A)/I$ is at most $\dim S$.  The same is thus true of
the coordinate ring of the associated variety $V(I/hI)$.  Thus, the
variety must have dimension at most that of $S$, but it also contains
$S$.  By results of Ginzburg \cite{Giprim}, $V(I/hI)$ must be the
closure of a single leaf of this dimension, and thus, we must have
$V(I/hI)=\bar S$.  This completes the proof.\qed

We now consider some consequences of Theorem \ref{unique-minimal}.
Recall that $\Lambda_\a$ is the simple object of $\cOg$ indexed by $\a$. 
Let $\fM_{\a,0}$ be the union of the closures of the symplectic leaves that
intersect the image in $\fM_0$ of $\supp \Lambda_\a$.  By Theorem \ref{unique-minimal},
$\fM_{\a,0}$ is equal to the closure of a single leaf, which we denote $\becircled\fM_{\a,0}$.
Furthermore, if $\la\in\Ht$ is the period of $\cD$, 
this leaf depends only on the coset of $\la$ in $H^2(\fM; \C)/H^2(\fM; \Z)$.

\begin{corollary}\label{mza}
Choose any $\la'\in H^2(\fM; \C)$
such that $\la-\la'\in H^2(\fM;\Z)$ 
and localization holds at $\la'$.
Let $\cD'$ be the quantization with period $\la'$, and let $A' := \secs(\cD')$.
Then the $A'$-module $L:= \secs(\cbi\otimes \Lambda_\a)$ is simple, and
$\fM_{\a,0} = \fM_L$.
\end{corollary}

\begin{proof}
Simplicity of $L$ follows from the fact that $\cbi\otimes \Lambda_\a$ is a simple $\cD'$-module
and localization holds at $\la'$.  The support of $\Lambda_\a$ is equal to that of $\cbi\otimes \Lambda_\a$,
so the image in $\fM_0$ of the support of $\Lambda_\a$ is equal to the support of the associated graded
of $L$.  Applying Theorem \ref{unique-minimal}, we are done.
\end{proof}

\vspace{-\baselineskip}
\begin{corollary}
If localization holds for $\cD$, then for all $\a\in\cI$, $\becircled\fM_{\a,0}$ is a special leaf.
\end{corollary}

The support of $\gsi_\a$ always contains the relative core component $X_\a$, but it may contain other
components of $\fM^+$, as well.  For example, if $\fM$ is a hypertoric
variety and the period of $\cD$ is integral in the sense of Section \ref{sec:integrality}, then the support of
$\gsi_\a$ is equal to $X_\a$ \cite[\S 6.3]{BLPWtorico}.  On the other
hand, if $\fM = T^*(G/B)$, then there exists $\a$ for which the
support of $\gsi_\a$ has multiple components unless $G =
\SL_r$ for $r\leq 7$.  Thus $\fM_{\a,0}$ always contains $X_{\a,0}$,
but it is possible that $X_{\a,0}$ is contained in a smaller leaf closure.

\begin{definition}\label{simple supports}
We call the pair $(\fM, \cD)$ {\bf interleaved} if localization holds for $\cD$ and,
for all $\a\in\cI$, $\fM_{\a,0}$ is the smallest special leaf
closure that contains $X_{\a,0}$.  If we have a notion of integrality (Section \ref{sec:integrality})
and the pair $(\fM, \cD)$ is interleaved for some (equivalently any) integral quantization $\cD$ for which localization
holds, then we will simply say that $\fM$ is interleaved.
\end{definition}

\begin{example}
As mentioned above, if $\fM$ is a hypertoric variety and the period of $\cD$
is integral, then the support of $\gsi_\a$ is equal to $X_\a$, thus $\fM$ is interleaved
by Theorem \ref{unique-minimal}.
\end{example}

\begin{example}
Finite and affine type A quiver varieties (which include finite type A Slodowy
slices) are interleaved; this follows from
Theorem \ref{special filtrations} and \cite[\S 5]{Webqui}.
\end{example}

\begin{example}
If $G$ is the adjoint group of type $F_4$, then $T^*(G/B)$ is {\bf not}
interleaved.  We will deduce this from Theorem
\ref{special filtrations}; see Remark \ref{not interleaved} for details.
\end{example}

The property of being interleaved will be used in the form of the following lemma,
which will be one of the main ingredients in the proof of Theorem \ref{special filtrations}.

\begin{lemma}\label{triangularity}
For all $\a\in\cI$,
there exist unique integers $\{\eta_{\a\b}\mid\b\in\cI\}$ such that 
$$\suppc\Lambda_\a = [X_\a] + \sum\eta_{\a\b} [X_\b],$$ 
where $\eta_{\a\b}$ can only be nonzero if 
$\b\leftharpoonup\a$.  In addition, if $(\fM, \cD)$ is interleaved, then
$\eta_{\a\b} \ne 0$ also implies 
$\fM_{\b,0} \subset \fM_{\a,0}$.
\end{lemma}

\begin{proof}
The existence and uniqueness of $\{\eta_{\a\b}\mid\b\in\cI\}$ follows from the fact that
the classes $\{[X_\b]\mid\b\in\cI\}$ form a basis for $\bm$.  Suppose that $\eta_{\a\b} \neq 0$.
Since $[X_\b]$ appears in $\suppc(\gsi_\a)$ and $\Lambda_\a$ is a quotient 
of $\gst_\a$, $[X_\b]$ must also appear in $\suppc(\gst_\a)$.
By Proposition \ref{geometric order}, this implies that 
$\b\leftharpoonup\a$.
Furthermore, $X_\b$ is contained
in the support of $\Lambda_\a$, and therefore $X_{\b,0}$ is contained in $\fM_{\a,0}$.
If $(\fM, \cD)$ is interleaved, this implies that
$\fM_{\b,0}\subset\fM_{\a,0}$.
\end{proof}

\section{Categorical filtrations}\label{sec:filtrations}
In this section we define categorical filtrations of $\HCa$, $\HCg$, $\cOa$, $\cOg$, and their derived categories.  These induce decompositions on their Grothendieck groups, and in Theorem \ref{special filtrations} and Corollary \ref{the special case} we relate the decomposition of $K(\cOg)$ to the
Beilinson-Bernstein-Deligne (BBD) decomposition of $\bmc$, using the characteristic cycle map.  We also relate this decomposition to a generalization
of Lusztig's theory of two-sided cells (Remark \ref{cell decomposition}).

\subsection{Filtration on Harish-Chandra bimodules}
Let $S\in\all$ be a symplectic leaf of $\fM_0$, and let $\bar S_\Delta := \big(\bar S\times\bar S\big)\cap \fZ_0\subset\fM_0\times\fM_0$.
Recall that $\fZ_0\cong\fM_0$ is the diagonal, so that $\bar S_\Delta\cong\bar S$.

\begin{definition}\label{HCa-filtration}
Let $\HCaS\subset\HCa$ be the full subcategory of algebraic Harish-Chandra bimodules $H$
such that for some (equivalently any) filtration of $H$, the coherent sheaf $\gr H$ on
$$\fZ_0\subset\fM_0\times\fM_0$$ is set-theoretically supported on $\bar S_\Delta$.
Let $\HCadS\subset\HCaS$ be the full subcategory supported on leaves strictly smaller than $S$.
Let $\dHCaS$ (respectively $\dHCadS$) be the full subcategory of $\dHCa$ consisting of objects with cohomology
in $\HCaS$ (respectively $\HCadS$).
\end{definition}

\begin{proposition}
  Any simple module $H$ in $\HCaS\setminus \HCadS$ has support equal to $\bar S_\Delta$.
\end{proposition}

\begin{proof}
  By Proposition
  \ref{unique-minimal}, there is a unique minimal symplectic leaf of
  $\fM_0\times \fM_0$ whose closure contains $\supp H$, given by the vanishing locus of
  the annihilator of $H$ as an $A\otimes A^{\op}$-module.  This must be
  of the form $S'\times S'$ for some $S'\in\all$, since \[\big(\bar S'\times \bar S''\big)\cap
  \fZ_0=\big((\bar S'\cap \bar S'')\times (\bar S'\cap \bar S'')\big)\cap \fZ_0.\]  Furthermore,
  we must have $\bar S'\supset S$, since part of the support of $H$ must
  intersect $S\times S$ (or we would have $H\in \HCadS$).  On the
  other hand, the dimension of the support of $H$ must be at least half of the dimension of
  $\bar S'\times \bar S'$, which implies $\dim S'=\dim S$, so $S'=S$.  
  Thus, the support of $H$ is contained in the irreducible variety $\bar S_\Delta$.
  Since they have the same dimension, we are done. 
\end{proof}

\vspace{-\baselineskip}
\begin{proposition}\label{ann-HC}
  The left annihilator of a Harish-Chandra bimodule is a primitive ideal.
\end{proposition}

\begin{proof}
  Note that the left annihilator $I$ of any simple
bimodule $H$ over any ring is prime, since if
$J_1,J_2\not\subset I$, we have that $J_1\cdot J_2\cdot H=J_1\cdot H=H$, so
$J_1\cdot J_2\not\subset I$.  

Now let $H$ be a filtered Harish-Chandra A-bimodule.
For any element $h\in H$, let $\bar h\in \gr H$ be its symbol, which we regard as a section of a sheaf on $\fZ_0\cong\fM_0$.
Choose an $h$ such that the section $\bar h$ is non-zero on a leaf $S\subset\supp H$ of maximal dimension
(here we take the support of $H$ as a left module).
Consider the filtered left submodule $A\cdot h\subset A$.
% with $(A\cdot h)(0)= A(0)\cdot h = \C h$.  
Using the Noetherian property, we can
find a simple quotient $L$ of $A\cdot h$ which is
supported on $S$.  Thus, there is a simple subquotient $L$ of $H$ whose
support has non-trivial intersection with $S$.  Let $I := \Ann(H)$ and $J := \Ann(L)$.
Then $I\subset J$, so $\bar S = V(\gr J)\subset V(\gr I)$.  However, $I$ is prime and $\dim V(\gr I) = \dim S = \dim V(\gr J)$,
thus $I = J$ \cite[Korollar 3.6]{BoKr}.
\end{proof}

\begin{remark}
By Proposition \ref{ann-HC}, we have that $\HCaS =
\HCadS$ unless $S$ is special.  If $S$ is special, then let $I$ be a
primitive ideal whose associated variety is $\bar S$.  In this case,
$A/I$ is an object in $\HCaS$, but not $\HCadS$.
\end{remark}

\begin{definition}\label{HCg-filtration}
Let $\HCgS\subset\HCg$ be the full subcategory of geometric Harish-Chandra bimodules $\cH$
that are set-theoretically supported on the preimage in $\fZ\subset\fM\times\fM$ of $\bar S_\Delta$.
Let $\HCgdS\subset\HCgS$ be the full subcategory supported on preimages of leaves strictly smaller than $S$.
Let $\dHCgS$ (respectively $\dHCgdS$) be the full subcategory of $\dHCg$ consisting of objects with cohomology
in $\HCgS$ (respectively $\HCgdS$).
\end{definition}

Let $\fJ_S$ be the localization of $\HCgS$ at $\HCgdS$.  
The Grothendieck group $K(\HCg)$ is filtered by the poset $\all$, with $K(\HCgS)/K(\HCgdS)\cong K(\fJ_S)$.
Using the Euler form on $K(\HCg)$, we may split this filtration to obtain a direct sum decomposition 
\begin{equation}\label{HC-decomp}
K(\HCg) \cong \bigoplus_{S\in\all} K(\fJ_S).
\end{equation}
See Remark \ref{DG quotient} for a more categorical interpretation of this decomposition.

If derived localization holds at $\la$, then the functors $\LLoc$ and
$\Rsecs$ induce an equivalence of categories between $\dHCgS$ and
$\dHCaS$ for all $S$.  In particular, as long as derived localization
holds at $\la$, the category $\fJ_S$ will be non-trivial if and only
if the leaf $S$ special, just as in $\HCa$; this is false if
derived localization doesn't hold at $\la$.

\begin{remark}
Losev and Ostrik prove that, if $\fM\to\fM_0$ is the Springer resolution of the nilpotent cone
of a simple Lie algebra, then $\fJ_S$ is an indecomposable 
multi-fusion category \cite[5.4 \& 5.5]{LO}.  It would be interesting to know if this result holds for more
general conical symplectic resolutions.
\end{remark}

\begin{remark}
Let $\nu$ be the map from $\fM$ to $\fM_0$.  For any symplectic leaf $S$,
let $\phi_S$ be the local system obtained
by restricting $R^{\codim S}\nu_!\C_\fM$ to $S$.  In other words, the stalks of $\phi_S$ are the
top cohomology groups of the fibers of $\nu$.
By \cite[8.9.8]{CG97}, there is an isomorphism
\begin{equation}\label{convolution:local systems}
\HZC \cong \bigoplus_{S\in\all}\End \phi_S.\
\end{equation} 
The ring homomorphism $\suppc:K(\HCg)_\C\to\HZC$ of Proposition
\ref{categorification} induces a map from $K(\fJ_S)_\C$ to $\End
\phi_S$.
% , where we regard $K(\fJ_S)_\C$ to $\End \phi_S$ as
 %subquotients of $K(\HCg)_\C$ and $\HZC$, respectively.  
If $\fM_0$ is an S3-variety of type A, then both $K(\fJ_S)_\C$ and $\End \phi_S$
are matrix algebras of the same rank, thus the
% in the former case since the
% cell representations are irreducible, and in the latter since all
% orbits are simply-connected (see \cite[Ch. 13]{Carter}).  Thus, the
map is an isomorphism.  In other types, this will not be the case; in
particular, 
% if the period of the quantization is regular, then 
$\fJ_S$ will be trivial unless $S$ is special.  
% The hypertoric case, in which all leaves are special, is being studied by
% Hilburn and the third author.
\end{remark}

\subsection{Filtration on category \texorpdfstring{$\cO$}{O}}
Definition \ref{HCa-filtration} (respectively \ref{HCg-filtration})
gives us a filtration of the monoidal category $\HCa$ (respectively $\HCg$) by sub-monoidal categories indexed by the poset of symplectic leaves.  
By Propositions \ref{tensor action}, \ref{Oa vs CMpa}, and \ref{Og vs CMpg}, $\cOa$ (respectively $\cOg$) is a module
category over $\HCa$ (respectively $\HCg$).
In this section we will define the analogous filtrations of the module categories.

\begin{definition}\label{Oa-filtration}
Let $\cOaS\subset\cOa$ be the full subcategory consisting of modules $N$ such that for some (equivalently any)
filtration of $N$, the coherent sheaf $\gr N$ on $\fM_0$ is set-theoretically supported on the closure of $S$.
Let $\cOadS\subset\cOaS$ be the full subcategory supported on leaves strictly smaller than $S$.
Let $\dOaS$ (respectively $\dOadS$) be the full subcategory of $D^b(\Amod)$ consisting of objects whose cohomology lies in $\cOaS$ (respectively $\cOadS$).
\end{definition}

\begin{definition}\label{Og-filtration}
Let $\cOgS\subset\cOg$ be the full subcategory of objects $\cN$
that are set-theoretically supported on the preimage in $\fM$ of the closure of $S$.
Let $\cOgdS\subset\cOgS$ be the full subcategory supported on preimages of leaves strictly smaller than $S$.
Let $D^b_{\cOgS}(\Dmod)$ (respectively $D^b_{\cOgdS}(\Dmod)$) be the full subcategory of $D^b(\Dmod)$ consisting of objects whose cohomology
lies in $\cOgS$ (respectively $\cOgdS$).
\end{definition}

The following straightforward proposition asserts that the above filtrations interact well with the structures
that we have already defined.

\begin{proposition}  Consider the tensor product and convolution actions of Proposition \ref{tensor action}.
\begin{enumerate}
\item The functor $\secs$ takes $\HCgS$ to $\HCaS$ and $\cOgS$ to $\cOaS$.
\item If $H\in\HCaS$ and $N\in\cOa$, then $H\otimes N\in\cOaS$.
\item If $\cH\in\HCgS$ and $\cN\in\cOg$, then $\cH\star \cN\in\cOgS$.
\item Each of the above statements holds when $S$ is replaced with $\partial S$.
\item Each of the above statements holds in the derived setting.
\end{enumerate}
\end{proposition}

Let $\fP_S$ denote the quotient of $\cOgS$ by $\cOgdS$.  Using the Euler form on $K(\cOg)$,
we obtain an orthogonal decomposition 
\begin{equation}\label{O-decomposition}
\bm\cong K(\cOg) \cong \bigoplus_{S\in\all} K(\fP_S),
\end{equation}
completely analogous to that of Equation \eqref{HC-decomp}.

\begin{remark}\label{DG quotient}
This decomposition can be given a categorical interpretation as follows.  For simplicity, we will consider a single step in the filtration.
Let $(\cOgdS)^\perp$ denote the full subcategory of $\cOgS$ consisting of objects $X$ such that, 
for every object $Y \in \cOgdS$, $\Ext^k(X,Y) = 0$ for all $k$; the
fact that $\cOg$ is Artinian with a projective generator guarantees
that the same is true of $\cOgS,\cOgdS$ and $(\cOgdS)^\perp$.
\begin{proposition}\label{perp category}
$\displaystyle (\cOgdS)^\perp\operatorname{-proj} \cong \fP_S\operatorname{-proj}.$
\end{proposition}
\begin{proof}
A projective module $P$ lives in $(\cOgdS)^\perp$ if and only if it
has $\Hom(P,Y)=0$ for every $Y\in  \cOgdS$.  
On the other hand, quotient to $\fP_S$ followed by its left adjoint
sends each module $N$ to its minimal submodule $N_S$ such that
$N/N_S\in \cOgdS$.  By assumption, the natural map $P\to P/P_S$ is 0,
so $P\cong P_S$.  This shows that the object $\bar P$ in the quotient
category satisfies $\Hom(\bar P,\bar N)\cong\Hom(P,N)$ for all modules
$N$.  In particular, the quotient induces a fully faithful functor
$\displaystyle (\cOgdS)^\perp\operatorname{-proj} \to
\fP_S\operatorname{-proj}.$

The exactness of the quotient functor means that its left
adjoint sends projectives to projectives; since this left adjoint
lands in $(\cOgdS)^\perp$, this provides a splitting to the fully
faithful functor of reduction, and thus induces an equivalence of categories.
\end{proof}

By Proposition \ref{perp category}, the decomposition
$$K(\cOgS) \cong K(\cOgdS) \oplus K(\fP_S)$$ may be identified with the decomposition 
$$K(\cOgS) = K(\cOgdS) \oplus K\big((\cOgdS)^\perp\big),$$ which is
categorified by the semiorthogonal decomposition of $\cOgS$ into $\cOgdS$ and $(\cOgdS)^\perp$.
\end{remark}

\subsection{Relation with the BBD filtration}
Let $\cF$ be the derived pushforward of the constant sheaf from $\fM$ to $\fM_0$, and let $\cF_S := \IC(\phi_S)[-\codim S]$.
By the BBD decomposition theorem, we have a canonical direct sum decomposition \cite[8.9.3]{CG97}
\begin{equation}\label{BBD sum}
\cF \cong \bigoplus_{S\in\all} \cF_S.
\end{equation}
Define a functor $\Sigma\colon D^b_{\mathscr{S}}(\M_0) \to \C\mmod$ by
\[\Sigma(-)=\hcd(\fM_0; -);\] applying it to both sides of \eqref{BBD sum} we obtain isomorphisms
\begin{equation}\label{BBD-decomposition}
\bmc \cong \Sigma(\cF) \cong  \bigoplus_{S\in\all} \Sigma(\cF_S).
\end{equation}

The isomorphism \eqref{convolution:local systems}
is a consequence of an isomorphism  of the convolution algebra $\HZC$ with $\End(\cF)$, so \eqref{BBD-decomposition} is compatible with the action of $\HZC$.
%  The rest of this section will be devoted to comparing the decompositions of Equations \eqref{O-decomposition}
% nd \eqref{BBD-decomposition}, which sometimes (but not always) coincide.

\begin{lemma}\label{perp}
The direct sum decomposition of Equation \eqref{BBD-decomposition} is orthogonal
with respect to the equivariant intersection form introduced in Section \ref{gco-intform}.
\end{lemma}

\begin{proof}
We will use the fact that the equivariant intersection form is compatible with the action of
the convolution algebra in the sense
that, for all $a\in\HZC$ and $b,c\in\Sigma(\cF)$, we have 
$(a\star b,c)=(b,a^*\star c)$, where $a\mapsto a^*$ is the
anti-automorphism of $\HZC$ given by flipping the two factors of $\fM\times \fM$.  To see this, note that $\HZC$ and $\Sigma(\cF) = \bmc$ are isomorphic to the $T$-equivariant versions of these groups, since the cohomology in lower degrees vanishes.  We can localize the equivariant groups to the $T$-fixed points, and using the projection formula for a proper pushforward, we get
\[(a \star b)|_\a = \sum_{\b\in I} \frac{a|_{(\a, \b)} \cdot b_{\b}}{e(\b)}\]
for all $\a \in I$.  Since $a^*|_{(\a,\b)} = a|_{(\b, \a)}$, the result follows.

Now let $e_S$ be the central idempotent in $\HZC$ which
projects $\Sigma(\cF)$ onto $\Sigma(\cF_S).$  Then $e_S^*$ is again a central
idempotent; we will prove by induction on $\all$ that $e_S^* = e_S$.
Assume that $e_{S'}^* = e_{S'}$ for all $S'<S$.
It is clear that $e_S$ is represented entirely by cycles on $\fZ$ that live over $\bar S\subset \fZ_0\cong\fM_0$,
therefore the image of $e_S^*$ is contained in $\bigoplus_{S'\leq S}\Sigma( \cF_{S'}).$
For all $S'<S$, our inductive hypothesis tells us that $\Sigma(\cF_{S'})$ is equal to the image of 
$e_{S'}^*$, and is therefore disjoint from the image of $e_S^*$.  Since the image of $e_S^*$
is invariant under the convolution algebra and complementary to the sum of the images of $e_{S'}^*$ for $S'<S$,
it must be equal to $\Sigma(\cF_{S'})$.  Thus $e_S^* = e_S$.

Suppose that $b\in \Sigma(\cF_S)$ and $c\in \Sigma(\cF_{S'})$ for some $S\neq S'$. 
Then \[(b,c)=(e_Sb,c)=(b,e_Sc)=(b,0)=0.\]
This establishes the result.
\end{proof} 

It is tempting to guess that two decompositions of $\bmc$ given in Equations \eqref{O-decomposition}
and \eqref{BBD-decomposition} coincide.
This cannot be correct in general, however, because $\fP_S$ is trivial
unless $S$ is special at a parameter where derived localization holds, while 
$\Sigma(\cF_S)$ is always nontrivial.  
The next natural guess is that the appropriately coarsened statement holds for special leaves;
that is, if $S$ is special, then the subspace $$K(\cOgS)_\C\subset K(\cOg)_\C\cong\bmc$$ should coincide
with the sum of the subspaces $$\Sigma(\cF_{S'})\subset\Sigma(\cF)\cong\bmc$$ corresponding to leaves $S'$
that are contained in the closure of $S$.
Even this statement fails in general (see Remark \ref{not interleaved}), but it holds if 
$(\fM, \cD)$ is interleaved (Definition \ref{simple supports}).  
We state this result below, and prove it at the end of this section.

\begin{theorem}\label{special filtrations}
Suppose that localization holds for $\cD$.
Then $(\fM,\cD)$ is interleaved if and only if for every special symplectic
leaf $S\in\spe$, the image of 
$K(\cOgS)_\C\subset K(\cOg)_\C$ under $\suppc$ is equal to 
$$\bigoplus_{S'\leq S} \Sigma(\cF_{S'}).$$
\end{theorem}

\begin{corollary}\label{the special case}
Suppose that $(\fM, \cD)$ is interleaved and all symplectic leaves are special.
Then the characteristic cycle isomorphism takes the categorical decomposition of $K(\cOg)$
from Equation \eqref{O-decomposition} to the BBD decomposition of $\bmc = \Sigma(\cF)$ from Equation \eqref{BBD-decomposition}.
\end{corollary}

\begin{proof}
By definition, the categorical decomposition is orthogonal with respect to the Euler pairing.
By Lemma \ref{perp}, the BBD decomposition is orthogonal with respect to the intersection
pairing.  The result then follows from Theorems \ref{support isomorphism} and \ref{special filtrations}.
\end{proof}

\vspace{-\baselineskip}
\begin{example}\label{ht and ftA}
If $\fM$ is a hypertoric variety or a finite type A quiver variety and $\cD$ is an integral quantization
for which localization holds, then $(\fM, \cD)$ is interleaved and all symplectic leaves are special, so the corollary applies.
\end{example}

\begin{remark}\label{not interleaved}
  One can use Theorem \ref{special filtrations} to show that $(\fM,\cD)$ is not interleaved
  by finding a special leaf for which the two vector spaces in question have different dimensions.  In the case where $\fM=T^*(G/B)$
  and the quantization is integral, the subspaces $K(\cOgS)_\C$
  and $\bigoplus_{S'\leq S} \Sigma(\cF_{S'})$ are
 sums of isotypic components for the action of the Weyl group on $K(\cOg)_\C\cong
  \Sigma(\cF)$; thus, we need only consider which simple
  representations appear in this subspace.  Using the notation of
  Carter's book \cite{Carter}, when $S$ is the nilpotent orbit $A_2$,
  the subspace $K(\cOgS)_\C$  is the sum of the isotypic components
  for the families
\[\{\phi_{1,24}\}\qquad\{\phi_{2,16}'', \phi_{4,13},  \phi_{2,16}'\}\qquad
\{\phi_{9,10}\}\qquad \{\phi_{8,9}''\};\]
these families are listed on  \cite[pg. 414]{Carter}.  The first index
is the dimension of the representation, so $\dim K(\cOgS)_\C=170$.

However, the corresponding piece of the BBD filtration
$\bigoplus_{S'\leq S} \Sigma(\cF_{S'})$ is larger; by the
chart on \cite[pg. 428]{Carter}, it also includes the isotypic
component for $\phi_{1,12}''$, which corresponds under the Springer
correspondence to the unique non-trivial local system on $A_2$; thus its
dimension is 171.  This shows that $T^*(G/B)$ is not interleaved for $G$
of type $F_4$ (we thank Victor Ostrik for pointing out this example to us). 
Inspection of the charts in Carter's book also shows
that the same is true for types $E_7$ and $E_8$.

In types $B$, $C$, and $D$, it is also easy to find examples where these two
filtrations do not match.  For example, in $C_4$, the representations
associated to the pairs of partitions $((2,1),(1))$ and
$((2,2),\emptyset)$ are both
associated to the nilpotent orbit with Jordan type $(4,2,2)$. However,
only the former is in the 2-sided cell of this special orbit; the
latter is in the 2-sided cell for $(4,4)$.  Outside of a few cases of
small rank, the variety $T^*(G/B)$ will be interleaved only in type A.
\end{remark}

To prove Theorem \ref{special filtrations}, we need a generalization of the functor $\Sigma$.
For all $S\in\all$, define a functor $\Sigma_S:D^b_{\mathscr{S}}(\M_0) \to \C\mmod$ by
$$\Sigma_S := H^{\dimfM}_{\fM_0^+\cap \bar S}(\fM_0; -\big).$$ 
For every $S$, the inclusion of $\fM^+_0\cap \bar S$ into $\fM^+_0$ induces a natural transformation $\Sigma_S\to\Sigma$.

\begin{lemma}\label{basis} For any $S \in \mathscr{S}$, the map
$\Sigma_S(\cF) \to \Sigma(\cF)$ is injective, and the image has basis
$\{[X_\a] \mid X_{\a,0} \subset \bar{S}\}.$
\end{lemma}
\begin{proof} Let $\M^+_{\bar S} := \M^+ \cap \nu^{-1}(\bar{S})$.  Then
Poincar\'e duality gives an isomorphism $$\Sigma_S(\cF) \cong H^{\dimfM}(\M,\M \setminus
\M^+_{\bar S}; \C) \cong H^{BM}_{\dimfM}(\M^+_{\bar S}; \C).$$ 
Since $\M^+_{\bar S}$ is purely $d$-dimensional,
this last group is the subgroup of $H^{BM}_{\dimfM}(\M^+; \C)$
spanned by the 
classes $[X_\alpha]$ for all $X_\a \subset \M^+_{\bar S}$.
\end{proof}

Let $\cF_{\le S} := \bigoplus_{S' \le S} \cF_{S'}$.

\begin{lemma} \label{BBD-filtration-basis} The image of the natural injection $\Sigma(\cF_{\le S}) \to \Sigma(\cF)$ 
is the same as the image of $\Sigma_S(\cF) \to \Sigma(\cF)$.
\end{lemma}

\begin{proof} It is clear that $\Sigma_S(\cF_{\le S}) \to \Sigma(\cF_{\le S})$ is an isomorphism, so it is enough to show that $\Sigma_S(\cF_{\le S}) \to \Sigma_S(\cF)$ is an isomorphism, or equivalently that $\Sigma_S(\cF_{S'}) = 0$ if $S' \not\le S$.  Let 
$j\colon \bar{S} \hookrightarrow \fM_0$ be the inclusion; then 
$$\Sigma_S(\cF_{S'}) = H^{\dimfM}_{\M_0^+ \cap \bar S}(\bar{S}; j^!\cF_S).$$  Since
$\cF_{S'}$ is an intersection cohomology complex and $S'\cap \bar{S} = \emptyset$, we
have $$j^!\cF_{S'} \in {}^p\!D_c^{\ge \dimfM + 1}(\bar{S});$$ see, for example, \cite[8.2.5]{HTT}.  
This implies that $H^k_{\M_0^+ \cap \bar S}(\bar{S}; j^!\cF_S)$ vanishes for $k \le \dim \M$ \cite[8.1.24]{HTT},  so 
$\Sigma_S(\cF_{S'}) = 0$, as desired.
\end{proof}

\begin{proofsf}
First suppose that $(\fM, \cD)$ is interleaved.
The image of $K(\cOgS)_\C$ is spanned by $\{\suppc \Lambda_\a\mid \fM_{\a,0}\subset \bar{S}\}$. 
By Lemma \ref{triangularity}, it is spanned by $\{[X_\a]\mid\fM_{\a,0}\subset \bar{S}\}$.
On the other hand, Lemmas \ref{basis} and \ref{BBD-filtration-basis} imply that $\bigoplus_{S'\leq S} \Sigma(\cF_{S'})$
is spanned by $\{[X_\a]\mid X_{\a,0}\subset \bar{S}\}$.
Since $S$ is special, $\fM_{\a,0}\subset \bar{S}$ if and only if $X_{\a,0}\subset \bar{S}$,
so the two vector spaces agree.

Now suppose that $(\fM, \cD)$ is not interleaved.  This means that there exists
an element $\a\in\cI$ and a special leaf $S$ such that $X_{\a,0}\subset \bar S\subsetneq \fM_{\a,0}$.
Lemma \ref{triangularity} says that the basis $\{\suppc \Lambda_\b\}_{\b \in \cI}$ is triangular with respect to the order $\leftharpoonup$, so if we write $[X_\a]$
in this basis, $\suppc \Lambda_\a$ must occur with non-zero coefficient.
It follows that $[X_\a]$ lies in $\bigoplus_{S'\leq S} \Sigma(\cF_{S'})$ but not in $K(\cOgS)_\C$.
\end{proofsf}

\subsection{The extreme pieces}\label{sec:extreme}
All of the structures discussed in this section are of particular interest when $S$ is equal to either the point stratum $\{o\}$
or the dense stratum $\becircled\fM_0$.  We begin with the point stratum.

The category $\cO_a^{\{o\}}$ is equal to the category of finite-dimensional $A$-modules.  Therefore, if localization holds,
we have $$K(\fP_{\{o\}}) = K(\cOg^{\{o\}}) \cong K(\cOa^{\{o\}}) = \Z\{[L_\a]\mid \dim L_\a<\infty\}.$$
By Lemmas \ref{basis} and \ref{BBD-filtration-basis}, the map $\suppc$ takes $K(\fP_{\{o\}})_\C$ 
to $\Sigma(\cF_{\{o\}})\cong H^{\dimfM}(\fM; \C)$. This map may or may not be an isomorphism.
The following theorem is an alternate version of Theorem \ref{special filtrations} that we may apply if we only care about the
point stratum; the proof is clear.

\begin{theorem}\label{point}
Suppose that localization holds for $\cD$.  Then the natural map from $K(\fP_{\{o\}})_\C$ to $H^{\dimfM}(\fM; \C)$
is always injective, and it is an isomorphism if and only if $\dim L_\a<\infty$ for all $\a\in\cI$ such that $X_{\a,0} = \{o\}$.
\end{theorem}

\begin{definition}
If $(\fM,\cD)$ satisfies hypotheses of Theorem \ref{point}, we will refer to the pair as {\bf fat-tailed}.
\end{definition}

Next, we turn our attention to the dense stratum.
A simple object of $\cOa$ lies in $\cOa^{\partial\becircled\fM_0}$ if and only if its annihilator is nonzero, so if localization holds,
we have
$$K(\fP_{\becircled\fM_0}) \cong \Z\{[L_\a]\mid \Ann(L_\a) = 0\}.$$
By Lemmas \ref{basis} and \ref{BBD-filtration-basis}, the map $\suppc$ takes $K(\cOa^{\partial\becircled\fM_0})_\C$
to $\bigoplus_{\bar S\neq \fM_0}\Sigma(\cF_S)$, so there is a naturally induced map of quotient spaces
from $K(\fP_{\becircled\fM_0})$ to $\Sigma(\cF_{\becircled\fM_0})$.  Furthermore, we have a nice interpretation
of the vector space $\Sigma(\cF_{\becircled\fM_0})$, as described in the following lemma.

\begin{lemma}\label{IH}
There is a canonical isomorphism $\Sigma(\cF_{\becircled\fM_0}) \cong I\! H^{2d}_{\bT}(\fM_0; \C)$.
\end{lemma}

\begin{proof}
We have $I\! H^*(\fM_0, \fM_0 \setminus \fM_0^+;\C) \subset H^*(\fM, \fM \setminus \fM^+;\C)$. Since the second group is zero except in degree $2d$, so is the first.  It follows that
the forgetful map induces an isomorphism 
\[ I\! H^{2d}_{\bT}(\fM_0, \fM_0 \setminus \fM_0^+;\C)\cong I\! H^{2d}(\fM_0, \fM_0 \setminus \fM_0^+;\C) = \Sigma(\cF_{\becircled\fM_0}).\]  
 Since $\dim \bT = 1$ and $\fM \setminus \fM_0^+$ contains no $\bT$-fixed points, the total dimension of 
$I\! H^*_\bT(\fM_0 \setminus \fM_0^+; \C)$ is finite.  The result now follows by applying the long exact sequence in $I\! H_{\bT}$ for the pair $(\fM_0, \fM_0 \setminus \fM_0^+)$ together with the fact that 
$I\! H^*_{\bT}(\fM_0; \C)$ is generated as an $H^*_\bT(pt)$-module by the part in degrees $< 2d$.
\end{proof}

The following theorem is the alternate version of Theorem \ref{special filtrations} that we may apply if we only care about the
dense stratum; again, the proof is clear.

\begin{theorem}\label{dense}
Suppose that localization holds for $\cD$.  Then the natural map from $K(\fP_{\becircled\fM_0})_\C$ to $I\! H^{2d}_{\bT}(\fM_0; \C)$
is always surjective, and it is an isomorphism if and only if $\fM_{\a,0} \neq \fM_0$ for all $\a\in\cI$ such that $X_{\a,0}$ is contained in the closure of a
non-dense leaf.
\end{theorem}

\begin{definition}
If $(\fM,\cD)$ satisfies hypotheses of Theorem \ref{dense}, we will refer to the pair as {\bf light-headed}.
\end{definition}

\begin{example}\label{fat-light}
If $\fM$ is a hypertoric variety or a finite type A quiver variety and
$\cD$ is an integral quantization for which localization holds, then $(\fM,\cD)$ is interleaved
(Example \ref{ht and ftA}), and therefore
both fat-tailed and light-headed.
\end{example}

\begin{example}\label{light-quiver}
If $\fM$ is a quiver variety of finite simply-laced type and $\cD$ is an integral quantization for which
localization holds, then $(\fM,\cD)$ is fat-tailed \cite[1.2]{BLet}.
\end{example}

\begin{conjecture}\label{GrDE}
If $\fM$ is a resolution of a transverse slice in the affine Grassmannian of finite simply-laced type 
and $\cD$ is an integral quantization for which
localization holds, then $(\fM,\cD)$ is light-headed.
\end{conjecture}

\begin{remark}
In type A, transverse slices in the affine Grassmannian coincide with quiver varieties \cite{MV08}, so Conjecture \ref{GrDE}
follows from Example \ref{fat-light}; the open cases are in types D and E.  We will revisit this conjecture
in Example \ref{awesome example}.
\end{remark}

\subsection{Cells}\label{sec:cells}
Throughout this section we will assume that localization holds.
Consider a pair of indices $\a,\b\in\cI$.

\begin{definition}
We say that $\a\Lleq \b$ if $\Ann L_\b\subset \Ann L_\a$.  
We say that $\a\Rleq \b$ if there exists a Harish-Chandra bimodule $H$
such that $L_\a$ is a subquotient of $H\otimes L_\b$.
We define a third pre-order on $\cI$ by putting $\a\Tleq\b$ if $\a\Lleq\b$ or $\a\Rleq\b$,
and then taking the transitive closure.  

If $\a\Tleq\b$ and $\b\Tleq\a$, we say that $\a$
and $\b$ lie in the same {\bf two-sided cell} of $\cI$.
\end{definition}

\begin{proposition}\label{two-sided-prop}
If $\a\Tleq\b$, then $\fM_{\a,0} \subset \fM_{\b,0}$.
\end{proposition}

\begin{proof}
It is enough to show that either $\a\Lleq\b$ or $\a\Rleq\b$ implies that $\fM_{\a,0} \subset \fM_{\b,0}$.
By Corollary \ref{mza} and the fact that localization holds, $\fM_{\a,0} = \fM_{L_\a}$
and $\fM_{\b,0} = \fM_{L_\b}$.  If $\a\Lleq \b$, then $\Ann L_\b\subset \Ann L_\a$,
so by definition of $\fM_L$, we have $\fM_{\a,0} \subset \fM_{\b,0}$.  

If $\a\Rleq \b$, then there exists an algebraic Harish-Chandra bimodule $H$ 
such that $L_\a$ is a subquotient of $H\otimes L_\b$.
Localizing, we obtain a geometric Harish-Chandra bimodule $\cH$ 
such that $\Lambda_\a$ is a subquotient of $\cH\otimes \Lambda_\b$.
This implies that the support of $\Lambda_\a$ is contained in the support of $\Lambda_\b$.
By definition of $\fM_{\a,0}$, this implies that $\fM_{\a,0} \subset \fM_{\b,0}$.
\end{proof}

\vspace{-\baselineskip}
\begin{remark}
By Proposition \ref{two-sided-prop}, we have a surjective map from the set of two-sided cells to the set
of special leaves that takes the cell containing $\a$ to the special
leaf whose closure is $\fM_{\a,0}$.  If $\fM$ is a hypertoric variety, this map is a bijection \cite[7.14]{BLPWtorico};
the same is true if $\fM = T^*(G/B)$ and the period of $\cD$ is integral.  However, it need not be a bijection
when $\fM = T^*(G/B)$ and the period is non-integral.

For example, let $G = G_2$, let $\a_1,\a_2\in\mathfrak{h}^* \cong \Ht$ be the short and long simple roots,
and consider the quantization with period $\a_1 + \a_2/2$.
In this case, the simple modules with highest weights
$\a_1-\a_2/2$ and $-11\a_1-5\a_2/2$ are both associated with the sub-regular nilpotent orbit.
\end{remark}

\begin{conjecture}\label{cells}
Suppose that $\cD$ is an integral quantization in the sense of Section \ref{sec:integrality}.
Then the map from the set of two-sided cells to $\spe$ taking the equivalence class of $\a$
to $\fM_{\a,0}$ is an isomorphism of posets.
\end{conjecture}

\begin{remark}\label{cell decomposition}
If Conjecture \ref{cells} holds, then the summands of $\bm$ in Equation \eqref{O-decomposition}
are simply the spans of the classes of the simple elements in each two-sided cell.
\end{remark}

\begin{remark}
One may refine the set of two-sided cells in two different ways:  by left cells (using the preorder $\Lleq$)
or by right cells (using the preorder $\Rleq$).  
By considering the left preorder, we can construct a surjective map from the set of left cells to the
set of irreducible components
of preimages in $\fM$ of special leaves in $\fM_0$.  If $\fM$ is a hypertoric variety, then this map is a bijection by
\cite[7.2]{BLPWtorico}.  If $\fM = T^*(G/B)$ and the period of $\cD$ is integral,
the same is true by the circle on the first page of the book of Borho,
Brylinski and MacPherson \cite{BBM89} (also based on work of Joseph,
Kashiwara and Duflo, among others).  As in the case of two-sided cells,
the non-integral case is more subtle.

By considering the right preorder, we can construct a surjective map from the set of right cells
to the set of irreducible components of varieties
of the form $\fM_{\a,0}^+ := \fM_{\a,0}\cap\fM^+$; such components are called {\bf orbital varieties}.
In the hypertoric case, this map is a bijection by \cite[7.11]{BLPWtorico}.
\end{remark}

\section{Twisting and shuffling functors}
\label{sec:twist-shuffl-funct}
The purpose of this section is to introduce two commuting collections of endofunctors of $\cOa$,
called twisting and shuffling functors.
The twisting functors, which operate by varying the period of the quantization, act on the entire category
$\Amod$, taking the subcategory $\cOa$ to itself.  These functors were introduced
in \cite[\S 6.4]{BLPWquant}, and they generalize Arkhipov's twisting functors on BGG category $\cO$ (see Remark \ref{Arkhipov}).
The shuffling functors operate by varying the choice of $\bT$, and therefore can only be defined on the category $\cOa$.
These functors generalize Irving's shuffling functors on BGG category $\cO$ (see Proposition \ref{shuf-twist}).

\subsection{Twisting functors}\label{twisting-gco}
Let $\la, \la'\in H^2(\fM; \C)$ be a pair of classes such that $\la - \la'\in H^2(\fM; \Z)$.
Let $A_\la$ and $A_{\la'}$ denote the algebras of $\bS$-invariant global sections of quantizations with periods $\la$ and $\la'$,
% and let $\cL$ be the line bundle on $\fM$ with Euler class $\la'-\la$.  
and let $\cOa\subset\Amod$ and $\cOa'\subset A_{\la'}\mmod$ be the associated categories.

In \cite[6.21]{BLPWquant}, we define an $(A_{\la'}, A_\la)$-bimodule $\bi$.  In the most general situation, 
$\bi$ is defined as a specialization of the space of sections of a quantized line bundle on the 
universal deformation $\scrM$ of $\fM$.  If localization holds at $\la'$, 
then $\bi$ can be described more simply as the bimodule of $\bS$-invariant global sections of $\cbi[\hmon]$ \cite[6.26]{BLPWquant}.
If $\fM$ is constructed as the symplectic quotient of a symplectic vector space, then $\bi$ can also be realized as a weight
space in a quotient of the Weyl algebra of the vector space \cite[6.28]{BLPWquant}.

Let $$\Phi^{\la'\!,\la}\colon D(A_\la\mmod)\to D(A_{\la'}\mmod)$$
be the functor obtained by derived tensor product with $\bi$. 

Let $\Pi\subset\Htr$ be the set of $\la$ such that
\begin{itemize}
\item $\la$ does not lie on the complexification of any of the hyperplanes of $\cHt$ (Section \ref{sec:nam}), and
\item there exists some conical symplectic resolution $\fM'$ of $\fM_0$ and some element $w$ of the
Namikawa Weyl group such that localization holds at $w\la$ on $\fM'$.
\end{itemize}
We prove in \cite[6.32]{BLPWquant} that, if $\la,\la'\in\Pi$, then the
functor $\Phi^{\la'\!,\la}$ is an equivalence, and it preserves bounded
derived categories; we will use the symbol $\Phi^{\la'\!,\la}$ for the
induced functor in this case.  By \cite[6.37]{BLPWquant}, in this case,
$\Phi^{\la'\!,\la}$ also takes the subcategory $D^b_{\cOa}(\Amod)$ to $D^b_{\cOa'}(\Amod)$.

\begin{remark}
By Remark \ref{dream}, the chambers of $\cHt$ are equal to $W$-translates of ample cones of conical symplectic
resolutions of $\fM_0$.  Thus $\Pi$ is the set of $W$-translates of classes that are not only ample on some resolution,
but deep enough in the ample cone so that localization holds on that resolution.  By Theorem \ref{localization},
the intersection of $\Pi$ with any chamber of $\cHt$ is nonempty (that is, it is always possible to go deep enough into the ample
cone so that localization holds).
\end{remark}

We define a {\bf pure twisting functor} to be an auto-equivalence of $D(A_{\la}\mmod)$ 
obtained by composing functors of the form $\Phi^{\la''\!,\la'}$ with
$\la',\la''\in \Pi$.  Such compositions
go through module categories for many different quantizations; we require that they pass only through
elements of $\Pi$, and that they begin and end at a single parameter $\la$.  
To define twisting functors in general, we incorporate the action of the Namikawa Weyl group.

For any $w\in W$, the rings $A_\la$ and $A_{w\la}$ are canonically isomorphic \cite[3.10]{BLPWquant}.
Let $$\Phi^{\la}_w\colon A_{w\la}\mmod\to A_{\la}\mmod$$
be the equivalence induced by this isomorphism; we will use the same
symbol to denote the induced functor on the derived category.

\begin{proposition}\label{w-O-commute}
The functor $\Phi_{w}^\la$ takes $\cOa$ to itself.
\end{proposition}
\begin{proof}
Recall that the isomorphism in \cite[3.10]{BLPWquant} that we use to define
$\Phi_{w}^\la$ arises from a $W$-action on the universal
 deformation $\scrM$.  This action commutes with that of $\bT$ since
 the isomorphism $A_\la\cong A_{w\la}$ induces the identity map on
 $A_\la(n)/A_{\la}(n-1)\cong A_{w\la}(n)/A_{w\la}(n-1)$ and preserves
 the grading by $\Z/n\Z$.  Thus it sends a non-commutative moment
 map for $\bT$ in $A_\la$ to one in $A_{w\la}$. Thus, the functor $\Phi_{w}^\la$ preserves category $\cOa$.
\end{proof}

We define a {\bf twisting functor} to be a composition of functors of the form $\Phi^{\la''\!,\la'}$ and their inverses
(passing only through elements of $\Pi$),
beginning at $D(A_{\la}\mmod)$ and ending at $D(A_{w\la}\mmod)$, followed by the functor $\Phi^{\la}_w$.  
Note that by \cite[6.25]{BLPWquant} there is a natural isomorphism  
$\Phi^{\la', \la} \circ \Phi^\la_w \cong \Phi^{\la'}_w \circ \Phi^{w\la', w\la}$, so
the set of twisting functors is closed under composition.

Let $E_{\operatorname{tw}}:=\displaystyle\Ht\smallsetminus\bigcup_{H\in\cHt}H_\C$.  The following theorem is proven in 
\cite[6.35 \& 6.37]{BLPWquant}.

\begin{theorem}\label{gco-twisting braid}
There is a natural homomorphism from $\pi_1(E_{\operatorname{tw}}/W, [\la])$ to the group of twisting functors on $D^b(A_{\la}\mmod)$,
preserving the full subcategory $\dOa$.
The subgroup $\pi_1(E_{\operatorname{tw}}, \la)$ maps to the group of pure twisting functors.
\end{theorem}

\begin{remark}\label{Arkhipov}
In the case of hypertoric varieties, twisting functors are studied in detail in \cite[\S 6]{GDKD} and \cite[\S 8.2]{BLPWtorico}.
In the case of the Springer resolution, we show in \cite[6.38]{BLPWquant} that they coincide with the twisting
functors defined by Arkhipov \cite{AS} (thus justifying the name).
\end{remark}

\begin{remark}\label{K-twisting}
On the level of the Grothendieck group, pure twisting functors act trivially \cite[6.39]{BLPWquant}, therefore
we obtain an action of $W$ on $K(\cOa)\cong K(\cOg)\cong\bm$.  This action coincides with the one arising from
the natural map from $\C[W]$ to the convolution algebra $\HZC$ \cite[6.40]{BLPWquant}.
\end{remark}

\subsection{Shuffling functors}\label{shuffling}
In this section we discuss shuffling functors, which are in a certain
sense ``dual'' to twisting functors (this will be explained in Section \ref{duality}).  Unlike twisting functors, these
are unavoidably tied to the subcategory $\cOa\subset \Amod$, since they are constructed by
varying the action of $\bT\cong\cs$.
We will fix a single quantization throughout this section, and we will assume that its period lies in $\Pi\cap\mathfrak{U}$.\footnote{In fact, it is easy to show, using twisting functors, that $\mathfrak{U}\subset\Pi$.}

\begin{lemma}\label{two-generics}
Let $G$ be the full group of Hamiltonian symplectomorphisms of $\fM$ that commute with $\bS$. 
 A maximal torus $T\subset G$ containing the image of $\bT$ is
 unique up to conjugation by the largest unipotent subgroup commuting
 with $\bT$.  
\end{lemma}

\begin{proof}
Let $C=C_G(\bT)$ be the centralizer of this cocharacter.  We wish to
show that this group is an extension of a torus by a unipotent
subgroup.  That is, we wish to show that any reductive subgroup of $C$
is a torus.  Any such subgroup must be contained in a maximal reductive
subgroup $G^r$.  Let $T$ be a maximal torus in $G^r$ containing $\bT$.
For all $\a\in\cI$, consider the map $\mg^r\to T_{p_\a}\fM$, and let $\mk\subset\mg^r$ denote 
the intersection over all $\a$ of the kernels of these maps.
Since the fixed points are isolated, we have $\mt\subset
\mathfrak{c}\cap \mg^r\subset\mk$.
We will show that $\mk = \mt$.  This will imply that $\mt = \mathfrak{c}\cap \mg^r$, and therefore
that $T$ is the identity component of $C\cap G^r$ and a maximal
reductive subgroup of $C$.  Since every torus containing $\bT$
is conjugate to $T$ under $C$ by the uniqueness of maximal
tori in $C$, the proof will be complete.

Since $\mk$ is invariant under the adjoint action of $\mt$, it must be a sum of
$\mt$ and some root spaces of $\mg^r$.  
Assume for the same of contradiction that there exists a root $\gamma$ such that $\mg_\gamma\subset\mk$, 
and let $G'\subset G^r$
be a semisimple subgroup whose Lie algebra $\mg'$ contains $\mg_\gamma$.
Consider a projective orbit $X\subset\fM$ of $G'$.  By assumption, for all $\a\in\cI$,
the map $\mg'\to T_{p_\a}X$ kills $\mg_\gamma$.
By the classification of projective homogeneous spaces,
this is only possible if $X$ is a point, thus all projective $G'$-orbits in $\fM$ are trivial.

Consider the action of $G'$ on the core $\fX\subset\fM$ (Remark \ref{core}).
If $G'$ acted nontrivially on $\fX$, then it would contain a nontrivial closed orbit \cite[Corollary 2]{BB80}, 
which we have seen is not possible.
Thus $G'$ fixes all of $\fX$.  Any Hamiltonian action
of a reductive group on a connected symplectic variety that fixes a Lagrangian
subvariety must be trivial, thus we obtain a contradiction.
\end{proof}

Let $\zeta:\bT\to T$ be the cocharacter of $T$ by which $\bT$ acts.  In this section we will vary $\zeta$,
and thereby vary the action of $\bT$.
We call a cocharacter {\bf generic} if $\fM^\bT$ is finite.
For any generic $\zeta$, we will write $\cOa^\zeta$ for the corresponding algebraic category $\cO$.
The set of non-generic cocharacters is equal to the intersection of the cocharacter lattice of $T$
with the union of a finite set $\cHs$ of hyperplanes in $\mt_\R$.

Let $D^b_{h}(\Dmod)$ be the full subcategory of $D^b(\Dmod)$ consisting of complexes with holonomic cohomology, and 
let $D^b_h(\Amod)$ be the full subcategory of $D^b(\Amod)$ that is taken to $D^b_h(\Dmod)$ by $\LLoc$.
Also let $\iota^{\zeta}:D^b_{\!\cOa^{\zeta}}(\Amod)\to D^b_h(\Amod)$ be the
inclusion functor; it is full and faithful by definition of $\mathfrak{U}$ (Section \ref{objects}).

\begin{proposition}
  The functor $\iota^{\zeta}$ has left and right
  adjoints  \[{}^L\pi^{\zeta}:D^b_h(\Amod)\to
  D^b_{\!\cOa^{\zeta}}(\Amod)\and {}^R\pi^{\zeta}:D^b_h(\Amod_h)\to
  D^b_{\!\cOa^{\zeta}}(\Amod). \]
\end{proposition}

\begin{proof}
  If $P$ is projective generator and $I$ an injective generator of $\cOa^{\zeta}$, then the functor
  $\Hom(P,-)$ induces an equivalence $D^b_{\!\cOa^{\zeta}}(\Amod)\cong
  D^b(\End(P)^{\op}\mmod)$, and similarly $\Hom(-,I)$ induces an equivalence with $D^b(\End(I)\mmod)^{\op}$. 
% The Ext-groups between any two holonomic
% $A$-modules are finite dimensional, in particular $\Ext_{A}(P,M)$ and $\Ext_{A}(M,I)$. 
In fact, there is a richer structure here:
replacing $P$ with a projective resolution of $P$ as an $A$-module,
for any object $N$ of $D^b_{\!\cOa^{\zeta}}(\Amod)$, we can think of $\Ext(P,N)$ as an
object in $D^b(\End(P)\mmod)^{\op}$.  
Similarly, we can think of $\Ext(N,I)$ as
an object in $D^b(\End(I)\mmod)$. (These are the Hom-spaces in the
usual dg-enhancement of $\Amod$.) Note that the hypothesis
that $\LLoc(N)$ is holonomic guarantees that these complexes are
finite dimensional. 
We can define
\[{}^L\pi^{\zeta}(N) :=\Ext_{\End(I)} (\Ext_{\cOa^{\zeta} }(N,I),I) \and
{}^R\pi^{\zeta}(N) := P \otimes_{\End(P) ^{\op}}\Ext_{\cOa^{\zeta} }(P,N).\]  
This completes the proof.
\end{proof} 

\vspace{-\baselineskip}
\begin{remark}
One can think of these two adjoints as ``projections'' onto $D^b_{\!\cOa^{\zeta}}(\Amod)$.  The
functor ${}^L\pi^{\zeta}$ is the derived functor of
taking the largest quotient of a module that lies in $\cOa^{\zeta}$, 
and ${}^R\pi^{\zeta}$ is the derived functor of taking the largest
such submodule.  It's clear that these functors are left/right exact, respectively. 
\end{remark}

Given  two different generic cocharacters $\zeta$ and $\zeta'$ of $T$, let
  $$\Psi^{\zeta'\!,\zeta} :=
  {}^L\pi^{\zeta'}\circ\iota^{\zeta}:D^b_{\!\cOa^{\zeta}}(\Amod)\to
  D^b_{\!\cOa^{\zeta'}}(\Amod),$$ and let $\Xi^{\zeta,\zeta'} = {}^R\pi^{\zeta}\circ\iota^{\zeta'}$ be its right adjoint. 
The following result, which should be regarded as an analogue of \cite[6.32]{BLPWquant}, 
was conjectured in a previous draft of this paper, and
has recently been proved by Losev \cite[7.3]{Losshuff}.

\begin{proposition}
\label{derived shuffling} 
The functor $\Psi^{\zeta'\!,\zeta}$ is an equivalence.
\end{proposition}

We define a {\bf pure shuffling functor} to be an endofunctor of $D^b_{\!\cOa^{\zeta}}(\Amod)$ 
obtained by composing functors of the form $\Psi^{\zeta''\!,\zeta'}$ for various generic cocharacters,
beginning and ending at a single generic cocharacter $\zeta$.
To define shuffling functors in general, we incorporate the action of the Weyl group.

Let $\rWeyl := N_G(T)/T$ be the Weyl group of $G$.  We use the blackboard-bold font to
distinguish $\rWeyl$ from the Namikawa Weyl group $W$,
which is typically different.  For example, if $\fM$ is a crepant resolution of $\C^2/\mck$,
then $W$ is isomorphic to the Weyl group corresponding to $\mck$ under the McKay correspondence, 
but $\rWeyl$ is trivial unless $\mck = \Z/2\Z$.

The action of $G$ on $\C[\fM]$ lifts canonically to an action on $A$.
The Weyl group $\rWeyl$ acts on the cocharacter lattice of $T$, and on the subset of generic cocharacters.
For all $w\in W$ and all generic $\zeta$, let $\zetaw := w\zeta w^{-1}$.
   Define a functor
$\Psi_{\bar w}: D^b(A\mmod)\to D^b(A\mmod)$ taking an $A$-module $N$
to the $A$-module  with the same underlying vector space, but with
action \[a\cdot x=(\bar wa)x \;\;\text{for all $x\in N$ and $a \in
  A$},\] for any  
$\bar w \in N_G(T)$ where the action on the left side is the action
on $\Psi_{\bar w}(N)$, and the action on the right is the original
action on $N$.  If $w$ is in the image of $\bar w$ in $\rWeyl$, we let $\Psi^\zetaw: D^b_{\!\cOa^{\zetaw}}(\Amod)\to
D^b_{\!\cOa^{\zeta}}(\Amod)$ denote the functor obtained by restricting $\Psi_{\bar w}$. Our notation for this functor is justified by the following lemma.

\begin{lemma}
Up to natural isomorphism, the functor $\Psi^\zetaw$ is independent of the choice of $\bar w$.  
\end{lemma}

\begin{proof}
  It suffices to prove that, for any $t\in T$, the functor from $\cOa^{\zeta}$ to itself
  given by twisting the module structure by the action of $t$ is 
  isomorphic to the identity functor.  

 By \cite[3.11]{BLPWquant}, the action of $T$ on $\fM$ admits a
 quantized moment map $\eta:U(\mt)\to A$.
The element $\xi \in A$ is the image of a generator of $\bt$ under a
quantized moment map for the $\bT$-action, so we can assume that $\eta$ extends this map.  In particular, $\eta$ induces an action of $\mt$ on any $A$-module $N$ in
$\cOa^{\zeta}$ which commutes with the action of $\xi$; since the $\xi$-weight spaces are finite-dimensional the $\mt$-action is semi-simple.

We can assume that $N$ is indecomposable, and so all the
$\mt$ weights lie in the same coset of the weight lattice of $T$
inside of $\mt^*$.  Thus, there is a character $\la\in \mt^*$
such that the action of $\mt$ on $N$ via $\eta'(x) := \eta(x)-\la(x)$ integrates
to an action $\rho\colon T \to \End(N)$.

Since $\la(x)$ is a scalar, we have
\[[\eta'(X),a]y=[\eta(X),a]y=\left.\frac{d}{ds}(e^{sX}\cdot a)y\right|_{s=0}\]
for any $y\in N$.
Integrating, we get an equality
$\rho(t)a\rho(t^{-1})=t\cdot a$ of operators in $\End(N)$ for any $t\in T$.
Thus for all $a\in A$ and $y\in N$, we have $(t\cdot a)\rho(t)y =
\rho(t)ay$.  In other words, the map $\rho(t)\colon N\to N$ intertwines
the $t$-twisted action with the original action.
\end{proof}

We define a {\bf shuffling functor} to be a composition of of functors of the form $\Psi^{\zeta''\!,\zeta'}$ 
and their inverses, beginning at $\cOa^{\zeta}$ and ending at $\cOa^{\zetaw}$,
followed by the functor $\Psi^\zetaw$.

With twisting functors, we have a result \cite[6.33]{BLPWquant} that says that twisting between parameters that
all lie within a fixed chamber of $\cHt$ is trivial.  We now establish the analogous result for shuffling functors, which we will
need in Section \ref{duality}.  For convenience, we will assume that derived localization holds.

\begin{lemma}\label{trivial shuffling}
Assume that derived localization holds.
Suppose that $\zeta$ and $\zeta'$ lie in the same chamber of $\cHs$.  Then the subcategories 
$D^b_{\!\cOa^\zeta}(\Amod),D^b_{\!\cOa^{\zeta'}}(\Amod)\subset D^b(\Amod)$
are equal, and $\Psi^{\zeta'\!,\zeta}$ is the identity functor.
\end{lemma}

\begin{proof}
Since derived localization holds, it is sufficient to prove that $\cOg^{\zeta} = \cOg^{\zeta'}$.
By the definition of geometric category $\cO$, it is sufficient to prove that $\zeta$ and $\zeta'$
induce the same relative core $\fM^+\subset\fM$.

Suppose not; this means that there exists $\a\in\cI$ such that the relative core components
$X_\a$ and $X_\a'$ defined by $\zeta$ and $\zeta'$ (Section \ref{sec:rel-core}) are different.
This in turn means that there is a character $\chi$ of $T$ such that the $\chi$-weight space of $T_{p_\a}\fM$ 
is nonzero and $\chi$ has opposite signs on $\zeta$ and $\zeta'$.
The vanishing set of $\chi$ is a hyperplane of $\cHs$ that separates $\zeta$ from $\zeta'$; this contradicts the fact
that $\zeta$ and $\zeta'$ lie in the same chamber.
\end{proof}

\begin{lemma}\label{prop:shuff-w}
There is a natural isomorphism  
$\displaystyle\Psi_w^{\zeta'}\circ \Psi^{\zeta'_w\!,\zeta_w}\cong\Psi^{\zeta'\!,\zeta} \circ \Psi_w^{\zeta}$.
\end{lemma}

\begin{proof}
By definition, we have $\Psi_{\bar w}\circ\iota^{\zeta_w}\cong \iota^{\zeta}\circ\Psi_w^{\zeta}$ for
any $\bar w\in N_G(T)$.  Then by adjointness, we have 
${}^L\pi^{\zeta'}\circ\Psi_{\bar{w}} \cong \Psi^{\zeta'}_w\circ {}^L\pi^{\zeta'_w}$, and
the result follows.
\end{proof}

We are now ready to state the analogue of Theorem \ref{gco-twisting braid}.  
Let  $E_{\operatorname{sh}}:= \displaystyle\mt\smallsetminus\bigcup_{H\in\cHt}H_\C$.

\begin{theorem}\label{gco-shuffling braid}
There is a natural homomorphism from $\pi_1(E_{\operatorname{sh}}/\rWeyl, [\zeta])$ to the group of shuffling functors on $D^b_{\!\cOa^\zeta}(\Amod)$.
The subgroup $\pi_1(E_{\operatorname{sh}}, \zeta)$ maps to the group of pure shuffling functors.
\end{theorem}
In fact, there are two natural such actions, intertwined by the
automorphism on $\pi_1(E_{\operatorname{sh}}/\rWeyl, [\zeta])$ induced
by complex conjugation.  One sends a minimal length oriented path in the Deligne
quiver to the functor $\Psi ^{\zeta,\zeta'}$, and the second sends
such a path to $\Xi ^{\zeta,\zeta'}$.  The second is the one that will
appear in the definition of symplectic duality (Definition
\ref{def:duality}).\medskip

\begin{proof}
Here we  follow 
the structure of the proof of
Theorem \ref{gco-twisting braid} in \cite[6.35]{BLPWquant}.    We
model the fundamental group of $E_{\operatorname{sh}}$ using the
Deligne groupoid, which is equivalent to
$\pi_1(E_{\operatorname{sh}})$.  The fundamental group of
$E_{\operatorname{sh}}/W$ is thus equivalent to the semi-direct product of $W$ with the
Deligne groupoid.

The
result \cite[7.3]{Losshuff} establishes that we have an action of the
Deligne groupoid, and Lemma \ref{prop:shuff-w} shows that this
action is compatible with the action of $W$ on the Deligne groupoid.
Thus we have an action of the semi-direct product, and therefore of
$\pi_1(E_{\operatorname{sh}}/\rWeyl, [\zeta])$.
\end{proof}

We have chosen to call the functors defined in this section ``shuffling functors" because they coincide
with Irving's shuffling functors \cite{Irvshuf} in the case of the Springer resolution.  More precisely, let
$\fM = T^*(G/B)$.  The group of Hamiltonian symplectomorphisms that commute with $\bS$ is
isomorphic to $G$ itself, and its Weyl group $\rWeyl$ is the usual Weyl group of $G$.  
(This example is unusual in that the Weyl group and the Namikawa Weyl group are isomorphic.)
Let $T\subset B\subset G$ be the unique maximal torus of $B$,
and let $\zeta:\bT\to T$ be a generic cocharacter with non-negative
weights on $\mb$.
Consider the shuffling functor
$$\Psi_w:=   \Psi_w^\zeta\circ\Psi^{\zetaw,\zeta} : D^b_{\!\cOa^\zeta}(\Amod)\to
D^b_{\!\cOa^\zeta}(\Amod).$$
As noted in \cite[\S 6.4]{BLPWquant}, the category $\cOa^\zeta$ for the period
$\la+\rho$ ($\la$ a dominant integral weight) is equivalent to
a regular infinitesimal block of BGG category $\cO$ via an equivalence defined by
Soergel \cite{Soe86}.

\begin{proposition}\label{shuf-twist}
For each $w\in \rWeyl$, the derived version of Soergel's equivalence
takes $\Psi_w$ to Irving's shuffling functor $C_{w_0ww_0}^{-1}[\ell(w)]$.
\end{proposition}

\begin{remark}
To prove Proposition \ref{shuf-twist}, we will make use of the fact that our twisting and shuffling functors commute,
which we will prove in the next section (Theorem \ref{twist-shuffle}).
\end{remark}

We first show that Proposition \ref{shuf-twist} is correct when applied to any Verma module.  For any $v \in W$, we denote by $M(v)$
the Verma module with highest weight $v(w_0\la-\rho)-\rho$.  It is an object in $\cOa^\zeta$.

\begin{lemma}
 For any Verma module $M(v)$, there is an isomorphism between the image
 $\Psi_w(M(v))$ and the image under the shuffling functor
 $C_{w_0ww_0}^{-1}(M(v))[\ell(w)]$.  
\end{lemma}
\begin{proof}
  We know from \cite[6.38]{BLPWquant} that Soergel's equivalence takes 
  the twisting functors defined in the previous section to Arkhipov's
  twisting functors.  
    Since all Verma modules are related by twisting functors, this
and Theorem \ref{twist-shuffle}  imply that if $\Psi_w$ and the shuffling functor $C_{w_0ww_0}^{-1}[\ell(w)]$ act the same way on
  one Verma module, then they have the same action on all of them.

The category $D^b_{\!\cOa^\zeta}(\Amod)$ can be identified, via localization and \cite[4.5]{BLPWquant}, with the 
derived category of $\la$-twisted D-modules on
$G/B$ which are smooth along the Schubert cells $X_w:=BwB/B$. 
By tensoring with a line bundle, we can further identify it with the derived category
of untwisted D-modules which are smooth along the Schubert cells.
Via this identification, the costandard objects of $\cOa^\zeta$ correspond to the D-modules
$\gst_v := (j_{v^{-1}})_*\fS_{X_{v^{-1}}}$, and
the functor $\Psi^\zeta_w$ corresponds to the pullback along the map $(\bar w\cdot ) \colon G/B\to G/B$.

Soergel's equivalence takes the Verma module $M(w_0vw_0)$ to $\gst_v$.
For any $w,v\in\rWeyl$, we have
\begin{multline*}
  \Ext^\bullet(\Psi_w\gst_{e},\gcst_{v})\cong \Ext^\bullet(\bar
  w^*\Psi^{\zetaw,\zeta}\gst_{e},\gcst_{v}) \cong
  \Ext^\bullet(\Psi^{\zetaw,\zeta}\gst_{e},\bar w_*\gcst_{v})\\ \cong
  \Ext^\bullet(\iota^{\zeta}\gst_{e},\iota^{\zetaw}\bar w_*\gcst_{v}) \cong
  \Ext^\bullet(\gst_{e},\bar w_*\gcst_{v}).
\end{multline*}
The D-module $\gst_{e}$ is supported at the point
${B/B}$  and $\bar w_*\gcst_{v}=j_*\fS_{wBv^{-1}B}$.  Thus, we
get trivial Exts unless $e\in wBv^{-1}B$, which only happens when
$v = w$.  When $v= w$, then $wBv^{-1}B/B$ is an affine space of dimension $\ell(w)$.
Thus we are reduced to a computation over the Weyl
algebra, and we obtain 
\begin{equation*}
  \Ext^\bullet(\Psi_w\gst_{e},\gcst_{v})\cong
  \begin{cases}
    0& v\neq w\\
    \C[-\ell(w)] & v=w.
  \end{cases}
\end{equation*}
This implies that $\Psi_w\gst_{e}\cong \gst_{w}[\ell(w)]$.

Now that we know how $\Psi_w$ acts on $\gst_e$, it remains to compute the
action of $C_{w_0ww_0}^{-1}[\ell(w)]$ on the corresponding Verma module $M(e)$.
We wish to show that $C_{w_0ww_0}^{-1}[\ell(w)]$ takes $M(e)$ to $M(w_0ww_0)[\ell(w)]$,
or equivalently that $C_{w}$ takes $M(w)$ to $M(e)$.
If $s$ is a simple reflection and $ws<w$, then by \cite[3.1]{Irvshuf} (using the notation of that paper), we have
\[C_sM(w)\cong C_sM(w_0,w_0w)\cong M(w_0,w_0ws)\cong M(ws).\] 
We can now prove the desired isomorphism by induction on length.
\end{proof}

\begin{proofIrv}
The usual $t$-structure on the derived BGG category $\cO$ is induced by the
exceptional collection of Verma modules as in \cite[Prop.~1]{Bez03}. Since the composition
$$F:=\Psi_w\circ C_{w_0ww_0}[-\ell(w)]$$ sends Vermas to Vermas,
it induces an auto-equivalence of the abelian category $\cO$ which
sends every simple to itself.  

Consider the trivial module $\C$ over $\mg$.  The functor
$\Psi_w$ sends $\C$ to itself, inducing
the identity on $\Ext_{\cO}(\C,\C)\cong H^*(G/B)$, since $\C$ lies in
category $\cO$ for every Borel. The same is true of
$C_{w_0ww_0}^{-1}[\ell(w)]$ since $\C$ is killed by translation to any
wall.  Thus, we have
an isomorphism $\C\cong F(\C)$ which induces the same isomorphism
$\Ext_{\cO}(\C,\C) \cong \Ext_{\cO}(F(\C),F(\C))$ as the functor $F$.
It follows that for any simple $L$ in $\cOa^\zeta$, we have a canonical isomorphism $f\colon \Ext^\bullet(\C,L)\cong
\Ext^\bullet(\C,F(L))$ of $H^*(G/B)$-modules induced by the functor $F$.

The Koszul dual form of Soergel's Endomorphismensatz \cite{Soe90}
states that for any two simple modules, we have an isomorphism
\[\Ext^\bullet(L,L')\cong
\Hom_{H^*(G/B)}(\Ext^\bullet(\C,L),\Ext^\bullet(\C,L')).\] This 
shows, in particular, that
\[\Ext^\bullet(L,F(L))\cong
\Hom_{H^*(G/B)}(\Ext^\bullet(\C,L),\Ext^\bullet(\C,F(L))),\] so the
isomorphism $f$ induces an isomorphism $L\cong F(L)$.  The
Endomorphismensatz similarly shows that this isomorphism induces the
same isomorphism of Yoneda algebras $\Ext^\bullet(\oplus L,\oplus
L)\cong \Ext^\bullet(\oplus F(L),\oplus F(L))$ as the functor
$F$. Thus, it induces an isomorphism between $F$ and the identity
functor.
\end{proofIrv}

We conclude this section by discussing the action of shuffling functors on the Grothendieck group. 
Just as we saw for twisting functors in Remark \ref{K-twisting}, we will find that the pure shuffling 
functors act trivially, and we are left with an action of the Weyl group $\rWeyl$. 

Recall from Theorem \ref{support isomorphism} that we have
$$\suppc:K(\cOa^\zeta)\cong K(\cOg^\zeta)\overset{\cong}{\longrightarrow} H^{2d}_{\fM^+_\zeta}(\fM; \Z),$$
and that this isomorphism intertwines the Euler form with the equivariant intersection form.
Furthermore, we have
$$H^{2d}_{\fM^+_\zeta}(\fM; \Z) \cong H^{2d}_{\fM^+_\zeta,T}(\fM; \Z)\hookto H^{2d}_T(\fM^T; \Z),$$
with image independent of $\zeta$.
In particular, this gives us a canonical way to identify the lattices
$K(\cOa^\zeta)\cong H^{2d}_{\fM^+_\zeta}(\fM; \Z)$ and $K(\cOa^{\zeta'})\cong H^{2d}_{\fM^+_{\zeta'}}(\fM; \Z)$ for any two generic cocharacters $\zeta$ and $\zeta'$.  
One can check that this identification sends $v_\a$ to $\pm
v_\a'$; the sign is given by the parity of the codimension of the
space of points that flow in to $p_\a$ for both $\zeta$ and $\zeta'$ inside the space of
points that flow in for $\zeta$ (note that this is symmetric under switching $\zeta$
and $\zeta'$).

This identification agrees with the map $K(\cOa^\zeta) \to K(\cOa^{\zeta'})$
induced by $\Psi^{\zeta',\zeta}$.  We omit a full proof
of this fact, since it will not be used later in the paper.  The proof
is similar to Theorem \ref{support isomorphism}; the key
is to show that the Euler pairing between $K(\cOa^\zeta)$ and $K(\cOa^{\zeta'})$
inside of $K(\Amod_h)$ agrees with the equivariant intersection pairing
on $H^{2d}_T(\fM^T; \Z)$.
Deformation arguments
show that it suffices to do this on a generic fiber of a twistor
deformation.  
Thus, we are reduced to calculating the Exts between
modules $M_L,M_L'$ over the
Weyl algebra deforming the structure sheaves of Lagrangian subspaces $L,L'$:
\[\Ext^i(M_L,M_{L'})=
\begin{cases}
  \C & i=\dim(L/(L\cap L'))\\
  0 & i\neq\dim(L/(L\cap L')).
\end{cases}
\] This implies the following proposition, which is an analogue of \cite[6.39]{BLPWquant}
(see Remark \ref{K-twisting}).

\begin{proposition}\label{K-shuffle}
Pure shuffling functors act trivially on the Grothendieck group of $\cOa^\zeta$.
\end{proposition} 

Impure shuffling functors, however, act in an interesting way.
Consider the impure shuffling functor $\Psi_w = \Psi_w^\zeta\circ\Psi^{\zetaw,\zeta}$.
We know from the above discussion that $\Psi^{\zetaw,\zeta}$ induces the aforementioned
canonical isomorphism from $H^{\dimfM}_{\fM^+_\zeta}(\fM; \Z)$ to $H^{\dimfM}_{\fM^+_{\zetaw}}\!(\fM; \Z)$.
The map induced by $\Psi_w^\zeta$ is given by choosing a lift $\bar w\in N(T)\subset G$ and considering
the automorphism of $\fM$ induced by $\bar w^{-1}$.  This automorphism intertwines the action of $\bT$ by $\zetaw$
with the action of $\bT$ by $\zeta$, and therefore induces an isomorphism from $H^{\dimfM}_{\fM^+_{\zetaw}}\!(\fM; \Z)$ to
$H^{\dimfM}_{\fM^+_\zeta}(\fM; \Z)$.  This isomorphism is different from the canonical one; in other words,
the automorphism of $K(\cOa^\zeta)\cong H^{\dimfM}_{\fM^+_\zeta}(\fM; \Z)$ induced by $\Psi_w$ is non-trivial.
These automorphisms are compatible with multiplication in the Weyl group, so we obtain a shuffling 
action of $\rWeyl$ on $K(\cOa^\zeta)$.

\begin{remark}\label{regular}
In the case of the Springer resolution, both $W$ and $\rWeyl$ are isomorphic to the ordinary
Weyl group.  Since pure twisting and shuffling functors act trivially on the Grothendieck group
(Remark \ref{K-twisting} and Proposition \ref{K-shuffle})
we obtain both a twisting and a shuffling action of $W$ on $K(\cOa^\zeta)$.  
Furthermore, the two actions commute with each other by Theorem \ref{twisting and shuffling commute},
which we will prove in the next section.
Indeed, what we obtain is isomorphic to the canonical action of $W\times W$ on $\C[W]$,
with one factor acting by left multiplication and the other by inverse right multiplication.
\end{remark}

\subsection{Twisting and shuffling commute}
The purpose of this section is to show that twisting and shuffling functors commute.
We begin with the pure ones.

\begin{lemma} \label{twist-shuffle} Let $\la, \la'\in\Pi$ be parameters with $\la-\la'\in H^2(\fM;\Z)$,
and $\zeta,\zeta'$ two generic cocharacters of $T$.  
Then we have a natural isomorphism of functors\footnote{Here we interpret the first $\Psi^{\zeta'\!,\zeta}$ 
for the quantization with period $\la'$,
and the second $\Psi^{\zeta'\!,\zeta}$ for the quantization with period $\la$.
We do not need to say anything about $\Phi^{\la'\!,\la}$, since this functor
is defined without reference to the choice of cocharacter.  Similar comments apply to the
statements of Lemmas \ref{w-twist-commute}, \ref{another stupid case}, and \ref{last one}.}
$$\Psi^{\zeta'\!,\zeta}\circ
  \Phi^{\la'\!,\la}\cong
  \Phi^{\la'\!,\la}\circ\Psi^{\zeta'\!,\zeta}.$$
\end{lemma}

\begin{proof}
First, we observe that the four functors above are not affected if we replace $\fM$ with some other
conical symplectic resolution $\fM'$ of $\fM_0$ \cite[3.9 \& 6.24]{BLPWquant}.  
By definition, localization holds for {\em some} resolution at every element of $\Pi$, thus we may assume that
localization holds at $\la'$.  By \cite[6.31]{BLPWquant}, 
the functor  $\Phi^{\la'\!,\la}$ can be written as the composition of the localization functor $\LLoc$ at
$\la$, the ``geometric twist" ${}_{\la'}\cT[\hmon]_{\la}\otimes_{\cD_\la}-$, and the derived sections
functor $\Rsecs$ at $\la'$.
Similarly, the functor $\Psi^{\zeta'\!,\zeta}$ (at either $\la$ or $\la'$) can be written as the composition of $\LLoc$,
the ``geometric shuffle" (defined in a way completely
analogous to that of $\Psi^{\zeta'\!,\zeta}$), and $\Rsecs$.
Thus it suffices to show that geometric twists commute with geometric shuffles.
This follows immediately from the fact that ${}_{\la'}\cT[\hmon]_{\la}\otimes_{\cD_\la}-$ and its adjoint commute
with the inclusion of $\cOg$ into $\cD\mmod_h$.
\end{proof}

We next move on to the various impure cases.

\begin{lemma}\label{another stupid case}
Let $\la\in\Pi$ and $w\in W$ be such that $w\cdot \la - \la\in H^2(\fM;\Z)$,
and let $\zeta,\zeta'$ be generic cocharacters of $T$.
Then we have a natural isomorphism of functors
$$\Psi^{\zeta'\!,\zeta}\circ
  \Phi^{\la}_w\cong
  \Phi^{\la}_w\circ\Psi^{\zeta'\!,\zeta}.$$
\end{lemma}

\begin{proof}
It is clear from the definition of $\Phi^{\la}_w$ that it commutes with inclusion functors and their adjoints,
and therefore with $\Psi^{\zeta'\!,\zeta}$.
\end{proof}

\begin{lemma}\label{last one}
Let $\la\in\Pi$ and $w\in W$ be such that $w\cdot \la - \la\in H^2(\fM;\Z)$.
Let $\zeta$ be generic, and let $v\in\rWeyl$ be arbitrary.
Then we have a natural isomorphism of functors
$$\Psi^\zeta_v\circ
  \Phi^{\la}_w\cong
  \Phi^{\la}_w\circ\Psi^{\zeta}_v.$$
\end{lemma}

\begin{proof}
This follows immediately from the fact that the canonical isomorphism $A_{\la} \cong A_{w\cdot \la}$ that was used to define the functor $\Phi^{\la}_w$ is $G$-equivariant.
\end{proof}

\begin{lemma}\label{w-twist-commute}
Let $\la, \la'\in\Pi$ be parameters with $\la-\la'\in H^2(\fM;\Z)$,
and $\zeta$ a generic cocharacter of $T$, and $w\in \rWeyl$.  
Then we have a natural isomorphism of functors
$$\Psi^{\zeta}_w\circ
  \Phi^{\la'\!,\la}\cong
  \Phi^{\la'\!,\la}\circ\Psi^{\zeta}_w.$$
\end{lemma}

\begin{proof}
For any object $N$ of $\cOa^{\zetaw}$ (with period $\la$),
we have 
$$\Psi^{\zeta}_w\circ\Phi^{\la'\!,\la}(N) = \Psi^{\zeta}_w\left(\bi\overset{L}{\otimes} N\right)
\cong\Psi^{\zeta}_w(\bi)\overset{L}{\otimes} \Psi^{\zeta}_w(N).$$
Here, by $\Psi^{\zeta}_w(\bi)$, we mean that we twist both the left {\em and} the right module structures
on $\bi$ by any lift $\bar w$ of $w$ to $N(T)\subset G$.  To prove the lemma,
it suffices to show that $\Psi^\zeta_w(\bi)\cong \bi$. 

By the same argument that we used at the beginning of the proof of Lemma \ref{twist-shuffle},
we may assume that localization holds at $\la'$.  This implies that $\bi\cong \secs(\cbi')$. Thus $\Psi^{\zeta}_w(\bi)$ is the
$\bS$-invariant sections of the pullback sheaf  $\bar{w}^*\cbi'$; however, the pullback of
$\cbi'$ by any group element is again a quantization of the same line
bundle, and thus isomorphic to $\cbi'$.  This completes the proof.
\end{proof}

The four preceding lemmas combine to give us the following theorem.

\begin{theorem}\label{twisting and shuffling commute}
Twisting functors commute with shuffling functors.
\end{theorem}

We end the section with a pair of conjectures, motivated by our study of twisting and shuffling
functors on hypertoric varieties.  Suppose that we have a notion of integral periods (Section \ref{sec:integrality}).
Fixing an integral parameter $\la\in\Pi$ and a generic cocharacter $\zeta$, we consider the
{\bf long twist} $\Phi^{\la,-\la}\circ\Phi^{-\la,\la}$
and the
{\bf long shuffle} $\Psi^{\zeta,-\zeta}\circ\Psi^{-\zeta,\zeta}$.
The first is a pure twisting functor, the second a pure shuffling functor;
in particular, they are both endofunctors of a single category $\dOa$.

\begin{conjecture}\label{serre}
Up to a shift, the long twist is isomorphic to to the right
Serre functor on $\dOa$ and the 
long shuffle is isomorphic to the left
Serre functor.
\end{conjecture}

\begin{remark}\label{O-serre}
  This conjecture is known to hold for BGG category $\cO$ by \cite[4.1]{MS} and for
  hypertoric category $\cO$ by \cite[6.11]{GDKD}.
\end{remark}

\begin{remark}
Conjecture \ref{serre} has recently been proven by Losev \cite[7.4 \& 7.7]{Losshuff}.
\end{remark}

\section{Examples}
\label{sec:examples}
The purpose of this section is to summarize the structures that we have defined so far for all known
classes of examples of conical symplectic resolutions.  Specifically, for each class, we will address the following
(the parenthetical section number indicates the point in this paper at which each of these topics was first discussed):

\renewcommand{\theenumi}{\roman{enumi}}
\renewcommand{\labelenumi}{(\theenumi)}
\begin{enumerate}
% \item the action of $\bS$ (which is not unique, but often has a most natural choice)
\item the group $G$ of Hamiltonian symplectomorphisms that commute with $\bS$, along with its Weyl group $\rWeyl$
(Section \ref{shuffling})
\item the vector space $H^2(\fM;\C)$ (or the full
  cohomology ring) along with the action of the Namikawa Weyl group $W$ (Section \ref{sec:nam})
\item the algebra $A$ of $\bS$-invariant global sections of a quantization (Section \ref{quant-quantizations})
\item the periods at which localization is known to hold (Section \ref{sec:localization})
\item Koszulity of $\cOa$ and $\cOg$ (Section \ref{hwksk})
\item the map $H^*(\fM;\C)\to Z(E)$ to the center of the Yoneda algebra of $\cOg$ (Section \ref{sec:hochsch-cohom-cent})
\item the poset $\all$ of symplectic leaves and the subposet $\spe$ of special leaves
% which we often only understand for integral quantizations 
(Section \ref{sec:simples})
% , and how they align with our cells.
\item the twisting and shuffling functors (Section \ref{sec:twist-shuffl-funct}).
\end{enumerate}

\subsection{Cotangent bundles of partial flag varieties}
\label{sec:flag-varieties}
Let $G$ be a semi-simple complex Lie group and $P\subset G$ a parabolic subgroup.
Let $\fM := T^*(G/P)$, equipped with the inverse scaling action on the fibers.
Up to modification of the $\bS$-action, these are the only known examples of conical symplectic resolutions
that are cotangent bundles.

The $G$-moment map $\fM\to\mg^*\cong\mg$ has as its image the closure of a nilpotent orbit
$O_P$;
the orbits that arise in this way are called {\bf Richardson}. 
%\cite[3.16]{Fu}  
The induced map from $\fM_0$
to the closure of the Richardson orbit is generically finite. 
%\cite[3.17]{Fu}.  
If it is generically one to one, then
$\fM_0$ is isomorphic to the normalization of the orbit closure.
If $G = \SL_r$, then every nilpotent orbit is Richardson, every nilpotent orbit closure is normal,
and the map from $\fM_0$ to the orbit closure is always an isomorphism.

\begin{enumerate}
\item If $P \ne G$, the group of Hamiltonian symplectomorphisms of $\fM$
  commuting with $\bS$ is the adjoint group $G/Z(G)$, and $\rWeyl$ is its Weyl group.
\item The cohomology ring is \[H^*(\fM; \C)\cong
  \C[\mt^{*}]^{\rWeyl_{\! P}}\Big{/}\;\C[\mt^{*}]^{\rWeyl_{\! P}}\cdot
  \C[\mt^{*}]_+^{\rWeyl},\] where $\rWeyl_{\! P}$ is the Weyl
  group of $P/[P,P]$ and $\mt$ is a Cartan subalgebra of $\mg$.    In particular,
  $H^2(\fM; \C)\cong (\mt^*)^{\rWeyl_{\! P}}$.  

We describe the Namikawa Weyl group only in 
the special case where $G = \SL_r$.  
Let $\mu$ be a composition of $r$.  This means that $\mu$ is a function $i \mapsto \mu_i$ from $\Z$ to $\mathbb{N}$ such that $\sum_{i} \mu_i = r$.  
Consider the parabolic subgroup $P = P_\mu\subset \SL_r$
of block-upper-triangular matrices with blocks of size $(\ldots,\mu_{-2},\mu_{-1},\mu_0,\mu_1,\mu_2,\ldots)$, in that order.  Let $\bar\mu$ denote
the partition with the same parts as the composition $\mu$, sorted into nonincreasing order, and let $\bar{\mu}^t$ be its transpose; in other words, 
$\bar{\mu}^t_j$ is the number of parts of $\bar\mu$ or $\mu$ 
that are greater than or equal to $j$.  Then $\fM_0$ is isomorphic to the closure of the nilpotent orbit in $\mg$ with
Jordan type $\mu$. The Namikawa Weyl group $W$ permutes parts of the composition of the same size; more precisely, we have
$$W\cong S_{\bar\mu^t_1-\bar\mu^t_2}\times \cdots \times S_{\bar\mu^t_{r-1}-\bar\mu^t_{r}} \times S_{\bar\mu^t_r}.$$
In particular, if $\bar\mu^t=(r)$ (in which case $P$ is a Borel subgroup), then $W=S_r$.  At the other extreme, if $\bar\mu^t=(1,\dots, 1)$
(in which case $P = G$), then $W$ is trivial.
\item Let $\la\in\Ht\subset\mt^*$, and let $A_\la$ be the invariant section ring of the quantization with period $\la$.
Then $A_\la$ is isomorphic to a quotient of $U(\mg)$ by a primitive
ideal; if $P=B$, then this ideal is generated by elements of the
center. By \cite[4.4]{BLPWquant}, $A_\la$ is also isomorphic to the ring of global $D$-modules on $G/P$,
twisted by $\la + \rho$.
\item By the work of Beilinson and Bernstein \cite{BB},
localization holds if and only if the inner product $\langle\la,\a\rangle$ is not a non-positive
integer for any positive root $\a\in\mt^*$.
\item If the period $\la$ is regular, then the category $\cOg$ is equivalent by Soergel's functor
to a regular infinitesimal block of parabolic BGG category $\cO$ \cite[3.5.1]{BGS96} (see also \cite[Proposition 2]{WebWO}).
 In particular, $\cOg$ is standard Koszul.
\item If the period $\la$ is regular and integral, then the center of $E$ is isomorphic to the center of the Koszul
dual category $\cOg^!$, which is a singular integral block of ordinary (not parabolic) BGG category $\cO$ \cite[1.1]{Back99}.  
The fact that the center of such a block is isomorphic
to the cohomology ring of $\fM$ is a consequence of Soergel's
Stuktursatz and Endomorphismensatz from \cite{Soe90}.  A slightly stronger statement is that 
Conjecture \ref{HH} holds in this case.
\begin{proposition}\label{prop:hochG/P}
The natural map from $H^*(\fM)$ to the center $Z(E)$ is an isomorphism.
\end{proposition}
\begin{proof}
By \cite[4.5]{BLPWquant}, $\Dmod$ is equivalent to the category of D-modules twisted by a line bundle $L$.
In particular, $L$ itself may be regarded as an object of $\Dmod$.  The Ext-algebra
$\Ext^\bullet(L,L)$ is isomorphic to the de Rham cohomology of $H^*(\fM)$,
and the map $H^*(\fM)\to \Ext^\bullet(L,L)$ from Section \ref{sec:hochsch-cohom-cent} realizes this isomorphism.

This implies that the map $H^*(\fM)\to Z(E)$ is
injective. By \cite[5.11]{Bru08}, the dimension of $Z(E)$ is the
same as the number of simple objects in $\cOg$, 
which is equal to $|\fM^\bT| = \dim H^*(\fM^\bT; \C) = \dim H^*(\fM; \C)$.
Thus, this map must be an isomorphism.
\end{proof}

\item 
Let $O_P\subset\mg$ be the Richardson nilpotent orbit with the property that $\fM_0$ is finite over $\overline{O}_P$.
The symplectic leaves of $\fM_0$ are the preimages of the $G$-orbits in $\overline{O}_P$.
For an integral parameter, the special leaves (equivalently, the special orbits) correspond
to those double cells which contain a shortest right coset
representative for the Weyl group of the parabolic.  
These are described in Carter's book \cite[\S 13]{Carter}.
If $\mg = \mathfrak{sl}_r$, then all leaves are special.
For non-integral parameters, the question of which orbits are special is
more complicated; we do not address it here.
\item In the case where $P$ is a Borel subgroup, our twisting functors agree with those
defined by Arkhipov (Remark \ref{Arkhipov}) under the equivalence between $\cOa$ and BGG category $\cO$.  Similarly, our shuffling functors agree with those
defined by Irving (see Proposition \ref{shuf-twist}).  Thus, when $P$ is the Borel we obtain two commuting actions of the generalized braid group $B_W$ on $\cOa$; at the level of the Grothendieck group, these descend to the left and right actions of $W$ on $\C[W]$ (Remark \ref{regular}).

When $P$ is an arbitrary parabolic, Soergel's functor can be used to identify $\cOg$ (and therefore $\cOa$ if localization holds) with an infinitesimal block of parabolic BGG category $\cO$.
Irving's shuffling functors all still make sense in the setting of parabolic BGG category $\cO$,
and they coincide with our shuffling functors.
On the other hand, not all of Arkhipov's twisting functors preserve parabolic BGG category $\cO$; our twisting functors are just those Arkhipov functors that preserve parabolic $\cO$.
Conjecture \ref{serre} holds in this case.
\end{enumerate}

\subsection{S3-varieties}
\label{sec:S3}

  Let $G$ be a simple complex algebraic group and let $e\in \mg$ be a
  nilpotent element; let $h,f$ be elements which satisfy the Chevalley
  relations of $\mathfrak{sl}_2$ together with $e$
and let $\mg=\oplus\mg_k$ be the decomposition of $\mg$ into eigenspaces for $h$.
  Choose a Cartan $\mt\subset
  \mg_0$ and let $T$ be the corresponding connected subgroup.  The space $\mg_{-1}$ has a
  symplectic form defined by  $\langle [-,-],e\rangle$, where $\langle
  -,-\rangle$ denotes the Killing form. We let $\ml\subset \mg_{-1}$ be a
  Lagrangian subspace with respect to this form. 
  Let $$\mm:=\ml\oplus\bigoplus_{k\leq -2}\mg_k$$ and let $M\subset G$ be the
  associated connected algebraic subgroup. 
We have a natural character $\chi:=\langle
  e,-\rangle\colon \mm\to \C$.

Let $P$ be a parabolic subgroup of $G$. Consider the moment map
  $\mu:T^*(G/P)\to \mg^*$, and let
  $\mu_\mm:T^*(G/P)\to \mm^*$ be the moment map obtained by projecting onto $\mm^*$.  
  As explained by Ginzburg and Gan \cite[\S 3.2]{GG}, the group $M$
  acts freely on $\mu_\mm^{-1}(\chi)$, so the
  quotient $\fXeP:=\mu_\mm^{-1}(\chi)/M$ is smooth; it is a symplectic resolution
  of the affine quotient $\fXePo := \Spec\C[T^*(G/P)]^M$.
  
If the Richardson orbit $O_P\subset\mg$ is simply connected or if $G = \operatorname{SL}_r$, then
$\fXePo$ is isomorphic to a transverse slice to the orbit $G\cdot e$ inside of $\overline{O}_P = G\cdot \fp^\perp\subset \mg^*$.  
More generally, it admits a finite map to such a slice.
If $e=0$, then $M$ is trivial, and $\fXeP = T^*(G/P)$, thus these spaces generalize those considered in the previous section. 
There seems to be no fixed name for $\fXeP$ and $\fXePo$ in the
literature; we have adopted the term {\bf S3-varieties}, as they have been studied (independently) by
Slodowy, Spaltenstein, and Springer.

\begin{remark}
A subtlety in the above construction is that $M$ is not reductive, and one usually only considers quotients by reductive groups.  
In this particular case, the fact that everything works as expected must be checked
carefully; this is shown in \cite{GG} using the freeness of the action of $M$ on the preimage of $\chi$.
% Obviously, in general, one should be very cautious about dividing by non-reductive groups.
\end{remark}

The usual inverse scaling action of $\bS$ on the fibers of
$T^*(G/P)$ does not descend to $\fXeP$, since $\mu_\mm$ is $\bS$-equivariant
and $\chi\in\mm^*$ is not $\bS$-invariant.  However, we can choose a
new $\bS$-action ($\rho^{\#}$ in the notation of \cite[\S 4]{GG}) on $T^*(G/P)$ (no longer
conical and not commuting with $G$) that {\em does} descend to a
conical action on $\fXeP$.
The grading induced on $\C[\fXePo]$ by the action of $\bS$ is called
the {\bf Kazhdan grading}; see \cite[\S 4]{GG} for more details.

\begin{enumerate}
\item To avoid confusion with the group $G$, 
let $G_\Ham$ denote the group of Hamiltonian symplectomorphisms of $\fXeP$
that commute with $\bS$.  Then $G_\Ham$ is a quotient of the simultaneous centralizer group
  $C_{G}(e,h,f)$ of the $\mathfrak{sl}_2\subset\mg$ spanned by $e,h$ and $f$.  
  If $P = B$ then $G_\Ham$ is just $C_G(e, h, f)/Z(G)$, and 
  $\rWeyl$ is its Weyl group.  However, it can be smaller in general; for instance if $e \in O_P$, then $\fXeP = \fXePo$ is a point, and $G_\Ham$ is trivial. 

Let us describe the group $G_\Ham$ explicitly when $G=\operatorname{GL}_r$.
(We use $\operatorname{GL}_r$ instead of $\operatorname{SL}_r$ for convenience here, but nothing substantial changes.)
The centralizer $C_{G}(e,h,f)$ is the product 
$\operatorname{GL}_{\gamma_1}\times \cdots \times \operatorname{GL}_{\gamma_r} \subset \operatorname{GL}_r$,
where ${\gamma_i}$ is the number of Jordan blocks of $e$ of size $i$ and the factor $\operatorname{GL}_{\gamma_i}$ is the endomorphism group of the sum of these blocks as $\mathfrak{sl}_2$-representations.  Suppose that the parabolic $P$ is described by 
a composition $\mu$, and the Jordan blocks of $e$ are given by a partition $\nu = (\nu_1, \nu_2, \dots)$ of $r$.  Note that $\fXeP$ is empty unless $\nu \le \bar{\mu}^t$ in the dominance order on partitions, so we will assume from now on that this is the case.

If the partitions $\nu$ and $\bar{\mu}^t$ have a different number of parts, pad the shorter one with zeros so they both have the same length $\ell$. 
Then let $J$ be the set of integers $1 \le j \le \ell$ for which 
\[\sum_{i = 1}^j \nu_i = \sum_{i=1}^j (\bar{\mu}^t)_i.\]
(Note that we always have $\ell \in J$.) 
Each $\gamma_k$ indexes a maximal block $\nu_{j_k} = \nu_{j_k+1} = \dots = \nu_{j_{k+1}-1}$
of equal parts of $\nu$.  Using the condition $\nu \le \bar{\mu}^t$ it is not hard to see
that there are three possibilities for $J \cap [j_k, j_{k+1}-1]$: either (1) it is empty, 
or (2) it consists of $j_{k+1}-1$ only, or (3) it contains all integers $j_k \le j < j_{k+1}$.  Then $G_\Ham$ is the quotient of $C_{G}(e,h,f) = \operatorname{GL}_{\gamma_1}\times \cdots \times \operatorname{GL}_{\gamma_r} \subset \operatorname{GL}_r$ by the group generated by all factors $\operatorname{GL_{\gamma_k}}$ for which (3) holds for $k$, together with all diagonal matrices of the form 
\[\operatorname{diag}(\lambda I_{\gamma_1}, \lambda I_{\gamma_2}, \dots, \lambda I_{\gamma_k}, I_{\gamma_{k+1}}, \dots)\]
where $I_{\gamma_i}$ is the identity matrix, $\lambda \in \C \setminus \{0\}$, and 
$k$ labels a block of type (2) or (3).

%$\gamma=(\gamma_1\geq \gamma_2\geq \cdots)$ is the partition
%transpose to the Jordan decomposition of $e$.  
Thus, the Weyl group
$\rWeyl$ in the $\operatorname{SL}_r/\operatorname{GL}_r$ case is the product of $S_{\gamma_k}$ for all $k$ not of type (3).

For example, if $\nu = (4, 4, 2, 2)$ then $C_G(e, h, f) \cong \operatorname{GL}_2 \times \operatorname{GL_2}$.  If $\bar{\mu}^t = (5, 4, 3)$ then we have $G_\Ham \cong (\operatorname{GL}_2 \times \operatorname{GL_2})/\C^*$.  If
$\bar{\mu}^t = (5, 3, 3, 1)$ then $G_\Ham \cong \operatorname{PGL}_2 \times \operatorname{PGL_2}$, and if $\bar{\mu}^t = (5, 3, 2, 2)$
then  $G_\Ham \cong \operatorname{PGL}_2$.

A maximal torus in $G_\Ham$ is somewhat easier to describe.  
For all roots $\a\in\mg^*\cong\mg$, let $e_\a$ be the orthogonal projection of $e$ onto the weight space $\mg_\a$.
  Then the Lie algebra of a maximal torus of $C_{G}(e,h,f)$ is \[\mt^e:=\mt\cap \mathfrak{c}_{\mg}(e,h,f)=\{t\in \mt \,|\, \a(t)=0 \text{
  for all $\a$ such that $e_\a\neq 0$} \}.\] 
If $G = \operatorname{GL}_n$, then the Lie algebra of a maximal torus in $G_\Ham$ is 
the quotient of $\mt^e$ by the span of all $\operatorname{diag}(I_m, 0)$ where $m = \sum_{i=1}^j \nu_i$ for some $j \in J$.
\item 
Since $\mm$ is nilpotent, $M$ is contractible, thus the $M$-equivariant cohomology of $T^*(G/P)$ coincides with the ordinary
cohomology.  This allows us to consider 
the Kirwan map 
\begin{equation}\label{kirwan map}
(\mt^*)^{\rWeyl_P} \cong H^2(T^*(G/P); \C) \cong H^2_{\! M}(T^*(G/P); \C) \to H^2(\fXeP; \C),
\end{equation}
which is always injective if $\fXeP$ is non-empty and positive dimensional.
In type A, the Kirwan map is also surjective by \cite[1.1]{BrO} (Brundan and Ostrik also give a presentation of the full cohomology ring of $\fXeP$ in type A).  The kernel of the map \eqref{kirwan map} is invariant under the action of
the Namikawa  Weyl group of $T^*(G/P)$ (see part (ii) of Section \ref{sec:flag-varieties}), so it induces an action on $H^2(\fXeP; \C)$, and 
the Namikawa Weyl group of $\fXeP$ is the quotient by the elements which 
act trivially there.

\item 
Let $A_\la$ be the $\bS$-invariant section algebra for the quantization of $\fXeP$ 
whose period is equal to the image of $\la\in H^2(G/P; \C)$.
This algebra has been studied in \cite{WebWO}, and it is isomorphic to a
 quotient of the usual $W$-algebra for the element $e$ by an explicit ideal. 
When $\la$ is in the image of the map
$H^2(G/P;\C)\to H^2(\fXeP;\C)$, the quantization of $\fXeP$ can be obtained from the quantization of $T^*(G/P)$ by quantum Hamiltonian reduction.
However, except in type A, this map may not be not surjective, and the quantizations which don't arise this way are more difficult to understand.
\item The question of when localization holds has a simple answer for quantizations obtained by Hamiltonian reduction from $T^*(G/P)$.
Choose a Borel subgroup $B$ such that $T\subset B\subset P\subset G$.
%and let $B_-$ be the unique Borel such that $B\cap B_-=T$.  
Consider an element $\la\in (\mt^*)^{\rWeyl_P}$, which includes into $H^2(\fXeP; \C)$
via Equation \ref{kirwan map}.
Let $\Delta_+(\mathfrak{p})\subset \mt^*$ be the set of positive roots $\a$ such that 
$\mg_{-\a}\not\subset\mathfrak{p}$.
The argument of \cite[5.1.2]{Gin08} is easily generalized to show that localization holds at $\la$ whenever 
$\langle\la, \a\rangle\notin \mathbb{Z}_{\leq 0}$ 
for all $\a\in\Delta_+(\mathfrak{p})$.

\item 
Let $L\subset G$ be a Levi subgroup such that $e$ is regular in $\mathfrak{l} = \operatorname{Lie}(L)$,
and let $\zeta$ be a cocharacter of $T$ commuting with $L$. 
If we choose $\zeta$ generically,
the sum of its nonnegative weight spaces will be the Lie algebra of a parabolic $R$ with Levi $L$. 
For example, if $G=\SL_r$, then the parabolic $R$ is the subgroup of
block diagonal matrices for some composition $\nu$ of $r$ with the same Jordan type as $e$. 
Since the action of $\bT$ by $\zeta$ fixes $e$, it descends to $\fXeP$.  

If $\la\in (\mt^*)^{\rWeyl_P}$ is dominant and integral, then it is shown in \cite{WebWO} that $\cOg\simeq\cOa$ is equivalent to the infinitesimal block of parabolic BGG category $\cO$ with parabolic $\fp$ and central character $\xi$, where $\xi$ is a central character of $U(\mg)$ corresponding to an integral highest weight whose stabilizer for the $\rho$-shifted action of $W$ on $\mt^*$ is $W_L$.
It follows that $\cOg$ is standard Koszul.
  
\item
The center of the Yoneda algebra of $\cOg$ is isomorphic to the center of the Koszul dual of $\cOg$.
Let $\rho_L\in\mt^*$ be the half the sum of the positive roots not in $\mathfrak{l}\subset\mg\cong\mg^*$.
If $\la + \rho_L$ is integral, then the parabolic-singular duality of Beilinson-Ginzburg-Soergel \cite{BGS96,Back99}
tells us that the Koszul dual of $\cOg$
is also a singular block of parabolic BGG category $\cO$.  (In this duality, the roles of the parabolic and the central
character are exchanged.  The larger the parabolic on one side, the more singular the character on the other side.)

In type A, the centers of these blocks were first computed by Brundan \cite{Bru06}; they were
shown to be isomorphic to the cohomology of $\fXeP$ independently in \cite[9.9]{kosdef} and \cite[1.1]{BrO}.
In \cite[3.5]{S3}, it is shown that the specific map $\gamma\colon H^*(\fXeP; \C)\to Z(E)$  of Conjecture \ref{HH} is an isomorphism.

\item The variety $\fXePo$ admits a finite map to the Slodowy slice $\cS$ to $e$ in $\mg$.
The symplectic leaves of $\fXePo$
are the preimages of the symplectic leaves of $\cS\cap\operatorname{nil}(\mg)$,
which are in turn the intersections of $\cS$ with the symplectic leaves of $\operatorname{nil}(\mg)$.
We conjecture that the same statement is true of special symplectic leaves.
The fact that the special symplectic leaves of $\cS\cap\operatorname{nil}(\mg)$ are 
the intersections of $\cS$ with the special symplectic leaves of $\operatorname{nil}(\mg)$ is proven by Losev \cite[1.2.2]{LosW}.

If $\mg = \mathfrak{sl}_r$ and the period is integral, then the above conjecture says that all leaves of $\fXePo$
are special.  This is true; it is known for $\overline{O}_P$ (see part (vii) of the previous section), and we 
may obtain the result for $\fXePo$ by applying Losev's operation $(\cdot)_\dagger$ to the relevant primitive ideals in $U(\mg)$.

\item As in \cite[2.8]{KR}, one can construct a Hamiltonian reduction
  functor sending modules over a quantization of $T^*(G/P)$ to modules
  over the corresponding quantization of $\fXeP$, and we show that all
  twisting functors between quantizations of $T^*(G/P)$ descend to
  twisting functors between the reduced quantizations of $\fXeP$.
  However, if the map $H^2(G/P;\C)\to H^2(\fXeP;\C)$ is not
  surjective, not every quantization of $\fXeP$ is a Hamiltonian
  reduction, so we cannot understand all twisting functors for $\fXeP$
  in terms of those on $T^*(G/P)$.
If this map is surjective, then we have $$E_{\operatorname{tw}}\cong (\mt^*)^{\rWeyl_P}\setminus
\{\la \mid \langle
\a^\vee ,\la\rangle=0\;\text{for some coroot $\a$ not orthogonal to $
  (\mt^*)^{\rWeyl_P}$}\}$$ and every twisting functor is obtained by
reduction from $T^*(G/P)$.  In particular, this holds if $G= \SL_r$.

The parameter space $E_{\operatorname{sh}}$ for shuffling functors is the complement of a hyperplane arrangement 
in the Lie algebra $\mt^{}_\Ham$ of a maximal torus in $G_\Ham$.  As explained in item (i) above, this
is a quotient of $\mt^e$.
If $G = \SL_r$, the dimension of this quotient is $m-|J|$, where $m$ is the number of 
Jordan blocks of $e$.  The root hyperplanes $\langle \a,t\rangle=0$
restrict to $\mt^e$, and the hyperplanes in $\mt_\Ham$ are exactly the 
projections of the ones which contain the kernel of the
projection $\mt^e \to \mt_\Ham$.  These are given by the usual equations
$\{a_i=a_j\}$, and such a hyperplane will appear if and only if
$i$ and $j$ belong to the same ``$J$-block''.

The arrangement is thus a product of type A hyperplane arrangements.  
However, as we have seen, 
the Weyl group $\rWeyl$ which acts on it may not 
the full permutation group associated to this arrangement, 
but instead is the subgroup of 
elements that permute Jordan blocks of the same size, except that a group of blocks of type (3) is not permuted.  If $P=B$, there are no type (3) blocks,
and so we have
 \[\rWeyl \cong S_{\bar\mu^t_1-\bar\mu_2^t}\times S_{\bar\mu^t_2-\bar\mu_3^t}\times \cdots \times
S_{\bar\mu_{\ell-1}^t-\bar \mu_\ell^t} \times S_{\bar\mu^t_\ell}.\]  Note that this coincides with the Namikawa Weyl
group for $T^*(\SL_r/P_\mu)$.

Note that S3-spaces for $SL_r$ are isomorphic to type A quiver
varieties.  We give a description of shuffling functors for all type A quiver varieties in Section \ref{sec:quiver-varieties} below.
\end{enumerate}

\subsection{Hypertoric varieties}
\label{sec:hypertoric-varieties}
Let $V$ be a symplectic vector space equipped with a linear symplectic action of a torus $K$.  The Hamiltonian reduction
$\fM$ of $V$ by $K$ is called a {\bf hypertoric variety}.  If we choose a generic character for $K$ as our GIT
parameter, the reduction is an orbifold; it will be smooth if and only if the matrix determined by the inclusion
of $K$ into a maximal torus of $\operatorname{Sp}(V)$ is unimodular \cite[3.2 \& 3.3]{BD}.  When smooth,
it is a conical symplectic resolution of the affine quotient $\fM_0$, where the action of $\bS$ is induced by the inverse scalar
action on $V$.

Let $$V = \bigoplus_{\chi\in\mk^*_\Z}V_\chi$$ be the decomposition of $V$ into weight spaces for $K$.
For simplicity, we assume that $V_0 = 0$.  (This assumption is harmless; the variety
$\fM$ is isomorphic to $V_0\times\fM'$, where $\fM'$ is built using the $K$ action on $V/V_0$.)
Choose an element $\xi\in\mk$ that is nonzero on every $\chi$ such that $V_\chi \neq 0$, and let
$$\Delta^+ := \{\chi\mid V_\chi\neq 0\;\;\text{and}\;\; (\chi,\xi)>0\}.$$

We make the additional assumption that, if $\dim V_\chi = 1$, then $\chi$ is in the span of $\Delta^+\smallsetminus\{\chi\}$.
This has the effect of ruling out certain redundancies; for example, it rules out the case
where $\dim V = 2$ and $K = \operatorname{Sp}(V)$, in which case $\fM$ would be a point.
In particular, every hypertoric variety can be constructed using a $V$ and a $K$ that satisfy this condition.
It will be a convenient assumption to have for part (ii) below, as well as for our discussion of symplectic duality
in the next section.

\begin{enumerate}
\item  The group $G$ of Hamiltonian symplectomorphisms of $\fM$ commuting with $\bS$ is 
isomorphic to $$\left(\prod_{\chi\in\Delta^+}\operatorname{GL}(V_\chi)\right)\Big{/} K.$$
Its Weyl group is a product of symmetric groups:
$$\rWeyl \cong \prod_{\chi\in\Delta^+}S_{\dim V_{\chi}}.$$
If we refine the decomposition $\oplus_{\chi\in\Delta^+}V_\chi$ to a decomposition into lines,
then the group $\tilde T \cong (\cs)^{\frac 1 2 \dim V}$ of automorphisms of this decomposition descends to a maximal torus
$T := \tilde T / K$ of $G$.  The natural basis for the cocharacter lattice of $\tilde T$ descends to a finite multiset of cocharacters
of $T$, which in turn define a weighted rational central multiarrangement $\mathcal{A}$ of hyperplanes in $\mt^*$. This hyperplane arrangement together with a character of $K$ gives 
the more usual combinatorial input data for constructing $\fM$.

\item The cohomology ring of $\fM$ was computed independently in \cite{Ko} and \cite{HS02}.
In degree 2, the Kirwan map
$$\mk^*\cong H^2_K(V; \C) \to H^2(\fM; \C)$$ is an isomorphism.  (Surjectivity was proven by Konno,
and injectivity is equivalent to our second assumption above.)

The Namikawa Weyl group is also isomorphic to a product of symmetric groups:
$$W \cong \prod_F S_{|F|},$$
where the product ranges over all rank 1 flats of $\mathcal{A}$.
(One may regard the set of rank 1 flats as the set underlying the multiset $\mathcal{A}$;
for an element $F$ of this set, $|F|$ is its multiplicity in $\mathcal{A}$.  Thus, if $\mathcal{A}$ contains $r$ copies of the same
hyperplane, we get a factor of $S_r$ in $W$.)  This group acts naturally on $\tilde\mt^*\cong\prod_F\C^F$ by permuting each summand.
It fixes $\mt^*$, and thus descends to an action on $\mk^*$ via the exact sequence
$$0\to\mt^*\to\tilde\mt^*\to\mk^*\to 0.$$

\item Let $\mathbb{D}$ be the Weyl algebra of the symplectic vector space $V$.
The invariant algebra $\mathbb{D}^K$ is called the {\bf hypertoric enveloping algebra} in \cite{BLPWtorico};
it was originally studied by Musson and Van den Bergh \cite{MVdB}.  Its center is isomorphic
to $\Sym\mk\cong\C[\mk^*]$, and the $\bS$-invariant section ring of the quantization of $\fM$ with period $\la\in H^2(\fM; \C)\cong\mk^*$
is isomorphic to the corresponding central quotient of $\mathbb{D}^K$ \cite[5.9]{BLPWtorico}.

\item A sufficient condition for localization to hold is given by Bellamy and Kuwabara \cite[5.8]{BeKu}.
Using their results, we give a different combinatorial condition in \cite[6.1]{BLPWtorico} that amounts
to checking that certain rational polyhedra contain lattice points.

\item  It is shown in \cite{BLPWtorico} that the arrangement $\cHt$ is the discriminantal arrangement of $\mathcal{A}$.
For any period that is sufficiently far away from the walls of $\cHt$
(the word for this in \cite{BLPWtorico} is {\bf regular}),
the category $\cOa$ is equivalent to the module categories over a finite dimensional
algebra introduced in \cite{GDKD}.  In particular, it is standard Koszul \cite[5.24]{GDKD}.

There is a unique notion of integrality for periods which satisfies the conditions of Section \ref{sec:integrality}; it is the same as the definition of integrality in \cite{GDKD}.
The precise recipe given there to associate an algebra to a particular quantization is straightforward
if the period $\la$ is integral, but in general it is somewhat tricky. Details are given in \cite[4.9]{BLPWtorico}.
The category $\cOg$ is always standard Koszul, because we can always twist by a line bundle to get to a regular period where localization holds.

\item When $\la$ is regular and integral, the fact that
$H^*(\fM; \C)$ is isomorphic to $Z(E)$ is proven in \cite[5.3]{BLPWtorico}, but we still need to show that the
homomorphism of Conjecture \ref{HH} is an isomorphism.

For every $\a\in\cI$, the simple module $\Lambda_\a$ in $\cOg$
is a quantization of the structure
sheaf of the relative core component $X_\a\subset\fM^+$.  Computing the Ext-algebra of $\Lambda_\a$
using a \v{C}ech spectral sequence, we see
that there is an isomorphism $H^*(X_\a; \C)\cong \Ext^\bullet(\Lambda_\a, \Lambda_\a)$ making the diagram
\[\tikz[->,very thick]{
\matrix[row sep=10mm,column sep=10mm,ampersand replacement=\&]{
\node (a) {$H^*(\fM)$}; \& \node (c) {$H^*(X_\a; \C)$};\\
\node (b) {$H\! H^*(\cD)$}; \& \node (d) {$\Ext^\bullet(\Lambda_\a, \Lambda_\a)$};\\
};
\draw (a) -- (b);
\draw (a) -- (c);
\draw (b) --(d);
\draw (c)--(d);
}\]
commute.  
Thus, the kernel of the map $H^*(\fM;\C)\to Z(E)$ is contained in the
intersection of the kernels of the maps $H^*(\fM;\C)\to H^*(X_\a; \C)$
for all $\a$.  It is shown in \cite[(34)]{HS02} that this intersection is
trivial, so the kernel of the $H^*(\fM;\C)\to Z(E)$ is trivial.  Since
the target and source have the same dimension, it must be an isomorphism. 
\item The poset $\all$ of symplectic leaves of $\fM_0$ is isomorphic to the poset of coloop-free flats of $\mathcal{A}$ \cite[2.3]{PW07}.
If the period of the quantization is integral, then all leaves are special.  This follows
from the reformulation \cite[7.4]{BLPWtorico} of work of Musson and
Van den Bergh \cite{MVdB}.  For non-integral weights, only some
leaves remain special; which ones remain can be deduced the
description of primitive ideals in $A$ given in \cite{MVdB}.

\item Twisting and shuffling functors for hypertoric varieties were studied in detail in \cite[\S 8]{BLPWtorico}.
In particular, Conjecture \ref{serre} is true \cite[8.19]{BLPWtorico}.
\end{enumerate}

\subsection{Hilbert schemes on ALE spaces}
\label{sec:hilbert-schemes-ale}

For any finite subgroup $\mck\subset \SL_2$, consider
the associated DuVal singularity $\C^2/\mck$, along with its unique
crepant resolution $\tCg$.  This is a conical symplectic resolution with respect to
the $\bS$-action induced by inverse scalar
multiplication on $\C^2$.  More generally, for any $r\in\mathbb{N}$,
the Hilbert scheme $\Hilb_r(\tCg)$ is a conical symplectic resolution of 
$\Sym^r(\C^2/\mck)\cong \C^{2r}/(\mck\wr S_r)$ \cite[Cor. 4]{Wangwreath}.

\begin{enumerate}
\item The group $G$ of Hamiltonian symplectomorphisms of $\Hilb_r(\tCg)$ commuting with
  $\bS$ is simply the group of linear symplectomorphisms of $\C^{2r}$ commuting with the
  action of $\mck\wr S_r$.  If $\mck=\Z/\ell\Z$, then
  $G\cong \C^*$, and if $\mck$ is of type $D$ or $E$, then $G$ is trivial.  In particular,
  we can find a Hamiltonian $\bT$-action with isolated fixed points if and only if $\mck=\Z/\ell\Z$.
  The Weyl group $\rWeyl$ is always trivial.
\item The Namikawa Weyl group $W$ is $W_{\mathfrak{G}}\times
\Z/2\Z$, where  $W_{\mathfrak{G}}$ is the Weyl group of $\mathfrak{G}$
if $r>1$.  The codimension two stratum corresponding to the factor $W_{\mathfrak{G}}$ is the set of points where $0$ lies in the support of the ideal, and the stratum corresponding to the factor $\Z/2\Z$ is given by points where the ideal has a point of multiplicity two. It follows from work of Nakajima \cite{Nak-book} that for $r>1$,
  we have
  $H^2(\Hilb_r(\tCg);\Z)\cong H^2(\tCg;\Z)\oplus \Z\delta$. The action
  of the Namikawa Weyl group on this space is via the action of
  $W_{\mathfrak{G}} $ on $H^2(\tCg;\Z)$ and $\Z/2\Z$ on $\Z\delta$ by
  negation.  Thus, $H^2(\tCg;\Z)$ is isomorphic as a $\Z W$-module to the
  root lattice of the finite dimensional simple Lie algebra
  $\mathfrak{G}$ associated to $\mck$ via the McKay correspondence,
  and $H^2(\Hilb_r(\tCg);\Z)$ to the root lattice of its affinization.
  
\item The algebra $A$ is isomorphic to a 
spherical symplectic reflection algebra for the representation of
$\mck\wr S_r$ on $\C^{2r}$ \cite{EGGO,Gorremark}.  
% We use $H$ to
% denote the full SRA, and $e$ its spherical idempotent.
\item Which periods localization holds for is still not completely understood.
For the case $\mck=\{1\}$, the
  answer is quite simple: localization holds at all parameters not of
  the form $\nicefrac {-1}2- \nicefrac mk$ for $m\leq 0$, $1<k\leq r$
  and $(m,k)=1$.
  For general $\mck$, this is a much more complex question, though
  some progress has been made in work of McGerty-Nevins \cite{MN2} and Jenkins
  \cite[\S 6-7]{Jenkins}. 
\end{enumerate}
We only have a category $\cO$ when $\mck=\Z/\ell\Z$, since
  in the other cases, there is only the trivial $\bT$-action.  From
  now on, we will only consider this case.  

\begin{enumerate}
\setcounter{enumi}{4}
\item The
  category $\cOa$ is closely related to the category $\cO$ defined by
  \cite{GGOR} for the Cherednik algebra of the complex reflection
  group $\Z/\ell\Z\wr S_r$.  The category $\cOa$ is the image of the
  GGOR category $\cO$ under the functor $M\mapsto eM$, where $e$ is the spherical
  idempotent in the full symplectic reflection algebra.
If the period lies in the set $\mathfrak{U}$ (these periods are called {\bf spherical}),
then this functor is an equivalence.

  The Koszulity of $\cOa$ at spherical integral parameters and
  thus of $\cOg$ for arbitrary integral parameters is proven by Chuang and Miyachi \cite{CM}.  This was extended to all other
  choices of spherical parameters by Rouquier, Shan, Vasserot, and
  Varagnolo \cite{RSVV,SVV}.
\item It is shown in \cite[3.5]{S3} that the map $H^*(\fM;\C)\to Z(E)$ is an isomorphism.
\item The special leaves for $\mck=\{1\}$ are described by Losev
  \cite[5.8.1]{LosSRA}. 
For $\ell>1$, the affinization of $\Hilb^r(\tCg)$ is
$\Sym^r(\C^2/\mck)$. The leaves of this variety are in bijection with
partitions $\nu$ of integers $r'\leq r$, where the parts of the partition are
the multiplicities of the points in $(\C^2\setminus \{0\})/\mck$ that
occur (and thus $\{0\}$ necessarily has multiplicity $r-r'$). It
follows from work of Shan and Vasserot that 
the special leaves in the integral case are those where the
partition is
$\nu=(1^{r'})$; more generally, when $k=\nicefrac{m}{e}$ with
$(m,e)=1$, it follows that the special leaves are those where all
parts of $\nu$ are $e$ or $1$ and $e\mid (r-r')$.  Note that this is
neither a subset nor a superset of the special leaves in the integral case.
\item Since $G = \cs$, we have $E_{\operatorname{sh}}\cong \cs$.  Thus shuffling functors
are expected to give an action of $\Z\cong \pi_1(E_{\operatorname{sh}})$, 
which we expect agrees with the powers of the Serre functor. Twisting functors are more interesting; 
even though the space $H^2(\Hilb^r(\tCg); \R)$ is independent of $r$, the hyperplane arrangement $\cHt$ is not.  
The hyperplanes are described by Gordon \cite[\S 4.3]{Gorqui}; the extra hyperplanes for $r>1$ reflect the fact
that $\Hilb^r(\tCg)$ is not the only conical symplectic resolution of $\Sym^r(\C^2/\mck)$.
If $\ell = 1$, then the twisting functors simply consist of the action of $\Z$ by powers of the Serre functor.
If $\ell > 1$, we defer to the next section on quiver varieties.
\end{enumerate}

\subsection{Quiver varieties}
\label{sec:quiver-varieties}
Let $\quiv$ be a finite quiver without oriented cycles.  Let $V$ be the set of vertices of $\quiv$,
and let $\Bw,\Bv\in\mathbb{N}^V$ be dimension vectors.  Whenever we
have two weights $\mu\leq \nu$ for the Kac-Moody algebra $\mg_\quiv$ associated to
$\quiv$ with $\nu$ dominant and $\mu\leq \nu$ in the usual root order, there are associated dimension vectors $\Bw$ and $\Bv$ given by
$\Bw_i:=\a_i^\vee(\nu)$ and $\sum \Bv_i\a_i:= \nu-\mu$.  
Our assumptions assure that these numbers are in $\mathbb{N}$.

The quiver variety $\tilde \fQ^\nu_\mu$ is a smooth open subvariety of the cotangent bundle to the moduli stack of framed
representations of $\quiv$, where $\Bw$ is the dimension of the framing and $\Bv$ is the dimension of the representation
\cite{Nak94, Nak98}; its affinization is denoted $\fQ^\nu_\mu$.  Like hypertoric varieties, a quiver variety may be described
as a Hamiltonian reduction of a symplectic vector space by the group
$$\operatorname{GL}_{\Bv} \;\; :=\;\; \prod_{i\in V}\operatorname{GL}_{\Bv_i}.$$ 
The smooth variety $\tilde \fQ^\nu_\mu$ is obtained by using a nontrivial
GIT parameter specified by Nakajima, while the affine variety $\fQ^\nu_\mu$ is obtained as the affine quotient.

There are various actions of $\bS$ that we could choose with respect to which
$\tilde \fQ^\nu_\mu$ is a conical projective resolution of $\fQ^\nu_\mu$.  For example, we could mimic the choice
that we made for hypertoric varieties and take the $\bS$-action induced by
the inverse scaling action on the vector space; this has $n=2$.  Alternatively, the orientation of $\quiv$ determines
a Lagrangian subspace of our symplectic vector space, and the $\bS$-action such that 
this subspace has weight $-1$ and its complement has weight $0$ induces an $\bS$-action on $\tilde \fQ^\nu_\mu$ with $n=1$.  The fact
that $\tilde \fQ^\nu_\mu$ is conical with respect to this action follows from the assumption that $\quiv$ has no oriented
cycles.  This is the $\bS$-action that we use below.

\begin{remark}
The class of type A S3-varieties coincides with the class of quiver varieties for which $\quiv$ is 
a type A Dynkin diagram with some choice of orientation \cite{Maf}.  However, the $\bS$-action that we used in Section \ref{sec:S3}
had $n=2$, whereas here we are using an action with $n=1$.
\end{remark}

\begin{remark}
If $\quiv$ is an affine Dynkin graph with some choice of orientation and $\nu$
is the highest weight of the basic representation of $\mg_\quiv$,
then affine quiver variety $\fQ^\nu_\mu$ is isomorphic to $\Sym^r(\C^2/\mck)$,
%Note that these papers only discuss the case when $\mu=\la-n\delta$. However, by work of Maffei
%\cite{MaffWeyl}, the affine quiver variety $\fQ^\la_\mu$ only depends
%on the orbit of $\mu$ in the Weyl group.  In the basic representation,
%every non-zero weight space is in the orbit of $\la-n\delta$ for some
%$n$, so every nonempty quiver variety in this case will be a
%resolution of $\fQ^\la_{\la-n\delta}$ and thus the section algebra
%coincides with the corresponding spherical symplectic reflection
%algebra for $S_n\wr H$ 
where $\mck\subset \SL_2$ is the finite
subgroup corresponding to $\quiv$ under the McKay correspondence.
Thus, the class of varieties discussed in Section \ref{sec:hilbert-schemes-ale} is a subset
of the class of affine type quiver varieties.
Once again, we used an $\bS$-action with $n=2$ in that section, whereas here we are using one with $n=1$.
\end{remark}

\begin{enumerate}
\item
Consider the group
$$G_\Bw\;\;\;\; := \;\;\prod_{\text{$i\in V$}}\operatorname{GL}_{\Bw_i}
\;\;\times\prod_{\text{$(i, j)\in V\times V$}}\operatorname{GL}_{n_{ij}},$$
where $n_{ij}$ is the number of edges from $i$ to $j$.  In the Hamiltonian reduction construction described above,
this is precisely the group of those automorphisms of the orientation-determined
Lagrangian subspace that commute with the action of $\operatorname{GL}_{\Bv}$.
The group $G$ of Hamiltonian symplectomorphisms of $\tilde \fQ^\nu_\mu$ commuting with $\bS$
is isomorphic to the quotient of $G_\Bw$ by its center:
$$G\;\cong\; G_\Bw / Z(G_\Bw) \;\cong\; G_\Bw / (\cs)^{V}.$$
The group $(\cs)^{V}$ embeds into $G_\Bw$ using the
coboundary formula, with $(z_i)_{i\in V}$ landing on $z_i$ times the identity matrix in the factor $\operatorname{GL}_{\Bw_i}$ and
  $z_i z_j^{-1}$ times the identity matrix in the factor $\operatorname{GL}_{n_{ij}}$.

Note that in the special case where $\quiv$ is a tree with some choice of orientation, 
we have $$G \;\cong\; \operatorname{\operatorname{PGL}}_{\Bw} := \left(\prod \operatorname{GL}_{\Bw_i}\right) \Big{/} \cs.$$
Another special case that will be of interest to us is where $\quiv$ is an $r$-cycle with some choice of orientation,
in which case
$$G \;\cong\; \big(\operatorname{\operatorname{PGL}}_{\Bw} \times \;\cs\big) \;\big{/}\; (\Z/r\Z).$$
In both of these two special cases, the Weyl group is isomorphic to the same product of symmetric groups:
$$\rWeyl \;\;\cong\;\; \prod_{i\in V} S_{\Bw_i}.$$
\item The cohomology ring of a quiver variety is poorly
understood; in particular, surjectivity of the Kirwan map
$$H_{\operatorname{GL}_{\Bv}}^*\!(pt; \C) \to H^*(\tilde\fQ^\nu_\mu; \C)$$ is an important and long-standing conjecture.
% For convenience, we will assume here the Kirwan map is surjective in degree 2.
The rank of the kernel of the Kirwan map in degree 2 is equal to the codimension in $\R^{V}$
of the affine span of the face of the weight polytope
of the representation $V_\nu$ which contains $\mu$.
In particular, for $\mu$ in the interior of the weight polytope, the Kirwan map is injective in degree 2.

For a fixed $\nu,\mu$,
consider the set $S$ of simple roots $\a_i$ with $\mu+\a_i\leq \nu$
and $\a_i^\vee(\mu)=0$.  In finite type, the Namikawa Weyl
group $W$ for $\fQ^\nu_\mu$ is the subgroup of the Weyl group of the Kac-Moody algebra
$\mg_{\quiv}$ generated by $\a_i$ in $S$.  We can show this by noting
that in finite type, \cite[3.27]{Nak98} shows that
codimension 2 strata of $\fQ^\nu_\mu$ are in bijection with the
connected components of the Dynkin subdiagram with vertices given by
$S$: the strata correspond to the weights $\mu+\a_D$ for
$\a_D$ the highest root of a connected component.  Each such stratum
contributes a copy of the Weyl group of the subdiagram. 

In infinite type, this group sits inside the Namikawa Weyl group, but
it may be a proper subgroup, as the case of $\Sym^r(\C^2/\mck)$ shows.

\item Except in special cases in which quiver varieties coincide with other known classes of varieties,
such as S3-varieties, hypertoric varieties, or Hilbert schemes on ALE spaces,
 the $\bS$-invariant section
algebra $A$ has not been studied.
% The cases of S3-varieties and hypertoric varieties are discussed above.

\item While the question of when localization holds is
  interesting, we know of no progress outside the cases of
  finite and affine type A quivers 
  (discussed elsewhere in this paper), other than
  the general results of \cite{BLPWquant, MN, MN2}.

\item The categories $\cOa$ and $\cOg$ are studied by the fourth author in
\cite{Webqui}, with relatively explicit descriptions in the finite
and affine cases using steadied quotients of weighted KLR algebras.
In finite or affine type A, the resulting category $\cOg$ is Koszul at
integral parameters.  In
the finite case, this follows from coincidence with blocks of
parabolic category $\cO$ for $\mathfrak{sl}_m$; in the affine case,
this is shown in \cite{Webqui} based on the Koszul duality
results in \cite{SVV}.    

\item It is not clear whether the map $H^*(\tilde\fQ^\nu_\mu; \C) \to Z(E)$ 
is injective or surjective even for integral periods;
resolving this question is closely tied to the question of Kirwan
surjectivity for quiver varieties.

\item 
If $\quiv$ is a finite type ADE Dynkin diagram with some choice of orientation, 
then the symplectic leaves of $\fQ^\nu_\mu$ are in
bijection with dominant weights $\nu'$ such that $\nu'\leq \nu$ and
$\nu'\geq w\cdot \mu$ for all $w$ in the Weyl group of
$\mg_\quiv$. In the integral case, all leaves are special \cite[5.4]{Webqui}. 
In affine type, the poset of leaves becomes more complicated and there
are non-special leaves; the poset of special
leaves for an integral period in affine type A is described in
\cite[5.10]{Webqui}.  
  
\item Assume that Kirwan surjectivity holds in degree 2.  The twisting
  functors for a quiver variety give an action of a subgroup of the Artin braid group of the
  corresponding root system; we can obtain an action of the whole
  braid group if we allow functors between different quantizations.
  These functors can also be constructed from the categorified quantum
group which acts on these categories by the main theorem of
\cite{Webcatq}: the twisting action is given by the
Chuang-Rouquier braid complexes, as shown by Bezrukavnikov and Losev \cite{BLet}.

In the finite and affine cases, shuffling functors also have an
algebraic description, given in \cite{Webqui}.  
In cases other than affine type A, they
correspond to braiding functors from \cite{Webmerged}, while in affine type A they correspond to change-of-charge functors from \cite{WebRou}.
\end{enumerate}

\subsection{Affine Grassmannian slices}
\label{sec:affine-grassm-slic}

Let $G$ be a semi-simple algebraic group $G$ over $\C$, 
and let $G\pptpp$ be the group of $\C\pptpp$-points of $G$.
This has a ``complementary'' pair of subgroups,
$G[[t]]$  and $G_1[t^{-1}]$, where $G_1[t^{-1}]$ is equal to
the kernel of the evaluation map $G[t^{-1}]\to G$. There is a
natural Poisson structure on the affine Grassmannian $\Gr:=
G\pptpp/G[[t]]$ whose Poisson leaves are the intersections of the
orbits of these two subgroups.

For any cocharacter $\la\colon \mathbb{G}_m\to G$, we obtain a point
$t^\la\in G\pptpp$, which descends to an element $[t^\la]\in\Gr$.  
For any pair of dominant coweights $\la$ and $\mu$, we can consider
the intersection\footnote{Following the notational convention in \cite{KWWY}, the $\bar \la$ on the left-hand side reflects the fact
that we have taken the closure of $G[[t]]\cdot[t^\la]$ on the right-hand side.}
 \[\Gr^{\bar \la}_\mu\;\;:=\;\;\overline{G[[t]]\cdot[t^\la]}\cap
G_1[t^{-1}]\cdot [t^\mu].\] 
This is a transverse slice to the orbit
$G[[t]]\cdot[t^\mu]$ in the closure $\overline{G[[t]]\cdot[t^\la]}$.
It is a conical symplectic singularity with respect to the $\bS$-action by loop rotation \cite[2.7]{KWWY}.
It may or may not admit a conical symplectic resolution; a necessary and sufficient criterion is given in \cite[2.9]{KWWY}.
In type A, such a resolution always exists.
 
\begin{enumerate}
\item The group of Hamiltonian symplectomorphisms commuting with $\bS$ is
the simultaneous centralizer of $t^\la$ and $t^\mu$ in $G$, which is typically a torus.
\addtocounter{enumi}{1}  
\item The question of how to quantize this variety has been considered
  by the fourth author jointly with Kamnitzer, Weekes and Yacobi; there
  is a conjectural identification of the quantizations of this
  symplectic variety with a quotient of a shifted Yangian
  \cite[4.8]{KWWY}.
\item[(iv-viii)] At the moment, these questions have
  not been addressed.  No serious study of the categories
  $\cOa$ and $\cOg$ has been done, aside from Brundan and Kleshchev's
  work in type A \cite{BKyang}.  In the type A case, the varieties and
  their resolutions coincide with type A quiver varieties or type A S3-varieties, by work of Maffei \cite{Maf} and \Mirkovic-Vybornov
  \cite{MV08}, so the results of previous sections can be applied.
\end{enumerate}

\section{Symplectic duality}\label{duality}
In this section we describe a close relationship between the categories associated to certain
pairs of symplectic varieties.  In a number of special cases, we expect this relationship to provide connections between
previously studied geometric and categorical constructions, including two superficially different sets of link 
invariants \cite{MS,SS} (see Section \ref{knot}).

Our relationship is defined at the categorical level, but it has two very concrete cohomological consequences.  
The first (Section \ref{filtrations}) arises by passing to Grothendieck groups; we obtain a duality of vector spaces 
that explains previously known numerical identities in the combinatorics of matroids and illuminates the phenomena
of Schur-Weyl duality and level-rank duality in representation theory.
The second (Section \ref{gm}) arises by considering the centers of the universal deformations of the Yoneda
algebras of our categories.  The relationship that we see was originally observed
in certain special cases by Goresky and MacPherson \cite{GM};
by regarding this relationship as a shadow of symplectic duality, 
we generate new classes of examples and provide an explanation for the examples observed in \cite{GM}.

Throughout this section,
we will assume that every conical symplectic resolution comes equipped with a set of ``integral periods" in $\Ht$,
consistent with the three conditions in Section \ref{sec:integrality}.  We will always work with a quantization
for which localization holds, so that we need not distinguish between $\cOa$ and $\cOg$.
We will assume that Conjecture \ref{geom-Koszul} holds, so that our category $\cO$ is standard Koszul
(and therefore Koszul by Theorem \ref{standard Koszul implies Koszul}).  In particular, this means that
$\cO$ comes equipped with a graded lift $\tilde\cO$, that is, a mixed
category whose degrading is $\cO$ (see Section \ref{hwksk}).
Finally, we will assume that the twisting and shuffling actions lift naturally to 
$D^b(\tilde\cO)$; this is the case in all of the examples from Section \ref{sec:examples} where a Koszul grading is known to exist.

\subsection{The definition}\label{sd-def}
Consider a conical symplectic resolution $\fM$, equipped with a Hamiltonian action of $\bT$, commuting with $\bS$, 
such that $\fM^\bT$ is finite.  
We denote by $\cO$ the category $\cOa\simeq\cOg$ for an integral period at which localization holds.
The fact that we do not need to specify the period of the quantization follows from Lemma \ref{geometric twist}, which says
that the categories $\cOg$ associated to any two integral parameters are canonically equivalent.
We will write $\pi_1(E_{\operatorname{tw}}/W)$ to denote $\pi_1(E_{\operatorname{tw}}/W, [\la])$ for any integral $\la$
sufficiently deep in the ample cone of $\fM$.

Let $G$ be the group of Hamiltonian symplectomorphisms of $\fM$ that commute with $\bS$, and let $T\subset G$
be a maximal torus containing the image of $\bT$. By Lemma
\ref{two-generics}, this is unique if $G$ is reductive, and more
generally unique up to conjugation by the unipotent radical of $C_G(\bT)$.
Let $\zeta\in \mt_{\R}$ be the cocharacter of $T$ by which $\bT$ acts, and
let $C$ be the chamber of $\cHt$ containing $\zeta$.
By Lemma \ref{trivial shuffling}, we could replace $\zeta$ (and with it the action of $\bT$)
by any other element of $C$ without changing $\cO$.
We will write $\pi_1(E_{\operatorname{sh}}/\rWeyl)$ to denote $\pi_1(E_{\operatorname{sh}}/\rWeyl, [\zeta])$ for any $\zeta\in C$.

Let $\fM^!$ be another conical symplectic resolution on which $\bT$ acts with isolated fixed points, commuting with $\bS$.  
We denote all of the corresponding
structures related to $\fM^!$ with an upper shriek; for example, the fixed points of $\fM^!$ will be indexed by the set $\cI^!$,
the group $\pi_1(E^!_{\operatorname{sh}}/\rWeyl^!)$ will act on $D^b(\tilde\cO^!)$, and so on.

\begin{definition}\label{def:duality}
A {\bf symplectic duality} from $\fM$ to $\fM^!$ consists of 
\begin{itemize}
\item a bijection $\a\mapsto\a^!$ from $\cI$ to $\cI^!$ which is order-reversing for the geometric order $\leftharpoonup$ defined in Section \ref{gco-intform},
\item a bijection $S\mapsto S^!$ from $\spe_\fM$ to $\spe_{\fM^!}$ which is order-reversing for the closure order,
\item group isomorphisms $W\cong \rWeyl^!$ and $\rWeyl\cong W^!$,
\item a pair of linear isomorphisms $\mt_\R\cong H^2(\fM^!; \R)$ and $H^2(\fM; \R)\cong\mt_\R^!$, which identify the lattice
of cocharacters with the lattice of integer homology classes, and
\item a Koszul duality from $\tilde \cO$ to $\tilde\cO^!$ (Definition \ref{def:dual}).
%an equivalence $D^b(\tilde\cO)\simeq D^b(\tilde\cO^!)$ that sends $\tilde L_\a$ to $\tilde P^!_{\a^!}$ 
%and $\tilde I_\a$ to $\tilde L^!_{\a^!}$.
%In particular, this means that $\cO$ and $\cO^!$ are Koszul dual to each other.
\end{itemize}
These structures are required to satisfy the following conditions:
\begin{itemize}
\item The bijection of fixed points is compatible with the bijection of special leaves
via the operation that associates a special leaf $\becircled\fM_{\a,0}\in\spe$ to an element $\a\in\cI$
(see Corollary \ref{mza} and the preceding discussion).
That is, for any $\a\in\cI$, we require that $$(\becircled\fM_{\a,0})^! =\, \becircled\fM^!_{\!\a^!\! ,0}\;.$$

\item The isomorphism $\mt_\R\cong H^2(\fM^!; \R)$, intertwines the action of $\rWeyl$ with that of $W^!$,
takes the arrangement $\cHs$ to $\cHt^!$, 
% takes the dominant cone in $\mt_\R$ to the movable cone in $H^2(\fM^!; \R)$, 
and takes the chamber
$C\subset\mt_\R$ to the ample cone in $H^2(\fM^!; \R)$.
Furthermore, all of the analogous statements hold for the isomorphism $H^2(\fM; \R)\cong\mt_\R^!$.
In particular, this means that we have canonical isomorphisms
$$\pi_1(E_{\operatorname{sh}}/\rWeyl)\cong \pi_1(E^!_{\operatorname{tw}}/W^!)
\and
\pi_1(E_{\operatorname{tw}}/W) \cong \pi_1(E^!_{\operatorname{sh}}/\rWeyl^!).$$

\item The Koszul duality from $\tilde\cO$ to $\tilde\cO^!$ exchanges
  twisting functor $\Phi^{*,*}$  (as defined in Section
\ref{twisting-gco})  and shuffling functor  $\Xi^{*,*}$   (as defined
in Section \ref{shuffling}) and similarly
with $\tilde\cO$ and $\tilde\cO^!$ reversed.
That is, the equivalence $D^b(\tilde\cO)\to D^b(\tilde\cO^!)$ takes the shuffling action of 
$\pi_1(E_{\operatorname{sh}}/\rWeyl)$ on $D^b(\tilde\cO)$ to the twisting action of
$\pi_1(E^!_{\operatorname{tw}}/W^!)$ on $D^b(\tilde\cO^!)$,
and vice versa.
\end{itemize}

% Sometimes we will simply write that $\fM$ and $\fM^!$ are dual, by which we mean that there exist parameters
% $B, C$ and $B^!, C^!$ along with a symplectic duality between the triples $(\fM, B, C)$ and $(\fM^!, B^!, C^!)$.
\end{definition}

\begin{remark}
Symplectic duality is symmetric;
that is, if there is a symplectic duality from $\fM$ to $\fM^!$, then
there is a symplectic duality from $\fM^!$ to $\fM$.
To see this, we invoke Proposition \ref{Koszul-symmetric} and Remark \ref{Koszul-symmetric-derived}, which say that
if $\Psi:D^b(\tilde\cO)\to D^b(\tilde\cO^!)$ is a Koszul duality from $\tilde\cO$ to $\tilde\cO^!$, then 
the composition of $\Psi^{-1}$ with the derived Nakayama functor $\R\EuScript{N}$ on
$D^b(\tilde \cO)$ is a Koszul duality from $\cO^!$ to $\cO$.  

We still need to check that $\Psi^{-1}\circ \R\EuScript{N}$
exchanges twisting and shuffling functors.
Since $\tilde \cO$ has finite global dimension,
$\R\EuScript{N}$ is a right Serre functor \cite{MS}; by the uniqueness of Serre 
functors, it commutes with any equivalence of derived categories,
in particular with any twisting or shuffling functor.  Thus, since $\Psi$ exchanges twisting and shuffling functors, so does $\Psi^{-1}\circ \R\EuScript{N}$.
\end{remark}

\subsection{Examples of symplectic dualities}\label{examples-known}
In this section we describe all of the examples of pairs of conical symplectic resolutions that we know to be dual,
along with some conjectural generalizations of these examples.

\subsubsection{Cotangent bundles of flag varieties}

\begin{theorem}\label{Springer-duality}
Let $G$ be a reductive algebraic group with Langlands dual ${}^LG$,
and let $B\subset G$ and ${}^LB\subset {}^LG$ be Borel subgroups.  Then
$T^*(G/B)$ is symplectic dual to $T^*({}^LG/{}^LB)$.
\end{theorem}

\begin{proof}
For a generic cocharacter $\zeta$ of $G$, the fixed points of $T^*(G/B)$ are indexed by the Weyl group $W \cong W^!$,
and the order-reversing bijection of $W$ is given by sending $w$ to $w^{-1}w_0$.  The fact that this bijection
induces an order-reversing bijection of special nilpotent orbits is proven in \cite[3.3]{KL79}.  The $W$-equivariant
linear isomorphisms are part of the package of Langlands duality.
The Koszul duality is proven in \cite[1.1.3]{BGS96}\footnote{This paper actually proves 
that a regular integral block of BGG category $\cO$ is self-dual, 
but those categories are isomorphic for Langlands dual groups, 
since they can be computed in terms of the Weyl group \cite{Soe90}.},
and the fact that twisting and shuffling are exchanged is proven 
in \cite[Theorem 39]{MOS}.
\end{proof}

\subsubsection{S3-varieties} 
Next, we consider S3-varieties associated to $SL_r$, as described in Section \ref{sec:S3}.
For a composition $\mu$ of $r$, we define  
a new composition $\mu^o$ by $\mu^o_i := \mu_{-i}$.
Also, recall that $\bar\mu$  denotes the partition of $r$ 
obtained by sorting the positive entries of $\mu$, and
 $\bar\mu^t$ denotes the transposed partition.  Note that $\bar\mu = \overline{\mu^o}$. 

Fix a pair of compositions $\mu$ and $\nu$ of $r$.
Let $e\colon \C^r\to \C^r$ be a nilpotent element in Jordan normal
form with block sizes given by $\nu$ in order.  
Let $\fX_\mu^\nu$ be the S3-variety $\fX^e_{P_\mu}$ that was
introduced in Section \ref{sec:S3}; it is nonempty if and only if $\bar\nu \le \bar\mu^t$ in the dominance order.

Let $T_\nu^\mu$ be a maximal torus 
of the group $G_\Ham$ for the variety $\fX_\mu^\nu$, as 
described in item (i) of Section \ref{sec:S3}.  The description of
the cohomology of these varieties in \cite{BrO} gives a natural 
isomorphism $\operatorname{Lie} T_\nu^\mu \cong H^2(\fX_{\mu^o}^\nu; \C)$.
Let $C^\mu_\nu$ be the unique chamber of the arrangement $\cHs$ for $\fX^\mu_\nu$ which lies on the positive side of every root hyperplane which appears. 
The following theorem appears in \cite[5.32]{Webqui}.

\begin{theorem}\label{S3-duality}
The variety $\fX_\nu^\mu$ is symplectic dual to $\fX_{\mu}^\nu$, where the action of $\bT$ on $\fX_\nu^\mu$
is given by a cocharacter in $C^\mu_\nu$ and the action of $\bT$ on $\fX_{\mu}^\nu$ is given by a cocharacter in $-C^\nu_{\mu}$.
\end{theorem}

\begin{remark}
Theorem \ref{S3-duality} does not appear to be fully symmetric; of
course, by negating the isomorphisms $\mt_\R\cong H^2(\fM^!; \R)$ and
{\it vice versa}, we can switch the sign of the chambers appearing,
and thus the role of $\mu$ and $\nu$. 

Alternatively, we could take
$\fX^\nu_{\mu^o}$ with the chamber $C^\nu_{\mu^o}$.  In this case, the
symmetry depends on the $G$-equivariant isomorphism
$\fX_{\nu^o}^{\mu^o}\cong\fX_\nu^\mu$, using the automorphism of $\mathfrak{sl}_r$ given 
by the adjoint action of any representative of $w_0$.  See
\cite[9.3]{kosdef} for the analogous statement about algebras.
\end{remark}

\begin{remark}\label{GMS3}
If we take $\nu$ to be a composition with $r$ parts each of size 1, Theorem \ref{S3-duality}
specializes to the statement that $T^*(G/P_\mu)$ is symplectic dual to the Slodowy slice to the nilpotent
orbit of Jordan type $\bar\mu^t$ inside of the full nilpotent cone.  If we further specialize to the case where $\mu=\nu$,
we obtain Theorem \ref{Springer-duality} for $\mathfrak{sl}_r$.
\end{remark}

\begin{remark}\label{quiver-Gr}
More generally, we expect that a quiver variety whose quiver is a finite ADE Dynkin diagram with some choice of
orientation (Section \ref{sec:quiver-varieties}) will be dual to a
slice in the affine Grassmannian for the Langlands dual group
(Section \ref{sec:affine-grassm-slic}).  Since quiver varieties exist
for every integral highest weight, we should consider them as
associated to the simply connected group for that Dynkin diagram, and
thus consider the affine Grassmannian of the adjoint form.  In type A, both of these varieties are type A S3-varieties, and the precise
statement that we want is given in Theorem \ref{S3-duality}.

We note that this conjectural duality provides a connection between two well-known constructions
of weight spaces of irreducible representations of simply laced simple Lie algebras.  One, due to Nakajima,
realizes these weight spaces as top homology groups of quiver varieties \cite[10.2]{Nak98}.  The other,
using the geometric Satake correspondence of
Ginzburg \cite[3.11 \& 5.2]{G-GS} and \Mirkovic-Vilonen \cite{MV-GS}, 
realizes them as top-degree $\bT$-equivariant intersection cohomology groups
of slices in the affine Grassmannian.  See Example \ref{awesome example} for an explanation of how symplectic
duality (conjecturally) allows us to identify these two vector spaces.
\end{remark}

\subsubsection{Hypertoric varieties} Next, we consider symplectic duality for hypertoric varieties.  Let $X$ and $X^!$ be a pair of unimodular,
Gale dual polarized arrangements \cite[2.17]{BLPWtorico}.
% and let $\bf{X}$ and $\bf{X^!}$ be regular, integral, quantized polarized arrangements
% that are linked to $X$ and $X^!$, respectively \cite[2.12]{BLPWtorico}.
These data can be used to construct hypertoric varieties $\fM$ and $\fM^!$ with specified actions of $\bT$ \cite[\S 5.1]{BLPWtorico}.

\begin{theorem}\label{hypertoric-duality}
The hypertoric varieties $\fM$ and $\fM^!$ are symplectic dual.
\end{theorem}

\begin{proof}
The order-reversing bijection on fixed points is given in \cite[2.10]{GDKD}.
Symplectic leaves of $\fM_0$ and $\fM^!_0$ (all of which are special) are indexed by coloop-free flats of the hyperplane
arrangements $\mathcal{A}$ and $\mathcal{A}^!$ associated to $X$ and $X^!$ \cite[2.3]{PW07}, and it is well-known that such flats
are in order-reversing bijection for Gale dual arrangements.  The compatibility of the bijections follows from \cite[7.16]{BLPWtorico}.

The group isomorphisms $W\cong \rWeyl^!$ and $\rWeyl\cong W^!$ are described in \cite[\S 8.1]{BLPWtorico},
and the equivariant isomorphisms of vector spaces with hyperplane arrangements are straightforward
from the combinatorics of Gale duality.  The Koszul duality between $\cO$ and $\cO^!$ is proven in \cite[4.7 \& 4.10]{BLPWtorico},
and the fact that twisting and shuffling are exchanged is \cite[8.24 \& 8.26]{BLPWtorico}.
\end{proof}

\subsubsection{Affine type A quiver varieties} A fourth example of symplectic duality is given by quiver varieties for affine type A
quivers.  We leave most of the combinatorics to the papers
\cite{Webqui,WebRou} which treat this case in more detail, and only give a rough outline below.  

Fix a positive integer $e$, and consider quiver varieties for $\slehat$.  That is, we take a quiver
whose underlying graph is an $e$-cycle, which we will identify with the Cayley graph of
$\Z/e\Z$ for the generators $\{\pm 1\}$.  
Fix a highest weight $\nu=\sum \nu_i\omega_i$ for $\slehat$ as in Section
\ref{sec:quiver-varieties}; let $\ell := \sum\nu_i$ be the {\bf level} of
this highest weight.  Pick a basis of the framing vector spaces, which
have total dimension $\ell$.  The Lie algebra of the torus $\mt$ is spanned by
the cocharacters
\begin{itemize}
\item $\varepsilon_j$ which acts with weight 1 on the $j^\text{th}$ basis
  vector in the framing space, and
\item $\gamma$ which acts with weight 1 on every clockwise oriented
  edge of the cycle (and thus weight -1 on counterclockwise oriented
  edges).  
\end{itemize}
There are certain distinguished choices of $\zeta$ which contain a
representative of each chamber of $\cHs$.  We call these {\bf Uglov} actions,
since they naturally correspond to the choice of charges for a higher
level Fock space {\it \`a la} Uglov \cite{Uglov}.  Let
$\mathbf{s}=(s_1,\dots, s_\ell)$ be a collection of integers such that
there are precisely $\nu_j$ of the entries of this sequence such that
$s_i\equiv j\pmod e$.  We let $\zeta_{\mathbf{s}}^+$ be the cocharacter whose
derivative is $\ell\gamma+(s_i\ell+ie) \varepsilon_j$, and let $C^\pm_{\mathbf{s}}$ be
the chamber of $\cHs$ containing it. The action of $\zeta_{\mathbf{s}}^+$ always has
isolated fixed points, and every chamber contains one of these
cocharacters, as shown in \cite[5.16]{Webqui}.
The fixed points of an Uglov action are in canonical bijection with
$\ell$-multipartitions.  We wish to visualize these partitions as abaci as in, for
example, \cite[\S 2.1]{Tingcomb}.  

We have 
$\ell$ runners on our abacus, numbered from bottom to top, 
each of which we visualize as a copy
of the real line with slots at each integer which can hold a bead.
To each multipartition $$\xi=(\xi^{(1)}_1\geq
\xi^{(1)}_2\geq \cdots; \xi^{(2)}_1\geq
\xi^{(2)}_2\geq \cdots; \dots; \xi^{(\ell)}_1\geq
\xi^{(\ell)}_2\geq \cdots),$$ we associate the abacus where on the
$k^\text{th}$ runner, we fill the beads at $$\xi^{(k)}_1+s_{k},
\xi^{(k)}_2+s_{k}-1, \dots, \xi^{(k)}_j+s_{k}-j+1,\dots,$$ and no
others.  Note this means that every position is filled at sufficiently
negative integers, and open at sufficiently positive.  

The combinatorics of the duality is encapsulated in the map between
fixed points.  This is given by cutting the abacus into $e\times \ell$
rectangles; that is rectangles consisting of the $me,me+1,\dots,
me+e-1$ positions of each runner as $m$ ranges over $\Z$.  Then we
flip the rectangle, so that the first runner becomes the beads at the
points $m\ell$ for $m\in \Z$, the second runner becomes the beads the
points $m\ell+1$, etc. as in the picture below. 
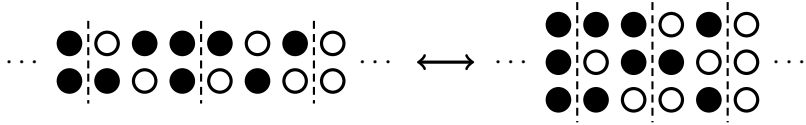
\[\begin{tikzpicture}
\node (a) at (-3,0){
    \begin{tikzpicture}[very thick,scale=.5]
      \draw[thick, densely dashed] (-.5,1.6) -- (-.5,-.6);
      \draw[thick, densely dashed] (-3.5,1.6) -- (-3.5,-.6);
      \draw[thick, densely dashed] (2.5,1.6) -- (2.5,-.6);
      \node[circle, draw, inner sep=3pt] at (0,0){}; \node[circle,
      draw, inner sep=3pt,fill=black] at (0,1) {}; \node[circle, draw, inner
      sep=3pt,fill=black] at (1,0) {}; \node[circle, draw, inner sep=3pt] at
      (1,1) {}; \node[circle, draw, inner sep=3pt] at (2,0) {};
      \node[circle, draw, inner sep=3pt,fill=black] at (2,1) {}; \node[circle,
      draw, inner sep=3pt] at (3,0) {}; \node[circle, draw, inner
      sep=3pt] at (3,1) {}; \node[circle, draw, inner sep=3pt,fill=black] at
      (-1,0) {}; \node[circle, draw, inner sep=3pt,fill=black] at (-1,1) {};
      \node[circle, draw, inner sep=3pt] at (-2,0) {}; \node[circle,
      draw, inner sep=3pt,fill=black] at (-2,1) {}; \node[circle, draw, inner
      sep=3pt,fill=black] at (-3,0) {}; \node[circle, draw, inner sep=3pt] at
      (-3,1) {}; \node[circle, draw, inner sep=3pt,fill=black] at (-4,0) {};
      \node[circle, draw, inner sep=3pt,fill=black] at (-4,1) {}; \node at
      (-5.2,.5) {$\cdots$}; \node at (4.2,.5) {$\cdots$};
    \end{tikzpicture}};
\node (b) at (3,0){    \begin{tikzpicture}[very thick,scale=.5]
      \draw[thick, densely dashed] (-.5,1.6) -- (-.5,-1.6);
      \draw[thick, densely dashed] (-2.5,1.6) -- (-2.5,-1.6);
      \draw[thick, densely dashed] (1.5,1.6) -- (1.5,-1.6);
      \node[circle, draw, inner sep=3pt,fill=black] at (0,0){}; 
      \node[circle, draw, inner sep=3pt] at (0,1) {}; 
      \node[circle, draw, inner sep=3pt] at (1,0) {}; 
      \node[circle, draw, inner sep=3pt,fill=black] at (1,1) {}; 
      \node[circle, draw, inner sep=3pt] at (2,0) {};
      \node[circle, draw, inner sep=3pt] at (2,1) {}; 
      \node[circle, draw, inner sep=3pt] at (-1,-1) {}; 
      \node[circle, draw, inner sep=3pt,fill=black] at (-2,-1) {}; 
      \node[circle, draw, inner sep=3pt,fill=black] at (-1,0) {}; 
      \node[circle, draw, inner sep=3pt,fill=black] at (-1,1) {};
      \node[circle, draw, inner sep=3pt] at (-2,0) {}; 
      \node[circle, draw, inner sep=3pt,fill=black] at (-2,1) {}; 
      \node[circle, draw, inner sep=3pt,fill=black] at (-3,0) {}; 
      \node[circle, draw, inner sep=3pt,fill=black] at (-3,1) {}; 
      \node[circle, draw, inner sep=3pt] at (0,-1) {};
      \node[circle, draw, inner sep=3pt,fill=black] at (1,-1) {}; 
      \node[circle, draw, inner sep=3pt,fill=black] at (-3,-1) {}; 
      \node[circle, draw, inner sep=3pt] at (2,-1) {}; 
      \node at (-4.2,0) {$\cdots$}; 
      \node at (3.2,0) {$\cdots$};
    \end{tikzpicture}};
\draw[<->,very thick] (a)--(b);
\end{tikzpicture}\] The lefthand picture above corresponds
to \[e=3,\quad\ell=2,\quad\mathbf{s}=(0,1),\quad\xi^{(1)}=(2,1),\,\xi^{(2)}=(2,1,1,1),\]
while the dual righthand picture corresponds to
\[e=2,\quad\ell=3,\quad\mathbf{t}=(0,0,1),\quad\xi^{(1)}=(2),\,\xi^{(2)}=(1,1),\,\xi^{(3)}=(1).\]

If we fix the triple $(\nu,\mu,\mathbf{s})$, and perform the duality
above on the abacus for a multipartitions with this weight and charge,
the resulting weights and charge $(\mu^!,\nu^!,\mathbf{t})$ are
combinatorially determined, as discussed in \cite[\S 5.3]{Webqui}.  We
can also associate to this
the triple $(\tilde{\fQ}^\nu_\mu, T, C_{\mathbf{s}}^+)$.  When the
combinatorial data is switched by rank-level duality, we obtain a
symplectic duality \cite[5.25]{Webqui}.

\begin{theorem}\label{hkr-mkr}
The variety $\tilde{\fQ}^\nu_\mu$ is symplectic dual to $\tilde{\fQ}^{\mu^!}_{\nu^!}$, where the action of $\bT$ on $\tilde{\fQ}^\nu_\mu$
is given by a cocharacter in $C_{\mathbf{s}}^+$ and the action of $\bT$ on $\tilde{\fQ}^{\mu^!}_{\nu^!}$ is given by a cocharacter in $-C_{\mathbf{t}}^+$.
% The triples $(\tilde{\fQ}^\nu_\mu, T,
% C_{\mathbf{s}}^+)$ and $(\tilde{\fQ}^{\mu^!}_{\nu^!},T,C_{\mathbf{t}}^-)$
% are symplectic dual.
\end{theorem}

\begin{remark}
The proof of Theorem \ref{hkr-mkr}
ultimately relies on Koszul duality results for
certain categories of affine representations and category $\cO$ for
Cherednik algebras based on work of Rouquier, Shan, Varagnolo, and Vasserot 
 \cite{RSVV,SVV}.
\end{remark}

Theorem \ref{hkr-mkr} has the following special case.
Let $\hkr$ be the Hilbert scheme of
$r$ points on a crepant resolution of $\C^2/\mck$, where $\mck :=
\Z/k\Z$ acts effectively and symplectically on $\C^2$.  Let $\mkr$ be
the moduli space of torsion-free sheaves $E$ on $\mathbb{P}^2$ with
$\rk E = k$ and $c_2(E) = r$, along with a framing $\Phi: E|_{\mathbb
  P^1}\overset{\sim}\to \mathcal O_{\mathbb P^1}^{\oplus k}$.  
On $\hkr$, the torus $T$ is 1-dimensional, induced by the
symplectic action on $\C^2$ commuting with $\mck$. On $\mkr$, the torus
$T^!$ is naturally identified with $\cs$ times the projective diagonal matrices in
$\operatorname{PGL}_k$; let $\vartheta_i$ denote the weights of $\zeta$ in
$\operatorname{PGL}_k$ (thus only well-defined up to simultaneous translation) and $h$
the weight in $\cs$.  The hyperplanes in $\cHs^!$ are the points of $\mt^!$ for which
$\vartheta_i-\vartheta_j=mh$ for $m\in [-k+1,k-1]$, along with the single additional hyperplane $h=0$.  Note
that these are precisely the GIT walls for $\hkr$ as described by
Gordon \cite[\S 4.3]{Gorqui} (Gordon's $H_i$ is our
$\vartheta_i-\vartheta_{i+1}$).  Let $C_+$ be the positive chamber in
$\mt_\R\cong \R$ and let $C_-$ be the chamber in $\mt^!_\R$
where $\vartheta_i\ll \vartheta_{i+1}$ and $h<0$.

\begin{corollary}\label{the real hkr-mkr}
The variety $\hkr$ is symplectic dual to $\mkr$.
\end{corollary}

Further specializing to the case where $k=1$, we have 
$\mathcal{H}(1,r)\cong\operatorname{Hilb}^r\C^2\cong \mathcal{M}(1,r)$.
The Hilbert scheme $\operatorname{Hilb}^r\C^2$ does not satisfy our assumption that the minimal leaf is a point,
but we may replace it with the reduced Hilbert scheme $\operatorname{Hilb}_0^r\C^2$ (in which the center of mass
is required to lie at the origin) without affecting category $\cO$ or any of its structure.

\begin{corollary}
The reduced Hilbert scheme $\operatorname{Hilb}_0^r\C^2$ is self-dual.
\end{corollary}

\begin{remark}\label{instanton}
More generally, we expect the moduli space of $G$-instantons on a
crepant resolution of $\C^2/\mck$ to be dual to the moduli space of
$G'$-instantons on a crepant resolution of $\C^2/\mck'$, where
$G$ is matched to $\mck'$ and $G'$ is matched to $\mck$ via
the McKay correspondence.  Corollary \ref{the real hkr-mkr} constitutes the special case
where both $G$ and $G'$ are of type A.

Braverman and Finkelberg have suggested that resolutions of slices
in the ``double affine Grassmannian'' should be isomorphic to certain spaces of instantons.
Via this philosophy, our conjecture may be regarded as an affine version of the conjecture in Remark \ref{quiver-Gr}.
\end{remark}

\subsection{Duality of cones}\label{sec:cones}
The notion of symplectic duality is in fact more naturally defined at the level of cones, as we explain below.
% Let $\fM_0$ be a Poisson cone that admits a conical symplectic resolution along with a Hamiltonian $\bT$-action, commuting with $\bS$,
% with isolated fixed points.
% The first thing that we will do is explain how (almost) all of the structure involved in the definition of symplectic
% duality can be expressed entirely in terms of $\fM_0$; that is, it does not involve choosing a particular resolution.

\begin{proposition}
Let $\fM$ and $\fM'$ be resolutions of the same cone $\fM_0$.
\begin{enumerate}
\item $G\cong G'$, and therefore $\rWeyl\cong\rWeyl'$.
Given a maximal torus $T\subset G$, the hyperplane arrangements
$\cHs$ and $\cHs'$ in $\mt_\R$ coincide.
\item $W\cong W'$, and we have a $W$-equivariant isomorphism 
$H^2(\fM; \R)\cong H^2(\fM'; \R)$ taking $\cHt$ to $\cHt'$.
\item For any $\la\in\Ht\cong H^2(\fM'; \C)$, we have $\spe_{\fM} = \spe_{\fM'}$.
\item Given a period $\la\in\Ht$ and a generic cocharacter $\zeta$ of $T$,
the categories $\cOa$ and $\cOa'$ associated to $\fM$ and $\fM'$ are canonically equivalent.
\end{enumerate}
\end{proposition}

\begin{proofnoqed}
\begin{enumerate}
\item The Lie algebra
$\mg$ of $G$ is isomorphic to the Lie algebra of Hamiltonian
vector fields on $\fM_0$ with $\bS$-weight 1; this tells us that
$\mg\cong\mg'$.  Let $\tilde G$ be a simultaneous cover of both $G$ and $G'$; then
    $\tilde G$ acts on both $\fM$ and $\fM'$.  Since an automorphism
    of $\fM$ or $\fM'$ is trivial if and only if it induces the
    trivial automorphism of $\fM_0$, the maps $\tilde G\to G$ and
    $\tilde G\to G'$ have the same kernel.

    Let $T\subset G$ be a maximal torus.  Any cocharacter
    $\zeta$ of $T$ induces an action of $\bT$ on $\fM$ with isolated fixed points if and
    only if the induced action on $\fM_0$ has isolated fixed points, thus the
    hyperplane arrangement $\cHs$ in $\mt$ is independent of the
    choice of resolution.

  \item The fact that the Namikawa Weyl group is determined completely
    by $\fM_0$ is immediate from Namikawa's definition.  The fact that
    the groups $H^2(\fM; \R)$ and $H^2(\fM'; \R)$ are canonically
    isomorphic for two different resolutions $\fM$ and $\fM'$ is
    explained in \cite[2.18]{BLPWquant}.  This isomorphism is clearly
    $W$-equivariant, and takes the arrangement $\cHt$ to $\cHt'$.

\item
The fact that the poset $\spe$ of special leaves does not depend on the choice of resolution
follows from the fact that the algebra $A$ does not depend on the choice of resolution; that is, the algebra of $\bS$-invariant
global sections of the quantization of $\fM$ with period $\la$ is canonically isomorphic to the algebra of $\bS$-invariant
global sections of the quantization of $\fM'$ with period $\la$ \cite[3.9]{BLPWquant}.

\item The category $\cOa$ is defined in terms of the algebra $A$ and the subalgebra $A^+\subset A$,
and we have established that these structures do not depend on the choice of resolution.\qed
\end{enumerate}
\end{proofnoqed}

% In the
% definition of symplectic duality, we fixed a $B$-dominant chamber $C$ of
% $\cHs$, and took $\cO$ to be the category $\cOa$ for any $\zeta\in
% C$ and $\la$ deep in the ample cone of $\fM$.  
% The ample cone of
%  $\fM$ is one chamber of $\cHt$; the other chambers consist of
% ample cones of other resolutions (these are the movable chambers)
% and their $W$-translates (Remark \ref{dream}).  Now, since we do not want to fix
% a resolution $\fM$, we need to fix a $B$-dominant chamber $\Csh$ of
% $\cHs$ as well as a movable chamber $\Ctw$ of $\cHt$, and we
% obtain a category $\cO(\Ctw,\Csh)$.

We are now prepared to define symplectic duality of cones.  
Let $\fM_0$ be a Poisson cone that admits a conical symplectic resolution, which we do not fix.  
Let $G$ be the group of Hamiltonian symplectomorphisms that commute with $\bS$.  Fix a
Borel subgroup $B\subset G$, and assume that one  (and
thus any) maximal torus $T\subset B$ fixes only the cone point of $\fM_0$.

For any movable chamber $\Ctw$ of $\cHt$ and $B$-dominant chamber $\Csh$ of $\cHs$,
we obtain a category $\cO(\Ctw,\Csh)$ by choosing the unique conical symplectic resolution $\fM$
with ample cone $\Ctw$, taking a quantization with period sufficiently deep in the ample cone,
and allowing $\bT$ to act by a cocharacter in $\Csh$.

\begin{definition}\label{dual cones}
A symplectic duality between $(\fM_0, B)$ and $(\fM_0^!, B^!)$ consists of 
\begin{itemize}
\item an order-reversing bijection
$S\mapsto S^!$ from $\spe_{\fM}$ to $\spe_{\fM^!}$ (for any choice of $\fM$ and $\fM^!$);
\item group isomorphisms $W\cong \rWeyl^!$ and $\rWeyl\cong W^!$;
\item a pair of linear isomorphisms $\mt_\R\cong H^2(\fM^!; \R)$ and $H^2(\fM; \R)\cong\mt_\R^!$, 
which are equivariant with respect to the isomorphisms of the previous item, compatible with the lattices, take $\cHt$ to $\cHs^!$ and $\cHs$ to $\cHt^!$,
and take movable twisting chambers to dominant shuffling chambers;
\item for any $\Ctw,\Csh$ and $\Ctw^!,\Csh^!$ related by the linear isomorphisms above,
a Koszul duality from $\cO(\Ctw,\Csh)$ to $\tilde\cO^!(\Ctw^!,\Csh^!)$.
%an equivalence $D^b(\tilde\cO(\Ctw,\Csh)\simeq D^b(\tilde\cO^!(\Ctw^!,\Csh^!))$ that sends 
%graded lifts of simples to graded lifts of projectives and graded lifts of injectives to graded lifts of simples.
\end{itemize}
These structures are required to satisfy the following conditions:
\begin{itemize}
\item Let $L$ be a simple object of $\cO(\Ctw,\Csh)$, and suppose that $\tilde L$ is sent
by the Koszul duality functor to a graded lift of the projective
cover of the simple object $L^!$ of $\cO^!(\Ctw^!,\Csh^!)$.  Then we require that our bijection of special leaves
takes the dense leaf in $\fM_{0,L}$ to the dense leaf in $\fM^!_{0,L^!}$.
% \item The isomorphism $\mt_\R\cong H^2(\fM^!; \R)$ intertwines the action of $\rWeyl$ with that of $W^!$.
\item The equivalences $D^b(\tilde\cO(\Ctw,\Csh))\simeq D^b(\tilde\cO^!(\Ctw^!,\Csh^!))$ interchange twisting and shuffling functors.
\end{itemize}
\end{definition}

\begin{remark}
Essentially, Definition \ref{dual cones} says that for any choice of $\fM$ resolving $\fM_0$
and $\fM^!$ resolving $\fM^!_0$, there are corresponding actions of $\bT$ on both sides
such that $\fM$ is dual to $\fM^!$.  (It also says that the linear
isomorphisms involved in all of these dualities can be chosen consistently.)
It is straightforward to promote Theorems \ref{Springer-duality},
\ref{S3-duality}, \ref{hypertoric-duality}, and \ref{hkr-mkr} to dualities between cones.
\end{remark}

\subsection{Duality of leaf closures and slices}
\label{sec:funct-strata-slic}
Let $\fM\to\fM_0$ be a conical symplectic resolution, and $S\subset\fM_0$ a symplectic leaf.  The closure $\bar{S}\subset\fM_0$
is again a Poisson cone, which may or may not admit a symplectic resolution.  For example, if $\fM_0$ is the nilpotent
cone in $\mathfrak{sl}_r$ and $S$ is a nilpotent orbit, then $\bar{S}$ admits a symplectic resolution of the form $T^*(\SL_r/P)$ for some $P$.
For other simple Lie algebras, however, $\bar{S}$ may admit no symplectic resolution.  If $\fM_0$ admits a Hamiltonian action of $\bT$,
commuting with $\bS$ and fixing only the cone point, then this restricts to an action on $\bar{S}$ with the same properties.

Let $p\in S$ be any point.  We say that another Poisson cone $\fN_0$ is a {\bf slice} to $S$ at $p$
if a formal neighborhood of $p$ in $\fM_0$ is isomorphic to a formal neighborhood of $p$ in $S$ times a formal
neighborhood of the cone point in $\fN_0$.  Assuming that such an $\fN_0$ exists\footnote{Kaledin \cite[1.6]{Kal09}
shows that we can always find a symplectic singularity $\fN_0$ with this property, but he does not prove
that the Poisson structure on $\fN_0$ is always homogeneous with respect to a conical $\bS$ action.  He does,
however, conjecture that this is the case \cite[1.8]{Kal09}.}, it will always admit a conical symplectic
resolution; in an infinitesimal neighborhood of the cone point, this resolution is obtained by base change along the inclusion
of $\fN_0$ into $\fM_0$.  However, even if $\fM_0$ admits a Hamiltonian $\bT$-action that commutes with $\bS$ 
and fixes only the cone point, $\fN_0$ might not admit such an action.

Let $\fM_0$ and $\fM_0^!$ be dual in the sense of Definition \ref{dual cones}.
Let $S$ be a special leaf of $\fM_0$, and let $S^!$ be the corresponding special leaf of $\fM_0^!$.
Let $\fN_0$ be a slice to $S$ at a point $p\in S$.

\begin{conjecture}\label{leaf-slice}
Suppose that $\bar{S}^!$ admits a conical symplectic resolution and $\fN_0$ admits a Hamiltonian $\bT$ action commuting 
with $\bS$ and fixing only a point.  Then $\bar{S}^!$ is dual to $\fN_0$.
\end{conjecture}

% We can verify Conjecture \ref{leaf-slice} for every example of symplectic duality that we know, as we explain below.

\begin{example}
Conjecture \ref{leaf-slice} is true for S3-varieties in type A.
We showed in Theorem \ref{S3-duality} that $\fX^\mu_\nu$ is dual to $\fX^\nu_{\mu^o}$.
The leaf closures in the affinization of $\fX^\mu_\nu$ have resolutions of the form $\fX^\mu_\rho$
and the slices in the affinization of $\fX^\nu_{\mu^o}$ have resolutions of the form $\fX^\rho_{\mu^o}$,
where $\bar\rho$ lies between $\bar\nu$ and $\bar\mu^t$ in the dominance order.

More generally, Conjecture \ref{leaf-slice} should hold for finite-type quiver varieties and slices in the affine
Grassmannian (Remark \ref{quiver-Gr}), as both of these families of varieties are closed under the operations
of leaf-closure and slice.
\end{example}

\begin{example}
Conjecture \ref{leaf-slice} is true for hypertoric varieties.  Special leaves correspond to coloop-free flats
\cite[2.1]{PW07}, leaf closures correspond to restrictions at flats \cite[2.1]{PW07}, and slices correspond
to localization at flats \cite[2.4]{PW07}.  Thus Conjecture \ref{leaf-slice} for hypertoric varieties follows from
Theorem \ref{hypertoric-duality} and the interchange of localization and restriction under
Gale duality \cite[2.6]{GDKD}.
\end{example}

\begin{example}
Conjecture \ref{leaf-slice} is true for affine type A quiver varieties.
A special leaf of $\fQ^\nu_{\mu}$ is indexed by a
highest weight $\xi$, a weight $\varpi$ and an integer $r$.
From the description of leaves and slices in \cite[\S 6]{Nak94}, the closure of this leaf
is isomorphic to $\fQ^\nu_{\varpi}\times \Sym^r\big(\C^2/(\Z/e\Z)\big)$, and its
slice is isomorphic to $\fQ^{\varpi-r\delta}_{\mu}\times \Sym^r\big(\C^2/(\Z/e\Z)\big)$.  These
switch roles under rank-level duality by \cite[5.18]{Webqui}, and the result follows.
\end{example}

\begin{remark}
In all of the examples that we know, $\bar{S}^!$ admits a conical symplectic resolution
if and only if $\fN_0$ admits a Hamiltonian $\bT$-action commuting with $\bS$ and fixing only a point.  
This suggests that there might be a more general notion of
duality than Definition \ref{dual cones} in which both of these conditions are relaxed, and one
holds on one side if and only if the other holds on the other side.
\end{remark}

\subsection{Duality of leaf filtrations}\label{filtrations}
Suppose that $\fM$ is symplectic dual to $\fM^!$.
For every special leaf $S\subset\fM$, let
$$D_S := K(\cOgS)_\C = \C\big\{[\Lambda_\a] \mid \fM_{\a,0}\subset \bar S \big\}$$
and
$$E_S := K(\cOgdS)_\C = \C\big\{[\Lambda_\a] \mid \fM_{\a,0}\subsetneq \bar S \big\}.$$
This is precisely the filtration of $K(\cOg)_\C$ whose associated graded gives us the direct sum decomposition
of Equation \eqref{O-decomposition}; that is, $D_S/E_S \cong K(\fP_S)_\C$.
If $\fM$ and $\fM^!$ are both interleaved
(Definition \ref{simple supports}), then this filtration agrees, via the characteristic cycle isomorphism,
with the BBD filtration of $H^{\dimfM}_{\fM^+}(\fM; \C)$ (Theorem \ref{special filtrations}).

Consider the perfect pairing between $K(\cO)_\C$ and $K(\cO^!)_\C$ for which the simples form dual bases, under the bijection between simples provided by Koszul duality.

\begin{proposition}\label{cohsd}
For each special leaf $S$, the above pairing induces a perfect pairing between $D_S/E_S$
and $D^!_{S^!}/E^!_{S^!}$.
\end{proposition}

\begin{proof}
For every special leaf $S$, we have
\begin{eqnarray*}
D_S^\perp \cap D^!_{S^!}&=& \C\big\{[\Lambda_\a] \mid \fM_{\a,0}\subset \bar S \big\}^\perp
\;\cap\; \C\big\{[\Lambda^!_{\a^!}] \mid \fM^!_{\a^!,0}\subset \bar S^! \big\}\\
&=& \C\big\{[\Lambda^!_{\a^!}] \mid \fM_{\a,0}\not\subset \bar S \big\}
\;\cap\; \C\big\{[\Lambda^!_{\a^!}] \mid \fM^!_{\a^!,0}\subset \bar S^! \big\}\\
&=& \C\big\{[\Lambda^!_{\a^!}] \mid \fM^!_{\a^!,0}\not\supset \bar S^! \big\}
\;\cap\; \C\big\{[\Lambda^!_{\a^!}] \mid \fM^!_{\a^!,0}\subset \bar S^! \big\}\\
&=&  \C\big\{[\Lambda^!_{\a^!}] \mid \fM^!_{\a^!,0}\subsetneq \bar S^! \big\}\\
&=& E^!_{S^!}.
\end{eqnarray*}
By symmetry, we also have $(D^!_{S^!})^\perp \cap D_S = E_S$.
This completes the proof.
\end{proof}

\begin{example}
If $\fM=\fM^! =T^*(\SL_r/B)$ and the period is integral, then the space $K(\cO)_\C$ can be
  identified with $\C[S_n]$, and the space $D_S/E_S$ is the unique
  subquotient which is isomorphic to the isotypic component of the
  Springer representation for $S$.
\end{example}

\begin{example}\label{dualities}
If $\fM$ and $\fM^!$ are hypertoric and the period is regular and integral, the dimension of the space
  $D_S/E_S$ is $T_{\mathcal{A}^F}(1,0)T_{\mathcal{A}_F}(0,1)$ where $T_{\mathcal{A}^F}$ and $T_{\mathcal{A}_F}$ are
  the Tutte polynomials of the restriction and localization of the
  hyperplane arrangement at the coloop-free flat $F$ corresponding to the
  leaf $S$.  The dual variety $\fM^!$ is the variety associated to the
  Gale dual hyperplane arrangement, and Gale duality takes restrictions to localizations and vice versa, along
  with reversing the variables in the Tutte polynomial.
\end{example}

\begin{corollary}\label{ih-and-top}
If $(\fM,\cD)$ is fat-tailed and $(\fM^!,\cD^!)$ is light-headed, then
the vector space $H^{\dim\fM}(\fM; \C)$ is dual to $I\! H_{\bT}^{\dim\fM^!}(\fM_0^!; \C)$.
\end{corollary}

\begin{proof}
Taking $S$ to be the minimal leaf $\{o\}$,
Proposition \ref{cohsd} says that $D_{\{o\}}$ is dual to 
$D^!_{\becircled\fM^!_0}/E^!_{\becircled\fM^!_0}$.
By fat-tailedness (Theorem \ref{point}), we have
$$D_{\{o\}} \cong K(\fP_{\{o\}})_\C \cong \Sigma(\cF_{\{o\}}) \cong H^{\dimfM}(\fM; \C).$$
By light-headedness (Theorems \ref{dense}), we have
$$D^!_{\becircled\fM^!_0}/E^!_{\becircled\fM^!_0} \cong K(\fP^!_{\becircled\fM^!_0})_\C\cong
\Sigma(\cF^!_{\becircled\fM^!_0})\cong I\! H_{\bT}^{\dim\fM^!}(\fM_0^!; \C).$$
This completes the proof.
\end{proof}

\begin{remark}\label{gr}
The vector space $I\! H_{\bT}^{\dim\fM}(\fM; \C)$ is naturally filtered,
with $k^\text{th}$ filtered piece $H^{\dim\fM - 2k}_\bT(pt; \C)\cdot I\! H_{\bT}^{2k}(\fM; \C)$, and the associated graded
is isomorphic to $I\! H^*(\fM; \C)$.  Thus Corollary \ref{ih-and-top} says that the dimension of the total intersection
cohomology of $\fM_0$ is equal to the dimension of the top homology of $\fM^!$.
If you have a pair of conical symplectic resolutions that you think might be dual, this is the first calculation that you should do.
\end{remark}

\begin{example}\label{awesome example}
Let $G$ be a simple algebraic group, simply laced and simply connected.
Fix a pair of $\la,\mu$ of dominant weights for $G$, and consider the Nakajima quiver variety $\tilde \fQ^\la_\mu$ 
that we discussed in Section \ref{sec:quiver-varieties}.
The top homology group of $\tilde \fQ^\la_\mu$ is isomorphic to the $V(\la)_\mu$, the $\mu$-weight space of the irreducible
representation $V(\la)$ \cite[10.2]{Nak98}.  

Let $G^L$ be the Langlands dual group, and consider the slice $\Gr^{\bar \la}_\mu$ in the affine Grassmannian for $G^L$
that we discussed in Section \ref{sec:affine-grassm-slic}.  This variety always admits a Hamiltonian $\bT$-action
fixing only the cone point, and the intersection cohomology group 
$I\! H_{\bT}^{\dim\Gr^{\bar \la}_\mu}(\Gr^{\bar \la}_\mu; \C)$ is also isomorphic to the $V(\la)_\mu$, the $\mu$-weight space of the irreducible representation $V(\la)$ \cite[3.11 \& 5.2]{G-GS}.

If $\la$ is a sum of minuscule weights, then there exists a symplectic resolution $\tilde \Gr^{\bar \la}_\mu$ 
of $\Gr^{\bar \la}_\mu$ and a Hamiltonian $\bT$-action
on $\tilde \fQ^\la_\mu$ with isolated fixed points, and we conjecture that $\tilde \Gr^{\bar \la}_\mu$
is symplectic dual to $\tilde \fQ^\la_\mu$ (Remark \ref{quiver-Gr}).  We know that $\tilde \fQ^\la_\mu$ is fat-tailed
(Example \ref{light-quiver}) and we conjecture that $\tilde \Gr^{\bar \la}_\mu$ is light-headed (Conjecture \ref{GrDE}).
Assuming these two conjectures, Corollary \ref{ih-and-top} would provide an isomorphism
\begin{equation}\label{quiver-satake}
H_{\dim \tilde \fQ^\la_\mu}(\tilde \fQ^\la_\mu; \C) \cong H^{\dim \tilde \fQ^\la_\mu}(\tilde \fQ^\la_\mu; \C)^*
\cong I\! H_{\bT}^{\dim\Gr^{\bar \la}_\mu}(\Gr^{\bar \la}_\mu; \C)
\end{equation} between the two geometric realizations
of $V(\la)_\mu$.

Note that if $G$ is of type A, then both $\tilde \fQ^\la_\mu$ and $\tilde \Gr^{\bar \la}_\mu$ are S3-varieties \cite{Maf, MV08}.
In this case, Conjecture \ref{GrDE} is covered by Example \ref{fat-light}, 
and the symplectic duality statement is Theorem \ref{S3-duality}. 
\end{example}

\begin{remark}
We defined a pairing between $K(\cO)$ and $K(\cO^!)$ by making the simple bases on each side
be orthonormal.  However, in the proof of Proposition \ref{cohsd}, we only used that they are orthogonal.
There is some evidence to suggest that there is a more natural pairing that is orthogonal but not orthonormal.

To define such a pairing, we suppose that there is a function  $\epsilon:\cI\to\{\pm 1\}$
such that for all $\a,\a'\in\cI$, $$\epsilon(\a) = \epsilon(\a') \implies \Ext^1(\Lambda_\a, \Lambda_{\a'}) = 0.$$
Such a function always exists for regular blocks of BGG category $\cO$ (it is given by the sign function on the Weyl group), 
as well as blocks
of hypertoric category $\cO$ (it is the same function that controls the orientations in \cite[4.3]{GDKD}).
We know of no examples for which such a function does not exist.
As long as $\cO$ is indecomposable, any such function would be unique up to a global sign.

Given such a function, define a new pairing by putting $$\left\langle[\Lambda_\a],[\Lambda^!_{\b^!}]\right\rangle_{\!\!\epsilon} := \epsilon(\a)\delta_{\a\b}.$$
In other words, we take the simples to form twisted orthonormal bases, where the twist is determined by $\epsilon$.
The reason that this pairing might be more natural is that the standards and the projectives would also form
twisted orthonormal bases.\footnote{We leave this statement as an exercise for the reader.
The proof uses the fact that the isomorphism $K(\tilde\cO)\cong K(\tilde\cO^!)$ induced by Koszul duality
takes $q$ to $-q^{-1}$.}
In contrast, the untwisted pairing that we originally defined is not well behaved
with respect to projectives or standards.

We conjecture that the isomorphism in Equation \eqref{quiver-satake} between the two geometric realizations
of $V(\la)_\mu$ will only be compatible with the action of $G$ if we use this twisted pairing.
\end{remark}

\begin{example}\label{Skew-Howe duality}
In type A, the relationship between quiver variety geometry and affine Grassmannian geometry is somewhat special.  In particular, as observed by \Mirkovic-Vybornov, 
Nakajima quiver varieties of type A also arise as transverse slices in the affine Grassmannian of type A \cite{MV08}.  This geometric phenomenon is reflected algebraically in skew-Howe duality for representations of type A simple Lie algebras.  More precisely, this duality asserts that there is a canonical vector space isomorphism between weight spaces in representations of $\mathfrak{gl}_m$ and multiplicity spaces in representations of $\mathfrak{gl}_n$ in the $\mathfrak{gl}_m \times  \mathfrak{gl}_n$-module $\bigwedge{}^{\!\!N}(\C^n\otimes \C^m)$.  These weight spaces and multiplicity spaces can be interpreted geometrically using either the  Ginzburg-Nakajima construction or the geometric Satake construction.
The various geometric realisations of weight and multiplicity spaces are then summarized in the following diagram:\\

\centerline{\begin{tabular}{l|ll}
&  Geometric Satake & Ginzburg-Nakajima\\ \hline
Multiplicity space for $\mathfrak{gl}_m$&$H_{\dim \fM}(\fM;\C)$ & $\IH_{\T^!}^{\dim \fM^!}(\fM_0^!;\C)$\\
Weight space for $\mathfrak{gl}_n$ & $\IH_{\T^!}^{\dim \fM^!}(\fM_0^!;\C)$ & $H_{\dim \fM}(\fM;\C)$
\end{tabular}}\bigskip

\noindent
The observation of \Mirkovic-Vybornov is essentially that the varieties in the diagonal of the above matrix are identical.
Our additional observation, which is special to type A, is that the varieties in a given row (or column) will be symplectic duals.
Moreover, the isomorphism between the vector spaces in a given row (or column) is a cohomological consequence of symplectic duality; see Remark \ref{gr}. Thus, in type A, symplectic duality exchanges the geometric Satake realization of skew-Howe duality with the Ginzburg-Nakajima realization of skew-Howe duality.  
\end{example}

\begin{example}\label{Level-rank duality}
For affine type $\widehat{A}$ quiver varieties, symplectic duality recovers a duality in representation theory due to I.B. Frenkel known as level-rank duality.  The situation is quite reminiscent of that of skew-Howe duality in finite type A, and is discussed in more detail in \cite[\S 5.2]{Webqui}.  In this case, the commuting actions of the affine Lie algebras $\widehat{\mathfrak{sl}}_n$ and $\widehat{\mathfrak{sl}}_m$ on the semi-infinite wedge space give rise to canonical identifications between weight spaces in a level $m$ representation of $\widehat{\mathfrak{sl}}_n$ and multiplicity spaces in a level $n$ representation of $\widehat{\mathfrak{sl}}_m$.  (See, for example, equation A.5 in \cite{NakBranching}.)  The geometric Satake construction of representations in affine type is not completely understood.  However, in affine type $\widehat{A}$, Braverman-Finkelberg \cite{BF} have an explicit proposal for a geometric Satake construction of representations.  We then have the following diagram of geometric realisations of representations:\\

\centerline{\begin{tabular}{l|ll}
&  Geometric Satake & Nakajima\\ \hline
Multiplicity space in a level $n$ rep. of $\widehat{\mathfrak{sl}}_m$ &$H_{\dim \fM}(\fM;\C)$ & $\IH_{\T^!}^{\dim \fM^!}(\fM_0^!;\C)$\\
Weight space in a level $m$ rep. of $\widehat{\mathfrak{sl}}_n$  & $\IH_{\T^!}^{\dim \fM^!}(\fM_0^!;\C)$ & $H_{\dim \fM}(\fM;\C)$
\end{tabular}}\bigskip

\noindent
Here $\fM$ and $\fM^!$ are affine type $\widehat{A}$ quiver varieties.  The Nakajima column realizes the weight and multiplicity spaces as homology groups of Nakajima quiver varieties.  That the diagonal (resp. off-diagonal) entries in Geometric Satake column involve the same variety is part of the content of the Braverman-Finkelberg proposal for geometric Satake in affine type $\widehat{A}$.  
Our additional observation is that the varieties in a given row (or column) will be symplectic duals.
\end{example}

\begin{remark}
Etingof and Schedler \cite[1.3.1(b)]{ES11} conjecture that $H_{\dim\fM^!}(\fM^!; \C)$ has the same dimension
as the zeroth Poisson homology group $H\! P_0(\fM^!_0)$, which is defined as the ring of functions modulo the linear
span of all Poisson brackets.  By Corollary \ref{ih-and-top} and Remark \ref{gr}, 
this is equivalent to the conjecture that $I\! H^*(\fM; \C)$ has the same dimension as $H\! P_0(\fM^!_0)$.
In \cite[3.4]{Pro12}, we strengthen this conjecture by proposing
that they should be isomorphic as graded vector spaces (the grading on Poisson homology is induced by the $\bS$-action).

Furthermore, $I\! H^*(\fM; \C)$ admits the natural deformation $I\! H_{T}^*(\fM; \C)$ over $\mt$,
and $H\! P_0(\fM^!_0)$ admits the natural deformation $H\! P_0(\scrN^!)$ over $H^2(\fM^!)$ 
(recall that $\scrN^! := \Spec \C[\scrM^!]$ is a Poisson deformation of $\fM^!_0$ over $H^2(\fM^!; \C)$).
Finally, recall that we have $\mt\cong H^2(\fM^!; \C)$ as part of the package of symplectic duality, therefore
the two deformations share the same base.
The conjecture in \cite[3.4]{Pro12} asserts that these two deformations should be isomorphic, as well;
it is proven for hypertoric varieties \cite[3.1]{Pro12}.
\end{remark}

\subsection{Duality of localization algebras}\label{gm}
In this section we extend Conjecture \ref{HH} in order to relate
symplectic duality to a cohomological duality
first studied by Goresky and MacPherson \cite{GM}, and later by the authors \cite{kosdef}.  Let $\fM$ and $\fM^!$
be a symplectic dual pair.

Recall that we define $E$ to be the Yoneda algebra of $\cO$, and we conjecture that its center 
is isomorphic to $H^*(\fM; \C)$ (Conjecture \ref{HH}).  Let $\tilde E$ be the universal deformation of $E$,
as defined in \cite[4.2]{kosdef}; this is a flat deformation over the base $Z(E^!)_2^*$, the dual of the degree 2
part of the center of $E^!$.  Conjecture \ref{HH} for $\fM^!$ says that the center of $E^!$ is isomorphic to $H^*(\fM^!; \C)$,
which means that the base of the universal deformation is isomorphic to $H_2(\fM^!; \C)$.  As part of the package
of symplectic duality, this is isomorphic to $\mt$.

\begin{conjecture}\label{HH-eq}
The graded ring $Z(\tilde E)$ is isomorphic to $H^*_T(\fM; \C)$.
\end{conjecture}

\begin{remark}
Conjecture \ref{HH-eq} is a natural extension of Conjecture \ref{HH}; if you believe that $Z(E)$ is isomorphic
to $H^*(\fM; \C)$, and they both admit canonical flat deformations over the base $\mt$, it is natural to guess that
these deformations are the same.  Unfortunately, unlike in Conjecture \ref{HH}, we do not have a geometric definition
of a map from $H^*_T(\fM; \C)$ to $Z(\tilde E)$.
\end{remark}

\begin{remark}
Conjecture \ref{HH-eq} holds for hypertoric varieties \cite[8.5]{kosdef} and for S3-varieties in type A \cite[9.9]{kosdef}.
\end{remark}

In their paper \cite{GM}, Goresky and MacPherson observe a somewhat mysterious cohomological
relationship between certain pairs of varieties with torus actions.  Below we will describe this relationship,
and demonstrate that it is a consequence of symplectic duality and Conjecture \ref{HH-eq}.

Consider the ring homomorphisms
$$\Sym\mt^*\,\,\hookto\,\,\Sym H^2_T(\fM; \C)\,\,\to\,\, H^*_T(\fM; \C)\,\,\hookto\,\, 
H^*_T(\fM^T; \C)\,\,\cong\,\,\bigoplus_{\a\in\cI}H_T^*(p_\a; \C).$$
Dualizing, we obtain maps
$$\mt\,\,\twoheadleftarrow\,\, H_2^T(\fM; \C)\,\,\leftarrow\,\,\Spec H^*_T(\fM; \C)\,\,\twoheadleftarrow\,\,\bigsqcup_{\a\in\cI}\mt.$$
Let $H_\a\subset H_2^T(\fM; \C)$ be the image of the copy of $\mt$ indexed by $\a\in\cI$.  This is a linear subspace
that projects isomorphically to $\mt$ via the left-most map, and the union of these linear subspaces is equal
to the image of $\Spec H^*_T(\fM; \C)$ in $H_2^T(\fM; \C)$.

\begin{theorem}\label{gm-thm}
There is a canonical perfect pairing between $H_2^T(\fM; \C)$ and $H_2^{T^!}\!(\fM^!; \C)$.  Assuming 
that Conjecture \ref{HH-eq} holds, then:
\begin{itemize}
\item the kernel of the projection from $H_2^T(\fM; \C)$ to $\mt$ is the perp space to
the kernel of the projection from $H_2^{T^!}\!(\fM^!; \C)$ to $\mt^!$
\item for all $\a\in\cI$, $H_\a\subset H_2^T(\fM; \C)$ is the perp space to
$H^!_{\a^!}\subset H_2^{T^!}\!(\fM^!; \C)$.
\end{itemize}
\end{theorem}

\begin{proof}
This result follows immediately from the Koszul duality of $\cO$ and $\cO^!$
using \cite[1.2]{kosdef}. 
\end{proof}

\begin{remark}
The phenomenon in Theorem \ref{gm-thm} was observed by Goresky and MacPherson \cite[\S 8]{GM}
for the varieties described in Remark \ref{GMS3} (a special case of type A S3-varieties).
The connection to symplectic duality allowed us to find many new examples, such as hypertoric varieties
and more general type A S3-varieties.
\end{remark}

\subsection{Knot homologies and symplectic duality}\label{knot}
There are close relationships between the representation theory of quantum groups and low-dimensional topology.  Perhaps the best known example of such a relationship is the construction of Reshetikhin-Turaev \cite{RT} invariants of links in $S^3$; these invariants are defined using the braiding on the category of $U_q(\mathfrak{g})$ modules.  Much of the subsequent work categorifying the representation theory of quantum groups has been motivated by the desire to lift the polynomial invariants of Reshetikhin-Turaev to richer homological invariants.
The best known such lift -  at least to representation theorists - is Khovanov's $U_q(\mathfrak{sl}_2)$ link homology \cite{Kho00}, which categorifies the Jones polynomial.  

On the other hand, representations of quantum groups (and their categorifications) can be constructed using the geometry of conical symplectic resolutions. As a result, one expects to obtain constructions of knot homologies using the geometry of conical symplectic resolutions or the representation theory of Koszul algebras.  Indeed, such constructions have been studied by many authors in the last ten years.  The examples most closely related to the geometry of conical symplectic resolutions include:
\begin{itemize}
\item Sussan's algebraic construction of $U_q(\mathfrak{sl}_n)$ link homologies (for fundamental weights) using a graded version of parabolic BGG category $\cO$ for $\mathfrak{gl}_m$ \cite{Sussan};
\item Mazorchuk-Stroppel's algebraic construction of $U_q(\mathfrak{sl}_n)$ link homologies (for fundamental weights) using a graded version of singular BGG category $\cO$ for $\mathfrak{gl}_m$ \cite{MS-link};
\item Cautis-Kamnitzer's geometric constructions of $U_q(\mathfrak{sl}_n)$ link homologies (for fundamental weights) using the affine Grassmannian \cite{CK1,CK2,KamBD};

\item Seidel-Smith and Manolescu's constructions of Khovanov homology using Fukaya categories of type A quiver varieties \cite{Seidel-Smith, Manolescu};
\item Cautis's geometric construction of of $U_q(\mathfrak{sl}_n)$ link homologies (for all weights) using derived categories of coherent sheaves and subvarieties of the affine Grassmannian \cite{Cauclasp};
\item Smith-Thomas and Cautis-Licata's constructions of Khovanov homology using derived categories of coherent sheaves on Hilbert schemes of points on ALE spaces
\cite{Thomas,Cautis-Licata};
\item the fourth author's construction of $U_q(\mathfrak{g})$ link homologies, categorifying the entire family of RT polynomial link invariants \cite{Webqui} using category $\cO$ for quiver varieties.
\end{itemize} 

The geometry underlying these categories is not the same in each
case.  Rather, it seems to come in two different flavors, which are
related to two different ways of geometrizing a representation of
$\mathfrak{g}$: via quiver varieties, or via the affine Grassmannian
for the Langlands dual group ${}^LG$.  The work of Seidel-Smith, Manolescu, and Cautis-Kamnitzer
is on the affine Grassmannian side, while the work of the fourth author is on
the quiver variety side.  Philosophically, all of these approaches
involve defining a braid group action on certain geometrically defined
categories, and adding some special functors which can be used to
define cups and caps.  In all cases, we can interpret these data as
coming from ideas that we have discussed.

On the affine Grassmannian side,  the relevant braid group action comes from
twisting functors, while cups and caps arise from
Lagrangian correspondences; in very loose terms, one should think of
them as versions of pushforward and pullback from leaves.
  \begin{itemize}
  \item Seidel-Smith and Manolescu work in the Fukaya category. Their
    braid group action comes from parallel transport in a space of
    complex structures, which we interpret as a Fukaya version of
    twisting.  Their cup and cap functors are the functors on quilted
    Fukaya categories induced by natural Lagrangian correspondences.  
\item Cautis-Kamnitzer work with coherent sheaves on convolution
  varieties for affine Grassmannians, but nothing is lost by replacing
  these varieties with certain open subsets which are conical symplectic
  resolutions (they are also $\mathfrak{sl}_\infty$
  quiver varieties).  After this modification, their braid group
  action is obtained via tensor produces with associated gradeds of twisting bimodules, and
  their cups and caps via tensor products with the associated gradeds of certain Harish-Chandra bimodules.  
  Cautis has explained how their
  construction is a special case of a general construction from a
  categorical $\mathfrak{sl}_\infty$ action, and the quantum version
  of Cautis-Kamnitzer arises from applying this to
  $\mathfrak{sl}_\infty$ quiver varieties as in the paper \cite{Webcatq}.
  \end{itemize}

On the other hand, from the quiver variety perspective, the braid
group action one uses comes from shuffling functors, as suggested by
Nakajima's work on tensor product quiver varieties.  The cup and cap
functors are harder to describe in this case, but should be thought of
as some sort of restriction to slices.  The two different approaches to categorified knot invariants
can be summarized by the following table.\\

\centerline{\begin{tabular}{l|ll}
&  Affine Grassmannian & Quiver varieties\\ \hline
braid actions & twisting & shuffling\\
cups and caps & push to/pull from leaves & push to/pull from slices\\
examples & \cite{CK1,CK2,Cauclasp,KamBD,Seidel-Smith, Manolescu} & \cite{Webmerged,Webqui}
\end{tabular}}

\vspace{\baselineskip}
Strikingly, while these two contexts look very different,  the
basic geometric concepts involved (twisting/shuffling, leaves/slices)
are interchanged by our conjectural duality.  Thus, the fact that
these two constructions exist and give the same knot invariants in
type A serves as a powerful piece of evidence for our conjecture.
In other types, there is no construction which yet exists on the affine
Grassmannian side.

In type A, these knot homologies also have Koszul dual realizations,
due to Sussan and Mazorchuk-Stroppel; these fit with the left and
right hand columns of the table above, since the former uses twisting
functors for the braid group action and the latter uses shuffling functors.

%Indeed, the constructions of Sussan and Mazorchuk-Stroppel are
%related by Koszul duality, which exchanges the twisting and shuffling
%functors used in defining the braid group actions which give rise to
%these knot homologies.  Though the geometric constructions above
%sometimes use coherent sheaves or Fukaya categories, it is reasonable
%to expect that these constructions can also be reformulated using
%twisting or shuffling functors on geometric category $\cO$. After
%such a reformulation, the affine Grassmannian constructions of link
%homologies and Webster's quiver variety constructions should then be
%related, at least conjecturally, by the symplectic duality in Remark
%\ref{quiver-Gr}.  In fact, the constructions of Khovanov homology
%using Hilbert schemes of points on ALE spaces suggests the existence
%of a further dual construction, using moduli spaces of framed sheaves
%on $\mathbb{P}^2$, though to the best of our knowledge such a construction has not been given yet.  

\appendix
\section{An Ext-vanishing result (appendix by Ivan Losev)}

\newcommand{\Dcal}{\cD}
\newcommand{\param}{P}
\newcommand{\eu}{h\xi}
\renewcommand{\hbar}{h}
\newcommand{\Tor}{\operatorname{Tor}}
\newcommand{\quo}{/\!/}

Let $\param$ be a vector space over $\C$ equipped with a linear map
$\param\to \Ht$ whose image does not lie in any of the discriminant
hyperplanes. Recall that $\scrM$ is the universal deformation of
$\fM$, and $\scrN$ the affinization of this variety.  Let  $\scrN_\param$ be the fiber product
$\mathscr{N}\times_{\Ht}\param$. The fiber product $\scrM_\param=\mathscr{M}\times_{\Ht}\param$ is a fiberwise symplectic resolution of singularities for $\scrN_\param$;
let $\omega_\param$ be the fiberwise symplectic form on $\scrM_\param$.

Let $\Dcal_{\param,\hbar}$ be a $\bT\times \bS$-equivariant formal
quantization of $\scrM_\param$.  This means that $\Dcal_{\param,\hbar}$ is a $\bT\times \bS$-equivariant sheaf of $\C[\param][[\hbar]]$-algebras (that are flat over
$\C[[\hbar]]$ and are complete and separated in the $\hbar$-adic topology) with a fixed
isomorphism
$\theta:\Dcal_{\param,\hbar}/(\hbar)\xrightarrow{\sim}\fS_{\scrM_P}$. % such that
%\begin{itemize}
%\item $\bT$ fixes $\hbar$, while $\bS$ rescales $\hbar$: $t.\hbar=t^2\hbar$ for $t\in \bS$.
%\item The isomorphism $\theta$ is a $\C[\param]$-linear, $\bT\times \bS$-equivariant isomorphism of Poisson sheaves.
%\item 
As before, we choose $\eu\in \Gamma(\scrM_\param,
\Dcal_{\param,\hbar})$ such that $\xi:=h^{-1}(\eu)$ is a
non-commutative moment map for the action of $\bT$ on
$\cD_{P}:=\cD_{P,\hbar}[h^{-1}]$.  
%\end{itemize}

%We also assume that there is a lattice $\Gamma$ with a embedding $\Gamma\hookrightarrow \operatorname{Pic}(\scrM_0)$
%such that for any $\gamma\in \Gamma$ and every $p\in \param$ there is a $\Dcal_{,\hbar}$-$\Dcal_{p+\gamma,\hbar}$-bimodule
%$\Dcal_{p,\gamma,\hbar}$ such that $\Dcal_{p,\gamma,\hbar}$

Let $\A_{\param,\hbar}$ denote the subalgebra of all $\bS$-finite elements in $\Gamma(\scrM_\param, \Dcal_{\param,\hbar})$.
This is an algebra over $\C[\param][\hbar]$ equipped with an action of $\bT\times \bS$ by automorphisms
such that $\A_{P,0}:=\A_{P,\hbar}/(\hbar)$ is identified with $\C[\scrM_\param]=\C[\scrN_\param]$.
For $p\in \param$ we set $\A_{p,0}:=\A_{P,0}/ (\ker p)$, where we view
$p$ as a homomorphism $\C[\param]\rightarrow \C$; this is the space of
functions on the fiber of $\scrM_\param$ over the point $p$.  

Set $\A_{\param}:=\A_{\param,\hbar}/(\hbar-1)\cong
\A_{\param,\hbar}[\hbar^{-1}]^{\bS}, \A_{p}:=\A_{\param}/(\ker p)$. We
have  gradings on the algebras $\A_{\param,\hbar},\A_{\param,0}, \A_{\param}, \A_{p,0},\A_{p}$ coming from the $\bT$-action. The $i$th graded components will be denoted
by $\A_{\param}(i), \A_{\param,0}(i)$, etc. We remark that on $\A_{\param},\A_{p}$ this grading is
inner -- it comes from the inner derivation $[\xi,\cdot]$.
%, where $\eu$ denotes the image of $\eu\in \A_{\param,\hbar}$under
%the epimorphisms to the corresponding algebras. 
%We set \begin{align*}&\A_{?}(>0):=\bigoplus_{i>0}\A_{?}(i), \quad
%\A_{?}(\geqslant 0):=\bigoplus_{i\geqslant 0}\A_?(i),\\ &\A_{?}(\geqslant 0)^+:=\A_{?}(\geqslant 0)\cap \A_{?}\A_{?}(>0),\quad \A_?^0:=\A_{?}(\geqslant 0)/\A_?(\geqslant 0)^+,\end{align*}
%where ``$?$'' means ``$\param$'', ``$\param,0$'', etc.

%Similarly to \cite{W_OCat},\cite{toric},\cite{GL},  
As in Section \ref{sec:Oa} (and in previous works such as
\cite{LosO,BLPWtorico,GoLo}), we can consider the full subcategory
$\mathcal{O}_p$ in the category $\A_{p}\mmod$ of finitely generated
modules consisting of all modules where $\A_{p}^{+}$ acts locally
finitely.

%The main result of this note is as follows.

%\begin{theorem}\label{Thm:ext_coinc}
%There is a Zariski open subset $\param^0$ in the hyperplane $\hbar=1$ such that the natural functor $D^b(\mathcal{O}_p)\rightarrow D^b(\A_{p}\mmod)$ is a fully faithful embedding for each $p\in \param^0$.
%\end{theorem}

%In the proof we will make a number of assumptions that hold in applications.

%In fact, we will prove a formally weaker statement.
Let $\Delta_{\param}$ denote the left $\A_{\param}$-module $\A_{\param}/\A_{\param}\A_{\param}^{>0}$. We use the notation
$\Delta_{?}$ for various specializations of $\Delta_\param$. We remark that $\Delta_{p}$ is an object in $\mathcal{O}_p$.
Now consider the right $\A_\param$-module $\nabla_{\param}^\vee:=\A_{\param}/\A_{\param}^{<0}\A_{\param}$. The $\bT$-grading on $\A_{\param}$
induces a $\bT$-grading on $\nabla_\param^\vee$. We can consider the  specializations $\nabla^\vee_{?}$ of $\nabla^{\vee}_\param$.
We will see below that, for any $(p,h)\in \param\oplus\C$, all graded components $\nabla^\vee_{p,h}(i)$ are finite dimensional. Let $\nabla_{p,h}$ denote the restricted dual of $\nabla^\vee_{p,h}$, i.e., $\nabla_{p,h}=\bigoplus_i \left(\nabla^\vee_{p,h}(i)\right)^*$.
We have $\nabla_{p,h}(i)=0$ for $i>0$. So $\nabla_{p}$ lies in the ind completion of $\mathcal{O}_p$.
In fact, one can show that $\nabla_{p}$ is finitely generated and so lies in $\mathcal{O}_p$.

The purpose of this appendix is to prove that:
\begin{theorem}\label{Thm:ext_coinc1}
There is a non-empty Zariski open subset $\param^0$ in the hyperplane $\hbar=1$ such that $\Ext^i_{\A_{p}}(\Delta_{p},\nabla_{p})=0$
for $p\in \param^0$ and $i>0$.
\end{theorem}
There is a Zariski open subset of $\param$ where derived localization
holds. Since the algebra $\A_p$ has finite global dimension on this
open subset, the vanishing of the coherent
sheaves $\Ext^i(\Delta_P,\nabla_P)$ at $p$ for all $i$ is an open condition. The content of
this theorem is that this set is non-empty.  In fact, it contains all
but finitely many points of any affine line not parallel to the discriminant locus.

%Theorem \ref{Thm:ext_coinc} actually follows from Theorem
%\ref{Thm:ext_coinc1} by a standard argument given in \cite[3.3.2]{BGS96}.
%We are going to make some assumptions that presumably hold in our setting. We  assume that there is a lattice $\Gamma\subset \param$ with the following property: there is Zariski open subset $\param^0\subset \param$ such that for each $p\in \param^0$ there are $\gamma\in \Gamma$ and invertible $\bT$-equivariant $\A_{p+n\gamma}$-$\A_{p}$-bimodules $\A_{p, n\gamma}$ for all $n\gg 0$. Clearly, the tensor
%product with $\A_{p,n\gamma}$ induces an equivalence $\mathcal{O}_{p}\xrightarrow{\sim}\mathcal{O}_{p+n\gamma}$. We assume in addition
%that this equivalence maps $\Delta_{p}$ to $\Delta_{p+n\gamma}$ and $\nabla_{p}$ to $\nabla_{p+n\gamma}$. We will
%see that Theorem \ref{Thm:ext_coinc1} works for $\param^0$. The crucial property here is that
%$\Ext^i_{\A_{p}}(\Delta_{p},\nabla_{p})=\Ext^i_{\A_{1,p+n\gamma}}(\Delta_{1,p+n\gamma},\nabla_{1,p+n\gamma})$.
One corollary of our choice of $p$ is that the category
$\mathcal{O}_p$ becomes highest weight. Its standard objects are
indecomposable direct summands of $\Delta_{p}$ and the costandard
objects are indecomposable direct summands of $\nabla_{p}$. In
particular,
$\Ext_{\A_{p}}^1(\Delta_{p},\nabla_{p})=\Ext^1_{\mathcal{O}_p}(\Delta_{p},\nabla_{p})=0$.
Another corollary is that the algebra $\A_{p}$ has finite homological
dimension not exceeding $\dim \fM$.

We also would like to mention that our proof of Theorem
\ref{Thm:ext_coinc1} is inspired by the proof of an analogous
statement for Rational Cherednik algebras, see \cite{EtiSRA}.

%\section{Preliminary considerations}

\subsection{The proof}

We start with an easy lemma that is analogous to \cite[3.1.4]{GoLo}.
The proof is precisely like that of Lemma \ref{fdba}:

\begin{lemma}
The graded components of $\Delta_{\param},\nabla_{\param}^\vee$ are finitely generated $\C[\param]$-modules.\qed
\end{lemma}
Next, we will need a structural result related to symplectic $\C^\times$-actions. Let $\bT$ act on a smooth affine symplectic variety $X$ with finitely many fixed points. Let $x$ be one of the fixed points.
Then the linear action of $\bT$ on $T_xX$ is symplectic.

%Choose a $\bT$-eigenbasis $x_1,\ldots,x_m,y_1,\ldots,y_m$ of $T_xX$ corresponding to characters $\chi_1,\ldots,\chi_m,
%-\chi_1,\ldots,-\chi_m$, where $\chi_i>0$. Let $y$ be the image of $x$ in $X\quo \bT$. Consider the partial completion
%$\C[X]^{\wedge_y}:=\C[X\quo \bT]^{\wedge_y}\otimes_{\C[X\quo T]}\C[X]$. We have a natural homomorphism
%$\C[X]\rightarrow \C[X]^{\wedge_y}$ that is easily seen to be injective. Let us remark that the ideal in $\C[X]^{\wedge_y}$
%generated by eigenvectors of $\bT$ with positive characters is just $\C[X]^{\wedge_y}\C[X](>0)$.

\begin{lemma}\label{Lem:generat}
 There are homogeneous elements $x_1,\ldots,x_m\in \C[X]^{>0},
 m=\frac{1}{2}\dim X$ such that the differentials $d x_1,\ldots,d x_m$
 are linearly independent at the point $x$. Moreover, the ideal
 $\C[X]\C[X]^{>0}$ is a locally complete intersection  generated in a neighborhood
 of $x$ by the elements $x_1,\ldots,x_m$.
\end{lemma}
\begin{proof}
This is a standard fact that can be deduced, for example, from the Luna slice theorem.
\end{proof}

Now we are going to reinterpret the $\Ext$'s between $\Delta_{p}$ and $\nabla_{p}$ in terms of $\Tor$'s between
$\Delta_{p}$ and $\nabla^\vee_{p}$.

\begin{lemma}
We have $\Ext^i(\Delta_{p,h}, \nabla_{p,h})= (\Tor_{i}(\nabla_{p,h}^\vee, \Delta_{p,h}))^*$. Here both the $\Ext$'s
and the $\Tor$'s are taken over $\A_{p,h}$.
\end{lemma}
The spaces $\Tor_{i}(\nabla_{p,h}^\vee, \Delta_{p,h})$ are graded (via the $\bT$-action) and the graded
components are finite dimensional. This is because both modules are finitely generated and their graded
components are finite dimensional. The superscript $^*$ means the restricted dual.
%In fact, below we will see that the $\Tor$'s are finite dimensional.
\smallskip

\begin{proof}
Let $(P_\bullet,d)$ be a free $\bT$-equivariant resolution  of $\Delta_{p,h}$. Then the $\Ext$'s in interest are
the cohomology of the complex $\Hom(P_\bullet, \nabla_{p,h})=P_\bullet^*\otimes_{\A_{p,h}}\nabla_{p,h}$ and the
differential is $d^*$. But $P_\bullet^*\otimes_{\A_{p,h}}\nabla_{p,h}= (P_\bullet\otimes_{\A_{p,h}}\nabla_{p,h}^\vee)^*$
and the differential $d^*$ is the dual of the differential on the complex $P_\bullet\otimes_{\A_{p,h}}\nabla_{p,h}^\vee$.
The cohomology of the latter complex are the $\Tor$'s. Since the restricted duality $^*$ is an exact functor, we are done.
\end{proof}

So we only need to prove the analog of Theorem \ref{Thm:ext_coinc1} for $\Tor^{\A_{p}}_{i}(\nabla^\vee_{p},\Delta_{p})$.
First we are going to understand the behavior of the $\Tor$'s for $h=0$.

Let $\ell$ be a line in $\param$ not lying in the discriminant hyperplanes.
Let $M=\bigoplus_{i\in \Z} M(i),N=\bigoplus_{i\in \Z} N(i)$ be $\bT$-equivariant $\A_{\ell,0}$-modules with the following properties.

\begin{itemize}
\item[(i)] $M$ and $N$ are finitely generated.
\item[(ii)] $M(i)=N(-i)=0$ for $i\gg 0$ and $ M(i), N(i)$ are finitely generated $\C[\ell]$-modules for all $i$.
\item[(iii)] The specializations of $M,N$ at $p\in \ell\setminus\{0\}$ are $\Delta_{p,0},\nabla^\vee_{p,0}$, respectively.
\end{itemize}

\begin{proposition}\label{Lem:Ext_vanish1}
$\dim \Tor_i^{\A_{\ell,0}}(M,N)<\infty$ for all $i>0$.
\end{proposition}
\begin{proof}
The proof is in several steps.

{\it Step 1.} Let us show that $\Tor_i^{\A_{\ell,0}}(M,N)$ is a finitely generated $\C[\ell]$-module for any $i$.

First of all, we claim that the $\A_{\ell,0}$-module $\Tor_i^{\A_{\ell,0}}(M,N)$ is supported on the $\bT$-fixed
point set $\scrM_\ell^{\bT}$. Indeed, the condition $M(i)=0$ for $i\gg 0$ implies that $M$ is supported
on the repelling locus for the $\bT$-action. Similarly, the condition $N(i)=0$ for $i\ll 0$ implies that $N$ is supported
on the contracting locus. The intersection of the two loci is precisely the set of $\bT$-fixed points.
Let $x\in \scrN_{\ell}$ be a point that is not $\bT$-fixed. Form the completion $\A_{\ell,0}^{\wedge_x}$
of $\A_{\ell,0}$ at $x$, this algebra is flat over $\A_{\ell,0}$. From the above remarks,
$\A_{\ell,0}^{\wedge_x}\otimes_{\A_{\ell,0}}M=0$ or $\A_{\ell,0}^{\wedge_x}\otimes_{\A_{\ell,0}}N=0$.
Then we have $\A_{\ell,0}^{\wedge_x}\otimes_{\A_{\ell,0}}\Tor_i^{\A_{\ell,0}}(M,N)=
\Tor_i^{\A_{\ell,0}^{\wedge_x}}(\A_{\ell,0}^{\wedge_x}\otimes_{\A_{\ell,0}}M,
\A_{\ell,0}^{\wedge_x}\otimes_{\A_{\ell,0}}N)=0$. Our claim in the beginning of this paragraph follows.

Now the original claim of this step will follow if we check that the morphism $\scrN_{\ell}^{\bT}\rightarrow\ell$ is finite.
The latter is a consequence  of the following two observations. First, the morphism $\scrM_{\ell}^{\bT}\rightarrow
\scrN_{\ell}^{\bT}$ is proper, surjective and locally finite. The last property follows from the assumption that
$\bT$ has finitely many fixed points on each $\scrM_p$. Second, the
morphism $\scrM_{\ell}^{\bT}\to \scrN_{\ell}^{\bT}$
is also surjective, proper and locally finite.

{\it Step 2.} Let us show that $\Tor_{i}^{\A_{p,0}}(\Delta_{p,0}, \nabla^\vee_{p,0})=0$ for $i>0$.
Similarly to Step 1, the $\A_{p,0}$-module  $\Tor_{i}^{\A_{p,0}}(\Delta_{p,0}, \nabla^\vee_{p,0})$
is supported at the $\bT$-fixed points. Pick such a point $x$. It is enough to show that the localization
$\Tor_{i}^{\A_{p,0}}(\Delta_{p,0}, \nabla^\vee_{p,0})_x$ is zero. Let $x_1,\ldots,x_m$ be as in
Lemma \ref{Lem:generat} and let $y_1,\ldots,y_m$ be similar elements but with negative eigen-characters. Then
$d_xx_1,\ldots,d_x x_m, d_x y_1,\ldots,d_x y_m$ form a basis in $T_xX$. Locally, $\Delta_{p,0}$
is $\A_{p,0}/(x_1,\ldots,x_m)$, while $\nabla_{p,0}^\vee$ is $\A_{p,0}/(y_1,\ldots,y_m)$.
Then the $\Tor$ vanishing is a standard fact.
%The characters of $\bT$ on $T_x\scrM_p$
%are of the form $\chi_1,\ldots,\chi_m>0, -\chi_1,\ldots,-\chi_m$. It is standard fact -- which can be deduced,
%say, from the Luna slice theorem -- that locally in an \'{e}tale neighborhood of $x$ the $\bT$-action
%is linear with characters $\pm \chi_1,\ldots,\pm \chi_m$. Moreover, the ideal $\A_{p,0}\A_{p,0}(>0)$
%is generated by the linear functions $x_1,\ldots,x_m$ (corresponding to the characters $\chi_1,\ldots,\chi_m$),
%while $\A\A(<0)$ is generated by the linear functions $y_1,\ldots,y_m$. So we have
%\begin{align*}
%&\Tor_i^{\A_{p,0}}(\Delta_{p,0},\nabla^\vee_{p,0})_x= \Tor_i^{\A_{p,0}}(\A_{p,0}/\A_{p,0}\A_{p,0}(>0), \A_{p,0}/\A_{p,0}\A_{p,0}(<0))_x=\\
%&\Tor_i^{\C[x_1,\ldots,x_m,y_1,\ldots,y_m]}(\C[x_1,\ldots,y_m]/(x_1,\ldots,x_m), \C[x_1,\ldots,y_m]/(y_1,\ldots,y_m)).
%\end{align*}
%The latter Tor vanishes. This is a standard result that can be proved using Koszul resolutions.

{\it Step 3.} Let $\mathfrak{m}_p$ denote the maximal ideal of $p\in \ell$ in $\C[\ell]$. The previous
step implies $\Tor_{i}^{\A_{p,0}}(M/\mathfrak{m}_p M, N/\mathfrak{m}_p N)=0$ for $i>0$. But $M,N$
are flat over $\C[\ell\setminus \{0\}]$ by our assumptions. It follows that
the completion of $\Tor_i^{\A_{\ell,0}}(M,N)$ at $p$ is zero. Together with Step 1, this implies that
$\Tor_i^{\A_{\ell,0}}(M,N)$ is finite dimensional for  $i>0$.
\end{proof}

\begin{proposition}\label{Lem:Ext_vanish2}
Let $q\in \param$. Then 
$\dim \Tor_i^{\A_{q+\ell}}(\nabla_{q+\ell}^\vee,\Delta_{q+\ell})<\infty$ for all $i>0$.
\end{proposition}
\begin{proof}
The proof is again in several steps.

{\it Step 1.} Our goal is to deduce Proposition \ref{Lem:Ext_vanish2} from Proposition \ref{Lem:Ext_vanish1}.

Let us notice that $\gr \A_{q+\ell}=\A_{\ell,0}$. The space $\Tor_i^{\A_{q+\ell}}(\Delta_{q+\ell},\nabla_{q+\ell}^\vee)$
inherits a filtration from $\A_{q+\ell}$. We can lift a $\bT\times \bS$-graded free resolution for $\gr \Delta_{q+\ell}$
to a free $\bT$-graded resolution of $\Delta_{q+\ell}$. This easily implies that $\gr \Tor_i^{\A_{q+\ell}}(\Delta_{q+\ell},\nabla_{q+\ell}^\vee)$ is a subquotient of $\Tor_i^{\A_{\ell,0}}(\gr \Delta_{q+\ell}, \gr\nabla_{q+\ell}^\vee)$, where the filtrations on $\Delta_{q+\ell},\nabla_{q+\ell}^\vee$ are induced
from $\A_{q+\ell}$. So it is enough to show that $\Tor_i^{\A_{\ell,0}}(\gr \Delta_{q+\ell}, \gr\nabla_{q+\ell}^\vee)$
is finite dimensional for all $i>0$.

We are going to check the last claim using Proposition \ref{Lem:Ext_vanish1}. For this we only need to check that
$M:=\gr \Delta_{q+\ell}, N:=\gr\nabla_{q+\ell}$ satisfy the conditions (i)-(iii) above. (i) and (ii) are clear
(both $\Delta_{q+\ell},\nabla_{q+\ell}$ are cyclic modules and so are $\gr\Delta_{q+\ell},\gr\nabla_{q+\ell}^\vee$).
We only need to check (iii), i.e., that the fibers of  $\gr \Delta_{q+\ell},\gr\nabla_{q+\ell}^\vee$  at $p\neq 0$
are $\Delta_{p,0},\nabla_{p,0}^\vee$, respectively. We will give a proof for $\Delta$, for $\nabla$ it is analogous.

{\it Step 2.} In this step  we will reduce the proof of the equality between 
the fiber of $\gr\Delta_{q+\ell}$ at $p$ and $\Delta_{p,0}$ to the claim that a
certain completion vanishes. 

Consider the quotient $\A_{\ell+\hbar q,\hbar}$ of $\A_{\param,\hbar}$ by the ideal of the plane $(\ell+ h q,h)$ in $\C[\param,\hbar]$.
Consider the left ideal $I_{\ell+\hbar q,\hbar}:=\A_{\ell+\hbar q,\hbar}\A_{\ell+\hbar q,\hbar}^{>0}$. Of course,
$\Delta_{\ell+\hbar q,\hbar}=\A_{\ell+\hbar q,\hbar}/I_{\ell+\hbar q,\hbar}$. Let $\tilde{I}_{\ell+\hbar q,\hbar}$ denote the $\hbar$-saturation
of $I_{\hbar,\ell}$, i.e., $\tilde{I}_{\ell+\hbar q,\hbar}$ consists of all elements $a\in \A_{\ell+\hbar q,\hbar}$ such that
$\hbar^k a\in I_{\ell+\hbar q,\hbar}$ for some $k>0$. Let $\tilde{I}_{\ell,0}$ be the specialization of $\tilde{I}_{\ell+\hbar q,\hbar}$
at $\hbar=0$. Then, more or less by definition, $\gr \Delta_{q+\ell}=\A_{\ell,0}/\tilde{I}_{\ell,0}$.
So we only need to show that the $\A_{\ell,0}$-module $M_{\ell}:=\tilde{I}_{\ell,0}/I_{\ell,0}$ (that is the kernel
of a natural epimorphism $\Delta_{\ell,0}\rightarrow \gr \Delta_{q+\ell}$) is supported on $\Spec(\A_{0,0})$.

Since $M_{\ell}\subset \Delta_{\ell,0}$,  it follows that $M_{\ell}(i)=0$ for $i>0$. So $M_{\ell}$ is supported on the repelling
locus of the $\bT$-action. Therefore it is enough to show that any $\bT$-fixed point $x\in \scrN_p$ with $p\in \ell\setminus \{0\}$
does not lie in the support of $M_{p}$, the specialization of $M_\ell$ at $p$.
Let $y\in \scrN_\ell\quo \bT$ be the image of $x$. It will suffice to check that $M_p^{\wedge_y}:=\C[\scrN_p\quo \bT]^{\wedge_y}\otimes_{\C[\scrN_p\quo \bT]}M_p$ is zero.

{\it Step 3.} Here we will reduce the proof of $M_p^{\wedge_y}=0$ to the claim that a certain ideal in a non-commutative
completion of $\A_{p+\hbar q,\hbar}$ is $\hbar$-saturated.

By definition, $M_p$ is the specialization of $\tilde{I}_{\ell+\hbar q,\hbar}/I_{\ell+\hbar q,\hbar}$ at $(p,0)$. Let $\tilde{I}_{p+\hbar q,\hbar}$
be the $\hbar$-saturation of $I_{p+\hbar q,\hbar}$. Clearly, the specialization of $\tilde{I}_{\ell+\hbar q,\hbar}|_p$ of $\tilde{I}_{\ell+\hbar q,\hbar}$ at $p$
is contained in $\tilde{I}_{p+\hbar q,\hbar}$. The module $M_p^{\wedge_y}$ is the specialization of the quotient
$$\left((\A_{\hbar,p}^{\bT})^{\wedge_y}\otimes_{\A_{\hbar,p}^{\bT}}\tilde{I}_{\ell+\hbar q,\hbar}|_p\right)/
\left((\A_{\hbar,p}^{\bT})^{\wedge_y}\otimes_{\A_{\hbar,p}^{\bT}}I_{p+\hbar q,\hbar}\right)$$
at $\hbar=0$. Here $(\A_{\hbar,p}^{\bT})^{\wedge_y}$ is the completion of $\A_{\hbar,p}^{\bT}$ at the maximal ideal that is
the preimage under the epimorphism $\A_{\hbar,p}^{\bT}\twoheadrightarrow \C[\scrN_p\quo \bT]$ of the maximal ideal of $y$.
Similarly, to \cite{BLPWquant}, Step 3 of the proof of Proposition 5.4.4,
we see that  $(\A_{\hbar,p}^{\bT})^{\wedge_y}$ is a flat right module over $\A_{\hbar,p}^{\bT}$.
So $$(\A_{\hbar,p}^{\bT})^{\wedge_y}\otimes_{\A_{\hbar,p}^{\bT}}\tilde{I}_{\ell+\hbar q,\hbar}|_p\hookrightarrow
(\A_{\hbar,p}^{\bT})^{\wedge_y}\otimes_{\A_{\hbar,p}^{\bT}}\tilde{I}_{p+\hbar q,\hbar}.$$
So to check that $M_p^{\wedge_y}$ is zero it is enough to show that the quotient
$$\left((\A_{\hbar,p}^{\bT})^{\wedge_y}\otimes_{\A_{\hbar,p}^{\bT}}\tilde{I}_{p+\hbar q,\hbar}\right)/
\left((\A_{\hbar,p}^{\bT})^{\wedge_y}\otimes_{\A_{\hbar,p}^{\bT}}I_{p+\hbar q,\hbar}\right)$$
has no $\hbar$-torsion. 

In fact, we will need a few more reductions. First, let
$\widehat{\otimes}$ denote the tensor product followed by the $\hbar$-adic completion. The claim on the absence
of the $\hbar$-torsion is the same as $$(\A_{\hbar,p}^{\bT})^{\wedge_y}\widehat{\otimes}_{\A_{\hbar,p}^{\bT}}\tilde{I}_{p+\hbar q,\hbar}=
(\A_{\hbar,p}^{\bT})^{\wedge_y}\widehat{\otimes}_{\A_{\hbar,p}^{\bT}}I_{p+\hbar q,\hbar}.$$
Also let us notice that the left hand side lies in the $\hbar$-saturation of the right hand side. So it is enough
to show that the left ideal  $$(\A_{\hbar,p}^{\bT})^{\wedge_y}\widehat{\otimes}_{\A_{\hbar,p}^{\bT}}I_{p+\hbar q,\hbar}\subset \A_{\hbar,p}^{\wedge_{y,\hbar}}:=(\A_{\hbar,p}^{\bT})^{\wedge_y}\widehat{\otimes}_{\A_{\hbar,p}^{\bT}}\A_{p+\hbar q,\hbar}$$
is $\hbar$-saturated.

{\it Step 4.} Here we will investigate some properties of $(\A_{\hbar,p}^{\bT})^{\wedge_y}\widehat{\otimes}_{\A_{\hbar,p}^{\bT}}I_{p+\hbar q,\hbar}$.

Let us show that the left ideal of interest is closed in the $\hbar$-adic topology.
The algebra $\A_{p,0}^{\wedge_y}$ is the algebra of $\bT$-finite vectors in
the completion $\A_{p,0}^{\wedge_x}$. The latter is Noetherian. From here it is easy to deduce that $\A_{p,0}^{\wedge_y}$
is Noetherian, compare with \cite{Gin08}, the proof of Lemma 2.4.2. The usual Hilbert argument (for power series) can be used
now to show that any left ideal in $\A^{\wedge_{y,\hbar}}_{p+\hbar q,\hbar}$ is closed in the $\hbar$-adic topology.

Modulo $\hbar$, the left ideal $(\A_{p+\hbar q,\hbar}^{\bT})^{\wedge_y}\widehat{\otimes}_{\A_{p+\hbar q,\hbar}^{\bT}}I_{p+\hbar q,\hbar}$
is a complete intersection generated by some $\bT$-equivariant elements $x_1,\ldots,x_m\in \A_{p,0}^{>0}$, see Lemma \ref{Lem:generat}.
Let us lift $x_1,\ldots,x_m$ to $\bT$-semiinvariant elements $\tilde{x}_1,\ldots,\tilde{x}_m\in \A_{p+\hbar q,\hbar}^{>0}$. %We claim that the $\hbar$-saturation $J_\hbar$ %of $(\A_{p+\hbar q,\hbar}^{\bT})^{\wedge_y}\widehat{\otimes}_{\A_{p+\hbar q,\hbar}^{\bT}}I_{p+\hbar q,\hbar}$ is generated by $\tilde{x}_1,\ldots,
%\tilde{x}_m$ (and hence coincides with $(\A_{p+\hbar q,\hbar}^{\bT})^{\wedge_y}\widehat{\otimes}_{\A_{p+\hbar q,\hbar}^{\bT}}I_{p+\hbar q,\hbar}$;
%this will complete the proof of the proposition).
We claim that $\tilde{x}_1,\ldots,\tilde{x}_m$ generate the ideal $(\A_{p+\hbar q,\hbar}^{\bT})^{\wedge_y}\widehat{\otimes}_{\A_{p+\hbar q,\hbar}^{\bT}}I_{p+\hbar q,\hbar}$.
To establish this it is enough to check that any element of $\A^{\wedge_{y,\hbar},>0}_{p+\hbar q,\hbar}$ lies in the left ideal  generated by $\tilde{x}_1,\ldots,\tilde{x}_m$. This easily follow from the observations that 
\begin{itemize}
\item $\A^{\wedge_{y,\hbar}}_{p+\hbar q,\hbar}(i)$
is a closed and $\hbar$-saturated subspace of $\A^{\wedge_{y,\hbar}}_{p+\hbar q,\hbar}$ for every $i$ 
\item and that $\A^{\wedge_{y,\hbar},>0}_{p+\hbar q,\hbar}$
modulo $\hbar$ lies in the ideal generated by $x_1,\ldots,x_m$.
\end{itemize}

{\it Step 5.} This step will complete the proof of the claim that the left ideal $$(\A_{p+\hbar q,\hbar}^{\bT})^{\wedge_y}\widehat{\otimes}_{\A_{p+\hbar q,\hbar}^{\bT}}I_{p+\hbar q,\hbar}\subset \A_{p+\hbar q,\hbar}^{\wedge_{y,\hbar}}$$
is $\hbar$-saturated and hence the proof of the proposition, as well. 

Thanks to the previous step,  it is enough to show that the left ideal generated by $\tilde{x}_1,\ldots,\tilde{x}_m$ is $\hbar$-saturated.
This a corollary of a more general statement:
that the coisotropic complete intersection always admits a quantization, but we are going to provide a proof here
since we do not know a reference for that fact.

Assume the converse, let $a\in \A_{p+\hbar q,\hbar}^{\wedge_{y,\hbar}}$ be such that $\hbar a\in \A^{\wedge_{y,\hbar}}_{p+\hbar q,\hbar}(\tilde{x}_1,\ldots,\tilde{x}_m)$ but $a\not\in \A^{\wedge_{y,\hbar}}_{p+\hbar q,\hbar}(\tilde{x}_1,\ldots,\tilde{x}_m)$. 
Let $\hbar a=\sum_{i=1}^m \tilde{b}_i \tilde{x}_i$ for some elements $\tilde{b}_1,\ldots,\tilde{b}_m\in \A^{\wedge_{y,\hbar}}_{p+\hbar q,\hbar}$.
Let $b_1,\ldots,b_m$ be the classes of $\tilde{b}_1,\ldots,\tilde{b}_m$ modulo $\hbar$. Then not all of $b_1,\ldots,b_m$ are 0
and we have $\sum_{i=1}^m b_i x_i=0$. From the exactness of the Koszul complex, we deduce that there are elements
$b_{ij}\in \A^{\wedge_y}_{p,0}$ with $b_{ij}=-b_{ji}$ and $b_i=\sum_{j} b_{ij}x_j$. Let us choose liftings
$\tilde{b}_{ij}$ of $b_{ij}$ to $\A^{\wedge_{\hbar,y}}_{p+\hbar q,\hbar}$ so that $\tilde{b}_{ij}=-\tilde{b}_{ji}$. Set $c_{i}=\hbar^{-1}(\tilde{b}_i-\sum_{j=1}^m
\tilde{b}_{ij}\tilde{x}_j)$. We get
\begin{equation*}%\label{eq:eq1}
\hbar a=\sum_{i=1}^m \tilde{b}_i\tilde{x}_i=\sum_{i,j=1}^m \tilde{b}_{ij}\tilde{x}_j\tilde{x_i}+ \hbar c_{ij} \tilde{x}_i=\hbar(\sum_{i<j}\tilde{b}_{ij}\frac{1}{\hbar}[\tilde{x}_j,\tilde{x}_i]+\sum_{i,j}c_{ij}\tilde{x}_i).
\end{equation*}
To complete the proof it remains to check that $\frac{1}{\hbar}[\tilde{x}_i,\tilde{x}_j]$ lies in the left ideal
generated by $\tilde{x}_i$. This follows from the observation that $\frac{1}{\hbar}[\tilde{x}_i,\tilde{x}_j]\in \A_{p+\hbar q,\hbar}^{>0}$.
%The generation claim follows from the facts that %$(\A_{p+\hbar q,\hbar}^{\bT})^{\wedge_y}\widehat{\otimes}_{\A_{p+\hbar q,\hbar}^{\bT}}I_{p+\hbar q,\hbar}$ is closed in the $\hbar$-adic
%topology and is generated by
\end{proof}

\begin{corollary}\label{Cor:vanish}
Let $q\in\param^0$.
For each $n>0$  there is an open subset $\ell^0\subset\ell$ such that
$\Tor^{\A_{q+p}}_i(\nabla^\vee_{q+p},\Delta_{q+p})=0$ for $0<i<n$ and all $p\in \ell^0$.
\end{corollary}
\begin{proof}
The $\C[\ell]$-module $\Tor^{\A_{q+\ell}}_i(\nabla^\vee_{q+\ell},\Delta_{q+\ell})$ is finite dimensional
by Proposition \ref{Lem:Ext_vanish2} and so is supported in finitely many points of $\ell$. Let $\ell^0$
be the complement of the supports of $\Tor^{\A_{q+\ell}}_i(\nabla^\vee_{q+\ell},\Delta_{q+\ell})$
for $0<i<n$ in the intersection $\ell\cap \param^0$. We claim that $\Tor^{\A_{q+p}}_i(\nabla^\vee_{q+p},\Delta_{q+p})=0$ for
every $p\in \ell^0$. Indeed, let $z$ be a coordinate on $\ell$ near $p$ so that $\C[\ell]=\C[z]$ and the maximal
ideal of $p$ is generated by $z$. Then $\Tor_i^{\A_{q+p}}(\nabla_{q+p}^\vee,\Delta_{q+p})=\Tor_i^{\A_{q+\ell}}(\nabla_{q+\ell}^\vee, \Delta_{q+p})$. Also we have the short exact sequence
$$0\rightarrow \Delta_{q+\ell}\xrightarrow{z}\Delta_{q+\ell}\rightarrow \Delta_{q+p}\rightarrow 0$$
which yields the long exact sequence
\begin{align*}&\Tor_i^{\A_{q+\ell}}(\nabla_{q+\ell}^\vee, \Delta_{q+\ell})\xrightarrow{z}\Tor_i^{\A_{q+\ell}}(\nabla_{q+\ell},\Delta_{q+\ell})\rightarrow \Tor_i^{\A_{q+\ell}}(\nabla_{q+\ell}^\vee, \Delta_{q+p})\rightarrow \\&\rightarrow\Tor_{i-1}^{\A_{q+\ell}}(\nabla_{q+\ell}^\vee, \Delta_{q+\ell})\xrightarrow{z} \Tor_{i-1}^{\A_{q+\ell}}(\nabla_{q+\ell}^\vee, \Delta_{q+\ell}).\end{align*}
The first arrow is bijective, thanks to  our choice of $p$, for any $i$ with $0<i<n$, while the last arrow is bijective
for any $i$ with $1<i\leqslant n$. So we see that $\Tor_i^{\A_{q+p}}(\nabla_{q+p}^\vee,\Delta_{q+p})=0$ for all
$i$ with $1<i<n$. But $\Tor_1^{\A_{q+p}}(\nabla_{q+p}^\vee,\Delta_{q+p})=\Ext^1_{\A_{q+p}}(\Delta_{q+p},\nabla_{q+p})^*$.
The former is zero by our initial assumptions on $p$.
\end{proof}

\noindent {\bf Proof of Theorem \ref{Thm:ext_coinc1}:}
%Choose $p\in \param^0$ and let $\gamma$ be as in the discussion after
%the theorem. 
Let $\ell$ that passes through $p$ which is not parallel to any
discriminant hyperplane. Apply Corollary \ref{Cor:vanish} to $n=\dim X+1$ and the line
$\ell$. We will get $\Ext_{\A_{p'}}^i(\Delta_{p'},\nabla_{p'})=(\Tor_i^{\A_{p'}}(\nabla_{p'}^\vee, \Delta_{p'}))^*$
for any $i\in \{1,\ldots,\dim \fM\}$ provided $p'$ avoids the finitely
number of bad points. Hence we can choose $p'$
in the form $p+n\eta$ for $n\gg 0$; for $n$ sufficiently large, the
algebra $A_{p'}$ has finite global dimension equal to the dimension of
$\fM$ since 
localization holds by \cite[Corollary 5.17]{BLPWquant}.  The desired conclusion follows.\qed

\bibliography{./symplectic}
\bibliographystyle{amsalpha}
\end{document}